\documentclass[a4paper]{birkjour}

\usepackage[utf8]{inputenc}
\usepackage[T1]{fontenc}
\usepackage[english]{babel}
\usepackage{mathtools,amsmath,amssymb,amsfonts,amsthm,mathrsfs}
\mathtoolsset{showonlyrefs}
\usepackage{enumerate}
\usepackage[shortlabels]{enumitem}
\usepackage[pdftex]{hyperref}

\usepackage{geometry}
\usepackage{graphicx,color}
\usepackage{tikz,pgfplots,pgf}
\usetikzlibrary{patterns}
\usepackage{subcaption}
\usepackage[font=small]{caption}

\definecolor{fst-blue}{RGB}{0,128,255}
\definecolor{fst-orange}{RGB}{255,128,0}
\definecolor{fst-purple}{RGB}{85,0,130}
\definecolor{fst-red}{RGB}{255,0,130}
\definecolor{fst-gray}{RGB}{200,200,200}
\definecolor{fst-black}{RGB}{20,20,20}
\definecolor{fst-yellow}{RGB}{255,200,80}
\definecolor{fst-green}{RGB}{128,255,0}

\definecolor{darkpastelgreen}{rgb}{0.01, 0.75, 0.24}
\definecolor{darkpastelblue}{rgb}{0.47, 0.62, 0.8}
\definecolor{darktangerine}{rgb}{1.0, 0.66, 0.07}
\definecolor{bostonuniversityred}{rgb}{0.8, 0.0, 0.0}
\definecolor{darkgray}{rgb}{0.66, 0.66, 0.66}

\newtheorem{theorem}{Theorem}[section]
\newtheorem{lemma}[theorem]{Lemma}
\newtheorem{proposition}[theorem]{Proposition}
\newtheorem{corollary}[theorem]{Corollary}
\theoremstyle{definition}
\newtheorem{remark}[theorem]{Remark}
\theoremstyle{definition}
\newtheorem*{definition*}{Definition}

\numberwithin{equation}{section}

\begin{document}

\title[An indefinite parameter-dependent Neumann problem]{On the number of positive solutions \\ to an indefinite parameter-dependent \\ Neumann problem}

\author[G.~Feltrin]{Guglielmo Feltrin}
\address{
Department of Mathematics, Computer Science and Physics, University of Udine\\
Via delle Scienze 206, 33100 Udine, Italy}
\email{guglielmo.feltrin@uniud.it}

\author[E.~Sovrano]{Elisa Sovrano}
\address{
\'{E}cole des Hautes \'{E}tudes en Sciences Sociales,
Centre d'Analyse et de Math\'{e}matique Sociales (CAMS), CNRS\\
54 Boulevard Raspail, 75006 Paris, France}
\email{elisa.sovrano@ehess.fr}

\author[A.~Tellini]{Andrea Tellini}
\address{
Department of Applied Mathematics in Industrial Engineering,
E.T.S.I.D.I., Universidad Polit\'{e}cnica de Madrid\\
Ronda de Valencia 3, 28012 Madrid, Spain}
\email{andrea.tellini@upm.es}

\thanks{Work written under the auspices of the Grup\-po Na\-zio\-na\-le per l'Anali\-si Ma\-te\-ma\-ti\-ca, la Pro\-ba\-bi\-li\-t\`{a} e le lo\-ro Appli\-ca\-zio\-ni (GNAMPA) of the Isti\-tu\-to Na\-zio\-na\-le di Al\-ta Ma\-te\-ma\-ti\-ca (INdAM). The second author has been supported by the Fondation Sciences Math\'ematiques de Paris (FSMP) through the project: ``Reaction-Diffusion Equations in Population Genetics: a study of the influence of geographical barriers on traveling waves and non-constant stationary solutions''. The third author has received funding from project PGC2018-097104-B-100 of the Spanish Ministry of Science, Innovation and Universities and from the Escuela T\'{e}cnica Superior de Ingenier\'{i}a y Dise\~{n}o Industrial of the Universidad Polit\'{e}cnica de Madrid.
\\
\textbf{Preprint -- January 2021}} 



\date{}

\dedicatory{}

\begin{abstract}
We study the second-order boundary value problem
\begin{equation*}
\begin{cases}
\, -u''=a_{\lambda,\mu}(t) \, u^{2}(1-u), & t\in(0,1), \\
\, u'(0)=0, \quad u'(1)=0,
\end{cases}
\end{equation*}
where $a_{\lambda,\mu}$ is a step-wise indefinite weight function, precisely $a_{\lambda,\mu}\equiv\lambda$ in $[0,\sigma]\cup[1-\sigma,1]$ and $a_{\lambda,\mu}\equiv-\mu$ in $(\sigma,1-\sigma)$, for some $\sigma\in\left(0,\frac{1}{2}\right)$, with $\lambda$ and $\mu$ positive real parameters. We investigate the topological structure of the set of positive solutions which lie in $(0,1)$ as $\lambda$ and $\mu$ vary. Depending on $\lambda$ and based on a phase-plane analysis and on time-mapping estimates, our findings lead to three different (from the topological point of view) global bifurcation diagrams of the solutions in terms of the parameter $\mu$. Finally, for the first time in the literature, a qualitative bifurcation diagram concerning the number of solutions in the $(\lambda,\mu)$-plane is depicted. The analyzed Neumann problem has an application in the analysis of stationary solutions to reaction-diffusion equations in population genetics driven by migration and selection.
\end{abstract} 

\maketitle

\newpage

\tableofcontents

\section{Introduction and main results}\label{section-1}

In this paper, we are concerned with the parameter-dependent Neumann problem
\begin{equation}\label{eq-main}
\begin{cases}
\, -u''(t)=a_{\lambda,\mu}(t) \, g(u(t)), & t\in(0,1), \\
\, 0 \leq u(t) \leq 1, & t\in[0,1], \\
\, u'(0)=0, \quad u'(1)=0,
\end{cases}
\end{equation}
where the nonlinear term $g\colon [0,1]\to\mathbb{R}$ is 
\begin{equation}\label{eq-nonlin}
g(s)=s^2(1-s), \quad s\in[0,1],
\end{equation}
and the weight $a_{\lambda,\mu}\colon[0,1]\to\mathbb{R}$ is a step-wise function defined as
\begin{equation}\label{eq-weight}
a_{\lambda,\mu}(t)=
\begin{cases}
\, \lambda, & \text{if $t\in[0,\sigma]\cup[1-\sigma,1]$,} \\
\, -\mu, & \text{if $t\in(\sigma,1-\sigma)$,}
\end{cases}
\end{equation}
with $\lambda>0$, $\mu>0$, and $\sigma\in\left(0,\frac{1}{2}\right)$. 

The presence of a sign-changing weight term $a_{\lambda,\mu}$ is a necessary condition for the existence of non-constant positive solutions to~\eqref{eq-main} - it can be seen by integrating the differential equation in $(0,1)$ - and places~\eqref{eq-main} into the family of nonlinear problems with indefinite weight (cf.,~\cite{Fe-18} for an extensive bibliography on the subject). The logistic-type nonlinear term $g$ leads to the existence of two trivial solutions, namely $u\equiv0$ and $u\equiv1$. For this reason, we seek only non-constant positive solutions whose notion is made precise as follows.

\begin{definition*}
A \textit{solution} to problem~\eqref{eq-main} is a function $u\in \mathcal{C}^1([0,1])$ with $u'$ absolutely continuous, $0<u(t)<1$ for all $t\in[0,1]$, and such that $u$ solves the equation in~\eqref{eq-main} for a.e. $t\in(0,1)$ and satisfies the Neumann boundary conditions.
\end{definition*}

The particular form of problem~\eqref{eq-main} is relevant in the field of population genetics to study stationary solutions of reaction-diffusion equations driven by migration and selection processes acting on a single diallelic locus, as firstly introduced in~\cite{Na-75} (see also \cite{BrHe-90,Fl-75,He-81}). In this framework, $u(t)\in[0,1]$ denotes the frequency of one of the two alleles involved and $t$ represents the spatial variable. Moreover, the term $g(s)$ models the selection in case of complete dominance, and $a_{\lambda,\mu}(t)$ represents the environmental diversity which is reflected by the sign-changing in the direction of selection.

From the mathematical point of view, the last decade has experienced a huge interest in indefinite weight problems with a logistic-type nonlinearity from different perspectives
both in the ODE setting (e.g.,~\cite{BoFeSo-20,FeGi-20,FeSo-18non,FeSo-18na,LGMMTe-14,NaSuDu-19,Na-16,Na-20,NaSu-20,So-17jomb}) and in the PDE one (e.g.,~\cite{LGMMTe-13,LoNa-02,LoNaNi-13,LoNiSu-10,Na-16,NaNiSu-10,OmSo-20}).
The main questions addressed were the existence, the uniqueness as well as the multiplicity of non-constant positive solutions.
However, focusing on either a fixed number of sign-changes of the weight term or on the number of inflection points of the logistic term, no complete description of global bifurcation diagrams has yet been provided. We notice that some first steps in this direction are given in~\cite{FeSo-18non,NaSuDu-19,So-17jomb} through numerical investigations.

In the present work, we pursue the studies undertaken in~\cite{FeSo-18non,FeSo-18na}, where a general nonlinearity superlinear at $u=0$ is considered in place of $g$. Taking advantage of the precise expression of~\eqref{eq-main} and using the connection times approach (see~\cite{LGTe-14,LGTeZa-14,Te-18}), we  establish the shape of the bifurcation diagrams for the solutions of~\eqref{eq-main} by using $\mu$ and $\lambda$ as the main and the secondary bifurcation parameter, respectively.
Moreover, we obtain general non-existence, existence, and multiplicity results, depending on the values of $\mu$ and $\lambda$. Our first result in this sense provides sufficient conditions for the non-existence of solutions. 

\begin{theorem}\label{th-intro1}
The following assertions hold: 
\begin{enumerate}
\item[$(i)$] for all $\lambda>0$, there exists $\mu_0^{*}(\lambda)>0$ such that no solution to problem~\eqref{eq-main} exists for all $\mu\in(0,\mu_0^{*}(\lambda))$; 
\item[$(ii)$] there exists $\lambda^{*}>0$ such that for all $\lambda\in(0,\lambda^{*})$ there exists $\mu^{**}_{0}(\lambda)>\mu_0^{*}(\lambda)$ such that no solution to problem~\eqref{eq-main} exists for all $\mu>\mu_0^{**}(\lambda)$. 
\end{enumerate}
\end{theorem}
\noindent This result means, in particular, that no solution exists in a whole neighborhood of $(\lambda,\mu)=(\lambda,0)$ and in a whole neighborhood of $(\lambda,\mu)=(0,\mu)$. 
Thus, we provide a positive answer to~\cite[Conjecture~A]{FeSo-18non} in this framework.

As a counterpart of Theorem~\ref{th-intro1}, we obtain the following result of existence of solutions for problem~\eqref{eq-main}, with high multiplicity for some ranges of the parameters $\lambda$ and $\mu$.

\begin{theorem}\label{th-intro2}
	With the quantities given in Theorem~\ref{th-intro1}, the following assertions hold: 
	\begin{enumerate}
		\item[$(i)$] for all $\lambda\in(0,\lambda^{*})$, there exist $\mu^{*}_{1}(\lambda),\mu^{*}_{2}(\lambda),\mu^{**}_{2}(\lambda)\in(0,+\infty)$ satisfying
		\begin{equation*}
		\mu^{*}_{0}(\lambda)\leq\mu^{*}_{1}(\lambda)<\mu^{*}_{2}(\lambda)<\mu^{**}_{2}(\lambda)\leq\mu^{**}_{0}(\lambda),
		\end{equation*}
		such that
		\begin{itemize}[leftmargin=30pt,labelsep=5pt]
			\item[$(i.a)$] for all $\mu\in(\mu^{*}_{1}(\lambda),\mu^{**}_2(\lambda))$, problem~\eqref{eq-main} admits at least one solution;
			\item[$(i.b)$] for all $\mu\in(\mu^{*}_2(\lambda),\mu^{**}_2(\lambda))$, problem~\eqref{eq-main} admits at least two solutions;
		\end{itemize}
		\item[$(ii)$] for all $\lambda\in[\lambda^{*},+\infty)$, there exist $\mu^{*}_{1}(\lambda),\mu^{*}_{2}(\lambda),\mu^{*}_{4}(\lambda),\mu^{*}_{8}(\lambda)\in(0,+\infty)$ satisfying
		\begin{equation*}
		\mu^{*}_{0}(\lambda)\leq\mu^{*}_{1}(\lambda)<\mu^{*}_{2}(\lambda)\leq\mu^{*}_{4}(\lambda)\leq\mu^{*}_{8}(\lambda),
		\end{equation*}
		such that  
		\begin{itemize}[leftmargin=30pt,labelsep=5pt]
			\item[$(ii.a)$] for all $\mu>\mu^{*}_1(\lambda)$, problem~\eqref{eq-main} admits at least one solutions;
			\item[$(ii.b)$] for all $\mu>\mu^{*}_2(\lambda)$, problem~\eqref{eq-main} admits at least two solutions;
			\item[$(ii.c)$] for all $\mu>\mu^{*}_4(\lambda)$, problem~\eqref{eq-main} admits at least four solutions;
			\item[$(ii.d)$] for all $\mu>\mu^{*}_8(\lambda)$, problem~\eqref{eq-main} admits at least eight solutions.
		\end{itemize}
	\end{enumerate}
\end{theorem}

Accordingly, the above result provides a global description of the number of solutions of~\eqref{eq-main} as $(\lambda,\mu)$ varies in $(0,+\infty)\times\left(\mu^{*}_{1}(\lambda),\mu^{**}_{2}(\lambda)\right)$, where $\mu^{**}_{2}(\lambda)$ is meant as $+\infty$, for $\lambda\in[\lambda^{*},+\infty)$. We stress that Theorem~\ref{th-intro2} not only includes the high multiplicity for $\lambda\in[\lambda^{*},+\infty)$ and large $\mu$, as stated in \cite[Theorem~1.3]{FeSo-18non}, but, in addition, deals with the existence and multiplicity of the solutions of~\eqref{eq-main} in the whole $(\lambda,\mu)$-plane by dividing $\{(\lambda,\mu)\colon \lambda>0,\, \mu\in(\mu^{*}_{1}(\lambda),\mu^{**}_2(\lambda))\}$ in sub-regions in terms of the minimal number of solutions of~\eqref{eq-main}. As a consequence, apparently for the first time in the literature, we obtain the global bifurcation diagram in the $(\lambda,\mu)$-plane depicted in Figure~\ref{fig-bif-plane}. 
Another novelty of our results is that the value $\lambda^{*}$ is sharp in the sense that it divides the region of non-existence of solutions from that of high multiplicity when $\mu$ is large.
Moreover, as it will be apparent in Section~\ref{section-6}, the same value $\lambda^{*}$ discerns the behavior of the bifurcation diagrams with $\mu$ as bifurcation parameter, by distinguishing between those made by bounded branches connecting $0$ and $1$ (for $0<\lambda<\lambda^*$), and those with unbounded connected components (for $\lambda\in[\lambda^{*},+\infty)$). 

\begin{figure}[htb]
\begin{tikzpicture}
\begin{axis}[legend style={at={(axis cs:72,30)},anchor=north west},
  tick label style={font=\scriptsize},
  axis y line=middle, 
  axis x line=middle,
  xtick={0,15},
  ytick={0},
  xticklabels={$0$,$\lambda^*$}, 
  yticklabels={},
  xlabel={\small $\lambda$},
  ylabel={\small $\mu$},
every axis x label/.style={
    at={(ticklabel* cs:1.0)},
    anchor=west,
},
every axis y label/.style={
    at={(ticklabel* cs:1.0)},
    anchor=south,
}, set layers,
  width=10cm,
  height=6cm,
  xmin=-5,
  xmax=70,
  ymin=-3,
  ymax=30]
\addplot [fill=darkgray!50, draw=none, on layer=axis background, area legend] coordinates {(60, 0.00) (70, 0.00) (70.00, 0.06) (60, 0.06)};   
\addplot[draw=black, line width=0.6pt, smooth, samples=100, domain=4:70, forget plot] {0.25*x};	
\addplot [fill=bostonuniversityred!50, draw=none, on layer=axis background, area legend]coordinates {(0,0) (70.00, 17.50) (70.00, 0.00)};
\addplot [fill=white, draw=none, on layer=axis background, forget plot] coordinates {(0.00,0.00) (0.29,0.07) (0.40,0.10) (0.46,0.11) (0.58,0.14) (0.69,0.17) (0.80,0.20) (0.88,0.22) (0.92,0.23) (1.08,0.26) (1.17,0.29) (1.38,0.34) (1.62,0.39) (1.76,0.43) (1.84,0.45) (2.11,0.51) (2.36,0.57) (2.76,0.66) (3.34,0.79) (3.57,0.84) (3.68,0.87) (4.05,0.95) (4.79,1.11) (5.15,1.18) (5.53,1.26) (6.03,1.37) (6.45,1.45) (7.07,1.58) (7.28,1.62) (7.37,1.64) (7.62,1.68) (8.54,1.86) (9.21,1.98) (9.83,2.10) (10.88,2.29) (11.05,2.33) (11.12,2.34) (11.29,2.37) (11.97,2.48) (12.44,2.56) (12.89,2.63) (13.76,2.79) (13.93,2.81) (14.74,2.94) (15.11,3.00) (15.66,3.09) (16.11,3.16) (16.46,3.21) (16.58,3.23) (16.83,3.26) (17.83,3.41) (18.42,3.50) (18.79,3.55) (19.21,3.61) (19.60,3.66) (20.26,3.75) (20.61,3.80) (21.18,3.88) (21.68,3.95) (22.01,3.99) (22.11,4.00) (22.28,4.02) (23.43,4.17) (23.95,4.22) (24.86,4.35) (24.89,4.35) (25.79,4.46) (26.30,4.52) (27.42,4.65) (27.63,4.67) (27.75,4.68) (28.21,4.74) (28.55,4.77) (29.22,4.85) (29.47,4.87) (29.89,4.92) (30.70,5.00) (31.32,5.06) (32.18,5.16) (32.33,5.17) (33.16,5.25) (33.68,5.30) (34.71,5.40) (35.00,5.43) (35.18,5.45) (35.92,5.52) (35.98,5.53) (36.69,5.59) (36.84,5.60) (37.07,5.62) (38.22,5.72) (38.68,5.76) (39.39,5.83) (39.75,5.86) (40.46,5.92) (40.53,5.93) (41.29,5.99) (41.68,6.02) (42.37,6.07) (42.84,6.11) (43.94,6.20) (44.21,6.22) (44.40,6.24) (45.13,6.29) (45.42,6.32) (45.96,6.36) (46.05,6.36) (46.18,6.37) (47.54,6.47) (47.89,6.49) (48.40,6.53) (49.12,6.58) (49.74,6.61) (50.60,6.69) (50.70,6.69) (50.94,6.71) (51.58,6.75) (52.30,6.80) (52.78,6.83) (53.42,6.86) (53.90,6.90) (54.94,6.97) (55.26,6.99) (55.50,7.00) (57.10,7.10) (57.11,7.10) (57.11,7.10) (57.13,7.11) (58.03,7.16) (58.73,7.20) (58.95,7.21) (59.23,7.23) (60.36,7.29) (60.79,7.30) (61.35,7.35) (61.99,7.38) (62.63,7.42) (63.47,7.46) (63.62,7.47) (64.47,7.50) (65.26,7.56) (65.56,7.57) (66.32,7.61) (66.91,7.64) (67.65,7.68) (68.16,7.69) (68.56,7.72) (69.73,7.78) (71.00,7.79) (71.00,0) (0,0)};
\addplot [draw=black, line width=0.6pt, smooth, forget plot]coordinates {(0.00,0.00) (0.29,0.07) (0.40,0.10) (0.46,0.11) (0.58,0.14) (0.69,0.17) (0.80,0.20) (0.88,0.22) (0.92,0.23) (1.08,0.26) (1.17,0.29) (1.38,0.34) (1.62,0.39) (1.76,0.43) (1.84,0.45) (2.11,0.51) (2.36,0.57) (2.76,0.66) (3.34,0.79) (3.57,0.84) (3.68,0.87) (4.05,0.95) (4.79,1.11) (5.15,1.18) (5.53,1.26) (6.03,1.37) (6.45,1.45) (7.07,1.58) (7.28,1.62) (7.37,1.64) (7.62,1.68) (8.54,1.86) (9.21,1.98) (9.83,2.10) (10.88,2.29) (11.05,2.33) (11.12,2.34) (11.29,2.37) (11.97,2.48) (12.44,2.56) (12.89,2.63) (13.76,2.79) (13.93,2.81) (14.74,2.94) (15.11,3.00) (15.66,3.09) (16.11,3.16) (16.46,3.21) (16.58,3.23) (16.83,3.26) (17.83,3.41) (18.42,3.50) (18.79,3.55) (19.21,3.61) (19.60,3.66) (20.26,3.75) (20.61,3.80) (21.18,3.88) (21.68,3.95) (22.01,3.99) (22.11,4.00) (22.28,4.02) (23.43,4.17) (23.95,4.22) (24.86,4.35) (24.89,4.35) (25.79,4.46) (26.30,4.52) (27.42,4.65) (27.63,4.67) (27.75,4.68) (28.21,4.74) (28.55,4.77) (29.22,4.85) (29.47,4.87) (29.89,4.92) (30.70,5.00) (31.32,5.06) (32.18,5.16) (32.33,5.17) (33.16,5.25) (33.68,5.30) (34.71,5.40) (35.00,5.43) (35.18,5.45) (35.92,5.52) (35.98,5.53) (36.69,5.59) (36.84,5.60) (37.07,5.62) (38.22,5.72) (38.68,5.76) (39.39,5.83) (39.75,5.86) (40.46,5.92) (40.53,5.93) (41.29,5.99) (41.68,6.02) (42.37,6.07) (42.84,6.11) (43.94,6.20) (44.21,6.22) (44.40,6.24) (45.13,6.29) (45.42,6.32) (45.96,6.36) (46.05,6.36) (46.18,6.37) (47.54,6.47) (47.89,6.49) (48.40,6.53) (49.12,6.58) (49.74,6.61) (50.60,6.69) (50.70,6.69) (50.94,6.71) (51.58,6.75) (52.30,6.80) (52.78,6.83) (53.42,6.86) (53.90,6.90) (54.94,6.97) (55.26,6.99) (55.50,7.00) (57.10,7.10) (57.11,7.10) (57.11,7.10) (57.13,7.11) (58.03,7.16) (58.73,7.20) (58.95,7.21) (59.23,7.23) (60.36,7.29) (60.79,7.30) (61.35,7.35) (61.99,7.38) (62.63,7.42) (63.47,7.46) (63.62,7.47) (64.47,7.50) (65.26,7.56) (65.56,7.57) (66.32,7.61) (66.91,7.64) (67.65,7.68) (68.16,7.69) (68.56,7.72) (69.73,7.78) (70.00,7.79)};
\node at (axis cs:61.2,4.7) {\scriptsize{$\mu^{*}_0(\lambda)$}};
\addplot [fill=darkgray!50, draw=none, on layer=axis background, forget plot] coordinates {(0.00,0.00) (0.33,0.06) (0.46,0.08) (0.57,0.10) (0.66,0.11) (0.77,0.13) (0.92,0.16) (0.98,0.17) (1.14,0.20) (1.15,0.20) (1.31,0.23) (1.38,0.24) (1.54,0.26) (1.64,0.28) (1.84,0.31) (1.98,0.34) (2.30,0.39) (2.32,0.39) (2.64,0.45) (2.76,0.47) (3.03,0.51) (3.31,0.56) (3.68,0.61) (3.98,0.66) (4.61,0.76) (4.79,0.79) (5.33,0.87) (5.53,0.90) (5.93,0.96) (6.70,1.08) (7.37,1.18) (7.41,1.18) (8.07,1.28) (8.73,1.37) (9.21,1.44) (9.46,1.47) (10.13,1.57) (10.20,1.58) (10.85,1.67) (11.05,1.69) (11.44,1.74) (12.26,1.85) (12.89,1.93) (13.67,2.03) (14.07,2.08) (14.74,2.17) (15.10,2.21) (15.66,2.28) (16.39,2.37) (16.54,2.39) (16.58,2.39) (16.65,2.40) (17.99,2.55) (18.42,2.60) (19.16,2.69) (19.44,2.72) (20.26,2.80) (20.91,2.88) (21.63,2.96) (22.11,3.01) (22.39,3.04) (23.03,3.10) (23.55,3.16) (23.87,3.19) (23.95,3.20) (24.07,3.21) (25.37,3.34) (25.79,3.38) (26.46,3.44) (26.87,3.48) (27.63,3.55) (28.38,3.63) (28.82,3.67) (29.47,3.73) (29.90,3.76) (31.15,3.88) (31.32,3.89) (31.42,3.90) (31.93,3.95) (32.24,3.97) (32.96,4.03) (33.16,4.05) (33.46,4.07) (34.50,4.16) (35.00,4.19) (35.74,4.26) (36.05,4.29) (36.69,4.34) (36.84,4.35) (37.60,4.41) (38.00,4.44) (38.68,4.49) (39.17,4.53) (40.23,4.61) (40.53,4.63) (40.73,4.65) (41.45,4.70) (41.91,4.74) (42.30,4.76) (42.37,4.77) (42.46,4.77) (43.89,4.87) (44.21,4.90) (44.65,4.93) (45.48,4.98) (46.05,5.01) (46.84,5.08) (47.07,5.09) (47.66,5.13) (47.89,5.15) (48.67,5.20) (49.01,5.22) (49.74,5.25) (50.28,5.30) (51.17,5.35) (51.58,5.38) (51.88,5.40) (53.32,5.48) (53.42,5.49) (53.50,5.49) (54.02,5.53) (54.34,5.54) (55.12,5.59) (55.26,5.60) (55.45,5.61) (56.75,5.68) (57.11,5.69) (57.57,5.73) (58.38,5.77) (58.95,5.80) (59.69,5.84) (60.01,5.86) (60.79,5.89) (61.65,5.95) (61.79,5.96) (62.63,6.00) (63.29,6.03) (63.88,6.06) (64.47,6.08) (64.94,6.11) (65.97,6.17) (66.32,6.18) (66.59,6.20) (68.05,6.27) (68.16,6.27) (68.25,6.28) (69.02,6.32) (69.08,6.32) (69.91,6.36) (70.00,6.36) (70.00,0) (0,0)};
\addplot [color=black, line width=0.6pt, smooth, forget plot]coordinates {(0.00,0.00) (0.33,0.06) (0.46,0.08) (0.57,0.10) (0.66,0.11) (0.77,0.13) (0.92,0.16) (0.98,0.17) (1.14,0.20) (1.15,0.20) (1.31,0.23) (1.38,0.24) (1.54,0.26) (1.64,0.28) (1.84,0.31) (1.98,0.34) (2.30,0.39) (2.32,0.39) (2.64,0.45) (2.76,0.47) (3.03,0.51) (3.31,0.56) (3.68,0.61) (3.98,0.66) (4.61,0.76) (4.79,0.79) (5.33,0.87) (5.53,0.90) (5.93,0.96) (6.70,1.08) (7.37,1.18) (7.41,1.18) (8.07,1.28) (8.73,1.37) (9.21,1.44) (9.46,1.47) (10.13,1.57) (10.20,1.58) (10.85,1.67) (11.05,1.69) (11.44,1.74) (12.26,1.85) (12.89,1.93) (13.67,2.03) (14.07,2.08) (14.74,2.17) (15.10,2.21) (15.66,2.28) (16.39,2.37) (16.54,2.39) (16.58,2.39) (16.65,2.40) (17.99,2.55) (18.42,2.60) (19.16,2.69) (19.44,2.72) (20.26,2.80) (20.91,2.88) (21.63,2.96) (22.11,3.01) (22.39,3.04) (23.03,3.10) (23.55,3.16) (23.87,3.19) (23.95,3.20) (24.07,3.21) (25.37,3.34) (25.79,3.38) (26.46,3.44) (26.87,3.48) (27.63,3.55) (28.38,3.63) (28.82,3.67) (29.47,3.73) (29.90,3.76) (31.15,3.88) (31.32,3.89) (31.42,3.90) (31.93,3.95) (32.24,3.97) (32.96,4.03) (33.16,4.05) (33.46,4.07) (34.50,4.16) (35.00,4.19) (35.74,4.26) (36.05,4.29) (36.69,4.34) (36.84,4.35) (37.60,4.41) (38.00,4.44) (38.68,4.49) (39.17,4.53) (40.23,4.61) (40.53,4.63) (40.73,4.65) (41.45,4.70) (41.91,4.74) (42.30,4.76) (42.37,4.77) (42.46,4.77) (43.89,4.87) (44.21,4.90) (44.65,4.93) (45.48,4.98) (46.05,5.01) (46.84,5.08) (47.07,5.09) (47.66,5.13) (47.89,5.15) (48.67,5.20) (49.01,5.22) (49.74,5.25) (50.28,5.30) (51.17,5.35) (51.58,5.38) (51.88,5.40) (53.32,5.48) (53.42,5.49) (53.50,5.49) (54.02,5.53) (54.34,5.54) (55.12,5.59) (55.26,5.60) (55.45,5.61) (56.75,5.68) (57.11,5.69) (57.57,5.73) (58.38,5.77) (58.95,5.80) (59.69,5.84) (60.01,5.86) (60.79,5.89) (61.65,5.95) (61.79,5.96) (62.63,6.00) (63.29,6.03) (63.88,6.06) (64.47,6.08) (64.94,6.11) (65.97,6.17) (66.32,6.18) (66.59,6.20) (68.05,6.27) (68.16,6.27) (68.25,6.28) (69.02,6.32) (69.08,6.32) (69.91,6.36) (70.00,6.36)};
\node at (axis cs:50.8,8.0) [rotate=4]{\scriptsize{$\mu^{*}_1(\lambda)$}}; 
\addplot [preaction={fill, darkpastelblue!50}, pattern color=darkpastelblue!20, pattern=crosshatch dots, draw=none, on layer=axis background, area legend]coordinates {(0,0) (70.00, 17.50) (70.00, 30.00) (0, 30.00)};
\node at (axis cs:61.2,16.6) [rotate=18]{\scriptsize{$\mu^{*}_2(\lambda)$}};
\addplot[draw=black, line width=0.6pt, smooth, samples=100, domain=15:70, forget plot] {0.32*x};	
\node at (axis cs:50.5,17.6) [rotate=22]{\scriptsize{$\mu^{*}_4(\lambda)$}};
\addplot [fill=darktangerine!50, draw=none, on layer=axis background, area legend]coordinates  {(15, 0.32*15) (70, 0.32*70) (70, 0.4*70) (15, 0.4*15)};
\addplot[draw=black, line width=0.6pt, smooth, samples=100, domain=15:70, forget plot] {0.4*x};	
\node at (axis cs:61.2,25.9) [rotate=27]{\scriptsize{$\mu^{*}_8(\lambda)$}};
\addplot [fill=darkpastelgreen!50, draw=none, on layer=axis background, area legend]coordinates {(15, 0.4*15) (70, 0.4*70) (70,30) (15, 30)};
\addplot [fill=white, draw=none, on layer=axis background, forget plot] coordinates {(0.00,0.00) (0.06,0.01) (0.06,0.01) (0.07,0.01) (0.12,0.02) (0.13,0.03) (0.14,0.03) (0.18,0.04) (0.19,0.04) (0.22,0.04) (0.23,0.05) (0.26,0.05) (0.29,0.06) (0.29,0.06) (0.32,0.07) (0.35,0.07) (0.39,0.08) (0.43,0.09) (0.47,0.10) (0.52,0.11) (0.58,0.12) (0.59,0.12) (0.65,0.13) (0.70,0.15) (0.78,0.16) (0.87,0.18) (0.94,0.20) (1.03,0.22) (1.16,0.25) (1.17,0.25) (1.29,0.28) (1.40,0.30) (1.45,0.32) (1.55,0.34) (1.63,0.36) (1.74,0.39) (1.81,0.40) (1.86,0.41) (2.04,0.46) (2.07,0.47) (2.09,0.47) (2.20,0.50) (2.26,0.52) (2.32,0.53) (2.33,0.53) (2.34,0.54) (2.55,0.59) (2.59,0.60) (2.64,0.62) (2.78,0.66) (2.84,0.68) (2.94,0.70) (3.00,0.72) (3.10,0.75) (3.19,0.78) (3.23,0.79) (3.24,0.79) (3.36,0.83) (3.45,0.86) (3.55,0.89) (3.62,0.91) (3.67,0.93) (3.85,0.98) (3.88,0.99) (3.90,1.00) (4.00,1.03) (4.12,1.07) (4.14,1.08) (4.17,1.09) (4.34,1.15) (4.40,1.17) (4.48,1.20) (4.56,1.23) (4.66,1.27) (4.78,1.31) (4.80,1.32) (4.91,1.37) (4.99,1.40) (5.12,1.45) (5.17,1.47) (5.21,1.48) (5.37,1.55) (5.42,1.57) (5.43,1.58) (5.45,1.59) (5.63,1.67) (5.69,1.69) (5.78,1.73) (5.84,1.76) (5.95,1.81) (6.05,1.86) (6.12,1.89) (6.21,1.94) (6.26,1.96) (6.46,2.07) (6.46,2.07) (6.47,2.07) (6.47,2.07) (6.67,2.18) (6.72,2.21) (6.82,2.27) (6.87,2.29) (6.98,2.36) (7.07,2.41) (7.19,2.49) (7.24,2.52) (7.27,2.53) (7.35,2.59) (7.46,2.66) (7.50,2.68) (7.57,2.73) (7.66,2.79) (7.73,2.84) (7.76,2.86) (7.85,2.93) (7.98,3.02) (8.02,3.05) (8.04,3.07) (8.08,3.10) (8.22,3.21) (8.28,3.26) (8.41,3.36) (8.41,3.36) (8.53,3.47) (8.59,3.52) (8.70,3.62) (8.76,3.68) (8.79,3.71) (8.88,3.79) (8.94,3.85) (9.05,3.96) (9.11,4.02) (9.22,4.14) (9.28,4.20) (9.31,4.23) (9.42,4.36) (9.44,4.39) (9.57,4.53) (9.61,4.58) (9.67,4.66) (9.77,4.78) (9.83,4.86) (9.92,4.98) (10.06,5.17) (10.07,5.20) (10.09,5.21) (10.17,5.33) (10.22,5.42) (10.34,5.61) (10.37,5.64) (10.39,5.69) (10.51,5.88) (10.60,6.04) (10.65,6.12) (10.69,6.21) (10.78,6.37) (10.86,6.53) (10.91,6.63) (10.96,6.72) (11.04,6.89) (11.12,7.07) (11.16,7.16) (11.19,7.24) (11.28,7.44) (11.38,7.68) (11.40,7.73) (11.41,7.76) (11.51,8.01) (11.51,8.02) (11.60,8.28) (11.62,8.32) (11.64,8.38) (11.72,8.63) (11.77,8.77) (11.78,8.79) (11.82,8.94) (11.86,9.05) (11.90,9.18) (11.92,9.26) (11.93,9.31) (11.98,9.48) (12.01,9.59) (12.08,9.83) (12.10,9.93) (12.16,10.12) (12.19,10.27) (12.21,10.34) (12.27,10.56) (12.28,10.62) (12.28,10.64) (12.33,10.86) (12.36,10.97) (12.39,11.12) (12.41,11.22) (12.44,11.33) (12.45,11.38) (12.47,11.49) (12.51,11.70) (12.55,11.90) (12.59,12.07) (12.64,12.34) (12.65,12.41) (12.66,12.45) (12.67,12.53) (12.72,12.83) (12.74,12.93) (12.78,13.14) (12.79,13.21) (12.82,13.45) (12.85,13.61) (12.89,13.89) (12.91,13.97) (12.91,14.00) (12.93,14.13) (12.94,14.22) (12.97,14.41) (12.98,14.48) (13.00,14.62) (13.03,14.81) (13.04,15.00) (13.08,15.22) (13.09,15.32) (13.12,15.52) (13.13,15.64) (13.18,16.01) (13.18,16.03) (13.18,16.05) (13.19,16.12) (13.23,16.47) (13.23,16.55) (13.25,16.67) (13.27,16.90) (13.29,17.07) (13.32,17.32) (13.32,17.33) (13.34,17.59) (13.36,17.76) (13.38,17.97) (13.39,18.10) (13.40,18.19) (13.44,18.61) (13.44,18.62) (13.45,18.63) (13.45,18.66) (13.48,19.07) (13.49,19.14) (13.50,19.23) (13.52,19.51) (13.53,19.66) (13.55,19.85) (13.56,19.96) (13.57,20.17) (13.59,20.40) (13.60,20.47) (13.61,20.69) (13.62,20.86) (13.64,21.07) (13.64,21.21) (13.66,21.31) (13.68,21.67) (13.68,21.72) (13.69,21.76) (13.71,22.03) (13.72,22.22) (13.72,22.24) (13.72,22.27) (13.75,22.68) (13.75,22.76) (13.76,22.86) (13.77,23.14) (13.78,23.28) (13.79,23.45) (13.80,23.60) (13.81,23.79) (13.83,24.04) (13.83,24.06) (13.83,24.31) (13.86,24.53) (13.86,24.62) (13.86,24.83) (13.88,25.00) (13.89,25.19) (13.89,25.34) (13.90,25.47) (13.92,25.77) (13.92,25.86) (13.93,25.94) (13.95,26.34) (13.95,26.38) (13.95,26.41) (13.96,26.64) (13.97,26.69) (13.97,26.88) (13.97,26.90) (13.97,26.92) (13.99,27.36) (13.99,27.41) (14.00,27.48) (14.01,27.83) (14.02,27.93) (14.02,28.05) (14.03,28.31) (14.03,28.45) (14.05,28.61) (14.05,28.79) (14.06,28.97) (14.07,29.17) (14.07,29.27) (14.07,29.48) (14.09,29.74) (14.09,29.74) (14.09,29.77) (14.10,30.00) (0,30.00) (0,0)};
\addplot [draw=black, line width=0.6pt, smooth, forget plot] coordinates {(0.00,0.00) (0.06,0.01) (0.06,0.01) (0.07,0.01) (0.12,0.02) (0.13,0.03) (0.14,0.03) (0.18,0.04) (0.19,0.04) (0.22,0.04) (0.23,0.05) (0.26,0.05) (0.29,0.06) (0.29,0.06) (0.32,0.07) (0.35,0.07) (0.39,0.08) (0.43,0.09) (0.47,0.10) (0.52,0.11) (0.58,0.12) (0.59,0.12) (0.65,0.13) (0.70,0.15) (0.78,0.16) (0.87,0.18) (0.94,0.20) (1.03,0.22) (1.16,0.25) (1.17,0.25) (1.29,0.28) (1.40,0.30) (1.45,0.32) (1.55,0.34) (1.63,0.36) (1.74,0.39) (1.81,0.40) (1.86,0.41) (2.04,0.46) (2.07,0.47) (2.09,0.47) (2.20,0.50) (2.26,0.52) (2.32,0.53) (2.33,0.53) (2.34,0.54) (2.55,0.59) (2.59,0.60) (2.64,0.62) (2.78,0.66) (2.84,0.68) (2.94,0.70) (3.00,0.72) (3.10,0.75) (3.19,0.78) (3.23,0.79) (3.24,0.79) (3.36,0.83) (3.45,0.86) (3.55,0.89) (3.62,0.91) (3.67,0.93) (3.85,0.98) (3.88,0.99) (3.90,1.00) (4.00,1.03) (4.12,1.07) (4.14,1.08) (4.17,1.09) (4.34,1.15) (4.40,1.17) (4.48,1.20) (4.56,1.23) (4.66,1.27) (4.78,1.31) (4.80,1.32) (4.91,1.37) (4.99,1.40) (5.12,1.45) (5.17,1.47) (5.21,1.48) (5.37,1.55) (5.42,1.57) (5.43,1.58) (5.45,1.59) (5.63,1.67) (5.69,1.69) (5.78,1.73) (5.84,1.76) (5.95,1.81) (6.05,1.86) (6.12,1.89) (6.21,1.94) (6.26,1.96) (6.46,2.07) (6.46,2.07) (6.47,2.07) (6.47,2.07) (6.67,2.18) (6.72,2.21) (6.82,2.27) (6.87,2.29) (6.98,2.36) (7.07,2.41) (7.19,2.49) (7.24,2.52) (7.27,2.53) (7.35,2.59) (7.46,2.66) (7.50,2.68) (7.57,2.73) (7.66,2.79) (7.73,2.84) (7.76,2.86) (7.85,2.93) (7.98,3.02) (8.02,3.05) (8.04,3.07) (8.08,3.10) (8.22,3.21) (8.28,3.26) (8.41,3.36) (8.41,3.36) (8.53,3.47) (8.59,3.52) (8.70,3.62) (8.76,3.68) (8.79,3.71) (8.88,3.79) (8.94,3.85) (9.05,3.96) (9.11,4.02) (9.22,4.14) (9.28,4.20) (9.31,4.23) (9.42,4.36) (9.44,4.39) (9.57,4.53) (9.61,4.58) (9.67,4.66) (9.77,4.78) (9.83,4.86) (9.92,4.98) (10.06,5.17) (10.07,5.20) (10.09,5.21) (10.17,5.33) (10.22,5.42) (10.34,5.61) (10.37,5.64) (10.39,5.69) (10.51,5.88) (10.60,6.04) (10.65,6.12) (10.69,6.21) (10.78,6.37) (10.86,6.53) (10.91,6.63) (10.96,6.72) (11.04,6.89) (11.12,7.07) (11.16,7.16) (11.19,7.24) (11.28,7.44) (11.38,7.68) (11.40,7.73) (11.41,7.76) (11.51,8.01) (11.51,8.02) (11.60,8.28) (11.62,8.32) (11.64,8.38) (11.72,8.63) (11.77,8.77) (11.78,8.79) (11.82,8.94) (11.86,9.05) (11.90,9.18) (11.92,9.26) (11.93,9.31) (11.98,9.48) (12.01,9.59) (12.08,9.83) (12.10,9.93) (12.16,10.12) (12.19,10.27) (12.21,10.34) (12.27,10.56) (12.28,10.62) (12.28,10.64) (12.33,10.86) (12.36,10.97) (12.39,11.12) (12.41,11.22) (12.44,11.33) (12.45,11.38) (12.47,11.49) (12.51,11.70) (12.55,11.90) (12.59,12.07) (12.64,12.34) (12.65,12.41) (12.66,12.45) (12.67,12.53) (12.72,12.83) (12.74,12.93) (12.78,13.14) (12.79,13.21) (12.82,13.45) (12.85,13.61) (12.89,13.89) (12.91,13.97) (12.91,14.00) (12.93,14.13) (12.94,14.22) (12.97,14.41) (12.98,14.48) (13.00,14.62) (13.03,14.81) (13.04,15.00) (13.08,15.22) (13.09,15.32) (13.12,15.52) (13.13,15.64) (13.18,16.01) (13.18,16.03) (13.18,16.05) (13.19,16.12) (13.23,16.47) (13.23,16.55) (13.25,16.67) (13.27,16.90) (13.29,17.07) (13.32,17.32) (13.32,17.33) (13.34,17.59) (13.36,17.76) (13.38,17.97) (13.39,18.10) (13.40,18.19) (13.44,18.61) (13.44,18.62) (13.45,18.63) (13.45,18.66) (13.48,19.07) (13.49,19.14) (13.50,19.23) (13.52,19.51) (13.53,19.66) (13.55,19.85) (13.56,19.96) (13.57,20.17) (13.59,20.40) (13.60,20.47) (13.61,20.69) (13.62,20.86) (13.64,21.07) (13.64,21.21) (13.66,21.31) (13.68,21.67) (13.68,21.72) (13.69,21.76) (13.71,22.03) (13.72,22.22) (13.72,22.24) (13.72,22.27) (13.75,22.68) (13.75,22.76) (13.76,22.86) (13.77,23.14) (13.78,23.28) (13.79,23.45) (13.80,23.60) (13.81,23.79) (13.83,24.04) (13.83,24.06) (13.83,24.31) (13.86,24.53) (13.86,24.62) (13.86,24.83) (13.88,25.00) (13.89,25.19) (13.89,25.34) (13.90,25.47) (13.92,25.77) (13.92,25.86) (13.93,25.94) (13.95,26.34) (13.95,26.38) (13.95,26.41) (13.96,26.64) (13.97,26.69) (13.97,26.88) (13.97,26.90) (13.97,26.92) (13.99,27.36) (13.99,27.41) (14.00,27.48) (14.01,27.83) (14.02,27.93) (14.02,28.05) (14.03,28.31) (14.03,28.45) (14.05,28.61) (14.05,28.79) (14.06,28.97) (14.07,29.17) (14.07,29.27) (14.07,29.48) (14.09,29.74) (14.09,29.74) (14.09,29.77) (14.10,30.00)};
\node at (axis cs:11,23) [rotate=85]{\scriptsize{$\mu^{**}_0(\lambda)$}};
\addplot [fill=darkgray!50, draw=none, on layer=axis background, forget plot] coordinates {(0.00,0.00) (0.03,0.01) (0.04,0.01) (0.05,0.02) (0.07,0.02) (0.08,0.03) (0.09,0.03) (0.10,0.03) (0.11,0.04) (0.11,0.04) (0.13,0.04) (0.14,0.05) (0.15,0.05) (0.16,0.06) (0.18,0.06) (0.19,0.06) (0.20,0.07) (0.23,0.08) (0.23,0.08) (0.26,0.09) (0.28,0.09) (0.31,0.10) (0.33,0.11) (0.37,0.13) (0.38,0.13) (0.39,0.13) (0.46,0.16) (0.52,0.18) (0.55,0.19) (0.61,0.21) (0.65,0.23) (0.74,0.26) (0.77,0.27) (0.78,0.27) (0.87,0.31) (0.91,0.32) (0.92,0.32) (0.93,0.33) (1.04,0.37) (1.07,0.38) (1.12,0.40) (1.17,0.42) (1.22,0.44) (1.28,0.46) (1.30,0.47) (1.31,0.47) (1.38,0.50) (1.43,0.52) (1.50,0.54) (1.53,0.56) (1.55,0.57) (1.67,0.61) (1.68,0.62) (1.68,0.62) (1.69,0.62) (1.81,0.67) (1.84,0.68) (1.88,0.70) (1.93,0.72) (1.99,0.75) (2.06,0.78) (2.07,0.78) (2.14,0.81) (2.19,0.83) (2.27,0.87) (2.30,0.88) (2.31,0.89) (2.39,0.92) (2.44,0.94) (2.45,0.95) (2.47,0.95) (2.56,1.00) (2.60,1.02) (2.66,1.04) (2.69,1.05) (2.72,1.07) (2.76,1.09) (2.81,1.11) (2.86,1.14) (2.91,1.16) (2.93,1.17) (3.04,1.22) (3.06,1.23) (3.06,1.23) (3.07,1.24) (3.18,1.29) (3.21,1.31) (3.27,1.34) (3.30,1.35) (3.37,1.39) (3.42,1.42) (3.48,1.44) (3.52,1.47) (3.55,1.48) (3.64,1.53) (3.67,1.54) (3.67,1.55) (3.68,1.55) (3.79,1.61) (3.83,1.63) (3.89,1.67) (3.91,1.67) (3.92,1.68) (3.98,1.71) (4.03,1.74) (4.11,1.79) (4.13,1.80) (4.15,1.81) (4.19,1.84) (4.26,1.88) (4.29,1.89) (4.32,1.91) (4.38,1.95) (4.44,1.98) (4.50,2.02) (4.54,2.05) (4.59,2.08) (4.62,2.09) (4.70,2.14) (4.73,2.17) (4.74,2.17) (4.77,2.19) (4.85,2.24) (4.90,2.27) (4.93,2.30) (4.96,2.32) (4.99,2.34) (5.05,2.37) (5.08,2.39) (5.16,2.45) (5.19,2.47) (5.20,2.48) (5.23,2.49) (5.31,2.55) (5.36,2.59) (5.42,2.63) (5.46,2.66) (5.51,2.70) (5.53,2.71) (5.59,2.76) (5.64,2.80) (5.66,2.81) (5.71,2.84) (5.75,2.88) (5.79,2.91) (5.82,2.93) (5.86,2.97) (5.96,3.04) (5.97,3.05) (5.97,3.05) (5.98,3.06) (6.08,3.14) (6.12,3.17) (6.19,3.23) (6.22,3.25) (6.28,3.30) (6.30,3.32) (6.35,3.37) (6.41,3.41) (6.43,3.43) (6.49,3.49) (6.51,3.51) (6.58,3.57) (6.62,3.60) (6.69,3.67) (6.72,3.70) (6.73,3.71) (6.77,3.75) (6.82,3.80) (6.89,3.86) (6.93,3.90) (7.01,3.98) (7.03,4.00) (7.04,4.01) (7.07,4.04) (7.13,4.10) (7.19,4.17) (7.23,4.21) (7.30,4.29) (7.33,4.32) (7.35,4.33) (7.40,4.39) (7.43,4.42) (7.50,4.50) (7.53,4.53) (7.58,4.59) (7.63,4.64) (7.65,4.68) (7.72,4.76) (7.76,4.80) (7.81,4.86) (7.82,4.87) (7.84,4.90) (7.91,4.99) (7.96,5.05) (8.01,5.11) (8.08,5.20) (8.10,5.23) (8.11,5.24) (8.17,5.32) (8.19,5.35) (8.20,5.36) (8.27,5.45) (8.28,5.47) (8.31,5.51) (8.37,5.60) (8.42,5.66) (8.42,5.66) (8.46,5.73) (8.52,5.82) (8.55,5.86) (8.57,5.89) (8.64,5.99) (8.72,6.12) (8.72,6.12) (8.72,6.12) (8.72,6.12) (8.81,6.26) (8.88,6.37) (8.89,6.39) (8.91,6.43) (8.98,6.53) (9.03,6.62) (9.06,6.67) (9.09,6.73) (9.14,6.82) (9.18,6.89) (9.22,6.96) (9.26,7.04) (9.30,7.11) (9.34,7.17) (9.38,7.26) (9.43,7.35) (9.46,7.41) (9.49,7.47) (9.54,7.56) (9.58,7.65) (9.61,7.72) (9.64,7.78) (9.69,7.87) (9.72,7.96) (9.76,8.03) (9.80,8.11) (9.83,8.19) (9.86,8.27) (9.90,8.36) (9.95,8.46) (9.97,8.52) (9.99,8.57) (10.04,8.69) (10.10,8.83) (10.11,8.86) (10.12,8.88) (10.18,9.03) (10.19,9.06) (10.24,9.18) (10.25,9.20) (10.26,9.22) (10.31,9.37) (10.35,9.49) (10.38,9.55) (10.41,9.64) (10.44,9.73) (10.46,9.80) (10.50,9.91) (10.56,10.08) (10.57,10.09) (10.57,10.10) (10.59,10.15) (10.63,10.28) (10.67,10.41) (10.69,10.46) (10.71,10.55) (10.75,10.65) (10.76,10.71) (10.80,10.84) (10.84,10.97) (10.86,11.02) (10.86,11.03) (10.87,11.05) (10.92,11.23) (10.94,11.32) (10.95,11.33) (10.97,11.42) (10.99,11.48) (11.02,11.59) (11.03,11.62) (11.03,11.63) (11.04,11.68) (11.08,11.82) (11.11,11.94) (11.13,12.02) (11.17,12.17) (11.19,12.22) (11.19,12.24) (11.21,12.32) (11.24,12.42) (11.25,12.48) (11.27,12.55) (11.29,12.63) (11.30,12.70) (11.33,12.80) (11.34,12.84) (11.34,12.86) (11.35,12.91) (11.38,13.05) (11.41,13.16) (11.43,13.26) (11.48,13.47) (11.48,13.47) (11.48,13.47) (11.48,13.47) (11.53,13.68) (11.54,13.78) (11.57,13.90) (11.59,14.00) (11.61,14.08) (11.62,14.12) (11.63,14.20) (11.66,14.34) (11.67,14.39) (11.69,14.51) (11.70,14.56) (11.73,14.69) (11.74,14.78) (11.79,15.00) (11.79,15.00) (11.79,15.00) (11.79,15.00) (11.83,15.22) (11.84,15.31) (11.87,15.45) (11.87,15.48) (11.89,15.61) (11.91,15.68) (11.92,15.77) (11.94,15.87) (11.95,15.91) (11.95,15.92) (11.95,15.94) (11.98,16.14) (12.00,16.22) (12.02,16.37) (12.03,16.40) (12.05,16.53) (12.06,16.60) (12.09,16.82) (12.09,16.83) (12.09,16.84) (12.10,16.84) (12.13,17.07) (12.14,17.14) (12.16,17.28) (12.16,17.30) (12.18,17.45) (12.20,17.54) (12.22,17.71) (12.23,17.76) (12.23,17.78) (12.24,17.87) (12.27,18.02) (12.27,18.06) (12.28,18.13) (12.30,18.26) (12.31,18.37) (12.32,18.44) (12.33,18.50) (12.34,18.55) (12.35,18.67) (12.36,18.75) (12.39,18.96) (12.39,18.98) (12.39,18.99) (12.40,19.03) (12.42,19.24) (12.43,19.29) (12.44,19.37) (12.45,19.48) (12.46,19.59) (12.48,19.73) (12.49,19.77) (12.50,19.90) (12.51,19.98) (12.53,20.17) (12.54,20.20) (12.54,20.23) (12.55,20.33) (12.57,20.48) (12.57,20.51) (12.58,20.56) (12.59,20.73) (12.60,20.82) (12.62,20.95) (12.62,20.98) (12.63,21.12) (12.65,21.23) (12.66,21.34) (12.67,21.43) (12.67,21.49) (12.70,21.72) (12.70,21.73) (12.70,21.74) (12.70,21.78) (12.72,22.00) (12.73,22.04) (12.73,22.10) (12.75,22.26) (12.76,22.35) (12.77,22.48) (12.77,22.51) (12.79,22.65) (12.80,22.77) (12.81,22.86) (12.81,22.96) (12.82,23.03) (12.84,23.23) (12.84,23.27) (12.85,23.29) (12.86,23.42) (12.87,23.55) (12.87,23.57) (12.87,23.60) (12.89,23.81) (12.90,23.88) (12.90,23.97) (12.91,24.07) (12.92,24.18) (12.93,24.34) (12.93,24.34) (12.95,24.49) (12.96,24.60) (12.96,24.70) (12.97,24.80) (12.98,24.86) (12.99,25.07) (12.99,25.10) (13.00,25.13) (13.01,25.29) (13.02,25.39) (13.02,25.41) (13.02,25.43) (13.04,25.66) (13.04,25.71) (13.05,25.79) (13.06,25.93) (13.06,26.02) (13.07,26.15) (13.08,26.19) (13.08,26.33) (13.10,26.46) (13.10,26.51) (13.11,26.63) (13.11,26.73) (13.12,26.86) (13.13,26.94) (13.13,27.00) (13.15,27.22) (13.15,27.24) (13.15,27.27) (13.16,27.40) (13.16,27.43) (13.17,27.54) (13.17,27.55) (13.17,27.57) (13.19,27.81) (13.19,27.86) (13.19,27.92) (13.20,28.08) (13.21,28.16) (13.22,28.27) (13.22,28.35) (13.23,28.47) (13.24,28.62) (13.24,28.62) (13.24,28.62) (13.25,28.78) (13.26,28.90) (13.26,28.97) (13.27,29.08) (13.27,29.17) (13.28,29.32) (13.29,29.39) (13.29,29.44) (13.30,29.66) (13.30,29.69) (13.30,29.72) (13.32,29.91) (13.32,29.99) (13.32,30.00) (0,30.00) (0,0)};
\addplot [draw=black, line width=0.6pt, smooth, forget plot] coordinates {(0.00,0.00) (0.03,0.01) (0.04,0.01) (0.05,0.02) (0.07,0.02) (0.08,0.03) (0.09,0.03) (0.10,0.03) (0.11,0.04) (0.11,0.04) (0.13,0.04) (0.14,0.05) (0.15,0.05) (0.16,0.06) (0.18,0.06) (0.19,0.06) (0.20,0.07) (0.23,0.08) (0.23,0.08) (0.26,0.09) (0.28,0.09) (0.31,0.10) (0.33,0.11) (0.37,0.13) (0.38,0.13) (0.39,0.13) (0.46,0.16) (0.52,0.18) (0.55,0.19) (0.61,0.21) (0.65,0.23) (0.74,0.26) (0.77,0.27) (0.78,0.27) (0.87,0.31) (0.91,0.32) (0.92,0.32) (0.93,0.33) (1.04,0.37) (1.07,0.38) (1.12,0.40) (1.17,0.42) (1.22,0.44) (1.28,0.46) (1.30,0.47) (1.31,0.47) (1.38,0.50) (1.43,0.52) (1.50,0.54) (1.53,0.56) (1.55,0.57) (1.67,0.61) (1.68,0.62) (1.68,0.62) (1.69,0.62) (1.81,0.67) (1.84,0.68) (1.88,0.70) (1.93,0.72) (1.99,0.75) (2.06,0.78) (2.07,0.78) (2.14,0.81) (2.19,0.83) (2.27,0.87) (2.30,0.88) (2.31,0.89) (2.39,0.92) (2.44,0.94) (2.45,0.95) (2.47,0.95) (2.56,1.00) (2.60,1.02) (2.66,1.04) (2.69,1.05) (2.72,1.07) (2.76,1.09) (2.81,1.11) (2.86,1.14) (2.91,1.16) (2.93,1.17) (3.04,1.22) (3.06,1.23) (3.06,1.23) (3.07,1.24) (3.18,1.29) (3.21,1.31) (3.27,1.34) (3.30,1.35) (3.37,1.39) (3.42,1.42) (3.48,1.44) (3.52,1.47) (3.55,1.48) (3.64,1.53) (3.67,1.54) (3.67,1.55) (3.68,1.55) (3.79,1.61) (3.83,1.63) (3.89,1.67) (3.91,1.67) (3.92,1.68) (3.98,1.71) (4.03,1.74) (4.11,1.79) (4.13,1.80) (4.15,1.81) (4.19,1.84) (4.26,1.88) (4.29,1.89) (4.32,1.91) (4.38,1.95) (4.44,1.98) (4.50,2.02) (4.54,2.05) (4.59,2.08) (4.62,2.09) (4.70,2.14) (4.73,2.17) (4.74,2.17) (4.77,2.19) (4.85,2.24) (4.90,2.27) (4.93,2.30) (4.96,2.32) (4.99,2.34) (5.05,2.37) (5.08,2.39) (5.16,2.45) (5.19,2.47) (5.20,2.48) (5.23,2.49) (5.31,2.55) (5.36,2.59) (5.42,2.63) (5.46,2.66) (5.51,2.70) (5.53,2.71) (5.59,2.76) (5.64,2.80) (5.66,2.81) (5.71,2.84) (5.75,2.88) (5.79,2.91) (5.82,2.93) (5.86,2.97) (5.96,3.04) (5.97,3.05) (5.97,3.05) (5.98,3.06) (6.08,3.14) (6.12,3.17) (6.19,3.23) (6.22,3.25) (6.28,3.30) (6.30,3.32) (6.35,3.37) (6.41,3.41) (6.43,3.43) (6.49,3.49) (6.51,3.51) (6.58,3.57) (6.62,3.60) (6.69,3.67) (6.72,3.70) (6.73,3.71) (6.77,3.75) (6.82,3.80) (6.89,3.86) (6.93,3.90) (7.01,3.98) (7.03,4.00) (7.04,4.01) (7.07,4.04) (7.13,4.10) (7.19,4.17) (7.23,4.21) (7.30,4.29) (7.33,4.32) (7.35,4.33) (7.40,4.39) (7.43,4.42) (7.50,4.50) (7.53,4.53) (7.58,4.59) (7.63,4.64) (7.65,4.68) (7.72,4.76) (7.76,4.80) (7.81,4.86) (7.82,4.87) (7.84,4.90) (7.91,4.99) (7.96,5.05) (8.01,5.11) (8.08,5.20) (8.10,5.23) (8.11,5.24) (8.17,5.32) (8.19,5.35) (8.20,5.36) (8.27,5.45) (8.28,5.47) (8.31,5.51) (8.37,5.60) (8.42,5.66) (8.42,5.66) (8.46,5.73) (8.52,5.82) (8.55,5.86) (8.57,5.89) (8.64,5.99) (8.72,6.12) (8.72,6.12) (8.72,6.12) (8.72,6.12) (8.81,6.26) (8.88,6.37) (8.89,6.39) (8.91,6.43) (8.98,6.53) (9.03,6.62) (9.06,6.67) (9.09,6.73) (9.14,6.82) (9.18,6.89) (9.22,6.96) (9.26,7.04) (9.30,7.11) (9.34,7.17) (9.38,7.26) (9.43,7.35) (9.46,7.41) (9.49,7.47) (9.54,7.56) (9.58,7.65) (9.61,7.72) (9.64,7.78) (9.69,7.87) (9.72,7.96) (9.76,8.03) (9.80,8.11) (9.83,8.19) (9.86,8.27) (9.90,8.36) (9.95,8.46) (9.97,8.52) (9.99,8.57) (10.04,8.69) (10.10,8.83) (10.11,8.86) (10.12,8.88) (10.18,9.03) (10.19,9.06) (10.24,9.18) (10.25,9.20) (10.26,9.22) (10.31,9.37) (10.35,9.49) (10.38,9.55) (10.41,9.64) (10.44,9.73) (10.46,9.80) (10.50,9.91) (10.56,10.08) (10.57,10.09) (10.57,10.10) (10.59,10.15) (10.63,10.28) (10.67,10.41) (10.69,10.46) (10.71,10.55) (10.75,10.65) (10.76,10.71) (10.80,10.84) (10.84,10.97) (10.86,11.02) (10.86,11.03) (10.87,11.05) (10.92,11.23) (10.94,11.32) (10.95,11.33) (10.97,11.42) (10.99,11.48) (11.02,11.59) (11.03,11.62) (11.03,11.63) (11.04,11.68) (11.08,11.82) (11.11,11.94) (11.13,12.02) (11.17,12.17) (11.19,12.22) (11.19,12.24) (11.21,12.32) (11.24,12.42) (11.25,12.48) (11.27,12.55) (11.29,12.63) (11.30,12.70) (11.33,12.80) (11.34,12.84) (11.34,12.86) (11.35,12.91) (11.38,13.05) (11.41,13.16) (11.43,13.26) (11.48,13.47) (11.48,13.47) (11.48,13.47) (11.48,13.47) (11.53,13.68) (11.54,13.78) (11.57,13.90) (11.59,14.00) (11.61,14.08) (11.62,14.12) (11.63,14.20) (11.66,14.34) (11.67,14.39) (11.69,14.51) (11.70,14.56) (11.73,14.69) (11.74,14.78) (11.79,15.00) (11.79,15.00) (11.79,15.00) (11.79,15.00) (11.83,15.22) (11.84,15.31) (11.87,15.45) (11.87,15.48) (11.89,15.61) (11.91,15.68) (11.92,15.77) (11.94,15.87) (11.95,15.91) (11.95,15.92) (11.95,15.94) (11.98,16.14) (12.00,16.22) (12.02,16.37) (12.03,16.40) (12.05,16.53) (12.06,16.60) (12.09,16.82) (12.09,16.83) (12.09,16.84) (12.10,16.84) (12.13,17.07) (12.14,17.14) (12.16,17.28) (12.16,17.30) (12.18,17.45) (12.20,17.54) (12.22,17.71) (12.23,17.76) (12.23,17.78) (12.24,17.87) (12.27,18.02) (12.27,18.06) (12.28,18.13) (12.30,18.26) (12.31,18.37) (12.32,18.44) (12.33,18.50) (12.34,18.55) (12.35,18.67) (12.36,18.75) (12.39,18.96) (12.39,18.98) (12.39,18.99) (12.40,19.03) (12.42,19.24) (12.43,19.29) (12.44,19.37) (12.45,19.48) (12.46,19.59) (12.48,19.73) (12.49,19.77) (12.50,19.90) (12.51,19.98) (12.53,20.17) (12.54,20.20) (12.54,20.23) (12.55,20.33) (12.57,20.48) (12.57,20.51) (12.58,20.56) (12.59,20.73) (12.60,20.82) (12.62,20.95) (12.62,20.98) (12.63,21.12) (12.65,21.23) (12.66,21.34) (12.67,21.43) (12.67,21.49) (12.70,21.72) (12.70,21.73) (12.70,21.74) (12.70,21.78) (12.72,22.00) (12.73,22.04) (12.73,22.10) (12.75,22.26) (12.76,22.35) (12.77,22.48) (12.77,22.51) (12.79,22.65) (12.80,22.77) (12.81,22.86) (12.81,22.96) (12.82,23.03) (12.84,23.23) (12.84,23.27) (12.85,23.29) (12.86,23.42) (12.87,23.55) (12.87,23.57) (12.87,23.60) (12.89,23.81) (12.90,23.88) (12.90,23.97) (12.91,24.07) (12.92,24.18) (12.93,24.34) (12.93,24.34) (12.95,24.49) (12.96,24.60) (12.96,24.70) (12.97,24.80) (12.98,24.86) (12.99,25.07) (12.99,25.10) (13.00,25.13) (13.01,25.29) (13.02,25.39) (13.02,25.41) (13.02,25.43) (13.04,25.66) (13.04,25.71) (13.05,25.79) (13.06,25.93) (13.06,26.02) (13.07,26.15) (13.08,26.19) (13.08,26.33) (13.10,26.46) (13.10,26.51) (13.11,26.63) (13.11,26.73) (13.12,26.86) (13.13,26.94) (13.13,27.00) (13.15,27.22) (13.15,27.24) (13.15,27.27) (13.16,27.40) (13.16,27.43) (13.17,27.54) (13.17,27.55) (13.17,27.57) (13.19,27.81) (13.19,27.86) (13.19,27.92) (13.20,28.08) (13.21,28.16) (13.22,28.27) (13.22,28.35) (13.23,28.47) (13.24,28.62) (13.24,28.62) (13.24,28.62) (13.25,28.78) (13.26,28.90) (13.26,28.97) (13.27,29.08) (13.27,29.17) (13.28,29.32) (13.29,29.39) (13.29,29.44) (13.30,29.66) (13.30,29.69) (13.30,29.72) (13.32,29.91) (13.32,29.99) (13.32,30.00)};
\node at (axis cs:12.8,7.2) [rotate=70]{\scriptsize{$\mu^{**}_2(\lambda)$}};
\node at (axis cs:-2,-2) {\scriptsize{0}};
\addplot [mark=none,dashed,color=black] coordinates {(15,0) (15,30)};
\legend{
\scriptsize{$\# sol.= 0$}, \scriptsize{$\# sol.\geq 1$}, \scriptsize{$\# sol.\geq 2$}, \scriptsize{$\# sol.\geq 4$}, \scriptsize{$\# sol.\geq 8$}
}
\end{axis}
\end{tikzpicture}
\caption{Qualitative bifurcation digram of the solutions to~\eqref{eq-main} in the $(\lambda,\mu)$-plane. The curves $\mu^{*}_0(\lambda)$ and $\mu^{**}_0(\lambda)$  define the non-existence region (gray). The curves $\mu^{*}_1(\lambda)$ and $\mu^{**}_2(\lambda)$ mark out regions of existence: at least one solution in between $\mu^{*}_{1}(\lambda)$ and $\mu^{*}_{2}(\lambda)$ (red) and at least two solutions in between $\mu^*_2(\lambda)$ and $\mu^{**}_{2}(\lambda)$ (blue). For $\lambda\in[\lambda^{*},+\infty)$, at least four solutions in the region above $\mu^*_4(\lambda)$ (yellow), and at least eight solutions in the one above $\mu^*_8(\lambda)$ (green).}
\label{fig-bif-plane}
\end{figure}
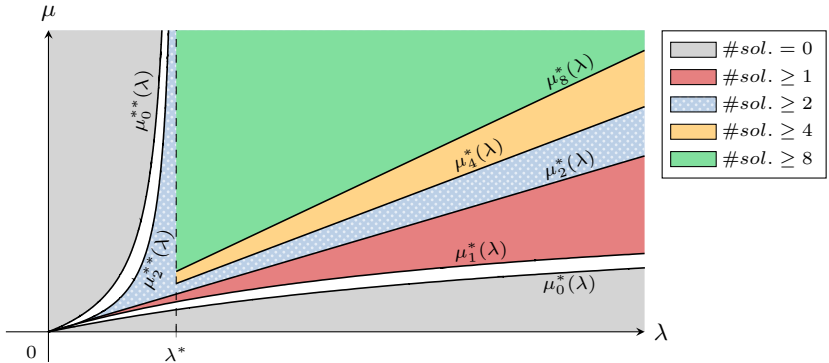

The indefinite weight problem~\eqref{eq-main} bears some similarities with the ones studied in the field of population genetics in~\cite{LoNa-02,LoNiSu-10,Na-16,Na-18,Na-20,NaNiSu-10,NaSu-20} that include the one-parameter Neumann problem
\begin{equation}\label{eq-aux-intro}
\begin{cases}
\, -u''(t)=\lambda \, \tilde a(t)\,g(u(t)), & t\in(0,1), \\
\, 0 \leq u(t) \leq 1, & t\in[0,1], \\
\, u'(0)=0, \quad u'(1)=0,
\end{cases}
\end{equation}
where $\tilde{a}\colon[0,1]\to\mathbb{R}$ is a sign-changing function. Concerning problem~\eqref{eq-aux-intro}, the main questions addressed have involved the uniqueness and the multiplicity of the solutions depending on the sign of $\int_0^1\tilde{a}(t)\,\mathrm{d}t$ (see the conjectures contained in~\cite{LoNa-02}). Indeed, when $\int_0^1\tilde{a}(t)\,\mathrm{d}t<0$, \cite[Theorem~1.3]{NaNiSu-10} is a general result stated in a high dimensional setting which ensures the existence of at least two solutions to~\eqref{eq-aux-intro} for $\lambda$ large enough. On the other hand, when $\int_0^1\tilde{a}(t)\,\mathrm{d}t\geq0$, the existence of at most one solution to~\eqref{eq-aux-intro} has been established in~\cite{Na-16,Na-18} under some additional assumptions on~$\tilde{a}$ and for $\lambda$ large enough. Despite that, in both cases, high multiplicity of solutions for problem~\eqref{eq-aux-intro} has been proved for some specific symmetric weights $\tilde{a}$ that change sign twice and $\lambda$ large enough: when $\int_0^1\tilde{a}(t)\,\mathrm{d}t<0$, \cite[Theorems~1.1,~1.2, and~1.3]{Na-20} ensure the existence of at least eight solutions to~\eqref{eq-aux-intro}; when $\int_0^1\tilde{a}(t)\,\mathrm{d}t\geq0$, \cite[Theorems~1.1 and~1.2]{NaSu-20} ensure the existence of at least three solutions to~\eqref{eq-aux-intro}.

To pursue the study of these one-parameter Neumann problems, we reduce the number of parameters in problem~\eqref{eq-main} by taking $\mu$ as a linear function of $\lambda$, that is $\mu=\lambda K$ for some $K>0$. Thus, we consider~\eqref{eq-aux-intro} with
\begin{equation}\label{eq-weight-intro}
\tilde{a}(t)=\tilde{a}_K(t):=\begin{cases}
\, 1, & \text{if $t\in[0,\sigma]\cup[1-\sigma,1]$,} \\
\, -K, & \text{if $t\in(\sigma,1-\sigma)$,}
\end{cases}
\end{equation}
and we prove the following result.

\begin{theorem}\label{th-intro3}
Let $\tilde{a}$ be as in~\eqref{eq-weight-intro}. Then, there exist $K_4,K_8\in(0,+\infty)$, with $K_4<K_8$, such that the following assertions hold:
\begin{enumerate}
\item[$(i)$] for all $K>0$, there exists $\lambda_1^{*}(K)>0$ such that problem~\eqref{eq-aux-intro} admits at least one solution for all $\lambda>\lambda_1^{*}(K)$;
\item[$(ii)$] for all $K>\frac{2\sigma}{1-2\sigma}$, problem~\eqref{eq-aux-intro} admits at least two solutions for all $\lambda\in[\lambda^{*},+\infty)$, where $\lambda^*$ is as in Theorem~\ref{th-intro1};
\item[$(iii)$] for all $K>K_4$, there exists $\lambda^{*}_4(K)>0$ such that problem~\eqref{eq-aux-intro} admits at least four solutions for all $\lambda>\lambda^{*}_4(K)$;
\item[$(iv)$] for all $K>K_8$, there exists $\lambda^{*}_8(K)>0$ such that problem \eqref{eq-aux-intro} admits at least eight solutions for all $\lambda>\lambda^{*}_8(K)$.
\end{enumerate}
\end{theorem}

\noindent We notice that the inequality $K>\frac{2\sigma}{1-2\sigma}$ in case~$(ii)$ implies that $\int_0^1\tilde{a}_K(t)\,\mathrm{d}t<0$ and so we enter in the framework treated in~\cite{Na-20}. Compared with the results contained in~\cite{Na-20}, Theorem~\ref{th-intro3} produces another high multiplicity result with a different weight function. Moreover, our theorem enriches the analysis of problem~\eqref{eq-aux-intro} by establishing intermediate multiplicity results, with possibly four solutions. We also refer to Corollary~\ref{th-intro3bis} for an equivalent version of Theorem~\ref{th-intro3} in the spirit of \cite{Na-20}, i.e., keeping $K$ fixed and varying $\sigma$ instead.

\medskip

The paper is organized as follows. In Section~\ref{section-2} and in Section~\ref{section-3}, we analyze the behavior of the solutions of the equations $u''=-\lambda g(u)$ and $u''=\mu g(u)$, respectively. We also present some preliminary results about the time-maps associated with these two equations in the phase-plane $(u,u')$ that are further discussed in Appendices~\ref{appendix-A} and~\ref{appendix-B}. In Section~\ref{section-4}, on the lines of~\cite{LGTe-14,LGTeZa-14}, we derive some technical time-map estimates to be exploited in Section~\ref{section-5} to prove Theorems~\ref{th-intro1} and~\ref{th-intro2}. In Section~\ref{section-6}, we discuss the shape of the global bifurcation diagrams of the solutions (with bifurcation parameter $\mu$), depending on $\lambda$. At last, in Section~\ref{section-7}, we study the one parameter problem~\eqref{eq-aux-intro}-\eqref{eq-weight-intro} and, among other results, we give the proof of Theorem~\ref{th-intro3}.

\section{Phase-plane analysis in $\mathopen{[}0,\sigma\mathclose{]}$ and $\mathopen{[}1-\sigma,1\mathclose{]}$}\label{section-2}

In this section, we analyze the equation in~\eqref{eq-main} in the intervals $\mathopen{[}0,\sigma\mathclose{]}$ and $\mathopen{[}1-\sigma,1\mathclose{]}$, that is where the weight function $a_{\lambda,\mu}\equiv\lambda>0$. Accordingly, we study
\begin{equation}\label{eq-0sigma}
u''=-\lambda g(u).
\end{equation}
Equivalently, we deal with the following planar system
\begin{equation}\label{syst-0sigma}
\begin{cases}
\, u'=v, \\
\, v' = -\lambda g(u),
\end{cases}
\end{equation}
whose associated energy is
\begin{equation}\label{eq-energy-lambda}
H_{\lambda}(u,v) := v^{2}+2\lambda G(u),
\end{equation}
where $G$ is the primitive of the function $g$ defined in~\eqref{eq-nonlin} vanishing at $u=0$, that is
\begin{equation}\label{def-G}
G(u)= \int_{0}^{u} g(\xi) \,\mathrm{d}\xi = \dfrac{u^{3}}{3} - \dfrac{u^{4}}{4}, \quad u\in[0,1].
\end{equation}
We remark that, since we search solutions $(u,v)$ of \eqref{syst-0sigma} with $u(t)\in(0,1)$ for all $t\in[0,1]$, we have to consider only the strip $[0,1]\times\mathbb{R}$ in the phase-plane $(u,v)$.

First, we focus our attention on the interval $[0,\sigma]$. The analysis in the interval $[1-\sigma,1]$ is analogous and we give the corresponding results at the end of this section. Since we are interested in solutions to the Neumann problem~\eqref{eq-main}, for all $s\in(0,1)$, let us consider the initial value problem
\begin{equation}\label{eq-initial0}
\begin{cases}
\, u'=v, \\
\, v' = -\lambda g(u), \\
\, u(0)=s, \\
\, v(0)=0,
\end{cases}
\end{equation}
and let $(u_{s},v_{s})$ be the unique solution of~\eqref{eq-initial0}, considered in its maximal interval of existence (contained in $\mathbb{R}$). 
For all $s\in(0,1)$, we denote by $T_{0}(s)$ the time taken by $(u_{s},v_{s})$ to go from the point $(s,0)$ to the line $\{0\}\times(-\infty,0)$ moving along the level line $H_{\lambda}(u,v)=2\lambda G(s)$ in the $(u,v)$-plane. We observe that $T_{0}(s)$ is well defined since $u_{s}$ is concave, due to the sign of $g$.

We set
\begin{equation*}
\mathcal{I}^{0}_{\lambda} = \bigl{\{} s\in(0,1) \colon T_{0}(s)>\sigma \bigr{\}}.
\end{equation*}
We stress that $\mathcal{I}^{0}_{\lambda}$ is the set of initial values $s\in(0,1)$ such that~\eqref{eq-initial0} has a positive solution defined in the whole interval $[0,\sigma]$.
The following proposition characterises $\mathcal{I}^{0}_{\lambda}$ in dependence of $\lambda$.

\begin{proposition}\label{pr-2.1}
There exist $\lambda^{*}>0$ and $s^{*}\in(0,1)$ such that
\begin{itemize}
\item[$(i)$] for all $0<\lambda<\lambda^{*}$, $\mathcal{I}^{0}_{\lambda}=(0,1)$;
\item[$(ii)$] for $\lambda=\lambda^{*}$, $\mathcal{I}^{0}_{\lambda}=(0,1)\setminus\{s^{*}\}$;
\item[$(iii)$] for all $\lambda>\lambda^{*}$, there exist $s_{0},s_{1}\in(0,1)$ with $s_{0}<s^{*}<s_{1}$ such that $\mathcal{I}^{0}_{\lambda}=(0,s_{0})\cup(s_{1},1)$.
\end{itemize}
\end{proposition}

\begin{proof}
Preliminarily, we explicitly define the function $s\mapsto T_{0}(s)$.
The equation of the level line associated with~\eqref{eq-energy-lambda} passing through an arbitrary point $(s,0)\in[0,1]\times\{0\}$ is
\begin{equation*}
v^{2} + 2 \lambda G(u) = 2 \lambda G(s).
\end{equation*}
Since $g(u)\geq0$ for all $u\in[0,1]$, from~\eqref{eq-initial0} we infer that $v_{s}(t)\leq0$ for all $t$. 
As a consequence, we have
\begin{equation*}
v = - \sqrt{2 \lambda \bigl{(} G(s)-G(u) \bigr{)}}
\end{equation*}
and thus
\begin{equation*}
\begin{aligned}
T_0(s) &= \int_{0}^{T_0(s)} 1 \,\mathrm{d}t 
= - \int_{0}^{T_0(s)} \dfrac{u'(t) \, \mathrm{d}t}{\sqrt{2 \lambda \bigl{(} G(s)-G(u(t)) \bigr{)}}}
\\
&=- \int_{s}^{u\left(T_0(s)\right)} \dfrac{\mathrm{d}u}{\sqrt{2 \lambda \bigl{(} G(s)-G(u) \bigr{)}}}
\\
&= \dfrac{1}{\sqrt{2\lambda}} \int_{0}^{s} \dfrac{\mathrm{d}u}{\sqrt{\dfrac{u^{4}-s^{4}}{4} - \dfrac{u^{3}-s^{3}}{3}}}.
\end{aligned}
\end{equation*}
Performing the change of variable $u=s \xi$, we obtain
\begin{equation*}
T_0(s) 
= \dfrac{1}{\sqrt{2\lambda}} \int_{0}^{1} \dfrac{s \, \mathrm{d}\xi}{\sqrt{\dfrac{s^{3}-(s\xi)^{3}}{3}-\dfrac{s^{4}-(s\xi)^{4}}{4}}}
= \dfrac{1}{\sqrt{2\lambda s}} \int_{0}^{1} \dfrac{\mathrm{d}\xi}{\sqrt{\dfrac{1-\xi^{3}}{3}-\dfrac{s(1-\xi^{4})}{4}}}.
\end{equation*}
Therefore, we have
\begin{equation}\label{eq-formulaT_0}
T_0(s) = \dfrac{1}{\sqrt{2\lambda s}} \hat{I}(s), \quad \text{for all $s\in(0,1)$},
\end{equation}
where
\begin{equation*}
\hat{I}(s) = \int_{0}^{1} \dfrac{\mathrm{d}\xi}{\sqrt{\dfrac{1-\xi^{3}}{3}-\dfrac{s(1-\xi^{4})}{4}}}.
\end{equation*}

We study the limit of $T_0(s)$ as $s\to 0^{+}$ and $s\to 1^{-}$.
As a consequence of the fact that $\hat{I}(0)\in(0,+\infty)$, we deduce
\begin{equation}\label{eq-limitTs-0}
\lim_{s\to 0^{+}} T_0(s) = + \infty.
\end{equation}
Moreover, observing that
\begin{equation}\label{eq-2.5}
\sqrt{\dfrac{1-\xi^{3}}{3}-\dfrac{1-\xi^{4}}{4}}=\dfrac{1-\xi}{2\sqrt{3}}\sqrt{3\xi^{2}+2\xi+1}, \quad \text{in $[0,1]$,}
\end{equation}
we immediately obtain that the improper integral $\hat{I}(1)$ diverges and so
\begin{equation}\label{eq-limitTs-1}
\lim_{s\to 1^{-}} T_0(s) = + \infty.
\end{equation}

Next, we analyze the monotonicity of the function $T_0(s)$. Accordingly, by differentiating \eqref{eq-formulaT_0}, we obtain
\begin{equation}\label{eq-der_T}
\begin{aligned}
T_0'(s) 
&= - \dfrac{\lambda}{(2\lambda s)^{\frac{3}{2}}}\hat{I}(s) + \dfrac{1}{\sqrt{2\lambda s}} \hat{I}'(s)
= \biggl{(} - \dfrac{1}{2s} + \dfrac{\hat{I}'(s)}{\hat{I}(s)} \biggr{)}T_0(s)
\\
&=\dfrac{T_0(s)}{2s\hat{I}(s)} \bigl{(} 2s\hat{I}'(s)-\hat{I}(s)\bigr{)},
\quad \text{for all $s\in(0,1)$.}
\end{aligned}
\end{equation}
We compute
\begin{equation*}
\hat{I}'(s) = \dfrac{1}{2}\int_{0}^{1} 
\dfrac{1-\xi^{4}}{4}
\biggl{(}\dfrac{1-\xi^{3}}{3}-\dfrac{s(1-\xi^{4})}{4}\biggr{)}^{\!-\frac{3}{2}} \, \mathrm{d}\xi >0, \quad \text{for all $s\in(0,1)$,}
\end{equation*}
and
\begin{equation*}
\hat{I}''(s) = \dfrac{3}{4}\int_{0}^{1} 
\biggl{(}\dfrac{1-\xi^{4}}{4}\biggr{)}^{\!2}
\biggl{(}\dfrac{1-\xi^{3}}{3}-\dfrac{s(1-\xi^{4})}{4}\biggr{)}^{\!-\frac{5}{2}} \, \mathrm{d}\xi > 0, \quad \text{for all $s\in(0,1)$.}
\end{equation*}
The function $\eta(s) := 2s\hat{I}'(s)-\hat{I}(s)$ is such that
\begin{equation*}
\eta'(s)= 2\hat{I}'(s)+2s\hat{I}''(s)-\hat{I}'(s)=\hat{I}'(s)+2s\hat{I}''(s)>0, \quad \text{for all $s\in(0,1)$,}
\end{equation*}
then $\eta$ vanishes at most in a point (independent of $\lambda$).
As a consequence, from \eqref{eq-der_T}, we deduce that the positive function $T_0$ has at most a critical point (independent of $\lambda$).
From \eqref{eq-limitTs-0} and \eqref{eq-limitTs-1}, we conclude that $T_{0}$ has a unique minimum point, which we denote by $s^{*}$ (we stress again that it does not depend of $\lambda$). In Figure~\ref{fig-02}, we represent the graph of $T_{0}$.

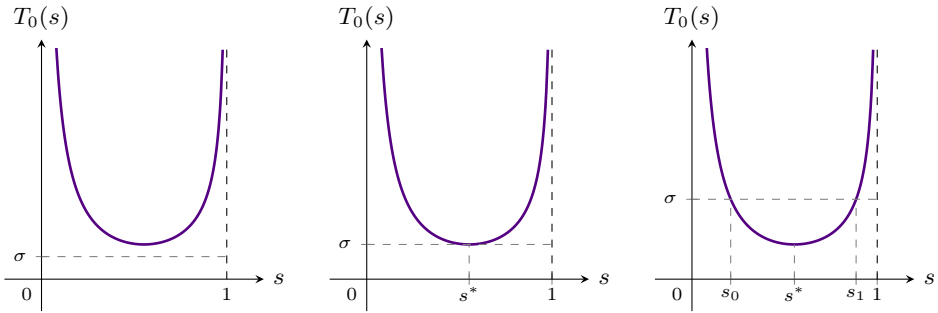
\begin{figure}[htb]
\centering
\begin{tikzpicture}
\begin{axis}[
  tick label style={font=\scriptsize},
  axis y line=middle, 
  axis x line=middle,
  xtick={0},
  ytick={0,0.56485},
  xticklabels={},
  yticklabels={,$\sigma$},
  xlabel={\small $s$},
  ylabel={\small $T_{0}(s)$},
every axis x label/.style={
    at={(ticklabel* cs:1.0)},
    anchor=west,
},
every axis y label/.style={
    at={(ticklabel* cs:1.0)},
    anchor=south,
},
  width=5cm,
  height=5.2cm,
  xmin=-0.2,
  xmax=1.2,
  ymin=-0.8,
  ymax=6]
\addplot[domain=0.08:0.978,color=fst-purple,line width=1pt,smooth,samples=100]{1/(x*sqrt(1-x)*(2-x))-1};
\addplot [mark=none,dashed,color=black] coordinates {(1,0) (1,5.8)};
\addplot [mark=none,dashed,color=gray] coordinates {(0,0.56485) (1,0.56485)};
\node at (axis cs:-0.08,-0.38) {\scriptsize{0}};
\node at (axis cs:1,-0.38) {\scriptsize{1}};
\end{axis}
\end{tikzpicture}
\quad
\begin{tikzpicture}
\begin{axis}[
  tick label style={font=\scriptsize},
  axis y line=middle, 
  axis x line=middle,
  xtick={0,0.553,1},
  ytick={0,0.869},
  xticklabels={},
  yticklabels={,$\sigma$},
  xlabel={\small $s$},
  ylabel={\small $T_{0}(s)$},
every axis x label/.style={
    at={(ticklabel* cs:1.0)},
    anchor=west,
},
every axis y label/.style={
    at={(ticklabel* cs:1.0)},
    anchor=south,
},
  width=5cm,
  height=5.2cm,
  xmin=-0.2,
  xmax=1.2,
  ymin=-0.8,
  ymax=6]
\addplot[domain=0.08:0.978,color=fst-purple,line width=1pt,smooth,samples=100]{1/(x*sqrt(1-x)*(2-x))-1};
\addplot [mark=none,dashed,color=black] coordinates {(1,0) (1,5.8)};
\addplot [mark=none,dashed,color=gray] coordinates {(0.553,0) (0.553,0.869)};
\addplot [mark=none,dashed,color=gray] coordinates {(0,0.869) (1,0.869)};
\node at (axis cs:-0.08,-0.38) {\scriptsize{0}};
\node at (axis cs:1,-0.38) {\scriptsize{1}};
\node at (axis cs:0.553,-0.38) {\scriptsize{$s^{*}$}};
\end{axis}
\end{tikzpicture}
\quad
\begin{tikzpicture}
\begin{axis}[
  tick label style={font=\scriptsize},
  axis y line=middle, 
  axis x line=middle,
  xtick={0,0.209,0.553,0.886,1},
  ytick={0,2},
  xticklabels={},
  yticklabels={,$\sigma$},
  xlabel={\small $s$},
  ylabel={\small $T_{0}(s)$},
every axis x label/.style={
    at={(ticklabel* cs:1.0)},
    anchor=west,
},
every axis y label/.style={
    at={(ticklabel* cs:1.0)},
    anchor=south,
},
  width=5cm,
  height=5.2cm,
  xmin=-0.2,
  xmax=1.2,
  ymin=-0.8,
  ymax=6]
\addplot[domain=0.08:0.978,color=fst-purple,line width=1pt,smooth,samples=100]{1/(x*sqrt(1-x)*(2-x))-1};
\addplot [mark=none,dashed,color=black] coordinates {(1,0) (1,5.8)};
\addplot [mark=none,dashed,color=gray] coordinates {(0.209,0) (0.209,2)};
\addplot [mark=none,dashed,color=gray] coordinates {(0.553,0) (0.553,0.869)};
\addplot [mark=none,dashed,color=gray] coordinates {(0.886,0) (0.886,2)};
\addplot [mark=none,dashed,color=gray] coordinates {(0,2) (1,2)};
\node at (axis cs:-0.08,-0.38) {\scriptsize{0}};
\node at (axis cs:1,-0.38) {\scriptsize{1}};
\node at (axis cs:0.209,-0.38) {\scriptsize{$s_{0}$}};
\node at (axis cs:0.553,-0.38) {\scriptsize{$s^{*}$}};
\node at (axis cs:0.886,-0.38) {\scriptsize{$s_{1}$}};
\end{axis}
\end{tikzpicture}
\caption{Qualitative graph of the function $T_{0}$ when $\lambda\in(0,\lambda^{*})$ (left), $\lambda=\lambda^{*}$ (center), and $\lambda>\lambda^{*}$ (right). The points $s^{*}$, $s_{0}$, and $s_{1}$ are defined in Proposition~\ref{pr-2.1}.}
\label{fig-02}
\end{figure}

At last, for all fixed $s\in(0,1)$, from \eqref{eq-formulaT_0} we have that the function $\lambda\mapsto T_0(s,\lambda)$ is strictly decreasing and it converges to $+\infty$ as $\lambda\to 0^{+}$, and to $0$ as $\lambda\to+\infty$. Thus, there exists a unique value of $\lambda$, that we denote by $\lambda^{*}$, such that 
\begin{equation*}
\min_{s\in(0,1)} T_0(s,\lambda^{*})=T_0(s^{*},\lambda^{*})=\sigma .
\end{equation*}
With this position, properties $(i)$, $(ii)$, and $(iii)$ immediately follow. We notice that the points $s_{0}=s_{0}(\lambda)$ and $s_{1}=s_{1}(\lambda)$ depend on $\lambda$ and, by the properties of $T_0$, they satisfy $0 < s_{0} < s^{*} < s_{1} < 1$ for all $\lambda>\lambda^{*}$. The proof is complete.
\end{proof}

For $\lambda>0$, we introduce the following set in the $(u,v)$-plane:
\begin{equation}\label{def-Gamma0}
\Gamma_{0}=\Gamma_{0}(\lambda)=\bigl{\{} \bigl{(} u_{s}(\sigma),v_{s}(\sigma) \bigr{)} \colon s\in \mathcal{I}_{\lambda}^{0} \bigr{\}}.
\end{equation}
Notice that $\Gamma_{0}\subseteq (0,1)\times(-\infty,0)$ and it is a continuous curve parametrized by $s\in \mathcal{I}_{\lambda}^{0}$. Moreover, from Proposition~\ref{pr-2.1} we have the following (see also Figure~\ref{fig-03}). 

\begin{corollary}\label{cor-2.1}
There exists $\lambda^{*}>0$ such that
\begin{itemize}
\item for $\lambda\in(0,\lambda^{*})$, $\overline{\Gamma_0(\lambda)}\cap (\{0\}\times(-\infty,0))=\emptyset$;
\item for $\lambda=\lambda^{*}$, $\overline{\Gamma_0(\lambda)}\cap(\{0\}\times(-\infty,0))=\{(0,-\xi^{*})\}$ for some $\xi^{*}\in(0,+\infty)$;
\item for $\lambda>\lambda^{*}$, $\overline{\Gamma_0(\lambda)}\cap (\{0\}\times(-\infty,0)) = \{(0,-\xi_{0}),(0,-\xi_{1})\}$ for some $0<\xi_{0}<\xi_{1}$;
\end{itemize}
where $\overline{\Gamma_0(\lambda)}$ denotes the closure of $\Gamma_0(\lambda)$ in $\mathbb{R}^{2}$.
\end{corollary}

\begin{figure}[htb]
\centering
\begin{tikzpicture}
\begin{axis}[
  tick label style={font=\scriptsize},
  axis y line=middle, 
  axis x line=middle,
  xtick={1},
  ytick={0},
  xticklabels={},
  yticklabels={},
  xlabel={\small $u$},
  ylabel={\small $v$},
every axis x label/.style={
    at={(ticklabel* cs:1.0)},
    anchor=west,
},
every axis y label/.style={
    at={(ticklabel* cs:1.0)},
    anchor=south,
},
  width=5cm,
  height=5cm,
  xmin=-0.2,
  xmax=1.2,
  ymin=-7,
  ymax=2]
\addplot [mark=none,dashed,color=black] coordinates {(1,-8) (1,3)};
\addplot [color=fst-blue,line width=1.2pt,smooth] coordinates {(0,0) (0.0106773, -0.00428203) (0.0205276, -0.0165198) (0.0296083, -0.0358854) (0.0379719, -0.0616512) (0.045667,-0.0931756) (0.0527386, -0.129892) (0.0592282, -0.171296) (0.0651747, -0.216941) (0.0706141, -0.26643) (0.0755804,-0.319404) (0.0801051, -0.375545) (0.0842181, -0.434566) (0.0879475, -0.496207) (0.0913198, -0.560235) (0.0943603,-0.626439) (0.0970927, -0.694624) (0.09954, -0.764617) (0.101724, -0.836256) (0.103665, -0.909394) (0.105383, -0.983896) (0.106898,-1.05964) (0.108228, -1.1365) (0.109392, -1.21438) (0.110405,-1.29318) (0.111287, -1.3728) (0.112052, -1.45316) (0.112717, -1.53417) (0.113298, -1.61575) (0.11381, -1.69784) (0.114268,-1.78035) (0.114688, -1.86322) (0.115085, -1.94638) (0.115473, -2.02978) (0.115867, -2.11335) (0.116281, -2.19703) (0.116731,-2.28075) (0.117231, -2.36447) (0.117795, -2.44813) (0.118438,-2.53167) (0.119176, -2.61502) (0.120023, -2.69815) (0.120995,-2.78098) (0.122108, -2.86346) (0.123377, -2.94553) (0.124818, -3.02713) (0.126448, -3.10819) (0.128285, -3.18866) (0.130346, -3.26847) (0.132648, -3.34754) (0.135212, -3.42581) (0.138056,-3.5032) (0.1412, -3.57962) (0.144666, -3.65498) (0.148476, -3.7292) (0.152654, -3.80217) (0.157223, -3.87378) (0.16221,-3.94392) (0.167641, -4.01243) (0.173547, -4.07919) (0.179957, -4.14401) (0.186904, -4.20672) (0.194424, -4.26712) (0.202555, -4.32495) (0.211337, -4.37997) (0.220814, -4.43185) (0.231034, -4.48024) (0.242049, -4.52475) (0.253916, -4.56489) (0.266696, -4.60012) (0.280458, -4.62979) (0.295278, -4.65316) (0.311237,-4.66933) (0.328429, -4.67725) (0.346956, -4.67568) (0.366935, -4.66312) (0.388495, -4.63779) (0.411782, -4.59754) (0.436964, -4.53979) (0.464229, -4.46136) (0.493796, -4.35841) (0.525913,-4.2262) (0.560873, -4.05888) (0.599011, -3.84915) (0.640726,-3.5879) (0.686484, -3.26358) (0.736843, -2.86152) (0.792471, -2.36282) (0.854176, -1.74294) (0.922946, -0.969639) (1, 0)};
\node at (axis cs:-0.06,-0.38) {\scriptsize{0}};
\node at (axis cs:1.06,-0.38) {\scriptsize{1}};
\end{axis}
\end{tikzpicture} 
\quad
\begin{tikzpicture}
\begin{axis}[
  tick label style={font=\scriptsize},
  axis y line=middle, 
  axis x line=middle,
  xtick={1},
  ytick={0},
  xticklabels={},
  yticklabels={},
  xlabel={\small $u$},
  ylabel={\small $v$},
every axis x label/.style={
    at={(ticklabel* cs:1.0)},
    anchor=west,
},
every axis y label/.style={
    at={(ticklabel* cs:1.0)},
    anchor=south,
},
  width=5cm,
  height=5cm,
  xmin=-0.2,
  xmax=1.2,
  ymin=-7,
  ymax=2]
\addplot [mark=none,dashed,color=black] coordinates {(1,-8) (1,3)};
\addplot [color=fst-blue,line width=1.2pt,smooth] coordinates {(0,0) (0.0104975, -0.00602399) (0.0198375, -0.0230038) (0.0281162,-0.0494937) (0.03542, -0.0842683) (0.0418266, -0.126285) (0.047406, -0.174653) (0.0522222, -0.228608) (0.056333, -0.287494) (0.0597915, -0.350744) (0.0626464, -0.417867) (0.0649426, -0.488435) (0.0667214, -0.562079) (0.0680212, -0.638473) (0.0688779, -0.717332) (0.0693247, -0.798407) (0.0693929, -0.881476) (0.069112, -0.966345) (0.0685095, -1.05284) (0.0676118, -1.14081) (0.0664436, -1.23011) (0.0650288, -1.32062) (0.0633899, -1.41223) (0.0615487, -1.50483) (0.0595259, -1.59834) (0.0573419, -1.69266) (0.0550159, -1.78772) (0.0525669, -1.88345) (0.0500135, -1.97977) (0.0473734, -2.07663) (0.0446645, -2.17395) (0.0419039, -2.27167) (0.0391088, -2.36975) (0.0362961, -2.46813) (0.0334824, -2.56674) (0.0306845, -2.66553) (0.0279189, -2.76446) (0.0252022, -2.86346) (0.0225512, -2.96247) (0.0199823, -3.06146) (0.0175127, -3.16035) (0.0151592, -3.2591) (0.0129391, -3.35764) (0.01087, -3.45592) (0.00896969, -3.55389) (0.00725646, -3.65147) (0.00574895, -3.74863) (0.00446632, -3.84528) (0.00342826, -3.94138) (0.00265512, -4.03686) (0.00216796, -4.13165) (0.00198863, -4.22571) (0.00213983, -4.31896) (0.00264529, -4.41135) (0.00352987, -4.5028) (0.00481964, -4.59325) (0.00654211, -4.68265) (0.00872636, -4.77091) (0.0114033, -4.85798) (0.0146057, -4.94377) (0.0183688, -5.02821) (0.0227304, -5.11122) (0.027731, -5.19269) (0.033415, -5.27254) (0.0398301, -5.35063) (0.0470287, -5.42682) (0.0550684, -5.50094) (0.0640126, -5.57278) (0.0739315, -5.6421) (0.0849034, -5.70858) (0.0970161, -5.77181) (0.110368, -5.8313) (0.125072, -5.88642) (0.141254, -5.93635) (0.159061, -5.98004) (0.178661, -6.01616) (0.200249, -6.04296) (0.224054, -6.05817) (0.250344, -6.05882) (0.279438, -6.04099) (0.311718, -5.99951) (0.347643, -5.92742) (0.387775, -5.81536) (0.432805, -5.65056) (0.483594, -5.4154) (0.541231, -5.0852) (0.607109, -4.62489) (0.683044, -3.98363) (0.771445, -3.08612) (0.875582, -1.81796) (1,0)};
\node at (axis cs:-0.06,-0.38) {\scriptsize{0}};
\node at (axis cs:1.06,-0.38) {\scriptsize{1}};
\end{axis}
\end{tikzpicture}
\quad
\begin{tikzpicture}
\begin{axis}[
  tick label style={font=\scriptsize},
  axis y line=middle, 
  axis x line=middle,
  xtick={1},
  ytick={0},
  xticklabels={},
  yticklabels={},
  xlabel={\small $u$},
  ylabel={\small $v$},
every axis x label/.style={
    at={(ticklabel* cs:1.0)},
    anchor=west,
},
every axis y label/.style={
    at={(ticklabel* cs:1.0)},
    anchor=south,
},
  width=5cm,
  height=5cm,
  xmin=-0.2,
  xmax=1.2,
  ymin=-7,
  ymax=2]
\addplot [mark=none,dashed,color=black] coordinates {(1,-8) (1,3)};
\addplot [color=fst-blue,line width=1.2pt,smooth] coordinates {(0,0) (0.0104177, -0.006791) (0.0195334, -0.0258159) (0.0274633, -0.0553109) (0.0343106, -0.0938043) (0.040167,-0.140063) (0.0451148, -0.193049) (0.0492278, -0.251884) (0.0525729, -0.315827) (0.0552106, -0.384244) (0.0571959, -0.456596) (0.0585793, -0.532422) (0.0594069, -0.611324) (0.0597213, -0.692961) (0.059562, -0.77704) (0.0589653, -0.863305) (0.0579653, -0.951536) (0.0565939, -1.04154) (0.0548809, -1.13315) (0.0528543, -1.22621) (0.0505407, -1.32061) (0.0479654, -1.41622) (0.0451523, -1.51294) (0.0421244, -1.61068) (0.0389035, -1.70935) (0.0355109, -1.80889) (0.0319669, -1.90922) (0.0282913, -2.01028) (0.0245034, -2.112) (0.0206219, -2.21433) (0.0166652, -2.31721) (0.0126514, -2.42059) (0.00859816, -2.52442) (0.00452311, -2.62863) (0.000443676, -2.73317)};
\addplot [color=fst-blue,line width=1.2pt,smooth] coordinates {(0.00127753, -6.01124) (0.0113713, -6.08494) (0.0226405, -6.15628) (0.0351888, -6.22499) (0.0491345, -6.29067) (0.0646136, -6.3528) (0.0817831, -6.41064) (0.100826, -6.46319) (0.121958, -6.50907) (0.145433, -6.54638) (0.171555, -6.57251) (0.20069, -6.58387) (0.233283, -6.57544) (0.269883, -6.54021) (0.311172, -6.46829) (0.358009, -6.34558) (0.411493, -6.15164) (0.473055, -5.85646) (0.544587, -5.415) (0.628649, -4.75812) (0.72879, -3.77663) (0.850077, -2.29201) (1, 0)}; 
\node at (axis cs:-0.06,-0.38) {\scriptsize{0}};
\node at (axis cs:1.06,-0.38) {\scriptsize{1}};
\end{axis}
\end{tikzpicture} 
\caption{Qualitative representation in the $(u,v)$-plane of the curve $\Gamma_{0}(\lambda)$ defined in~\eqref{def-Gamma0} when $\lambda\in(0,\lambda^{*})$ (left), $\lambda=\lambda^{*}$ (center), and $\lambda>\lambda^{*}$ (right).} 
\label{fig-03}
\end{figure}
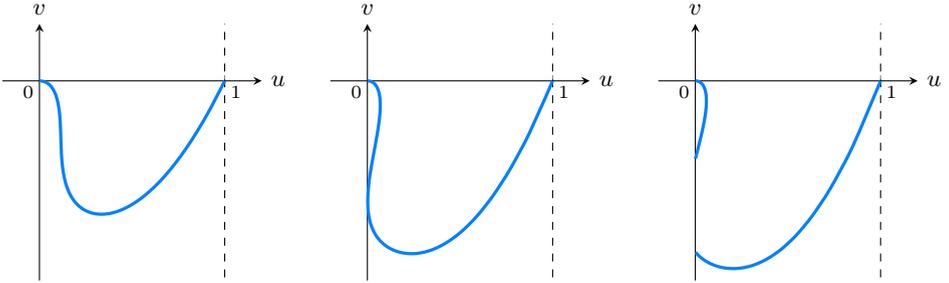

With the notation of Corollary~\ref{cor-2.1}, we observe that
\begin{equation}\label{vsi-order}
v_{s_{0}}(\sigma)=-\xi_{0}>-\xi_{1}=v_{s_{1}}(\sigma).
\end{equation}
Indeed, since the energy in~\eqref{eq-energy-lambda} is conserved and $G$ is increasing, we have
\begin{equation*}
(v_{s_{0}}(\sigma))^2=2\lambda G(s_{0})<2\lambda G(s_{1})=(v_{s_{1}}(\sigma))^2,
\end{equation*}
thus~\eqref{vsi-order} follows, since $v_{s_{i}}(\sigma)<0$, $i=0,1$.

The next result describes the behavior of the curve $\Gamma_0$ in the $(u,v)$-plane near the points $(0,0)$ and $(1,0)$. 

\begin{proposition}\label{pr-2.2}
Let $\lambda>0$. The curve $\Gamma_0=\Gamma_{0}(\lambda)$ satisfies the following properties:
\begin{itemize}
\item[$(i)$] there exists a neighborhood $\mathcal{U}_0=\mathcal{U}_0(\lambda)=[0,x_{0}(\lambda))\times(-\varepsilon(\lambda),\varepsilon(\lambda))$ of $(0,0)$ such that $\Gamma_0\cap\mathcal{U}_0$ can be parametrized, in the $(u,v)$-plane, as
\begin{equation*} 
\bigl{\{}(x,- \lambda \sigma x^2) \colon x\in[0,x_0(\lambda))\bigr{\}};
\end{equation*}
\item[$(ii)$] there exists a neighborhood $\mathcal{U}_1=\mathcal{U}_1(\lambda)=(x_{1}(\lambda),1]\times(-\varepsilon(\lambda),\varepsilon(\lambda))$ of $(1,0)$ such that $\Gamma_0\cap\mathcal{U}_1$ can be parametrized, in the $(u,v)$-plane, as
\begin{equation*}
\bigl{\{} (x,- \sqrt{\lambda} \tanh(\sqrt{\lambda}\sigma)(1-x)) \colon x\in(x_1(\lambda),1]\bigr{\}}.
\end{equation*}
\end{itemize}
\end{proposition}

\begin{proof}
Let $s\in\mathcal{I}_{\lambda}^{0}$.
Proceeding as in the proof of Proposition~\ref{pr-2.1}, for all $t\in[0,\sigma]$, we have
\begin{equation}\label{eq-2.t}
t = \dfrac{1}{\sqrt{2\lambda}} \int_{u_{s}(t)}^{s} \dfrac{\mathrm{d}u}{\sqrt{ \dfrac{u^{4}-s^{4}}{4} - \dfrac{u^{3}-s^{3}}{3} }}
= \dfrac{1}{\sqrt{2\lambda s}} \int_{\frac{u_{s}(t)}{s}}^{1} \dfrac{\mathrm{d}\xi}{\sqrt{ \dfrac{1-\xi^{3}}{3}-\dfrac{s(1-\xi^{4})}{4}}}. 
\end{equation}
Since $t$ is finite and the integrand is positive, passing to the limit as $s\to 0^{+}$, we deduce that
\begin{equation}\label{eq-u_s-unif}
\lim_{s\to 0^{+}}\dfrac{u_{s}(t)}{s}=1, \quad \text{uniformly in $t\in[0,\sigma]$.}
\end{equation}
Integrating~\eqref{eq-0sigma} in $[0,\sigma]$, we have
\begin{equation*}
u_{s}'(\sigma)=-\lambda \int_{0}^{\sigma} g(u_{s}(t)) \,\mathrm{d}t
\end{equation*}
and so
\begin{equation*}
\lim_{s\to0^{+}} \dfrac{u_{s}'(\sigma)}{s^{2}} 
= -\lambda \lim_{s\to0^{+}} \int_{0}^{\sigma} \dfrac{g(u_{s}(t))}{s^{2}} \,\mathrm{d}t.
\end{equation*}
Since $u_{s}(t)\leq s$ for all $t\in[0,\sigma]$ and $g$ is strictly increasing in $[0,2/3]$, we obtain that
\begin{equation*}
0 \leq \dfrac{g(u_{s}(t))}{s^{2}} \leq \dfrac{g(s)}{s^{2}}\leq 1, \quad \text{for all $t\in[0,\sigma]$ and $s\in[0,2/3]$}.
\end{equation*}
The dominated convergence theorem implies that
\begin{align*}
\lim_{s\to0^{+}} \dfrac{u_{s}'(\sigma)}{s^{2}}
&= -\lambda \int_{0}^{\sigma} \lim_{s\to0^{+}} \dfrac{g(u_{s}(t))}{s^{2}} \,\mathrm{d}t
\\
&= - \lambda \int_{0}^{\sigma} \lim_{s\to0^{+}} \dfrac{g(u_{s}(t))}{(u_{s}(t))^2} \dfrac{(u_{s}(t))^2}{s^{2}} \,\mathrm{d}t
= - \lambda \sigma,
\end{align*}
where the last equality follows from \eqref{eq-u_s-unif} and the facts that 
$\lim_{s\to0^{+}}u_{s}(t)=0$ and $g(u)=u^{2}+o(u^{2})$ for $u\to 0$. 
Finally, we conclude that
\begin{equation}\label{limit-v_s-to-0}
\lim_{s\to0^{+}} \dfrac{u_{s}'(\sigma)}{(u_{s}(\sigma))^{2}} = - \lambda \sigma,
\end{equation}
and so statement $(i)$ is proved.

As for $\Gamma_{0}$ near $(1,0)$, since
\begin{equation}\label{eq-2.10}
t = \dfrac{1}{\sqrt{2\lambda}} \int_{u_{s}(t)}^{s} \dfrac{\mathrm{d}u}{\sqrt{\dfrac{u^{4}-s^{4}}{4} - \dfrac{u^{3}-s^{3}}{3}}}, \quad \text{for all $t\in[0,\sigma]$,}
\end{equation}
from the fact that the integrand behaves like $1/(u-s)$ for $u$ near $s$ (cf.,~\eqref{eq-2.5}), we deduce that
\begin{equation}\label{eq-conv_us_1}
\lim_{s\to 1^{-}} u_{s}(t) = 1, \quad \text{uniformly in $t\in[0,\sigma]$.}
\end{equation}
Performing the changes of variable $u=1-z$ and $z=(1-s)\xi$ in \eqref{eq-2.10}, we obtain
\begin{align}
t &= -\dfrac{1}{\sqrt{2\lambda}} \int_{1-u_{s}(t)}^{1-s} \dfrac{\mathrm{d}z}{\sqrt{ \dfrac{(1-z)^{4}-s^{4}}{4} - \dfrac{(1-z)^{3}-s^{3}}{3}}}
\\
&= \dfrac{1}{\sqrt{2\lambda}} \int_{1}^{\frac{1-u_{s}(t)}{1-s}} \dfrac{(1-s) \, \mathrm{d}\xi}{\sqrt{ \dfrac{\bigl{(}(1-(1-s)\xi\bigr{)}^{4}-s^{4}}{4} - \dfrac{\bigl{(}(1-(1-s)\xi\bigr{)}^{3}-s^{3}}{3}}},
\label{eq-1-u_s}
\end{align}
for all $t\in[0,\sigma]$.
By De l'H\^{o}pital's rule, we have that
\begin{align*}
\lim_{s\to 1^{-}}\dfrac{G(s)-G(1-(1-s)\xi)}{(1-s)^{2}} 
&= \lim_{s\to 1^{-}}\dfrac{g(s)-g(1-(1-s)\xi)\xi}{-2(1-s)}\\
&=\lim_{s\to 1^{-}}\dfrac{g'(s)-g'(1-(1-s)\xi)\xi^{2}}{2} \\
&= \dfrac{g'(1)(1-\xi^{2})}{2} = \dfrac{\xi^{2}-1}{2},
\end{align*}
thus, for every $t\in[0,\sigma]$, letting
\begin{equation*}
\ell_{t} := \liminf_{s\to 1^{-}} \dfrac{1-u_{s}(t)}{1-s} \in [0,+\infty]
\end{equation*}
and passing to the $\liminf$ as $s\to 1^{-}$ in \eqref{eq-1-u_s}, we deduce
\begin{equation*}
t = \dfrac{1}{\sqrt{\lambda}} \int_{1}^{\ell_{t}} \dfrac{1}{\sqrt{\xi^2-1}}\, \mathrm{d}\xi
= \dfrac{1}{\sqrt{\lambda}} \log \left(\ell_{t}+\sqrt{\ell_{t}^{2}-1} \right) = \dfrac{1}{\sqrt{\lambda}} \cosh^{-1}(\ell_{t}).
\end{equation*}
By inverting the above equality, we obtain that
\begin{equation*}
\ell_{t} = \cosh(\sqrt{\lambda}t)\in(1,\cosh(\sqrt{\lambda}\sigma)],\quad \text{for all $t\in[0,\sigma]$.}
\end{equation*}
Proceeding in the same way with the $\limsup$, we thus conclude that
\begin{equation}\label{eq-1-u_s-unif}
\lim_{s\to 1^{-}} \dfrac{1-u_{s}(t)}{1-s} = \cosh(\sqrt{\lambda}t),\quad \text{for all $t\in[0,\sigma]$.}
\end{equation}
Arguing as above, integrating equation \eqref{eq-0sigma} and using the dominated convergence theorem, we deduce that
\begin{equation}\label{eq-u_s'_1}
\lim_{s\to1^{-}} \dfrac{u_{s}'(t)}{1-s} = -\lambda \lim_{s\to 1^{-}} \int_{0}^{\sigma} \dfrac{g(u_{s}(\xi))}{1-s} \,\mathrm{d}\xi = - \sqrt{\lambda} \sinh(\sqrt{\lambda}\sigma),
\end{equation}
where the last equality follows from \eqref{eq-conv_us_1}, \eqref{eq-1-u_s-unif}, and the fact that $g(u)=1-u+o(1-u)$ for $u\to 1$.
Finally, from \eqref{eq-1-u_s-unif} and \eqref{eq-u_s'_1}, we have
\begin{equation*}
\lim_{s\to1^{-}} \dfrac{u_{s}'(\sigma)}{1-u_{s}(\sigma)} = -\sqrt{\lambda} \tanh(\sqrt{\lambda}\sigma).
\end{equation*}
Then, statement $(ii)$ is proved.
\end{proof}

\medskip

We conclude this section by stating the analogous results in the interval $[1-\sigma,1]$. As above, for every $s\in(0,1)$, we introduce the initial value problem
\begin{equation}\label{eq-initial1}
\begin{cases}
\, u'=v, \\
\, v' = -\lambda g(u), \\
\, u(1)=s, \\
\, v(1)=0.
\end{cases}
\end{equation}
Let $(\hat{u}_{s},\hat{v}_{s})$ be the unique solution of~\eqref{eq-initial1}, considered in its maximal interval of existence (contained in $\mathbb{R}$). 
For all $s\in(0,1)$, we denote by $\hat{T}_{0}(s)$ the time taken by $(\hat{u}_{s},\hat{v}_{s})$ to go from the point $(s,0)$ to the line $\{0\}\times(0,+\infty)$ moving backwards along the level line $H_{\lambda}(u,v)=2\lambda G(s)$ in the $(u,v)$-plane. 

By setting
\begin{equation*}
\mathcal{I}^{1}_{\lambda} = \bigl{\{} s\in(0,1) \colon \hat{T}_{0}(s)>\sigma \bigr{\}},
\end{equation*}
we state the symmetric statement of Proposition~\ref{pr-2.2} as follows.

\begin{proposition}\label{pr-2.3}
There exist $\lambda^{*}>0$ and $s^{*}\in(0,1)$ such that
\begin{itemize}
\item[$(i)$] for all $\lambda\in(0,\lambda^{*})$, $\mathcal{I}^{1}_{\lambda}=(0,1)$;
\item[$(ii)$] for $\lambda=\lambda^{*}$, $\mathcal{I}^{1}_{\lambda}=(0,1)\setminus\{s^{*}\}$;
\item[$(iii)$] for all $\lambda>\lambda^{*}$, there exist $s_{0},s_{1}\in(0,1)$ with $s_{0}<s^{*}<s_{1}$ such that $\mathcal{I}^{1}_{\lambda}=(0,s_{0})\cup(s_{1},1)$.
\end{itemize}
\end{proposition}

Similarly, setting
\begin{equation}\label{def-Gamma1}
\Gamma_{1}=\Gamma_{1}(\lambda)=\bigl{\{} (\hat{u}_{s}(1-\sigma),\hat{v}_{s}(1-\sigma)) \colon s\in \mathcal{I}^{1}_{\lambda} \bigr{\}},
\end{equation}
the following results correspond to Corollary~\ref{cor-2.1} and Proposition~\ref{pr-2.2}.

\begin{corollary}\label{cor-2.2}
There exists $\lambda^{*}>0$ such that
\begin{itemize}
\item for $\lambda\in(0,\lambda^{*})$, $\overline{\Gamma_1(\lambda)}\cap (\{0\}\times(0,+\infty))=\emptyset$;
\item for $\lambda=\lambda^{*}$, $\overline{\Gamma_1(\lambda)}\cap(\{0\}\times(0,+\infty))=\{(0,\xi^{*})\}$ for some $\xi^{*}\in(0,+\infty)$;
\item for $\lambda>\lambda^{*}$, $\overline{\Gamma_1(\lambda)}\cap (\{0\}\times(0,+\infty)) = \{(0,\xi_{0}),(0,\xi_{1})\}$ for some $0<\xi_{0}<\xi_{1}$.
\end{itemize}
\end{corollary}

\begin{proposition}\label{pr-2.4}
Let $\lambda>0$. The curve $\Gamma_1=\Gamma_{1}(\lambda)$ satisfies the following properties:
\begin{itemize}
\item[$(i)$] there exists a neighborhood $\mathcal{U}_0=\mathcal{U}_0(\lambda)=[0,x_{0}(\lambda))\times(-\varepsilon(\lambda),\varepsilon(\lambda))$ of $(0,0)$ such that $\Gamma_1\cap\mathcal{U}_0$ can be parametrized, in the $(u,v)$-plane, as
\begin{equation*} 
\bigl{\{}(x, \lambda \sigma x^2) \colon x\in[0,x_0(\lambda))\bigr{\}};
\end{equation*}
\item[$(ii)$] there exists a neighborhood $\mathcal{U}_1=\mathcal{U}_1(\lambda)=(x_{1}(\lambda),1]\times(-\varepsilon(\lambda),\varepsilon(\lambda))$ of $(1,0)$ such that $\Gamma_1\cap\mathcal{U}_1$ can be parametrized, in the $(u,v)$-plane, as
\begin{equation*}
\bigl{\{} (x, \sqrt{\lambda} \tanh(\sqrt{\lambda}\sigma)(1-x)) \colon x\in(x_1(\lambda),1]\bigr{\}}.
\end{equation*}
\end{itemize}
\end{proposition}

\begin{remark}\label{rem-2.1}
By the symmetry of the weight function \eqref{eq-weight}, we stress that $\Gamma_{1}=\bigl{\{} (u,-v) \colon (u,v)\in\Gamma_{0} \bigr{\}}$ and thus the values $\lambda^{*}$, $s^{*}$, $s_{0}$, $s_{1}$, $\xi^{*}$, $\xi_{0}$, $\xi_{1}$ in Corollaries~\ref{cor-2.1} and~\ref{cor-2.2} are exactly the same. Moreover, the neighborhoods $\mathcal{U}_0$ and $\mathcal{U}_1$ can be taken to coincide in Propositions~\ref{pr-2.2} and~\ref{pr-2.4}.
\hfill$\lhd$
\end{remark}

\section{Phase-plane analysis in $\mathopen{(}\sigma,1-\sigma\mathclose{)}$}\label{section-3}

In this section, we analyze the equation in~\eqref{eq-main} in the interval $(\sigma,1-\sigma)$, that is where the weight function $a_{\lambda,\mu}\equiv-\mu<0$. Accordingly, we study 
\begin{equation}\label{eq-3.1}
u''(t)=\mu g(u)
\end{equation}
and we consider its associated energy 
\begin{equation}\label{eq-energy-mu}
H_\mu(u,v) :=v^2-2\mu G(u),
\end{equation}
where $G$ is defined as in~\eqref{def-G}. Let 
\begin{equation}\label{def-stable-manifold}
\mathcal{M}_{\mu}:=\bigl{\{} (u,v)\in [0,1]\times\mathbb{R} \colon v^{2} = 2\mu G(u) \bigr{\}}
\end{equation}
be the level line of $H_\mu$ passing through $(0,0)$.
We note that the curve $\mathcal{M}_{\mu}$ divides the strip $[0,1]\times\mathbb{R}$ into three connected open regions, which can be characterized by the sign of $H_{\mu}$. More precisely, $H_{\mu}(u,v)>0$ when $(u,v)$ belongs to the two unbounded regions ``outside'' $\mathcal{M}_{\mu}$, while $H_{\mu}(u,v)<0$ when $(u,v)$ belongs to the bounded region ``inside'' $\mathcal{M}_{\mu}$ (containing the segment $(0,1)\times\{0\}$). Clearly, by definition \eqref{def-stable-manifold}, $H_{\mu}(u,v)=0$ for all $(u,v)\in\mathcal{M}_{\mu}$.

First, since $\Gamma_{0}$ and  $\Gamma_{1}$ behave like a parabola near $(0,0)$ according to Proposition~\ref{pr-2.2}~$(i)$ and Proposition~\ref{pr-2.4}~$(i)$, we analyze the time $T_{p}=T_p(x)$ taken to reach the $u$-axis for the first time starting from points on the parabola
\begin{equation}\label{eq-parabola}
(x,y(x))=(x,-kx^2), \qquad x\in(0,1), \quad k>0,
\end{equation}
and moving along the level lines associated with~\eqref{eq-energy-mu}.

We observe that if $k^2\leq\frac{\mu}{6}$ the parabola \eqref{eq-parabola} does not intersect the manifold $\mathcal{M}_{\mu}$ in the strip $(0,1)\times\mathbb{R}$ of the $(u,v)$-plane. On the contrary, if $k^2>\frac{\mu}{6}$  there is a unique intersection whose abscissa is denoted by $x_p=x_p(\mu)$. We extend this definition in the case $k^2\leq\frac{\mu}{6}$ by setting $x_p=1$ so that $T_p$ is defined in $(0,x_p)$ for every $k>0$.

For $x\in(0,x_{p})$, let $m(x)$ be the abscissa of the intersection point between the level line of~\eqref{eq-energy-mu} through $(x,y(x))$ and the segment $(0,1)\times\{0\}$. From \eqref{eq-3.1} we have that $u$ is strictly decreasing along the level lines associated with~\eqref{eq-energy-mu} in $(0,1)\times(-\infty,0]$, as a consequence
\begin{equation}\label{eq-3.4}
0<m(x)<x<1.
\end{equation}
Moreover, $m(x)$ satisfies
\begin{equation}\label{eq-3.5}
(y(x))^2-2\mu G(x) = -2\mu G(m(x)).
\end{equation}
Hence, from \eqref{eq-energy-mu} and \eqref{eq-3.5}, the time $T_{p}$ we are interested in is given by
\begin{align}
T_p(x)&=\int_{m(x)}^x\frac{\mathrm{d}u}{\sqrt{(y(x))^2+2\mu(G(u)-G(x))}} \notag \\
&=\int_{m(x)-x}^0\frac{\mathrm{d}\tilde{u}}{\sqrt{(y(x))^2+2\mu(G(x+\tilde{u})-G(x))}} \notag \\
&=\int_{0}^{1}\frac{(x-m(x))\,\mathrm{d}\xi}{\sqrt{(y(x))^2+2\mu(G(x-(x-m(x))\xi)-G(x))}}, \label{eq-3.7}
\end{align}
where we have performed the changes of variable $u=x+\tilde{u}$ and $\tilde{u}=-(x-m(x))\xi$.

In the following result, we collect some properties of $T_{p}$.

\begin{proposition}\label{pr-3.1}
The function $T_p \colon (0,x_p)\to(0,+\infty)$ is continuous, strictly increasing, and satisfies:
\begin{itemize}
\item[$(i)$] $\displaystyle \lim_{x\to 0^{+}}T_p(x) = \tfrac{k}{\mu}$;
\item[$(ii)$] if $x_p<1$, then $\displaystyle \lim_{x\to (x_p)^{-}}T_p(x)=+\infty$.
\end{itemize}
\end{proposition}

\begin{proof}
The continuity follows directly from expression~\eqref{eq-3.7}. To prove the limit in~$(i)$, we first claim that the function $m(x)$ in a right neighborhood of $0$ behaves as follows
\begin{equation}\label{eq-3.8}
m(x)=x-\frac{k^2}{2\mu}x^2+o(x^2), \qquad \text{as $x\to0^{+}$.}
\end{equation}
Indeed, \eqref{eq-3.4} implies that
\begin{equation*}
\lim_{x\to 0^{+}} m(x)=0.
\end{equation*}
Recalling \eqref{eq-parabola}, we rewrite \eqref{eq-3.5} as
\begin{equation*}
\biggl{(}\frac{m(x)}{x}\biggr{)}^{\!3} \biggl{(}\frac{1}{3}-\frac{m(x)}{4}\biggr{)}=-\frac{k^2}{2\mu}x+\biggl{(}\frac{1}{3}-\frac{x}{4}\biggr{)}.
\end{equation*}
Thus, by taking the $\liminf$ and $\limsup$ as $x\to 0^{+}$ in the previous relation, we obtain that $l:=\liminf_{x\to 0^{+}}\frac{m(x)}{x}$ and $L:=\limsup_{x\to 0^{+}}\frac{m(x)}{x}$ satisfy $l^3=L^3=1$. Hence, it follows that
\begin{equation}\label{eq-3.9}
\lim_{x\to 0^{+}}\frac{m(x)}{x}=1,
\end{equation}
which verifies the first term in the expansion \eqref{eq-3.8}. From \eqref{eq-3.5}, we deduce that
\begin{equation*}
\frac{m(x)-x}{x^2}\left(\frac{1+\frac{m(x)}{x}+\frac{m(x)^2}{x^2}}{3}-\frac{x+m(x)+\frac{m(x)^2}{x}+\frac{m(x)^3}{x^2}}{4}\right)=-\frac{k^2}{2\mu},
\end{equation*}
which, thanks to \eqref{eq-3.9}, implies that
\begin{equation*}
\lim_{x\to 0^{+}}\frac{m(x)-x}{x^2}=-\frac{k^2}{2\mu}
\end{equation*}
and concludes the proof of \eqref{eq-3.8}. By computing
\begin{align*}
&G(x-(x-m(x))\xi)-G(x)= \\
&= \frac{-3x^2(x-m(x))\xi+3x(x-m(x))^2\xi^2-(x-m(x))^3\xi^3}{3} \\
&\quad -\frac{-4x^3(x-m(x))\xi+6x^2(x-m(x))^2\xi^2-4x(x-m(x))^3\xi^3+(x-m(x))^4\xi^4}{4},
\end{align*}
from \eqref{eq-3.8}, we obtain
\begin{equation*}
\lim_{x\to 0^{+}} \frac{G\left(x-(x-m(x))\xi\right)-G\left(x\right)}{x^4} = -\frac{k^2}{2\mu}\xi, \qquad \text{for every $\xi\in(0,1)$.}
\end{equation*}
Hence, it follows from \eqref{eq-3.7} that
\begin{align*}
\lim_{x\to 0^{+}} T_p(x) 
&=\lim_{x\to 0^{+}}\int_{0}^{1}\frac{\frac{x-m(x)}{x^{2}}}{\sqrt{\frac{(y(x))^{2}}{x^4}+2\mu\frac{G(x-(x-m(x))\xi)-G(x)}{x^4}}} \,\mathrm{d}\xi\\
& =\frac{k^{2}}{2\mu} \int_{0}^{1} \frac{\mathrm{d}\xi}{\sqrt{k^{2}(1-\xi)}}=\frac{k}{\mu},
\end{align*}
proving $(i)$.

\smallskip

Statement $(ii)$ directly follows by continuity, since the time taken to reach the $u$-axis on the manifold $\mathcal{M}_{\mu}$ is $+\infty$ and, as $x\to (x_p)^{-}$, the starting point approaches the intersection of the parabola and the manifold.

It remains to prove the monotonicity of $T_p(x)$. By differentiating \eqref{eq-3.7}, we obtain
\begin{equation*}
T_p'(x)=\int_{0}^{1} N(x,\xi)\bigl{(} (y(x))^{2}+2\mu(G(x-(x-m(x))\xi)-G(x))\bigr{)}^{\!-\frac{3}{2}}\,\mathrm{d}\xi,
\end{equation*}
where we have set
 \begin{align}\label{eq-3.10}
&N(x,\xi) = (1-m'(x))\bigl{(}(y(x))^{2}+2\mu(G(x-(x-m(x))\xi)-G(x))\bigr{)} \\
&\qquad -(x-m(x))\bigl{(}y(x)y'(x)+\mu (g(x-(x-m(x))\xi) (1-(1-m'(x))\xi) -g(x))\bigr{)}.
\end{align}
We aim to show that $N(x,\xi)>0$ for all $x\in(0,x_{p})$ and $\xi\in(0,1)$. To this end, we observe that, by differentiating \eqref{eq-3.5}, we get 
\begin{equation}\label{eq-3.11}
y(x)y'(x)-\mu g(x)=-\mu g(m(x))m'(x),
\end{equation}
thus \eqref{eq-3.10} and \eqref{eq-3.11} give that
\begin{align*}
N(x,1)&=(1-m'(x))\bigl{(}(y(x))^{2}+2\mu(G(m(x))-G(x))\bigr{)} \\
&\quad-(x-m(x))\bigl{(}y(x)y'(x)+\mu(g(m(x))m'(x)-g(x))\bigr{)} \\
&=0.
\end{align*}
Lemma~\ref{le-A.1} in Appendix~\ref{appendix-A} shows that $\partial_\xi N(x,\xi)<0$ for all $x\in(0,x_{p})$ and $\xi\in(0,1)$.
This concludes the proof.
\end{proof}

Second, according to Proposition~\ref{pr-2.2}~$(ii)$  and Proposition~\ref{pr-2.4}~$(ii)$, we analyze the time $T_{l}=T_l(x)$ taken to reach the $u$-axis for the first time starting from points on the line
\begin{equation}\label{eq-line}
(x,y(x))=(x,-k(1-x)), \qquad x\in(0,1), \quad k>0,
\end{equation}
and moving along the level lines associated with~\eqref{eq-energy-mu}.

We denote by $x_l=x_l(\mu)$ the abscissa of the unique intersection of such a line with the set $\mathcal{M}_{\mu}$. In analogy with the previous situation, for $x\in(x_{l},1)$, with $m(x)$ we still indicate the abscissa of the intersection point between the level line of~\eqref{eq-energy-mu} through $(x,y(x))$ and the segment $(0,1)\times\{0\}$. With this notation, the time-map $T_l(x)$ is defined for every $x\in(x_l,1)$ and is given by
\begin{align}
T_l(x)&=\int_{m(x)}^x\frac{\mathrm{d}u}{\sqrt{(y(x))^{2}+2\mu(G(u)-G(x))}} \notag \\
&=\int_{1-x}^{1-m(x)}\frac{\mathrm{d}\tilde{u}}{\sqrt{(y(x))^{2}+2\mu(G(1-\tilde{u})-G(x))}} \notag \\
&=\int_{1}^{\frac{1-m(x)}{1-x}}\frac{(1-x)\,\mathrm{d}\xi}{\sqrt{(y(x))^{2}+2\mu(G(1-(1-x)\xi)-G(x))}}. \label{eq-3.15}
\end{align}

The properties of $T_{l}$ are contained in the following result.

\begin{proposition}\label{pr-3.2}
The function $T_l \colon (x_l,1)\to(0,+\infty)$ is continuous, strictly decreasing, and satisfies:
\begin{enumerate}
\item[$(i)$] $\displaystyle \lim_{x\to 1^{-}}T_l(x)=\tfrac{1}{\sqrt{\mu}}\arctan \tfrac{k}{\sqrt{\mu}}$;
\item[$(ii)$] $\displaystyle \lim_{x\to (x_l)^{+}}T_l(x)=+\infty$.
\end{enumerate}
\end{proposition}

\begin{proof}
The continuity and $(ii)$ follow exactly as in Proposition~\ref{pr-3.1}. As for statement~$(i)$, we start by showing that
\begin{equation}\label{eq-3.13}
m(x)=1-\sqrt{1+\frac{k^{2}}{\mu}}(1-x)+o(1-x), \quad \text{as $x\to 1^{-}$.}
\end{equation}
Indeed, from \eqref{eq-3.11}, we obtain that
\begin{equation}\label{eq-3.14}
m'(x)=\frac{\mu g(x)-y(x)y'(x)}{\mu g(m(x))}>0, \quad \text{for all $x\in(x_{l},1)$,}
\end{equation}
since $g$ is positive in $(0,1)$, $y$ is negative and $y'$ is positive. Thus, $\lim_{x\to 1}m(x)$ exists and, by taking the limit in \eqref{eq-3.5}, it belongs to $G^{-1}(\{G(1)\})$. Since $G$ is strictly increasing in $[0,1]$, necessarily
\begin{equation*}
\lim_{x\to 1^{-}}m(x)=1.
\end{equation*}
If we denote
\begin{equation*}
l:=\liminf_{x\to 1^{-}} m'(x), \qquad L:=\limsup_{x\to 1^{-}} m'(x),
\end{equation*}
thanks to \eqref{eq-3.14} and the generalized De l'H\^{o}pital's rule, we deduce
\begin{align*}
\frac{\mu+k^{2}}{\mu l}&=\frac{\mu g'(1)-k^{2}}{\mu g'(1) l}
=\liminf_{x\to 1^{-}}\frac{\mu g'(x)-y'(x)^{2}-y(x)y''(x)}{\mu g'(m(x))m'(x)}\leq l
\\
&\leq L\leq \limsup_{x\to 1^{-}}\frac{\mu g'(x)-y'(x)^{2}-y(x)y''(x)}{\mu g'(m(x))m'(x)}=\frac{\mu g'(1)-k^{2}}{\mu g'(1) L} = \frac{\mu+k^{2}}{\mu L}.
\end{align*}
We thus have that all the inequalities in the previous relation are actually equalities and we conclude that
\begin{equation*}
\lim_{x\to 1^{-}}\frac{1-m(x)}{1-x}=\lim_{x\to 1^{-}}m'(x)=\sqrt{1+\frac{k^{2}}{\mu}}.
\end{equation*}
The proof of \eqref{eq-3.13} is complete.
From
\begin{align*}
&\lim_{x\to 1^{-}} \frac{G(1-(1-x)\xi)-G(x)}{(1-x)^{2}}
=\lim_{x\to 1^{-}}\frac{g(1-(1-x)\xi)\xi-g(x)}{-2(1-x)}=
\\
&=\lim_{x\to 1^{-}}\frac{g'(1-(1-x)\xi)\xi^{2}-g'(x)}{2}
=\frac{1-\xi^{2}}{2},
\qquad \text{for every $\xi\in(0,1)$,}
\end{align*}
by passing to the limit as $x\to 1^{-}$ in \eqref{eq-3.15}, we obtain
\begin{equation*}
\lim_{x\to 1^{-}}T_l(x)
=\int_{1}^{\sqrt{1+\frac{k^{2}}{\mu}}}\dfrac{\mathrm{d}\xi}{\sqrt{k^2+\mu (1-\xi^2)}}
=\frac{1}{\sqrt{\mu}}\arctan\frac{k}{\sqrt{\mu}},
\end{equation*}
which proves $(i)$. 

To study the monotonicity of $T_{l}$, we can repeat the computations of the proof of Proposition~\ref{pr-3.1} and conclude by Lemma~\ref{le-A.2}.
\end{proof}

\section{Analysis of connection times from $\Gamma_{0}$ to $\Gamma_{1}$}\label{section-4}

This section is devoted to the study of the time necessary to ``connect'', in the strip $[0,1]\times\mathbb{R}$ of the $(u,v)$-plane, the curve $\Gamma_{0}$ defined in~\eqref{def-Gamma0} with the curve $\Gamma_{1}$ defined in~\eqref{def-Gamma1} by moving along the level lines associated with~\eqref{eq-energy-mu}. We observe that connections are possible only inside the region ``embraced'' by $\mathcal{M}_{\mu}$ (see Figures~\ref{fig-05} and~\ref{fig-07}).
A fundamental property that we repeatedly use throughout this section is that $\Gamma_{0}$ and $\Gamma_{1}$ are mutually symmetric with respect to the axis $\{v=0\}$ (see Remark~\ref{rem-2.1}). 

Since the shapes of $\Gamma_{0}$ and $\Gamma_{1}$ depend on $\lambda$ in accord with Corollaries~\ref{cor-2.1} and~\ref{cor-2.2}, it is convenient to divide the analysis into two cases: $\lambda \in[\lambda^*,+\infty)$ and $\lambda \in(0,\lambda^{*})$.

\subsection{The case $\lambda\in[\lambda^*,+\infty)$}\label{section-4.1}

Let $\lambda\in[\lambda^*,+\infty)$ be fixed. 
From the behavior near $(0,0)$ of $\mathcal{M}_{\mu}$ defined in \eqref{def-stable-manifold}, which is given by
\begin{equation}\label{Mnear0}
|v| = \sqrt{2\mu G(u)} = \sqrt{\dfrac{2}{3}\mu} u^{\frac{3}{2}} + o(u^{\frac{3}{2}}), \quad \text{as $u\to0^{+}$,}
\end{equation}
and the properties of $\Gamma_{0}$ described in Corollary~\ref{cor-2.1}, we can ensure the existence of at least two intersections between $\Gamma_{0}$ and $\mathcal{M}_{\mu}$ in $(0,1)\times\mathbb{R}$. Accordingly, let $s_{0}^{\mathcal{M}},s_{1}^{\mathcal{M}}\in(0,1)$ with $s_{0}^{\mathcal{M}}<s_{1}^{\mathcal{M}}$ be such that
\begin{equation}\label{eq-4.1}
(u_{s_{0}^{\mathcal{M}}}(\sigma),v_{s_{0}^{\mathcal{M}}}(\sigma)),(u_{s_{1}^{\mathcal{M}}}(\sigma),v_{s_{1}^{\mathcal{M}}}(\sigma))\in \mathcal{M}_{\mu}.
\end{equation}
In principle, there could be more than two values $s\in(0,1)$ giving intersections between $\Gamma_{0}$ and $\mathcal{M}_{\mu}$; if it is the case, we consider only the closest to $0$ and to $1$, respectively. This implies that
\begin{equation}\label{uM-order}
u_{s_{0}^{\mathcal{M}}}(\sigma)<u_{s_{1}^{\mathcal{M}}}(\sigma).
\end{equation}
Indeed, thanks to our definition, we have that the points $(u_{s}(\sigma),v_{s}(\sigma))\in\Gamma_{0}$ lie outside the region ``embraced'' by $\mathcal{M}_{\mu}$ for $s\in(s_{0}^{\mathcal{M}},s_{0})\cup(s_{1},s_{1}^{\mathcal{M}})$, where $s_0$ and $s_1$ are the ones given by Proposition~\ref{pr-2.1}~$(iii)$. Therefore, if we assume by contradiction that~\eqref{uM-order} does not hold, by continuity and~\eqref{vsi-order} we have
\begin{equation*}
\left\{(u_{s}(\sigma),v_{s}(\sigma))\colon s\in[s_{0}^{\mathcal{M}},s_{0})\right\}\cap\left\{(u_{s}(\sigma),v_{s}(\sigma))\colon s\in(s_{1},s_{1}^{\mathcal{M}}]\right\}\neq\emptyset,
\end{equation*}
against the uniqueness of solution for the initial value problems~\eqref{eq-initial0}. 

Recalling the expression of the energy $H_{\mu}$ given in \eqref{eq-energy-mu}, we introduce the function
\begin{equation}\label{def-hs}
h_{\mu}(s)=H_{\mu}(u_{s}(\sigma),v_{s}(\sigma)), \quad s\in(0,s_{0}^{\mathcal{M}}]\cup[s_{1}^{\mathcal{M}},1),
\end{equation}
and we extend it by continuity in $s=0$ and $s=1$.
By the sign properties of $H_{\mu}$ in the regions separated by $\mathcal{M}_{\mu}$ (see the discussion in Section~\ref{section-3}), we deduce that $h_{\mu}(s)<0$ for all $s\in(0,s_{0}^{\mathcal{M}})\cup(s_{1}^{\mathcal{M}},1]$ and, moreover, $h_{\mu}$ vanishes in $0$, $s_{0}^{\mathcal{M}}$ and $s_{1}^{\mathcal{M}}$. 
Therefore, the function $h_{\mu}$ has at least one local minimum in $(0,s_{0}^{\mathcal{M}})$. 
At last, it is easy to infer that $s=1$ is a global minimum point of $h_{\mu}$. Indeed, by denoting $(m(u_{s}(\sigma)),0)$ the intersection point between the level line of~\eqref{eq-energy-mu} passing through $(u_{s}(\sigma),v_{s}(\sigma))$ and the $u$-axis, we have $\{h_{\mu}(s) \colon s\in(0,s_{0}^{\mathcal{M}})\cup(s_{1}^{\mathcal{M}},1)\}=\{-G(m(u_{s}(\sigma)))\colon s\in(0,s_{0}^{\mathcal{M}})\cup(s_{1}^{\mathcal{M}},1)\}$, by energy conservation, and the monotonicity of $G$ gives the assertion.

From now on, for sake of simplicity, we study the situation in which $h_{\mu}$ has a unique local minimum point in $(0,s_{0}^{\mathcal{M}})\cup(s_{1}^{\mathcal{M}},1)$. In this case, such a point belongs to $(0,s_{0}^{\mathcal{M}})$ and the  qualitative graph of $h_{\mu}$ is depicted in Figure~\ref{fig-05}~(left). The situation which includes multiple extrema points of $h_{\mu}$ can be treated similarly and contains the minimal configuration we investigate here.

Let us call $s_{0}^{\tau}$ the (local) minimum point of $h_{\mu}$ in $(0,s_{0}^{\mathcal{M}})$,
and $s_{1}^{\tau}$ the point in $(s_{1}^{\mathcal{M}},1)$ such that $h_{\mu}(s_{0}^{\tau})= h_{\mu}(s_{1}^{\tau})$. 
We stress that
\begin{equation*}
0 < s_{0}^{\tau} < s_{0}^{\mathcal{M}} < s_{1}^{\mathcal{M}} < s_{1}^{\tau} < 1
\end{equation*}
and that the level line $H_{\mu}(u,v)=H_{\mu}(u_{s_{0}^{\tau}}(\sigma),v_{s_{0}^{\tau}}(\sigma))$ is tangent to the curve $\Gamma_{0}$ in $(u_{s_{0}^{\tau}}(\sigma),v_{s_{0}^{\tau}}(\sigma))$. See Figure~\ref{fig-05}~(right) for a graphical representation.

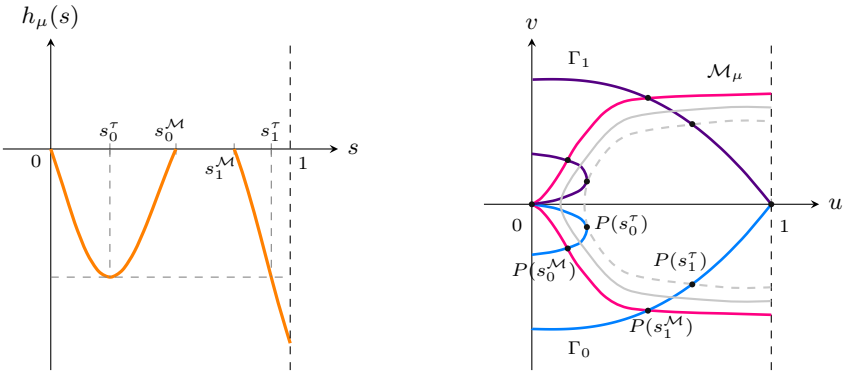
\begin{figure}[htb]
\begin{tikzpicture}
\begin{axis}[
  tick label style={font=\scriptsize},
  axis y line=middle, 
  axis x line=middle,
  xtick={0,0.247,0.523,0.766,0.921},
  ytick={0},
  xticklabels={},
  yticklabels={},
  xlabel={\small $s$},
  ylabel={\small $h_{\mu}(s)$},
every axis x label/.style={
    at={(ticklabel* cs:1.0)},
    anchor=west,
},
every axis y label/.style={
    at={(ticklabel* cs:1.0)},
    anchor=south,
},
  width=6cm,
  height=6cm,
  xmin=-0.2,
  xmax=1.2,
  ymin=-2,
  ymax=1]
\addplot [color=fst-orange,line width=1.2pt,smooth] coordinates {(0,0) (0.05, -0.328455) (0.1, -0.667727) (0.15, -0.937504) (0.2, -1.10656) (0.25, -1.15899) (0.3, -1.09412) (0.35, -0.926122) (0.4, -0.682006) (0.45, -0.398265) (0.5, -0.116374) (0.523, 0)};
\addplot [color=fst-orange,line width=1.2pt,smooth] coordinates {(0.766, 0) (0.8, -0.210962) (0.85, -0.579163) (0.9, -0.987341) (0.95, -1.39377) (1, -1.75699)};
\addplot [mark=none,dashed,color=black] coordinates {(1,-2) (1,2)};
\addplot [mark=none,dashed,color=gray] coordinates {(0.247,0) (0.247,-1.159)};
\addplot [mark=none,dashed,color=gray] coordinates {(0.921,0) (0.921,-1.159)};
\addplot [mark=none,dashed,color=gray] coordinates {(0,-1.159) (1,-1.159)};
\node[label={220:{\scriptsize{0}}}] at (axis cs:0.04,0.1) {};
\node[label={310:{\scriptsize{1}}}] at (axis cs:0.96,0.1) {};
\node[label={90:{\scriptsize{$s_{0}^{\tau}$}}}] at (axis cs:0.247,-0.1) {};
\node[label={90:{\scriptsize{$s_{0}^{\mathcal{M}}$}}}] at (axis cs:0.5,-0.1) {};
\node[label={270:{\scriptsize{$s_{1}^{\mathcal{M}}$}}}] at (axis cs:0.71,0.1) {};
\node[label={90:{\scriptsize{$s_{1}^{\tau}$}}}] at (axis cs:0.921,-0.1) {};
\end{axis}
\end{tikzpicture}
\qquad\qquad
\begin{tikzpicture}
\begin{axis}[
  tick label style={font=\scriptsize},
  axis y line=middle, 
  axis x line=middle,
  xtick={1},
  ytick={0},
  xticklabels={},
  yticklabels={},
  xlabel={\small $u$},
  ylabel={\small $v$},
every axis x label/.style={
    at={(ticklabel* cs:1.0)},
    anchor=west,
},
every axis y label/.style={
    at={(ticklabel* cs:1.0)},
    anchor=south,
},
  width=6cm,
  height=6cm,
  xmin=-0.2,
  xmax=1.2,
  ymin=-12,
  ymax=12]
\addplot [mark=none,dashed,color=black] coordinates {(1,-12) (1,12)};
\addplot [color=fst-red,line width=1pt,smooth] coordinates {(0,0) (0.0333333, 0.353815) (0.0666667, 0.987827) (0.1, 1.79072) (0.133333, 2.71948) (0.166667, 3.74743) (0.2, 4.6) (0.3,6.6) (0.4,7.5) (0.6, 7.8) (0.8,7.9) (1,8)};
\addplot [color=fst-red,line width=1pt,smooth] coordinates {(0,0) (0.0333333, -0.353815) (0.0666667, -0.987827) (0.1, -1.79072) (0.133333, -2.71948) (0.166667, -3.74743) (0.2, -4.6) (0.3,-6.6) (0.4,-7.5) (0.6, -7.8) (0.8,-7.9) (1,-8)};
\addplot [color=fst-blue,line width=1pt,smooth] coordinates {(0, -9) (0.0031357, -9.02216) (0.0178438, -9.0298) (0.0332339, -9.03671) (0.0493606, -9.04243) (0.0662844, -9.04639) (0.0840732, -9.04788) (0.102803, -9.04604) (0.12256, -9.03981) (0.14344, -9.02788) (0.165554, -9.00866) (0.189025, -8.98016) (0.213998, -8.93994) (0.240635, -8.88496) (0.269127, -8.81146) (0.299692, -8.71467) (0.332586, -8.58865) (0.368111, -8.42583) (0.40662, -8.21656) (0.448538, -7.94844) (0.49437, -7.60538) (0.54473, -7.1663) (0.600368, -6.6033) (0.662208, -5.87893) (0.731405, -4.94234) (0.809422, -3.72329) (0.898136, -2.12304) (1, 0)};
\addplot [color=fst-blue,line width=1pt,smooth] coordinates {(0,0) (0.1, -0.3) (0.2,-1) (0.23,-1.8) (0.2,-2.8) (0.1,-3.4) (0, -3.65)};
\addplot [color=fst-purple,line width=1pt,smooth] coordinates {(0,0) (0.1, 0.3) (0.2,1) (0.23,1.8) (0.2,2.8) (0.1,3.4) (0, 3.65)};
\addplot [color=fst-purple,line width=1pt,smooth] coordinates {(0.0031357, 9.02216) (0.0178438, 9.0298) (0.0332339, 9.03671) (0.0493606, 9.04243) (0.0662844, 9.04639) (0.0840732, 9.04788) (0.102803, 9.04604) (0.12256, 9.03981) (0.14344, 9.02788) (0.165554, 9.00866) (0.189025, 8.98016) (0.213998, 8.93994) (0.240635, 8.88496) (0.269127, 8.81146) (0.299692, 8.71467) (0.332586, 8.58865) (0.368111, 8.42583) (0.40662, 8.21656) (0.448538, 7.94844) (0.49437, 7.60538) (0.54473, 7.1663) (0.600368, 6.6033) (0.662208, 5.87893) (0.731405, 4.94234) (0.809422, 3.72329) (0.898136, 2.12304) (1, 0)};
\addplot [color=fst-gray,line width=0.8pt,smooth,dashed] coordinates {(1, -6) (0.8, -6) (0.4, -5) (0.25, -2.3) (0.22, 0) (0.25, 2.3) (0.4, 5) (0.8, 6) (1, 6)};
\addplot [color=fst-gray,line width=0.8pt,smooth] coordinates {(1, -7) (0.8, -7) (0.4, -6.2) (0.22, -3.2) (0.12, 0) (0.22, 3.2) (0.4, 6.2) (0.8, 7) (1, 7)};
\node[label={220:{\scriptsize{0}}}] at (axis cs:0.04,0.2) {};
\node[label={310:{\scriptsize{1}}}] at (axis cs:0.96,0.2) {};
\node[label={270:{\scriptsize{$\Gamma_{0}$}}}] at (axis cs:0.2,-8.6) {};
\node[label={90:{\scriptsize{$\Gamma_{1}$}}}] at (axis cs:0.2,8.6) {};
\node[label={90:{\scriptsize{$\mathcal{M}_{\mu}$}}}] at (axis cs:0.81,7.5) {};
\node[label={0:{\scriptsize{$P(s_{0}^{\tau})$}}}] at (axis cs:0.18,-1.5) {};
\node[label={90:{\scriptsize{$P(s_{1}^{\tau})$}}}] at (axis cs:0.62,-6.2) {};
\node[label={240:{\scriptsize{$P(s_{0}^{\mathcal{M}})$}}}] at (axis cs:0.25,-2.6) {};
\node[label={270:{\scriptsize{$P(s_{1}^{\mathcal{M}})$}}}] at (axis cs:0.54,-6.7) {};
\node[circle,fill,inner sep=0.8pt,fst-black] at (axis cs:0.23,1.65) {};
\node[circle,fill,inner sep=0.8pt,fst-black] at (axis cs:0.67,5.8) {};
\node[circle,fill,inner sep=0.8pt,fst-black] at (axis cs:0.15,3.2) {};
\node[circle,fill,inner sep=0.8pt,fst-black] at (axis cs:0.485,7.7) {};
\node[circle,fill,inner sep=0.8pt,fst-black] at (axis cs:0.23,-1.65) {};
\node[circle,fill,inner sep=0.8pt,fst-black] at (axis cs:0.67,-5.8) {};
\node[circle,fill,inner sep=0.8pt,fst-black] at (axis cs:0.15,-3.2) {};
\node[circle,fill,inner sep=0.8pt,fst-black] at (axis cs:0.485,-7.7) {};
\node[circle,fill,inner sep=0.8pt,fst-black] at (axis cs:0,0) {};
\node[circle,fill,inner sep=0.8pt,fst-black] at (axis cs:1,0) {};
\end{axis}
\end{tikzpicture}
\caption{For $\lambda\in[\lambda^*,+\infty)$, qualitative representations of the graph of the function $h_{\mu}$ defined in~\eqref{def-hs} (left) and of $\Gamma_0$ (blue), $\Gamma_1$ (violet), $\mathcal{M}_{\mu}$ (pink) along with some level lines of~\eqref{eq-energy-mu} (gray) in the $(u,v)$-plane (right). We set $P(s)=(u_{s}(\sigma),v_{s}(\sigma))$.}    
\label{fig-05}
\end{figure}

Since we can parametrize the curve $\Gamma_{0}\cap\{(u,v)\in(0,1)\times\mathbb{R}\colon H_{\mu}(u,v)>0\}$ using the parameter $s\in(0,s_{0}^{\mathcal{M}})\cup(s_{1}^{\mathcal{M}},1)$, we divide the study of the connection times into four cases according to whether $s$ belongs to one of the following intervals: $(0,s_{0}^{\tau})$, $(s_{0}^{\tau},s_{0}^{\mathcal{M}})$, $(s_{1}^{\mathcal{M}},s_{1}^{\tau})$, and $(s_{1}^{\tau},1)$. We treat $s=s_{0}^{\tau}$ and $s=s_{1}^{\tau}$ as limit cases.
The graphs of these connection times are shown in Figure~\ref{fig-06}.

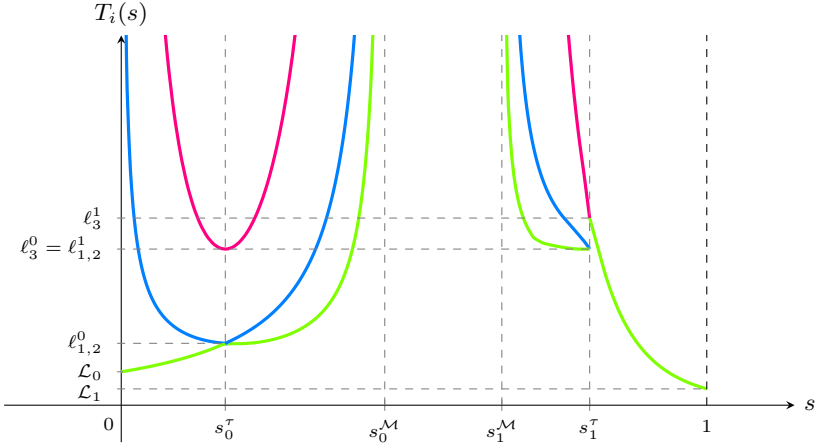
\begin{figure}[htb]
\begin{tikzpicture}
\begin{axis}[
  tick label style={font=\scriptsize},
  axis y line=middle, 
  axis x line=middle,
  xtick={0,0.1777,0.45,0.65,0.8},
  ytick={0},
  xticklabels={},
  yticklabels={},
  xlabel={\small $s$},
  ylabel={\small $T_{i}(s)$},
every axis x label/.style={
    at={(ticklabel* cs:1.0)},
    anchor=west,
},
every axis y label/.style={
    at={(ticklabel* cs:1.0)},
    anchor=south,
},
  width=12cm,
  height=7cm,
  xmin=-0.2,
  xmax=1.15,
  ymin=-0.9,
  ymax=9]
\addplot [color=fst-blue,line width=1.2pt,smooth] coordinates {(0.0045, 13.6009) (0.009, 8.40187) (0.0135, 6.3544) (0.018, 5.22222) (0.0225, 4.49225) (0.027, 3.97792) (0.0315, 3.59387) (0.036, 3.29513) (0.0405, 3.05559) (0.045, 2.85898) (0.0495, 2.6946) (0.054, 2.55511) (0.0585, 2.43527) (0.063, 2.33127) (0.0675, 2.24023) (0.072, 2.15996) (0.0765, 2.08875) (0.081, 2.02526) (0.0855, 1.96839) (0.09, 1.91726) (0.0945, 1.87116) (0.099, 1.82949) (0.1035, 1.79174) (0.108, 1.7575) (0.1125, 1.7264) (0.117, 1.69816) (0.1215, 1.6725) (0.126, 1.64921) (0.1305, 1.62808) (0.135, 1.60897) (0.1395, 1.59171) (0.144, 1.57618) (0.1485, 1.56228) (0.153, 1.5499) (0.1575, 1.53897) (0.162, 1.5294) (0.1665, 1.52114) (0.171, 1.51413) (0.1755, 1.50832) (0.18, 1.50368) (0.1845, 1.50017) (0.1777, 1.49617)}; 
\addplot [color=fst-green,line width=1.2pt,smooth] coordinates {(0.1777, 1.49617) (0.2025, 1.49696) (0.207, 1.49879) (0.2115, 1.50166) (0.216, 1.50558) (0.2205, 1.51055) (0.225, 1.51658) (0.2295, 1.52369) (0.234, 1.53189) (0.2385, 1.54122) (0.243, 1.5517) (0.2475, 1.56337) (0.252, 1.57628) (0.2565, 1.59046) (0.261, 1.60597) (0.2655, 1.62288) (0.27, 1.64126) (0.2745, 1.66118) (0.279, 1.68273) (0.2835, 1.70602) (0.288, 1.73116) (0.2925, 1.75827) (0.297, 1.7875) (0.3015, 1.81901) (0.306, 1.85298) (0.3105, 1.88962) (0.315, 1.92917) (0.3195, 1.97189) (0.324, 2.01808) (0.3285, 2.06811) (0.333, 2.12238) (0.3375, 2.18135) (0.342, 2.24558) (0.3465, 2.31571) (0.351, 2.39249) (0.3555, 2.47682) (0.36, 2.56978) (0.3645, 2.67267) (0.369, 2.78705) (0.3735, 2.91487) (0.378, 3.05854) (0.3825, 3.2211) (0.387, 3.40642) (0.3915, 3.61957) (0.396, 3.86721) (0.4005, 4.15839) (0.405, 4.50567) (0.4095, 4.92704) (0.414, 5.44919) (0.4185, 6.11361) (0.423, 6.98858) (0.4275, 8.19521) (0.432, 9.97148)};
\addplot [color=fst-green,line width=1.2pt,smooth] coordinates {(0, 0.8125) (0.0045, 0.823232) (0.009, 0.834184) (0.0135, 0.845361) (0.018, 0.856771) (0.0225, 0.868421) (0.027, 0.880319) (0.0315, 0.892473) (0.036, 0.904891) (0.0405, 0.917582) (0.045, 0.930556) (0.0495, 0.94382) (0.054, 0.957386) (0.0585, 0.971264) (0.063, 0.985465) (0.0675, 1.) (0.072, 1.01488) (0.0765, 1.03012) (0.081, 1.04573) (0.0855, 1.06173) (0.09, 1.07813) (0.0945, 1.09494) (0.099, 1.11218) (0.1035, 1.12987) (0.108, 1.14803) (0.1125, 1.16667) (0.117, 1.18581) (0.1215, 1.20548) (0.126, 1.22569) (0.1305, 1.24648) (0.135, 1.26786) (0.1395, 1.28986) (0.144, 1.3125) (0.1485, 1.33582) (0.153, 1.35985) (0.1575, 1.38462) (0.162, 1.41016) (0.1665, 1.43651) (0.171, 1.46371) (0.1777, 1.4918)}; 
\addplot [color=fst-blue,line width=1.2pt,smooth] coordinates {(0.1777, 1.4918) (0.18, 1.52083) (0.1845, 1.55085) (0.189, 1.5819) (0.1935, 1.61404) (0.198, 1.64732) (0.2025, 1.68182) (0.207, 1.71759) (0.2115, 1.75472) (0.216, 1.79327) (0.2205, 1.83333) (0.225, 1.875) (0.2295, 1.91837) (0.234, 1.96354) (0.2385, 2.01064) (0.243, 2.05978) (0.2475, 2.11111) (0.252, 2.16477) (0.2565, 2.22093) (0.261, 2.27976) (0.2655, 2.34146) (0.27, 2.40625) (0.2745, 2.47436) (0.279, 2.54605) (0.2835, 2.62162) (0.288, 2.70139) (0.2925, 2.78571) (0.297, 2.875) (0.3015, 2.9697) (0.306, 3.07031) (0.3105, 3.17742) (0.315, 3.29167) (0.3195, 3.41379) (0.324, 3.54464) (0.3285, 3.68519) (0.333, 3.83654) (0.3375, 4.) (0.342, 4.17708) (0.3465, 4.36957) (0.351, 4.57955) (0.3555, 4.80952) (0.36, 5.0625) (0.3645, 5.34211) (0.369, 5.65278) (0.3735, 6.) (0.378, 6.39063) (0.3825, 6.83333) (0.387, 7.33929) (0.3915, 7.92308) (0.396, 8.60417) (0.4005, 9.40909)};
\addplot [color=fst-red,line width=1.2pt,smooth] coordinates {(0.05625, 13.1918) (0.0675, 10.4071) (0.07875, 8.49608) (0.09, 7.12831) (0.10125, 6.12346) (0.1125, 5.37529) (0.12375, 4.8177) (0.135, 4.40787) (0.14625, 4.11737) (0.1575, 3.92711) (0.16875, 3.82448) (0.18, 3.80152) (0.19125, 3.85388) (0.2025, 3.98026) (0.21375, 4.18209) (0.225, 4.46353) (0.23625, 4.83168) (0.2475, 5.29703) (0.25875, 5.87428) (0.27, 6.58351) (0.28125, 7.45205) (0.2925, 8.5172) (0.30375, 9.83056)};
\addplot [color=fst-red,line width=1.2pt,smooth] coordinates {(0.656667, 587.754) (0.663333, 255.479) (0.67, 152.494) (0.676667, 103.686) (0.683333, 75.7382) (0.69, 57.9029) (0.696667, 45.6906) (0.703333, 36.9048) (0.71, 30.3484) (0.716667, 25.3157) (0.723333, 21.3653) (0.73, 18.2075) (0.736667, 15.645) (0.743333, 13.5389) (0.75, 11.7891) (0.756667, 10.3216) (0.763333, 9.08065) (0.77, 8.02372) (0.776667, 7.11769) (0.783333, 6.33652) (0.79, 5.65952) (0.8, 4.55448)};
\addplot [color=fst-green,line width=1.2pt,smooth] coordinates {(0.8, 4.55448) (0.823333, 3.3505) (0.83, 3.03783) (0.836667, 2.75975) (0.843333, 2.51176) (0.85, 2.29005) (0.856667, 2.09136) (0.863333, 1.91292) (0.87, 1.75231) (0.876667, 1.60747) (0.883333, 1.4766) (0.89, 1.35815) (0.896667, 1.25075) (0.903333, 1.15322) (0.91, 1.06451) (0.916667, 0.983704) (0.923333, 0.909994) (0.93, 0.842666) (0.936667, 0.781087) (0.943333, 0.724697) (0.95, 0.672994) (0.956667, 0.625536) (0.963333, 0.581924) (0.97, 0.541805) (0.976667, 0.504859) (0.983333, 0.470802) (0.99, 0.439377) (0.996667, 0.410353) (1, 0.396675)};
\addplot [color=fst-blue,line width=1.2pt,smooth] coordinates {(0.655, 41.1679) (0.66, 21.9394) (0.665, 15.5313) (0.67, 12.3284) (0.675, 10.4074) (0.68, 9.12745) (0.685, 8.21376) (0.69, 7.52899) (0.695, 6.9968) (0.7, 6.57143) (0.705, 6.22373) (0.71, 5.93427) (0.715, 5.68962) (0.72, 5.48016) (0.725, 5.29885) (0.73, 5.14041) (0.735, 5.0008) (0.74, 4.87688) (0.745, 4.76616) (0.75, 4.66667) (0.755, 4.57679) (0.76, 4.49522) (0.79, 4) (0.8, 3.8)};
\addplot [color=fst-green,line width=1.2pt,smooth] coordinates {(0.655, 11.4643) (0.66, 7.77634) (0.665, 6.3978) (0.67, 5.65816) (0.675, 5.19088) (0.68, 4.86639) (0.685, 4.62665) (0.69, 4.44158) (0.695, 4.29396) (0.7, 4.1732) (0.705, 4.07239) (0.72, 3.95) (0.74, 3.88) (0.76, 3.83) (0.78, 3.8) (0.8, 3.8)};
\node[label={220:{\scriptsize{0}}}] at (axis cs:0.02,0.1) {};
\node[label={270:{\scriptsize{1}}}] at (axis cs:1,0.1) {};
\node[label={270:{\scriptsize{$s_{0}^{\tau}$}}}] at (axis cs:0.1777,0.1) {};
\node[label={270:{\scriptsize{$s_{0}^{\mathcal{M}}$}}}] at (axis cs:0.45,0.1) {};
\node[label={270:{\scriptsize{$s_{1}^{\mathcal{M}}$}}}] at (axis cs:0.65,0.1) {};
\node[label={270:{\scriptsize{$s_{1}^{\tau}$}}}] at (axis cs:0.8,0.1) {};
\node[label={180:{\scriptsize{$\mathcal{L}_{0}$}}}] at (axis cs:0,0.8) {};
\node[label={180:{\scriptsize{$\mathcal{L}_{1}$}}}] at (axis cs:0,0.3) {};
\node[label={180:{\scriptsize{$\ell_{1,2}^{0}$}}}] at (axis cs:0,1.5059) {};
\node[label={180:{\scriptsize{$\ell_{3}^{0}=\ell_{1,2}^{1}$}}}] at (axis cs:0,3.8) {};
\node[label={180:{\scriptsize{$\ell_{3}^{1}$}}}] at (axis cs:0,4.555) {};
\addplot [mark=none,dashed,color=gray] coordinates {(-0.007,0.396675) (1,0.396675)};
\addplot [mark=none,dashed,color=gray] coordinates {(-0.007,0.8125) (0.003,0.8125)};
\addplot [mark=none,dashed,color=gray] coordinates {(-0.007,1.5059) (0.1777,1.5059)};
\addplot [mark=none,dashed,color=gray] coordinates {(-0.007,3.8) (0.8,3.8)};
\addplot [mark=none,dashed,color=gray] coordinates {(-0.007, 4.555) (0.8, 4.555)};
\addplot [mark=none,dashed,color=black] coordinates {(1,0) (1,12)};
\addplot [mark=none,dashed,color=gray] coordinates {(0.1777,0) (0.1777,10)};
\addplot [mark=none,dashed,color=gray] coordinates {(0.45,0) (0.45,10)};
\addplot [mark=none,dashed,color=gray] coordinates {(0.65,0) (0.65,10)};
\addplot [mark=none,dashed,color=gray] coordinates {(0.8,0) (0.8,10)};
\end{axis}
\end{tikzpicture}
\caption{For $\lambda\in[\lambda^*,+\infty)$, qualitative graphs of $T_{i}$, with $i=1,2,3$, which are the times necessary to connect $\Gamma_{0}$ to $\Gamma_{1}$, as functions of the initial data $u(0)=s$ of~\eqref{eq-initial0}: $T_1$~(green), $T_2$ (blue), and $T_3$ (pink).}    
\label{fig-06}
\end{figure}

\smallskip
\noindent
\textit{Case~1.} Let $s\in (0,s_{0}^{\tau})$. In this case, by the properties of $h_{\mu}$, there are three solutions $\xi^{s}_{1},$ $\xi^{s}_{2},$ and $\xi^{s}_{3}$ of $h_{\mu}(s)=h_{\mu}(\xi)$ with
\begin{equation}\label{eq-3points}
0 < \xi^{s}_{1}=s < s_{0}^{\tau} < \xi^{s}_{2} < s_{0}^{\mathcal{M}}  < s_{1}^{\mathcal{M}}< \xi^{s}_{3} < s_{1}^{\tau} < 1.
\end{equation}
Let $x_{i}(s)=u_{\xi^{s}_{i}}(\sigma)$ and $y_{i}(s)=-v_{\xi^{s}_{i}}(\sigma)$, so that the three points $(x_{i}(s),y_{i}(s))$, $i=1,2,3$, belong to the intersection between $\Gamma_{1}$ and the level lines of~\eqref{eq-energy-mu} passing through $(u_{s}(\sigma),v_{s}(\sigma))$. We remark that, for all $s\in(0,s_{0}^{\tau})$,
\begin{equation}\label{xi-order}
x_{1}(s)<x_{2}(s)<x_{3}(s).
\end{equation}
Indeed,  let $\bar{s}\in(0,s_{0}^{\tau})$ be fixed and let $\mathcal{O}$ denote the level line of~\eqref{eq-energy-mu} passing through $(u_{\bar{s}}(\sigma),v_{\bar{s}}(\sigma))$. Thanks to our minimality assumption, by comparing the values of the energy~\eqref{eq-energy-mu} on the points of $\Gamma_{0}$ with the value of the energy on $\mathcal{O}$ (see Figure~\ref{fig-05}), we have that the points $(u_{s}(\sigma),v_{s}(\sigma))\in\Gamma_{0}$ lie outside the region ``embraced'' by $\mathcal{O}$ for $s\in(0,\xi_{1}^{\bar{s}})\cup(\xi_{2}^{\bar{s}},s_{0}^{\mathcal{M}})\cup(s_{1}^{\mathcal{M}},\xi_{3}^{\bar{s}})$, while they lie inside for $s\in(\xi_{1}^{\bar{s}},\xi_{2}^{\bar{s}})\cup(\xi_{3}^{\bar{s}},1)$. Therefore, if we assume by contradiction that $x_{2}(\bar{s})\leq x_{1}(\bar{s})$, by continuity we have
\begin{equation*}
\left\{(u_{s}(\sigma),v_{s}(\sigma))\colon s\in(0,\xi_{1}^{\bar{s}}]\right\}\cap\left\{(u_{s}(\sigma),v_{s}(\sigma))\colon s\in[\xi_{2}^{\bar{s}},s_{0}^{\mathcal{M}})\right\}\neq\emptyset,
\end{equation*}
against the uniqueness of solution for the initial value problems~\eqref{eq-initial0}. A similar reasoning, together with~\eqref{uM-order}, shows that $x_{2}(\bar{s})<x_{3}(\bar{s})$.

The time necessary for a point $(u_{s}(\sigma),v_{s}(\sigma))\in\Gamma_0$ to reach $\Gamma_{1}$ the first, the second, and the third time is defined as follows
\begin{equation}\label{def-T123}
\begin{aligned}
T_{i}(s)&= \int_{m(u_{s}(\sigma))}^{u_{s}(\sigma)} \dfrac{\mathrm{d}u}{\sqrt{(v_{s}(\sigma))^2+2\mu(G(u)-G(u_{s}(\sigma)))}} 
\\& \quad + \int_{m(u_{s}(\sigma))}^{x_{i}(s)} \dfrac{\mathrm{d}u}{\sqrt{(v_{s}(\sigma))^2+2\mu(G(u)-G(u_{s}(\sigma)))}}, \quad i=1,2,3.
\end{aligned}
\end{equation}
Therefore, all the three connection times $T_{i}$, $i=1,2,3$, are continuous since $s\mapsto (u_{s}(\sigma),v_{s}(\sigma))$ is continuous. Moreover, due to \eqref{xi-order}, they satisfy, for all $s\in(0,s_{0}^{\tau})$,
\begin{equation}\label{eq-ineq-t}
0 < T_{1}(s) < T_{2}(s) < T_{3}(s).
\end{equation}

First, we claim that
\begin{equation}\label{eq-aux-time1}
\lim_{s\to 0^{+}} T_{1}(s) = \dfrac{2 \lambda \sigma}{\mu} =: \mathcal{L}_{0}.
\end{equation}
To prove the claim, for all $k\in(0,+\infty)$, let us define
\begin{equation*}
\mathfrak{P}(k):=\bigl{\{}(x,-k x^{2})\colon x\in[0,1] \bigr{\}},
\end{equation*}
and denote through $T_{\mathfrak{P}(k)}(P)$ the time necessary for a point $P=(x_{P},y_{P})\in \mathfrak{P}(k)$ to reach $(x_{P},-y_{P})$ by moving along the level lines of~\eqref{eq-energy-mu}.
By Proposition~\ref{pr-2.2}~$(i)$, we have that $\Gamma_{0}$ behaves like the parabola $\mathfrak{P}(\lambda\sigma)$ in a neighborhood of $(0,0)$. As a consequence, for every $k_{-},k_{+}\in(0,+\infty)$ with $k_{-}<\lambda \sigma<k_{+}$, there exist two neighborhoods $\mathcal{U}_{0}(k_{-})$ and $\mathcal{U}_{0}(k_{+})$ of $(0,0)$ such that in $\mathcal{U}_{0}(k_{-})\cap\mathcal{U}_{0}(k_{+})$ the curve $\Gamma_{0}$ is between the two parabolas $\mathfrak{P}(k_{-})$ and $\mathfrak{P}(k_{+})$. 
Therefore, we have that
\begin{equation*}
T_{\mathfrak{P}(k_{-})}(\zeta^{-}_{s}) < T_{1}(s) < T_{\mathfrak{P}(k_{+})}(\zeta^{+}_{s}), \quad \text{for all $s$ sufficiently close to $0$,}
\end{equation*}
where $\zeta^{-}_{s},\zeta^{+}_{s}$ are the abscissas of the first intersection points between the level line of~\eqref{eq-energy-mu} passing through $(u_{s}(\sigma),v_{s}(\sigma))$ and the parabolas $\mathfrak{P}(k_{-})$ and $\mathfrak{P}(k_{+})$, respectively (for $s$ small enough). Using the energy conservation, we have $\zeta^{\pm}_{s}\to0$ as $s\to0^{+}$. Next, from Proposition~\ref{pr-3.1}~$(i)$, we infer
\begin{equation*}
\dfrac{2 k_{-}}{\mu} = \lim_{s\to 0^{+}} T_{\mathfrak{P}(k_{-})}(\zeta^{-}_{s}) \leq \lim_{s\to 0^{+}} T_{1}(s) \leq \lim_{s\to 0^{+}} T_{\mathfrak{P}(k_{+})}(\zeta^{+}_{s}) = \dfrac{2 k_{+}}{\mu}.
\end{equation*}
From the arbitrariness of $k_{-}$ and $k_{+}$, \eqref{eq-aux-time1} follows and the claim is proved.

Next, we claim that
\begin{equation*}
\lim_{s\to 0^{+}} T_{2}(s) = \lim_{s\to 0^{+}} T_{3}(s) = + \infty.
\end{equation*}
Since $0<m(u_{s}(\sigma))<u_{s}(\sigma)<s$ for all $s\in(0,s_{0}^{\mathcal{M}})\cup(s_{1}^{\mathcal{M}},1)$,
we have that
\begin{equation*}
\lim_{s\to0^{+}} (u_{s}(\sigma),v_{s}(\sigma)) = (0,0),
\qquad
\lim_{s\to 0^{+}} m(u_{s}(\sigma)) = 0.
\end{equation*}
Recalling \eqref{eq-3points}, the facts that $h_{\mu}(s_{0}^{\mathcal{M}})=0$, $h_{\mu}<0$ in $(s_{0}^{\tau},s_{0}^{\mathcal{M}})$, and $h_{\mu}(s)\to0$ as $s\to0^{+}$ (see also Figure~\ref{fig-05}), we deduce that $\lim_{s\to0^{+}}\xi^{s}_{2}=s_{0}^{\mathcal{M}}$ and so 
\begin{equation*}
\lim_{s\to 0^{+}} x_{2}(s) = u_{s_{0}^{\mathcal{M}}}(\sigma) > 0.
\end{equation*}
Consequently, we have
\begin{align*}
\lim_{s\to 0^{+}} T_{3}(s) \geq \lim_{s\to 0^{+}} T_{2}(s) 
&\geq\lim_{s\to 0^{+}} \int_{m(u_{s}(\sigma))}^{x_{2}(s)} \dfrac{\mathrm{d}u}{\sqrt{(v_{s}(\sigma))^2+2\mu(G(u)-G(u_{s}(\sigma)))}}
\\
&= \int_{0}^{u_{s_{0}^{\mathcal{M}}}(\sigma)} \dfrac{\mathrm{d}u}{\sqrt{2\mu G(u)}}
= + \infty,
\end{align*}
since $G(u)= u^{3}/3 + o(u^{3})$ as $u\to0^{+}$.
Then, the claim is proved.

Third, since $\lim_{s\to (s_{0}^{\tau})^{-}} \xi^{s}_{1} = \lim_{s\to (s_{0}^{\tau})^{-}} \xi^{s}_{2} = s_{0}^{\tau}$, and $\lim_{s\to (s_{0}^{\tau})^{-}} \xi^{s}_{3} = s_{1}^{\tau}$, we have
\begin{equation*}
\lim_{s\to (s_{0}^{\tau})^{-}} x_{1}(s) = 
\lim_{s\to (s_{0}^{\tau})^{-}} x_{2}(s) =
u_{s_{0}^{\tau}}(\sigma), 
\qquad
\lim_{s\to (s_{0}^{\tau})^{-}} x_{3}(s) = u_{s_{1}^{\tau}}(\sigma).
\end{equation*}
Therefore, we obtain
\begin{equation*}
\lim_{s\to (s_{0}^{\tau})^{-}} T_{1}(s) = \lim_{s\to (s_{0}^{\tau})^{-}} T_{2}(s) = \ell^{0}_{1,2},
\qquad \lim_{s\to (s_{0}^{\tau})^{-}} T_{3}(s) = \ell^{0}_{3},
\end{equation*}
where
\begin{align}
\ell^{0}_{1,2} &:= 2\int_{m(u_{s_{0}^{\tau}}(\sigma))}^{u_{s_{0}^{\tau}}(\sigma)} \dfrac{\mathrm{d}u}{\sqrt{(v_{s_{0}^{\tau}}(\sigma))^2+2\mu(G(u)-G(u_{s_{0}^{\tau}}(\sigma)))}},
\label{def-l120}
\\
\ell^{0}_{3} &:= \int_{m(u_{s_{0}^{\tau}}(\sigma))}^{u_{s_{0}^{\tau}}(\sigma)} \dfrac{\mathrm{d}u}{\sqrt{(v_{s_{0}^{\tau}}(\sigma))^2+2\mu(G(u)-G(u_{s_{0}^{\tau}}(\sigma)))}}
\\ &\quad +\int_{m(u_{s_{0}^{\tau}}(\sigma))}^{u_{s_{1}^{\tau}}(\sigma)} \dfrac{\mathrm{d}u}{\sqrt{(v_{s_{0}^{\tau}}(\sigma))^2+2\mu(G(u)-G(u_{s_{0}^{\tau}}(\sigma)))}}.
\label{def-l30}
\end{align}
We notice that $0<\ell^{0}_{1,2}<\ell^{0}_{3}<+\infty$ since $u_{s_{0}^{\tau}}(\sigma)<u_{s_{1}^{\tau}}(\sigma)$.

\smallskip
\noindent
\textit{Case~2.} Let $s\in (s_{0}^{\tau},s_{0}^{\mathcal{M}})$. Here, by arguing as in Case~1, one can prove that there are three intersection points $(x_{i}(s),y_{i}(s))$, $i=1,2,3$,  between $\Gamma_{1}$ and the level line of~\eqref{eq-energy-mu} passing through $(u_{s}(\sigma),v_{s}(\sigma))$, and satisfying \eqref{xi-order}. Therefore, we define the times $T_{i}$, with $i=1,2,3$, as in \eqref{def-T123}, which are continuous and satisfy \eqref{eq-ineq-t} for $s$ in the considered range.
A simple argument shows that
\begin{equation*}
\lim_{s\to (s_{0}^{\tau})^{+}} T_{1}(s) = \lim_{s\to (s_{0}^{\tau})^{+}} T_{2}(s) = \ell^{0}_{1,2},
\qquad \lim_{s\to (s_{0}^{\tau})^{+}} T_{3}(s) = \ell^{0}_{3}.
\end{equation*}
From
\begin{align*}
&\lim_{s\to (s_{0}^{\mathcal{M}})^{-}} (x_{1}(s),y_{1}(s)) = (0,0),
\\
&\lim_{s\to (s_{0}^{\mathcal{M}})^{-}} (v_{s}(\sigma))^2- 2\mu G(u_{s}(\sigma)) = 
\lim_{s\to (s_{0}^{\mathcal{M}})^{-}} (y_{1}(s))^2- 2\mu G(x_{1}(s)) = 0,
\\
&\lim_{s\to (s_{0}^{\mathcal{M}})^{-}} u_{s}(\sigma) = u_{s_{0}^{\mathcal{M}}}(\sigma) \in(0,1),
\\
&\lim_{s\to (s_{0}^{\mathcal{M}})^{-}} m(u_{s}(\sigma))=m(u_{s_{0}^{\mathcal{M}}}(\sigma)) = 0,
\end{align*}
we have
\begin{equation*}
\lim_{s\to (s_{0}^{\mathcal{M}})^{-}} T_{1}(s) = \lim_{s\to (s_{0}^{\mathcal{M}})^{-}} T_{2}(s) = \lim_{s\to (s_{0}^{\mathcal{M}})^{-}} T_{3}(s) = + \infty.
\end{equation*}

\smallskip
\noindent
\textit{Case~3.} Let $s\in(s_{1}^{\mathcal{M}},s_{1}^{\tau})$. Here, by arguing as in Case~1, one can prove that there are three intersection points $(x_{i}(s),y_{i}(s))$, $i=1,2,3$,  between $\Gamma_{1}$ and the level line of~\eqref{eq-energy-mu} passing through $(u_{s}(\sigma),v_{s}(\sigma))$, and satisfying \eqref{xi-order}. Therefore, we define the times $T_{i}$, with $i=1,2,3$, as in~\eqref{def-T123}, which are continuous and satisfy~\eqref{eq-ineq-t} for $s$ in the considered range.
From
\begin{align*}
&\lim_{s\to (s_{1}^{\mathcal{M}})^{+}} (x_{1}(s),y_{1}(s)) = (0,0),
\\
&\lim_{s\to (s_{1}^{\mathcal{M}})^{+}} (v_{s}(\sigma))^2- 2\mu G(u_{s}(\sigma)) = 
\lim_{s\to (s_{1}^{\mathcal{M}})^{+}} (y_{1}(s))^2- 2\mu G(x_{1}(s)) = 0,
\\
&\lim_{s\to (s_{1}^{\mathcal{M}})^{+}} m(u_{s}(\sigma))=m(u_{s_{1}^{\mathcal{M}}}(\sigma)) = 0,
\\
&\lim_{s\to (s_{1}^{\mathcal{M}})^{+}} u_{s}(\sigma) = u_{s_{1}^{\mathcal{M}}}(\sigma) \in(0,1),
\end{align*}
we have
\begin{equation*}
\lim_{s\to (s_{1}^{\mathcal{M}})^{+}} T_{1}(s) = \lim_{s\to (s_{1}^{\mathcal{M}})^{+}} T_{2}(s) =\lim_{s\to (s_{1}^{\mathcal{M}})^{+}} T_{3}(s) =+\infty.
\end{equation*}
A simple argument shows that
\begin{equation*}
\lim_{s\to (s_{1}^{\tau})^{-}} T_{1}(s) = \lim_{s\to (s_{1}^{\tau})^{-}} T_{2}(s) = \ell^{1}_{1,2},
\qquad
\lim_{s\to (s_{1}^{\tau})^{-}} T_{3}(s) = \ell^{1}_{3},
\end{equation*}
where
\begin{align*}
\ell^{1}_{1,2} &:= \int_{m(u_{s_{1}^{\tau}}(\sigma))}^{u_{s_{1}^{\tau}}(\sigma)} \dfrac{\mathrm{d}u}{\sqrt{(v_{s_{1}^{\tau}}(\sigma))^2+2\mu(G(u)-G(u_{s_{1}^{\tau}}(\sigma)))}}
\\ &\quad +\int_{m(u_{s_{1}^{\tau}}(\sigma))}^{u_{s_{0}^{\tau}}(\sigma)} \dfrac{\mathrm{d}u}{\sqrt{(v_{s_{1}^{\tau}}(\sigma))^2+2\mu(G(u)-G(u_{s_{1}^{\tau}}(\sigma)))}},
\\
\ell^{1}_{3} &:= 2 \int_{m(u_{s_{1}^{\tau}}(\sigma))}^{u_{s_{1}^{\tau}}(\sigma)} \dfrac{\mathrm{d}u}{\sqrt{(v_{s_{1}^{\tau}}(\sigma))^2+2\mu(G(u)-G(u_{s_{1}^{\tau}}(\sigma)))}}.
\end{align*}
We notice that $0<\ell^{1}_{1,2}<\ell^{1}_{3}<+\infty$. Moreover, since $m(u_{s_{0}^{\tau}}(\sigma))=m(u_{s_{1}^{\tau}}(\sigma))$ and $(v_{s_{0}^{\tau}}(\sigma))^2-2\mu G(u_{s_{0}^{\tau}}(\sigma))=(v_{s_{1}^{\tau}}(\sigma))^2-2\mu G(u_{s_{1}^{\tau}}(\sigma))$, we have that
\begin{equation*}
\ell_{3}^{0} = \ell_{1,2}^{1}.
\end{equation*}

\smallskip
\noindent
\textit{Case~4.} Let $s\in (s_{1}^{\tau},1)$. 
In this case, we have one intersection point between $\Gamma_{1}$ and the level line of~\eqref{eq-energy-mu} passing through $(u_{s}(\sigma),v_{s}(\sigma))$, that is $(x_{1}(s),y_{1}(s))=(u_{s}(\sigma),-v_{s}(\sigma))$. Therefore, we define only $T_{1}$ as in \eqref{def-T123}, which is continuous and can be equivalently written as
\begin{equation}\label{def-T1_4.2}
T_{1}(s)= 2 \int_{m(u_{s}(\sigma))}^{u_{s}(\sigma)} \dfrac{\mathrm{d}u}{\sqrt{(v_{s}(\sigma))^2+2\mu(G(u)-G(u_{s}(\sigma)))}}.
\end{equation}
By arguing as in Case~3, it is easy to prove that
\begin{equation*}
\lim_{s\to (s_{1}^{\tau})^{+}} T_{1}(s) = \ell^{1}_{3}
\end{equation*}
and, using Proposition~\ref{pr-2.4}~$(ii)$ and Proposition~\ref{pr-3.2}~$(i)$, that
\begin{equation}\label{eq-aux-time2}
\lim_{s\to 1^{-}} T_{1}(s) = \dfrac{2}{\sqrt{\mu}}\arctan\left(\dfrac{\sqrt{\lambda} \tanh(\sqrt{\lambda}\sigma)}{\sqrt{\mu}}\right) =: \mathcal{L}_{1}.
\end{equation}

\medskip

We conclude the analysis of the case $\lambda\in[\lambda^{*},+\infty)$ by investigating further properties of the above time-maps $T_{i}$, with $i=1,2,3$, which will be crucial in the next section.

First, we prove that
\begin{equation}\label{eq:order_L0_L1}
\lim_{s\to 0^{+}} T_{1}(s) = \mathcal{L}_{0} > \mathcal{L}_{1} = \lim_{s\to 1^{-}} T_{1}(s).
\end{equation}
Precisely, from~\eqref{eq-aux-time1} and~\eqref{eq-aux-time2}, we show that
\begin{equation*}
\dfrac{\lambda \sigma}{\mu} > \dfrac{1}{\sqrt{\mu}}\arctan\left(\dfrac{\sqrt{\lambda} \tanh(\sqrt{\lambda}\sigma)}{\sqrt{\mu}}\right), \quad \text{for all $\mu,\lambda,\sigma\in(0,+\infty)$.}
\end{equation*}
Accordingly, we consider the function $\eta\colon(0,+\infty)\to\mathbb{R}$ of class $\mathcal{C}^{\infty}$ defined as
\begin{equation*}
\eta(x) = x\arctan\left(\dfrac{\sqrt{\lambda} \tanh(\sqrt{\lambda}\sigma)}{x}\right).
\end{equation*}
For all $x\in(0,+\infty)$ we have
\begin{align*}
&\eta'(x)
=\arctan\left(\dfrac{\sqrt{\lambda} \tanh(\sqrt{\lambda}\sigma)}{x}\right) 
- \dfrac{\sqrt{\lambda} \tanh(\sqrt{\lambda}\sigma) x}{\lambda \tanh^{2}(\sqrt{\lambda}\sigma)+x^{2}},
\\
&\eta''(x)
= - \dfrac{2(\sqrt{\lambda} \tanh(\sqrt{\lambda}\sigma))^{3}}{(\lambda \tanh^{2}(\sqrt{\lambda}\sigma)+x^{2})^{2}}.
\end{align*}
From $\lim_{x\to+\infty}\eta'(x)=0$ and $\eta''(x)<0$ for all $x\in(0,+\infty)$, we infer that $\eta$ is a strictly increasing function.
Since $\lim_{x\to + \infty} \eta(x) = \sqrt{\lambda} \tanh(\sqrt{\lambda}\sigma)$, in order to conclude, we only need to prove that
\begin{equation*}
\sqrt{\lambda} \tanh(\sqrt{\lambda}\sigma) < \lambda \sigma.
\end{equation*}
This is obvious from the fact that $\tanh(z)<z$ for all $z>0$.

\medskip

To conclude the case $\lambda\in[\lambda^{*},+\infty)$, we extend the connection times $T_{i}(s)$, $i=1,2,3$, as follows. We set $T_1(0)=\mathcal{L}_0$, $T_1(s_0^{\tau})=\ell_{1,2}^0$, $T_1(s_1^{\tau})=\ell_{1,2}^1$, and $T_1(1)=\mathcal{L}_1$, so that $T_1$ is lower-semicontinuous in $D_1:=[0,s_0^{\mathcal{M}})\cup(s_1^{\mathcal{M}},1]$. Moreover, by setting $T_2(s_0^{\tau})=\ell_{1,2}^0$, $T_2(s_1^{\tau})=\ell_{1,2}^1$, $T_3(s_0^{\tau})=\ell_{3}^0$, and $T_3(s_1^{\tau})=\ell_{3}^1$, we have that $T_i$ is continuous in $D_i:=(0,s_0^{\mathcal{M}})\cup(s_1^{\mathcal{M}},s_1^{\tau}]$, with $i=2,3$.
In addition, we observe that
\begin{equation} \label{eq-order-Ti-general}
T_{1}(s)\leq T_{2}(s)\leq T_{3}(s), \quad \text{for all $s\in[0,1]\cap D_2$,}
\end{equation}
with strict inequalities for $s\neq s_{0}^{\tau}$ and $s\neq s_{1}^{\tau}$.

At last, by exploiting the symmetry of the problem (cf.~Remark~\ref{rem-2.1}) and \eqref{def-T123}, we have that
\begin{equation}\label{eq-symmetry}
\begin{aligned}
&\mathrm{Im}\bigl{(} T_{2}|_{(0,s_{0}^{\tau})} \bigr{)} =
\mathrm{Im}\bigl{(} T_{1}|_{(s_{0}^{\tau},s_{0}^{\mathcal{M}})} \bigr{)},
\\
&\mathrm{Im}\bigl{(} T_{3}|_{(0,s_{0}^{\tau})} \bigr{)} =
\mathrm{Im}\bigl{(} T_{1}|_{(s_{1}^{\mathcal{M}},s_{1}^{\tau})} \bigr{)},
\\
&\mathrm{Im}\bigl{(} T_{3}|_{(s_{0}^{\tau},s_{0}^{\mathcal{M}})} \bigr{)} =
\mathrm{Im}\bigl{(} T_{2}|_{(s_{1}^{\mathcal{M}},s_{1}^{\tau})} \bigr{)}.
\end{aligned}
\end{equation}

\subsection{The case $\lambda \in (0,\lambda^{*})$}\label{section-4.2}

Let $\lambda \in (0,\lambda^{*})$ be fixed. The properties of $\Gamma_{0}$ described in Corollary~\ref{cor-2.1} and the ones of $\Gamma_{1}$ described in Corollary~\ref{cor-2.2} ensure the existence of $\tilde{\mu}(\lambda)>0$ such that for $\mu=\tilde{\mu}(\lambda)$ the curve $\mathcal{M}_{\mu}$ is tangent to $\Gamma_{0}$ and $\Gamma_{1}$.

If $\mu\in(0,\tilde{\mu}(\lambda))$, then there are at least two intersections between $\Gamma_{0}$ and $\mathcal{M}_{\mu}$ in $(0,1)\times\mathbb{R}$ (as in the case $\lambda\in[\lambda^{*},+\infty)$). Accordingly, let $s_{0}^{\mathcal{M}},s_{1}^{\mathcal{M}}\in(0,1)$ with $s_{0}^{\mathcal{M}}<s_{1}^{\mathcal{M}}$ be such that \eqref{eq-4.1} holds.
In principle, there could be more than two values $s\in(0,1)$ giving intersections between $\Gamma_{0}$ and $\mathcal{M}_{\mu}$; if it is the case, we consider only the closest to $0$ and to $1$, respectively.
As a consequence, when $\mu\in(0,\tilde{\mu}(\lambda))$, the situation is exactly that described in Section~\ref{section-4.1} (for $\lambda\in[\lambda^{*},+\infty)$) and represented in Figures~\ref{fig-05} and~\ref{fig-06}.

If $\mu=\tilde{\mu}(\lambda)$, then $s_{0}^{\mathcal{M}}=s_{1}^{\mathcal{M}}=:s^{\mathcal{M}}$ and the curve $\Gamma_{0}$ is tangent to $\mathcal{M}_{\mu}$ at $\left(u_{s^{\mathcal{M}}}(\sigma),v_{s^{\mathcal{M}}}(\sigma)\right)$. The behavior of the time-maps differs from the one described in Section~\ref{section-4.1} (for $\lambda\in[\lambda^{*},+\infty)$) since the gap between the two vertical asymptotes $s=s_{0}^{\mathcal{M}}$ and $s=s_{1}^{\mathcal{M}}$ is dropped, thus producing a unique asymptote $s=s^{\mathcal{M}}$. 

The difference between the cases $\lambda \in (0,\lambda^{*})$ and $\lambda\in[\lambda^{*},+\infty)$ is evident when $\mu>\tilde{\mu}(\lambda)$, namely when there are no intersections between $\Gamma_{0}\cup\Gamma_{1}$ and $\mathcal{M}_{\mu}$ in $(0,1)\times\mathbb{R}$. 
Recalling the definition of the function $h_{\mu}\colon [0,1]\to\mathbb{R}$ given in \eqref{def-hs} and its sign properties, we deduce that $h_{\mu}(s)<0$ for all $s\in(0,1]$ and, as before, $s=1$ is a global minimum point (see Figure~\ref{fig-07}).

In general, since $h_{\mu}$ could be strictly decreasing in $[0,1]$, for all $s\in(0,1)$, we can only have one intersection between $\Gamma_{1}$ and the level line of~\eqref{eq-energy-mu} passing through $(u_{s}(\sigma),v_{s}(\sigma))$, that is $(u_{s}(\sigma),-v_{s}(\sigma))$. If that is the case, for all $s\in(0,1)$, we define $T_{1}$ as in \eqref{def-T1_4.2},
whose graph is depicted in Figure~\ref{fig-noisola}. As in Section~\ref{section-4.1}, we have \eqref{eq:order_L0_L1}.

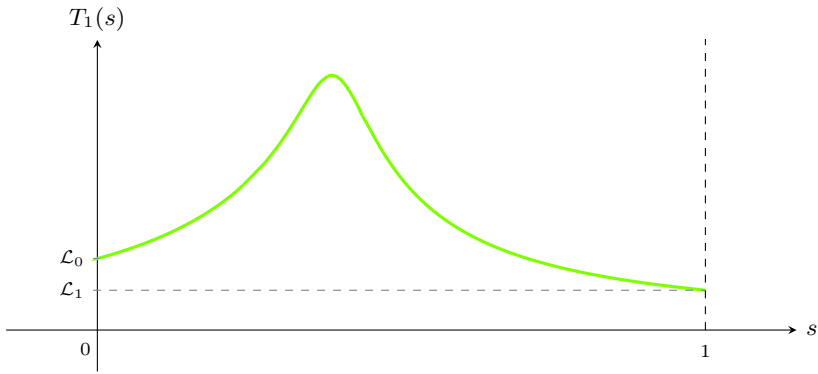
\begin{figure}[htb]
\begin{tikzpicture}
\begin{axis}[
  tick label style={font=\scriptsize},
  axis y line=middle, 
  axis x line=middle,
  xtick={0},
  ytick={0},
  xticklabels={},
  yticklabels={},
  xlabel={\small $s$},
  ylabel={\small $T_{1}(s)$},
every axis x label/.style={
    at={(ticklabel* cs:1.0)},
    anchor=west,
},
every axis y label/.style={
    at={(ticklabel* cs:1.0)},
    anchor=south,
},
  width=12cm,
  height=6cm,
  xmin=-0.15,
  xmax=1.15,
  ymin=-1,
  ymax=7] 
\addplot [color=fst-green,line width=1.2pt,smooth] coordinates {(0.436, 5.16194) (0.445726, 4.91447) (0.45461, 4.67408) (0.463537, 4.44493) (0.472509, 4.22903) (0.481529, 4.02702) (0.490601, 3.83872) (0.499726, 3.66352) (0.508907, 3.50056) (0.518147, 3.34894) (0.527449, 3.20771) (0.536816, 3.076) (0.54625, 2.95299) (0.555754, 2.83789) (0.565331, 2.73002) (0.574984, 2.62875) (0.584716, 2.53352) (0.59453, 2.44382) (0.604428, 2.3592) (0.614415, 2.27923) (0.624492, 2.20357) (0.634664, 2.13186) (0.644933, 2.06381) (0.655302, 1.99915) (0.665775, 1.93763) (0.676356, 1.87904) (0.687046, 1.82316) (0.697851, 1.76982) (0.708773, 1.71884) (0.719816, 1.67008) (0.730984, 1.62338) (0.74228, 1.57863) (0.753708, 1.5357) (0.765272, 1.49448) (0.776975, 1.45488) (0.788822, 1.4168) (0.800817, 1.38015) (0.812963, 1.34486) (0.825266, 1.31085) (0.837729, 1.27805) (0.850356, 1.24641) (0.863153, 1.21585) (0.876124, 1.18633) (0.889273, 1.1578) (0.902605, 1.1302) (0.916125, 1.10349) (0.929838, 1.07763) (0.94375, 1.05258) (0.957865, 1.0283) (0.972189, 1.00476) (1, 0.96)};
\addplot [color=fst-green,line width=1.2pt,smooth] coordinates {(0.379, 6.12121) (0.384531, 6.14342) (0.393185, 6.11633) (0.401864, 6.01611) (0.410571, 5.85323) (0.419308, 5.64481) (0.428078, 5.40917) (0.436, 5.16194)};
\addplot [color=fst-green,line width=1.2pt,smooth] coordinates {(0.24, 3.5) (0.273198, 3.99792) (0.281748, 4.1474) (0.290292, 4.30619) (0.298832, 4.47473) (0.30737, 4.65324) (0.315908, 4.84144) (0.32445, 5.03832) (0.332997, 5.24159) (0.341552, 5.44704) (0.350117, 5.64769) (0.358695, 5.83299) (0.367288, 5.98855) (0.379, 6.12121)}; 
\addplot [color=fst-green,line width=1.2pt,smooth] coordinates {(0,1.728) (0.000999551, 1.72815) (0.0109462, 1.76897) (0.0208063, 1.81107) (0.0305828, 1.85455) (0.0402787, 1.89945) (0.049897, 1.94583) (0.0594407, 1.99378) (0.0689125, 2.04338) (0.0783154, 2.0947) (0.0876522, 2.14786) (0.0969256, 2.20295) (0.106139, 2.26008) (0.115294, 2.31938) (0.124394, 2.38097) (0.133441, 2.44501) (0.142439, 2.51165) (0.15139, 2.58106) (0.160296, 2.65343) (0.169161, 2.72899) (0.177986, 2.80795) (0.186775, 2.89057) (0.195529, 2.97715) (0.204252, 3.06798) (0.212946, 3.16342) (0.221613, 3.26386) (0.230256, 3.36971) (0.24, 3.5)}; 
\node[label={220:{\scriptsize{0}}}] at (axis cs:0.02,0.1) {};
\node[label={270:{\scriptsize{1}}}] at (axis cs:1,0.1) {};
\node[label={180:{\scriptsize{$\mathcal{L}_{0}$}}}] at (axis cs:0.01,1.728) {};
\node[label={180:{\scriptsize{$\mathcal{L}_{1}$}}}] at (axis cs:0.01,0.96) {};
\addplot [mark=none,dashed,color=gray] coordinates {(-0.007,0.96) (1,0.96)};
\addplot [mark=none,dashed,color=gray] coordinates {(-0.007,1.728) (0.003,1.728)};
\addplot [mark=none,dashed]coordinates {(1,0) (1,12)};
\end{axis}
\end{tikzpicture}
\caption{For $\lambda\in(0,\lambda^{*})$ and $\mu>\tilde{\mu}(\lambda)$, qualitative graph of $T_{1}$, which is the time necessary to connect $\Gamma_{0}$ to $\Gamma_{1}$, as a function of the initial data $u(0)=s$ of~\eqref{eq-initial0}.}   
\label{fig-noisola}
\end{figure}

We enhance the qualitative description of the case $\mu>\tilde{\mu}(\lambda)$ dealing with $\mu$ sufficiently close to $\tilde{\mu}(\lambda)$. Preliminarily, we observe that if $\mu=\tilde{\mu}(\lambda)$ the function $h_{\mu}$ vanishes for $s=0$ and for $s=s_{0}^{\mathcal{M}}=s_{1}^{\mathcal{M}}=s^{\mathcal{M}}$, and so $h_{\mu}|_{[0,s^{\mathcal{M}}]}$ has a global minimum point. Therefore, by continuity, if we slightly increase $\mu$, $h_{\mu}$ has a local maximum point $s_{1}^{\omega}\in(0,1)$ corresponding to a small perturbation of the point $s^{\mathcal{M}}$, and the function $h_{\mu}|_{[0,s_{1}^{\omega}]}$ has a global minimum point $s=s_{0}^{\tau}$. 
As in Section~\ref{section-4.1}, we investigate the case in which $s_{0}^{\tau}$ is the unique local minimum of $h_{\mu}$ in $(0,1)$ (and so $s_{1}^{\omega}$ is the unique local maximum point of $h_{\mu}$). More complicated situations can be investigated by adapting the analysis here performed.
We call $s_{0}^{\omega}$ the point in $[0,s_{1}^{\omega}]$ such that $h_{\mu}(s_{0}^{\omega})=h_{\mu}(s_{1}^{\omega})$.
Since $s=1$ is a global minimum point of $h_{\mu}$, we define $s_{1}^{\tau}$ the point in $[s_{1}^{\mathcal{M}},1]$ such that $h_{\mu}(s_{0}^{\tau})=h_{\mu}(s_{1}^{\tau})$. 
We stress that
\begin{equation*}
0 < s_{0}^{\omega} < s_{0}^{\tau} < s_{1}^{\omega} < s_{1}^{\tau} < 1
\end{equation*}
and also that the level line $H_{\mu}(u,v)=H_{\mu}(u_{s}(\sigma),v_{s}(\sigma))$ is tangent to the curve $\Gamma_{0}$ in $(u_{s_{0}^{\tau}}(\sigma),v_{s_{0}^{\tau}}(\sigma))$ and in $(u_{s_{1}^{\omega}}(\sigma),v_{s_{1}^{\omega}}(\sigma))$.
See Figure~\ref{fig-07} for a graphical representation.

\begin{figure}[htb]
\begin{tikzpicture}
\begin{axis}[
  tick label style={font=\scriptsize},
  axis y line=middle, 
  axis x line=middle,
  xtick={0,0.0270902,0.278743,0.655019,0.868187},
  ytick={0},
  xticklabels={},
  yticklabels={},
  xlabel={\small $s$},
  ylabel={\small $h_{\mu}(s)$},
every axis x label/.style={
    at={(ticklabel* cs:1.0)},
    anchor=west,
},
every axis y label/.style={
    at={(ticklabel* cs:1.0)},
    anchor=south,
},
  width=6cm,
  height=6cm,
  xmin=-0.2,
  xmax=1.2,
  ymin=-2,
  ymax=1]
\addplot [color=fst-orange,line width=1pt,smooth] coordinates {(0, 0.) (0.05, -0.224447) (0.1, -0.435037) (0.15, -0.611684) (0.2, -0.738611) (0.25, -0.805952) (0.3, -0.810752) (0.35, -0.757271) (0.4, -0.656549) (0.45, -0.525284) (0.5, -0.384138) (0.55, -0.255641) (0.6, -0.161901) (0.65, -0.122359) (0.7, -0.151817) (0.75, -0.258935) (0.8, -0.445351) (0.85, -0.705514) (0.9, -1.02725) (0.95, -1.393) (1, -1.78158)};
\addplot [mark=none,dashed,color=black] coordinates {(1,-2) (1,2)};
\addplot [mark=none,dashed,color=gray] coordinates {(0.278743,0) (0.278743,-0.816243)};
\addplot [mark=none,dashed,color=gray] coordinates {(0.655019,0) (0.655019,-0.122004)};
\addplot [mark=none,dashed,color=gray] coordinates {(0.868187,0) (0.868187,-0.816243)};
\addplot [mark=none,dashed,color=gray] coordinates {(0,-0.816243) (1,-0.816243)};
\addplot [mark=none,dashed,color=gray] coordinates {(0,-0.122004) (1,-0.122004)};
\node[label={220:{\scriptsize{0}}}] at (axis cs:0.04,0.1) {};
\node[label={310:{\scriptsize{1}}}] at (axis cs:0.96,0.1) {};
\node[label={90:{\scriptsize{$s_{0}^{\omega}$}}}] at (axis cs:0.0670902,-0.1) {};
\node[label={90:{\scriptsize{$s_{0}^{\tau}$}}}] at (axis cs:0.278743,-0.1) {};
\node[label={90:{\scriptsize{$s_{1}^{\omega}$}}}] at (axis cs:0.655019,-0.1) {};
\node[label={90:{\scriptsize{$s_{1}^{\tau}$}}}] at (axis cs:0.868187,-0.1) {};
\end{axis}
\end{tikzpicture}
\qquad\qquad
\begin{tikzpicture}
\begin{axis}[
  tick label style={font=\scriptsize},
  axis y line=middle, 
  axis x line=middle,
  xtick={1},
  ytick={0},
  xticklabels={},
  yticklabels={},
  xlabel={\small $u$},
  ylabel={\small $v$},
every axis x label/.style={
    at={(ticklabel* cs:1.0)},
    anchor=west,
},
every axis y label/.style={
    at={(ticklabel* cs:1.0)},
    anchor=south,
},
  width=6cm,
  height=6cm,
  xmin=-0.2,
  xmax=1.2,
  ymin=-12,
  ymax=12]
\addplot [mark=none,dashed,color=black] coordinates {(1,-12) (1,12)};
\addplot [color=fst-red,line width=0.8pt,smooth] coordinates {(0,0) (0.0333333, 0.353815) (0.0666667, 0.987827) (0.1, 1.79072) (0.133333, 2.71948) (0.166667, 3.74743) (0.2, 4.6) (0.3,6.6) (0.4,7.5) (0.6, 7.8) (0.8,7.9) (1,8)};
\addplot [color=fst-red,line width=0.8pt,smooth] coordinates {(0,0) (0.0333333, -0.353815) (0.0666667, -0.987827) (0.1, -1.79072) (0.133333, -2.71948) (0.166667, -3.74743) (0.2, -4.6) (0.3,-6.6) (0.4,-7.5) (0.6, -7.8) (0.8,-7.9) (1,-8)};
\addplot [color=fst-blue,line width=0.8pt,smooth] coordinates {(0,0) (0.1, -0.3) (0.25,-0.9) (0.3,-1.6) (0.25,-2.8) (0.25,-3.8) (0.3,-4.8) (0.4,-5.7) (0.5,-5.9) (0.6,-5.8) (0.7,-5.2) (0.8,-3.8) (0.9,-2) (1,0)};
\addplot [color=fst-purple,line width=0.8pt,smooth] coordinates {(0,0) (0.1, 0.3) (0.25,0.9) (0.3,1.6) (0.25,2.8) (0.25,3.8) (0.3,4.8) (0.4,5.7) (0.5,5.9) (0.6,5.8) (0.7,5.2) (0.8,3.8) (0.9,2) (1,0)};
\addplot [color=fst-gray,line width=0.8pt,smooth,dashed] coordinates {(1, -5) (0.8, -5) (0.5,-4.3) (0.4, -3.8) (0.32, -2.3) (0.29, 0) (0.32, 2.3) (0.4, 3.8) (0.5,4.3) (0.8, 5) (1, 5)};
\addplot [color=fst-gray,line width=0.8pt,smooth] coordinates {(1, -6) (0.8, -6) (0.4, -5) (0.25, -2.3) (0.22, 0) (0.25, 2.3) (0.4, 5) (0.8, 6) (1, 6)};
\addplot [color=fst-gray,line width=0.8pt,smooth,dashed] coordinates {(1, -7) (0.8, -7) (0.4, -6.2) 
(0.22, -3.2) (0.12, 0) (0.22, 3.2) (0.4, 6.2) (0.8, 7) (1, 7)};
\node[label={220:{\scriptsize{0}}}] at (axis cs:0.04,0.2) {};
\node[label={310:{\scriptsize{1}}}] at (axis cs:0.96,0.2) {};
\node[label={0:{\scriptsize{$\Gamma_{0}$}}}] at (axis cs:0.79,-3.5) {};
\node[label={0:{\scriptsize{$\Gamma_{1}$}}}] at (axis cs:0.79,3.5) {};
\node[label={90:{\scriptsize{$\mathcal{M}_{\mu}$}}}] at (axis cs:0.81,7.5) {};
\node[label={0:{\scriptsize{$P(s_{0}^{\tau})$}}}] at (axis cs:0.24,-1.5) {};
\node[label={90:{\scriptsize{$P(s_{1}^{\tau})$}}}] at (axis cs:0.64,-5.8) {};
\node[label={90:{\scriptsize{$P(s_{0}^{\omega})$}}}] at (axis cs:0.015,1.9) {};
\node[label={270:{\scriptsize{$P(s_{1}^{\omega})$}}}] at (axis cs:0.13,-2.2) {};
\addplot [color=black,line width=0.5pt] coordinates {(0.03, 3) (0.125,-0.4)};
\node[circle,fill,inner sep=0.8pt,fst-black] at (axis cs:0.125,0.4) {};
\node[circle,fill,inner sep=0.8pt,fst-black] at (axis cs:0.125,-0.4) {};
\node[circle,fill,inner sep=0.8pt,fst-black] at (axis cs:0.26,4.2) {};
\node[circle,fill,inner sep=0.8pt,fst-black] at (axis cs:0.26,-4.2) {};
\node[circle,fill,inner sep=0.8pt,fst-black] at (axis cs:0.3,1.5) {};
\node[circle,fill,inner sep=0.8pt,fst-black] at (axis cs:0.3,-1.5) {};
\node[circle,fill,inner sep=0.8pt,fst-black] at (axis cs:0.73,4.85) {};
\node[circle,fill,inner sep=0.8pt,fst-black] at (axis cs:0.73,-4.85) {};
\node[circle,fill,inner sep=0.8pt,fst-black] at (axis cs:0,0) {};
\node[circle,fill,inner sep=0.8pt,fst-black] at (axis cs:1,0) {};
\end{axis}
\end{tikzpicture}
\caption{For $\lambda\in(0,\lambda^{*})$, qualitative representations of the graph of the function $h_{\mu}$ defined in~\eqref{def-hs} (left) and of $\Gamma_0$ (blue), $\Gamma_1$ (violet), $\mathcal{M}_{\mu}$ (pink) along with some level lines of~\eqref{eq-energy-mu} (gray) in the $(u,v)$-plane (right). We set $P(s)=(u_{s}(\sigma),v_{s}(\sigma))$.} 
\label{fig-07}
\end{figure}
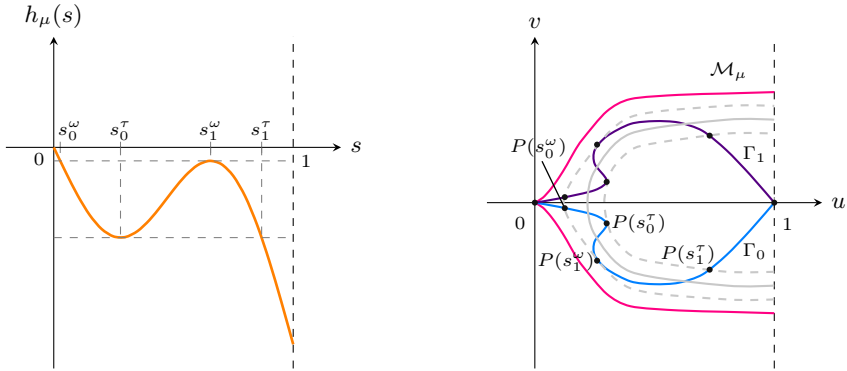

We are in position to study the connection times in the case $\mu>\tilde{\mu}(\lambda)$ with $\mu$ sufficiently close to $\tilde{\mu}(\lambda)$. If $s\in(0,s_{0}^{\omega})\cup(s_{1}^{\tau},1)$, we can define only the time $T_{1}(s)$ to reach $\Gamma_{1}$ from $(u_{s}(\sigma),v_{s}(\sigma))$, exactly as in \eqref{def-T1_4.2}. It follows that $T_{1}$ satisfies \eqref{eq:order_L0_L1} together with
\begin{equation*}
\lim_{s\to (s_{0}^{\omega})^{-}} T_{1}(s) =: \kappa_{0}\in(0,+\infty), \quad
\lim_{s\to (s_{1}^{\tau})^{+}} T_{1}(s) =: \kappa_{1}\in(0,+\infty).
\end{equation*}
If $s\in(s_{0}^{\omega},s_{0}^{\tau})\cup(s_{0}^{\tau},s_{1}^{\omega})\cup(s_{1}^{\omega},s_{1}^{\tau})$, we can define the times $T_{i}(s)$, for $i=1,2,3$, exactly as in Cases~1, 2, 3 discussed in Section~\ref{section-4.1} (cf.,~\eqref{def-T123}). 
We stress that they are bounded and some properties of symmetry, analogous to the ones in Section~\ref{section-4.1}, hold.
The following properties can be proved straightforwardly
\begin{align}
&\lim_{s\to 0^{+}} T_{1}(s) = \mathcal{L}_0,&
&\lim_{s\to 1^{-}} T_{1}(s) =\mathcal{L}_1,&
\\
&\lim_{s\to (s_{0}^{\omega})^{+}} T_{1}(s) = \kappa_{0},  &
&\lim_{s\to (s_{0}^{\omega})^{+}} T_{2}(s) = \lim_{s\to (s_{0}^{\omega})^{+}} T_{3}(s) =: \kappa_{2},&
\\
&\lim_{s\to s_{0}^{\tau}} T_{1}(s) = \lim_{s\to s_{0}^{\tau}} T_{2}(s) =: \kappa_{3},&
&\lim_{s\to s_{0}^{\tau}} T_{3}(s) = \kappa_{2},&\label{eq-tab}
\\
&\lim_{s\to s_{1}^{\omega}} T_{1}(s) = \kappa_{2},&
&\lim_{s\to s_{1}^{\omega}} T_{2}(s) = \lim_{s\to s_{1}^{\omega}} T_{3}(s) =: \kappa_{4},&
\\
&\lim_{s\to (s_{1}^{\tau})^{-}} T_{1}(s) = \lim_{s\to (s_{1}^{\tau})^{-}} T_{2}(s) = \kappa_{2},&
&\lim_{s\to (s_{1}^{\tau})^{-}} T_{3}(s) = \kappa_{1},&
\end{align}
where $\kappa_{i}\in(0,+\infty)$. Reasoning as in Section~\ref{section-4.1}, we can deduce that $\kappa_2<\kappa_1$. In Figure~\ref{fig-08}, we summarize the behavior of the time-maps $T_i$, with $i=1,2,3$, analyzed above. In particular, the case $\mu\in(\tilde{\mu}(\lambda),+\infty)$ with $\mu$ sufficiently close to $\tilde{\mu}(\lambda)$ provides the presence of a \textit{loop} in addition to the curve in Figure~\ref{fig-06}. 

Also in the case $\lambda\in(0,\lambda^{*})$, we set $T_1(0)=\mathcal{L}_0$, $T_1(s_0^{\tau})=\kappa_3$, $T_1(s_1^{\tau})=\kappa_2$, and $T_1(1)=\mathcal{L}_1$, so that $T_1$ is lower-semicontinuous in $D_1:=[0,1]$. Moreover, we extend by continuity the connection times $T_{i}(s)$, $i=2,3$, according to~\eqref{eq-tab} and we define their domains $D_i:=[s_0^{\omega},s_1^{\tau}]$, $i=2,3$. Therefore, \eqref{eq-order-Ti-general} holds true also in this case.
In addition, recalling that $\kappa_2<\kappa_1$, we observe that, for every $i=1,2,3$, $T_{i}\left(D_{i}\right)$ is an interval in $(0,+\infty)$. From~\eqref{eq-tab}, we can easily infer that $\bigcup_{i=1,2,3}T_i(D_i)$ is an interval too.

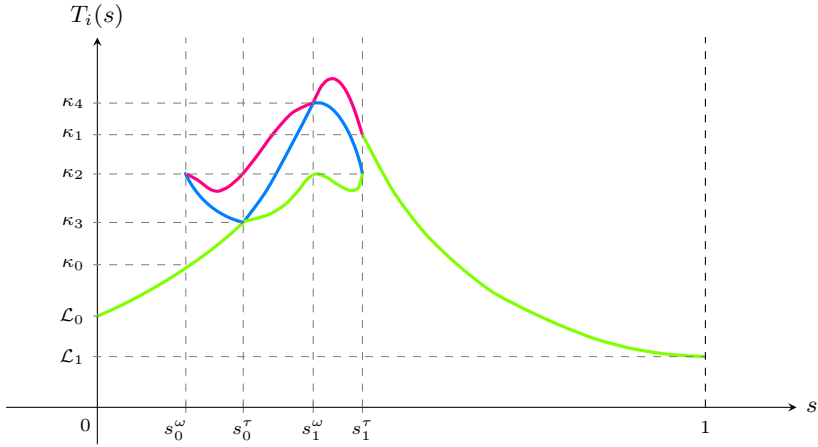
\begin{figure}[htb]
\begin{tikzpicture}
\begin{axis}[
  tick label style={font=\scriptsize},
  axis y line=middle, 
  axis x line=middle,
  xtick={0,0.1455,0.24,0.355,0.436},
  ytick={0},
  xticklabels={},
  yticklabels={},
  xlabel={\small $s$},
  ylabel={\small $T_{i}(s)$},
every axis x label/.style={
    at={(ticklabel* cs:1.0)},
    anchor=west,
},
every axis y label/.style={
    at={(ticklabel* cs:1.0)},
    anchor=south,
},
  width=12cm,
  height=7cm,
  xmin=-0.15,
  xmax=1.15,
  ymin=-0.7,
  ymax=7] 
\addplot [color=fst-green,line width=1.2pt,smooth] coordinates {(0.0000,1.7280) (0.0241,1.8480) (0.0491,1.9870) (0.0747,2.1420) (0.1005,2.3110) (0.1261,2.4920) (0.1513,2.6830) (0.1756,2.8810) (0.1987,3.0850) (0.2203,3.2920) (0.2400,3.5000)};
\addplot [color=fst-green,line width=1.2pt,smooth] coordinates {(0.4360,5.1620) (0.4862,4.1) (0.5384,3.3) (0.5923,2.7) (0.6478,2.2) (0.7045,1.85) (0.7623,1.5460) (0.8210,1.2880) (0.8803,1.1090) (0.9400,1.0020) (1.0000,0.9600)};
\addplot [color=fst-blue,line width=1.2pt,smooth] coordinates {(0.2400,3.5000) (0.2568,3.7540) (0.2686,3.9400) (0.2775,4.0930) (0.2855,4.2460) (0.2946,4.4350) (0.3061,4.6820) (0.3193,4.9670) (0.3327,5.2610) (0.3449,5.5320) (0.3545,5.7500)};
\addplot [color=fst-red,line width=1.2pt,smooth] coordinates {(0.3545,5.7500) (0.3665,6.0440) (0.3780,6.1930) (0.3887,6.2210) (0.3986,6.1540) (0.4077,6.0170) (0.4157,5.8360) (0.4227,5.6340) (0.4284,5.4380) (0.4329,5.2720) (0.4360,5.1620)};
\addplot [color=fst-blue,line width=1.2pt,smooth] coordinates {(0.3550,5.7600) (0.3681,5.7610) (0.3803,5.7000) (0.3916,5.5870) (0.4017,5.4370) (0.4108,5.2600) (0.4186,5.0700) (0.4251,4.8780) (0.4302,4.6980) (0.4339,4.5410) (0.4360,4.4200)};
\addplot [color=fst-green,line width=1.2pt,smooth] coordinates {(0.2400,3.5160) (0.2805,3.6420) (0.3093,3.8440) (0.3294,4.0880) (0.3445,4.3100) (0.3592,4.4180) (0.3770,4.3550) (0.3963,4.2100) (0.4147,4.1080) (0.4290,4.1510) (0.4360,4.4200)};
\addplot [color=fst-blue,line width=1.2pt,smooth] coordinates {(0.2400,3.5000) (0.2275,3.5430) (0.2157,3.5970) (0.2046,3.6620) (0.1942,3.7370) (0.1844,3.8230) (0.1753,3.9190) (0.1668,4.0260) (0.1590,4.1430) (0.1519,4.2710) (0.1455,4.4100) (0.1460,4.430)};
\addplot [color=fst-red,line width=1.2pt,smooth] coordinates {(0.1466,4.4180) (0.1639,4.31) (0.1757,4.21) (0.1861,4.1240) (0.1990,4.0920) (0.2158,4.1770) (0.2366,4.3890) (0.2612,4.7400) (0.2893,5.1810) (0.3201,5.5710) (0.3528,5.7600)};
\node[label={220:{\scriptsize{0}}}] at (axis cs:0.02,0.1) {};
\node[label={270:{\scriptsize{1}}}] at (axis cs:1,0.1) {};
\node[label={270:{\scriptsize{$s_{0}^{\omega}$}}}] at (axis cs:0.129,0.1) {};
\node[label={270:{\scriptsize{$s_{0}^{\tau}$}}}] at (axis cs:0.24,0.1) {};
\node[label={270:{\scriptsize{$s_{1}^{\omega}$}}}] at (axis cs:0.355,0.1) {};
\node[label={270:{\scriptsize{$s_{1}^{\tau}$}}}] at (axis cs:0.436,0.1) {};
\node[label={180:{\scriptsize{$\mathcal{L}_{0}$}}}] at (axis cs:0.01,1.728) {};
\node[label={180:{\scriptsize{$\mathcal{L}_{1}$}}}] at (axis cs:0.01,0.96) {};
\node[label={180:{\scriptsize{$\kappa_{0}$}}}] at (axis cs:0.01,2.7) {};
\node[label={180:{\scriptsize{$\kappa_{3}$}}}] at (axis cs:0.01,3.5) {};
\node[label={180:{\scriptsize{$\kappa_{1}$}}}] at (axis cs:0.01,5.16194) {};
\node[label={180:{\scriptsize{$\kappa_{4}$}}}] at (axis cs:0.01,5.76) {};
\node[label={180:{\scriptsize{$\kappa_{2}$}}}] at (axis cs:0.01,4.4180) {};
\addplot [mark=none,dashed,color=gray] coordinates {(-0.007,2.7) (0.1455,2.7)};
\addplot [mark=none,dashed,color=gray] coordinates {(-0.007,0.96) (1,0.96)};
\addplot [mark=none,dashed,color=gray] coordinates {(-0.007,1.728) (0.003,1.728)};
\addplot [mark=none,dashed,color=gray] coordinates {(-0.007,3.5) (0.24,3.5)};
\addplot [mark=none,dashed,color=gray] coordinates {(-0.007,5.16194) (0.436,5.16194)};
\addplot [mark=none,dashed,color=gray] coordinates {(-0.007,5.76) (0.36,5.76)};
\addplot [mark=none,dashed,color=gray] coordinates {(-0.007,4.418) (0.436,4.418)};
\addplot [mark=none,dashed]coordinates {(1,0) (1,12)};
\addplot [mark=none,dashed,color=gray]coordinates {(0.1455,0) (0.1455,7)};
\addplot [mark=none,dashed,color=gray] coordinates {(0.24,0) (0.24,7)};
\addplot [mark=none,dashed,color=gray] coordinates {(0.355,0) (0.355,7)};
\addplot [mark=none,dashed,color=gray] coordinates {(0.436,0) (0.436,7)};
\end{axis}
\end{tikzpicture}
\caption{For $\lambda\in(0,\lambda^{*})$ and $\mu>\tilde{\mu}(\lambda)$, qualitative graphs of $T_{i}$, with $i=1,2,3$, which are the times necessary to connect $\Gamma_{0}$ to $\Gamma_{1}$, as functions of the initial data $u(0)=s$ of~\eqref{eq-initial0}: $T_1$~(green), $T_2$ (blue), and $T_3$ (pink). The existence of the \textit{loop} is ensured only for $\mu$ near $\tilde{\mu}(\lambda)$.}   
\label{fig-08}
\end{figure}

\section{Proofs of Theorem~\ref{th-intro1} and Theorem~\ref{th-intro2}}\label{section-5}

In this section, we prove Theorems~\ref{th-intro1} and~\ref{th-intro2}. For the former, the strategy of the proof is to show that no connection from $\Gamma_{0}$ to $\Gamma_{1}$ along the level lines of~\eqref{eq-energy-mu} is possible in time $1-2\sigma$, which is the length of the interval $(\sigma,1-\sigma)$. For the latter, instead, we prove that the connection times $T_{i}$, $i=1,2,3$, defined in Section~\ref{section-4} allow us to connect $\Gamma_{0}$ to $\Gamma_{1}$ in several ways. For convenience, we point out explicitly the dependence on $\mu$ of the connection times and of all related quantities.

Before passing to the proofs, we need to introduce some ``barriers'', $\mathfrak{B}_{-}$ and $\mathfrak{B}_{+}$, in the $(u,v)$-plane which will control the connection times (see Figure~\ref{fig-09}).

\begin{figure}[htb]
	\begin{tikzpicture}
	\begin{axis}[
	tick label style={font=\scriptsize},
	axis y line=middle, 
	axis x line=middle,
	xtick={1},
	ytick={0},
	xticklabels={},
	yticklabels={},
	xlabel={\small $u$},
	ylabel={\small $v$},
	every axis x label/.style={
		at={(ticklabel* cs:1.0)},
		anchor=west,
	},
	every axis y label/.style={
		at={(ticklabel* cs:1.0)},
		anchor=south,
	},
	width=6cm,
	height=6cm,
	xmin=-0.2,
	xmax=1.2,
	ymin=-12,
	ymax=12]
	\addplot [mark=none,dashed,color=black] coordinates {(1,-12) (1,12)};
	\addplot [mark=none,dashed,color=black] coordinates {(0.618034,-11.46) (0.618034,11.46)};
	\addplot [color=fst-red,line width=0.8pt,smooth] coordinates {(0,0) (0.0333333, 0.353) (0.0666667, 0.987827) (0.1, 1.79072) (0.133333, 2.71948) (0.166667, 4) (0.2, 5.6) (0.3,9.6) (0.4,10.9) (0.6, 11.8) (0.8,11.9) (1,11.9)};
	\addplot [color=fst-red,line width=0.8pt,smooth] coordinates {(0,0) (0.0333333, -0.353) (0.0666667, -0.987827) (0.1, -1.79072) (0.133333, -2.71948) (0.166667, -4) (0.2, -5.6) (0.3, -9.6) (0.4, -10.9) (0.6, -11.8) (0.8, -11.9) (1, -11.9)};
	\addplot [color=fst-green,line width=0.8pt,smooth] coordinates {(0, 0) (0.0125, 0.0015625) (0.025, 0.00625) (0.0375, 0.0140625) (0.05, 0.025) (0.0625, 0.0390625) (0.075, 0.05625) (0.0875, 0.0765625) (0.1, 0.1) (0.1125, 0.126563) (0.125, 0.15625) (0.1375, 0.189063) (0.15, 0.225) (0.1625, 0.264063) (0.175, 0.30625) (0.1875, 0.351563) (0.2, 0.4) (0.2125, 0.451563) (0.225, 0.50625) (0.2375, 0.564063) (0.25, 0.625) (0.2625, 0.689063) (0.275, 0.75625) (0.2875, 0.826562) (0.3, 0.9) (0.3125, 0.976563) (0.325, 1.05625) (0.3375, 1.13906) (0.35, 1.225) (0.3625, 1.31406) (0.375, 1.40625) (0.3875, 1.50156) (0.4, 1.6) (0.4125, 1.70156) (0.425, 1.80625) (0.4375, 1.91406) (0.45, 2.025) (0.4625, 2.13906) (0.475, 2.25625) (0.4875, 2.37656) (0.5, 2.5) (0.5125, 2.62656) (0.525, 2.75625) (0.5375, 2.88906) (0.55, 3.025) (0.5625, 3.16406) (0.575, 3.30625) (0.5875, 3.45156) (0.6, 3.6) (0.6125, 3.75156) (0.618034, 3.81966) (0.618034, 3.81966) (0.625, 3.75) (0.65, 3.5) (0.675, 3.25) (0.7, 3.) (0.725, 2.75) (0.75, 2.5) (0.775, 2.25) (0.8, 2.) (0.825, 1.75) (0.85, 1.5) (0.875, 1.25) (0.9, 1.) (0.925, 0.75) (0.95, 0.5) (0.975, 0.25) (1,0)};
	\addplot [color=fst-green,line width=0.8pt,smooth] coordinates {(0, 0) (0.0125, 0.0015625) (0.025, -0.00625) (0.0375, -0.0140625) (0.05, -0.025) (0.0625, -0.0390625) (0.075, -0.05625) (0.0875, -0.0765625) (0.1, -0.1) (0.1125, -0.126563) (0.125, -0.15625) (0.1375, -0.189063) (0.15, -0.225) (0.1625, -0.264063) (0.175, -0.30625) (0.1875, -0.351563) (0.2, -0.4) (0.2125, -0.451563) (0.225, -0.50625) (0.2375, -0.564063) (0.25, -0.625) (0.2625, -0.689063) (0.275, -0.75625) (0.2875, -0.826562) (0.3, -0.9) (0.3125, -0.976563) (0.325, -1.05625) (0.3375, -1.13906) (0.35, -1.225) (0.3625, -1.31406) (0.375, -1.40625) (0.3875, -1.50156) (0.4, -1.6) (0.4125, -1.70156) (0.425, -1.80625) (0.4375, -1.91406) (0.45, -2.025) (0.4625, -2.13906) (0.475, -2.25625) (0.4875, -2.37656) (0.5, -2.5) (0.5125, -2.62656) (0.525, -2.75625) (0.5375, -2.88906) (0.55, -3.025) (0.5625, -3.16406) (0.575, -3.30625) (0.5875, -3.45156) (0.6, -3.6) (0.6125, -3.75156) (0.618034, -3.81966) (0.618034, -3.81966) (0.625, -3.75) (0.65, -3.5) (0.675, -3.25) (0.7, -3.) (0.725, -2.75) (0.75, -2.5) (0.775, -2.25) (0.8, -2.) (0.825, -1.75) (0.85, -1.5) (0.875, -1.25) (0.9, -1.) (0.925, -0.75) (0.95, -0.5) (0.975, -0.25) (1,0)};
	\addplot [color=fst-yellow,line width=0.8pt,smooth] coordinates {(0, 0) (0.0125, 0.0046875) (0.025, 0.01875) (0.0375, 0.0421875) (0.05, 0.075) (0.0625, 0.117188) (0.075, 0.16875) (0.0875, 0.229687) (0.1, 0.3) (0.1125, 0.379688) (0.125, 0.46875) (0.1375, 0.567188) (0.15, 0.675) (0.1625, 0.792188) (0.175, 0.91875) (0.1875, 1.05469) (0.2, 1.2) (0.2125, 1.35469) (0.225, 1.51875) (0.2375, 1.69219) (0.25, 1.875) (0.2625, 2.06719) (0.275, 2.26875) (0.2875, 2.47969) (0.3, 2.7) (0.3125, 2.92969) (0.325, 3.16875) (0.3375, 3.41719) (0.35, 3.675) (0.3625, 3.94219) (0.375, 4.21875) (0.3875, 4.50469) (0.4, 4.8) (0.4125, 5.10469) (0.425, 5.41875) (0.4375, 5.74219) (0.45, 6.075) (0.4625, 6.41719) (0.475, 6.76875) (0.4875, 7.12969) (0.5, 7.5) (0.5125, 7.87969) (0.525, 8.26875) (0.5375, 8.66719) (0.55, 9.075) (0.5625, 9.49219) (0.575, 9.91875) (0.5875, 10.3547) (0.6, 10.8) (0.6125, 11.2547) (0.618034, 11.459) (0.618034, 11.459) (0.625, 11.25) (0.65, 10.5) (0.675, 9.75) (0.7, 9.) (0.725, 8.25) (0.75, 7.5) (0.775, 6.75) (0.8, 6.) (0.825, 5.25) (0.85, 4.5) (0.875, 3.75) (0.9, 3.) (0.925, 2.25) (0.95, 1.5) (0.975, 0.75) (1,0)};
	\addplot [color=fst-yellow,line width=0.8pt,smooth] coordinates {(0, 0) (0.0125, -0.0046875) (0.025, -0.01875) (0.0375, -0.0421875) (0.05, -0.075) (0.0625, -0.117188) (0.075, -0.16875) (0.0875, -0.229687) (0.1, -0.3) (0.1125, -0.379688) (0.125, -0.46875) (0.1375, -0.567188) (0.15, -0.675) (0.1625, -0.792188) (0.175, -0.91875) (0.1875, -1.05469) (0.2, -1.2) (0.2125, -1.35469) (0.225, -1.51875) (0.2375, -1.69219) (0.25, -1.875) (0.2625, -2.06719) (0.275, -2.26875) (0.2875, -2.47969) (0.3, -2.7) (0.3125, -2.92969) (0.325, -3.16875) (0.3375, -3.41719) (0.35, -3.675) (0.3625, -3.94219) (0.375, -4.21875) (0.3875, -4.50469) (0.4, -4.8) (0.4125, -5.10469) (0.425, -5.41875) (0.4375, -5.74219) (0.45, -6.075) (0.4625, -6.41719) (0.475, -6.76875) (0.4875, -7.12969) (0.5, -7.5) (0.5125, -7.87969) (0.525, -8.26875) (0.5375, -8.66719) (0.55, -9.075) (0.5625, -9.49219) (0.575, -9.91875) (0.5875, -10.3547) (0.6, -10.8) (0.6125, -11.2547) (0.618034, -11.459) (0.618034, -11.459) (0.625, -11.25) (0.65, -10.5) (0.675, -9.75) (0.7, -9.) (0.725, -8.25) (0.75, -7.5) (0.775, -6.75) (0.8, -6.) (0.825, -5.25) (0.85, -4.5) (0.875, -3.75) (0.9, -3.) (0.925, -2.25) (0.95, -1.5) (0.975, -0.75) (1,0)};
	\addplot [color=fst-blue,line width=0.8pt,smooth] coordinates {(0,0) (0.1, -0.25) (0.25,-0.9) (0.3,-1.3) (0.33,-1.8) (0.37,-2.9) (0.4,-3.8) (0.43, -4.9) (0.46,-5.4) (0.5,-5.7) (0.6,-5.8) (0.7,-5.2) (0.8,-3.8) (0.9,-2) (1,0)};
	\addplot [color=fst-purple,line width=0.8pt,smooth] coordinates {(0,0) (0.1, 0.25) (0.25,0.9) (0.3,1.3) (0.33,1.8) (0.37,2.9) (0.4,3.8) (0.43, 4.9) (0.46,5.4) (0.5,5.7) (0.6,5.8) (0.7,5.2) (0.8,3.8) (0.9,2) (1,0)};
	\node[label={220:{\scriptsize{0}}}] at (axis cs:0.04,0.2) {};
	\node[label={310:{\scriptsize{1}}}] at (axis cs:0.96,0.2) {};
	\node[label={270:{\scriptsize{$x^{*}$}}}] at (axis cs:0.7,1) {};
	\node[label={0:{\scriptsize{$\Gamma_{0}$}}}] at (axis cs:0.58,-6.4) {};
	\node[label={0:{\scriptsize{$\Gamma_{1}$}}}] at (axis cs:0.58,6.4) {};
	\node[label={0:{\scriptsize{$\mathfrak{B}_{+}$}}}] at (axis cs:0.39,-1.5) {};
	\node[label={270:{\scriptsize{$\mathfrak{B}_{-}$}}}] at (axis cs:0.8,-8) {};
	\node[label={270:{\scriptsize{$\mathcal{M}_{\mu}$}}}] at (axis cs:0.13,-5.2) {};
	\end{axis}
	\end{tikzpicture}
	\caption{For $\lambda\in(0,\lambda^{*})$, qualitative representation in the $(u,v)$-plane of the curves $\Gamma_0$ and $\Gamma_1$, the manifold $\mathcal{M}_{\mu}$, and the curves $\mathfrak{B}_{-}$ and $\mathfrak{B}_{+}$.}  
	\label{fig-09}
\end{figure}
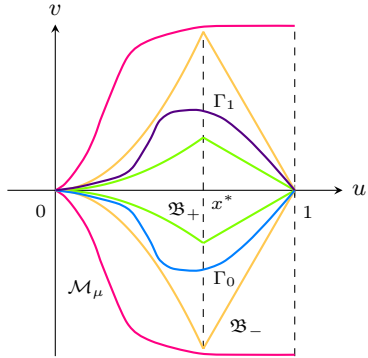

\smallskip
\noindent
\textit{Construction of ``interior barriers'' for all $\lambda>0$.} Let $\lambda> 0$ be fixed. According to the behavior of $\Gamma_{0}$ near $(0,0)$ and $(1,0)$ described in Proposition~\ref{pr-2.2}, we fix $k_{+}=k_{+}(\lambda)>0$ sufficiently small such that the curve
\begin{equation*}
\mathfrak{B}_{+}=\mathfrak{B}_{+}(\lambda) := \bigl{\{} (x,-k_{+}(\lambda)\min\{x^{2},1-x\}) \colon x\in[0,1]\bigr{\}},
\end{equation*}
lies between the curve $\Gamma_{0}=\Gamma_{0}(\lambda)$ and the segment $[0,1]\times\{0\}$. Observe that $\mathfrak{B}_{+}$ is the graph of the maximum function between the parabola $\{(x,-k_{+}x^{2}) \colon x\in[0,1]\}$ and the line $\{(x,-k_{+}(1-x))\colon x\in[0,1]\}$ which intersect for $x=x^{*}:=\frac{\sqrt{5}-1}{2}$. In particular, $\mathfrak{B}_{+}$ connects the points $(0,0)$ and $(1,0)$ in the $(u,v)$-plane.

Thanks to~\eqref{def-stable-manifold}, for sufficiently small $\mu$ the curve $\mathcal{M}_{\mu}$ intersects $\mathfrak{B}_{+}$ in two points in $(0,1)\times(-\infty,0)$, whose abscissas are denoted by $x_{0}^{\mathcal{M}}=x_{0}^{\mathcal{M}}(\mu)$ and $x_{1}^{\mathcal{M}}=x_{1}^{\mathcal{M}}(\mu)$, in such a way that $x_{0}^{\mathcal{M}}<x^{*}<x_{1}^{\mathcal{M}}$.

Let us define now, for all $x\in(0,x_{0}^{\mathcal{M}})\cup(x_{1}^{\mathcal{M}},1)$, the time $T_{+}(x,\mu)$ to connect the point $(x,-k_{+}\min\{x^{2},1-x\})\in\mathfrak{B}_{+}$ to $(x,k_{+}\min\{x^{2},1-x\})$, which is the symmetric point with respect to the $u$-axis, along the level lines of \eqref{eq-energy-mu}. Thanks to Propositions~\ref{pr-3.1} and~\ref{pr-3.2} we have that $T_{+}(\cdot,\mu)$ is strictly increasing in $(0,x_{0}^{\mathcal{M}})$ and strictly decreasing in $(x_{1}^{\mathcal{M}},1)$. Moreover, we extend $T_{+}(x,\mu)$ by continuity at $x=0$ and $x=1$. Now, by reasoning as in the proof of \eqref{eq:order_L0_L1}, one can prove that
\begin{equation*}
\min_{x\in[0,x_{0}^{\mathcal{M}})\cup(x_{1}^{\mathcal{M}},1]} T_{+}(x,\mu)=T_{+}(1,\mu)=\frac{1}{\sqrt{\mu}}\arctan\frac{k_{+}}{\sqrt{\mu}}.
\end{equation*}
Since the right-hand side converges to $+\infty$ as $\mu\to 0^{+}$, this implies that there exists $\mu_{+}=\mu_{+}(\lambda)>0$ such that
\begin{equation}\label{eq:limit_T_mu-A1}
T_{+}(x,\mu)>1-2\sigma, \quad \text{for all $\mu\in(0,\mu_{+})$ and $x\in[0,x_{0}^{\mathcal{M}}(\mu))\cup(x_{1}^{\mathcal{M}}(\mu),1]$}.
\end{equation}
Next, we compare the connection times from $\Gamma_{0}$ to $\Gamma_{1}$ introduced in Section~\ref{section-4} with $T_{+}$. We observe that, for sufficiently small $\mu$, such connection times look like in Figure~\ref{fig-06}. Moreover, for all $s\in(0,s_0^{\mathcal{M}}(\mu))\cup(s_{1}^{\mathcal{M}}(\mu),1)$, we denote by $\zeta_{s}$ the smallest abscissa of the intersections of $\mathfrak{B}_{+}$ with the level line of \eqref{eq-energy-mu} passing through $(u_{s}(\sigma),v_{s}(\sigma))$. Since $\zeta_{s}<u_{s}(\sigma)$, we have
\begin{equation}\label{eq:limit_T_mu-A2}
T_{1}(s,\mu)>T_{+}(\zeta_{s},\mu), \quad \text{for all $s\in(0,s_{0}^{\mathcal{M}}(\mu))\cup(s_{1}^{\mathcal{M}}(\mu),1)$}.
\end{equation}
Finally, by combining~\eqref{eq:limit_T_mu-A1},  \eqref{eq:limit_T_mu-A2}, and~\eqref{eq-order-Ti-general}, we conclude that
for all $\lambda>0$, there exists $\mu_{+}=\mu_{+}(\lambda)>0$ such that
\begin{equation}\label{eq:limit_T_mu-A}
T_{i}(s,\mu)>1-2\sigma, \quad \text{for all $\mu\in(0,\mu_{+})$ and  $s\in D_{i}(\mu)$,}
\end{equation}
for all $i=1,2,3$, where we recall that $D_{i}(\mu)$ is the domain of $T_i$. 

\smallskip
\noindent
\textit{Construction of ``exterior barriers'' for $\lambda\in(0,\lambda^{*})$.} According to the behavior of $\Gamma_{0}$ near $(0,0)$ and $(1,0)$ described in Proposition~\ref{pr-2.2}, we fix $k_{-}=k_{-}(\lambda)>0$ sufficiently large such that $\Gamma_{0}=\Gamma_{0}(\lambda)$ lies between the curve
\begin{equation*}
\mathfrak{B}_{-}=\mathfrak{B}_{-}(\lambda) := \bigl{\{} (x,-k_{-}(\lambda)\min\{x^{2},1-x\}) \colon x\in[0,1]\bigr{\}}
\end{equation*}
and the segment $[0,1]\times\{0\}$. Observe that, similarly as before, $\mathfrak{B}_{-}$ is the graph of the maximum function between the parabola $\{(x,-k_{-}x^{2}) \colon x\in[0,1]\}$ and the line $\{(x,-k_{-}(1-x))\colon x\in[0,1]\}$ which intersect at 
\begin{equation*}
(x^{*},y^{*}):=\biggl{(} \dfrac{\sqrt{5}-1}{2}, -k_{-} \dfrac{3-\sqrt{5}}{2}\biggr{)}.
\end{equation*}
In particular, $\mathfrak{B}_{-}$ connects the points $(0,0)$ and $(1,0)$ in the $(u,v)$-plane (see Figure~\ref{fig-09}).

Recalling the definition of $\mathcal{M}_{\mu}$ given in \eqref{def-stable-manifold} and property \eqref{Mnear0}, by
Proposition~\ref{pr-2.2}, for every $\mu>0$ there exists a neighborhood $\mathcal{U}_0(\mu)=[0,x_{0}(\mu))\times(-\varepsilon(\mu),\varepsilon(\mu))$ of $(0,0)$ such that $\mathfrak{B}_{-}\cap\mathcal{U}_0(\mu)$ lies in the region between $\mathcal{M}_{\mu}$ and $[0,1]\times\{0\}$. We fix $\hat{\mu}>0$ and we stress that the neighborhood $\mathcal{U}_0(\hat{\mu})$ satisfies the above property for every $\mu\geq\hat{\mu}$. At this stage, we fix $x_{0}(\hat{\mu})$ (which does not depend on $\mu$) and for every $\mu\geq\hat{\mu}$ and every  $u\in[x_{0}(\hat{\mu}),1]$, the corresponding points $(u,v_{\mu})\in \mathcal{M}_{\mu}$ satisfy
\begin{equation*}
|v_{\mu}|=\sqrt{2\mu G(u)}\to+\infty, \quad \text{as $\mu\to+\infty$.}
\end{equation*}
From this discussion we deduce the existence of $\bar{\mu}(\lambda)>0$ such that for every $\mu>\bar{\mu}(\lambda)$ the curves $\mathcal{M}_{\mu}$ and $\mathfrak{B}_{-}$ do not intersect in $(0,1]\times\mathbb{R}$. As a consequence, for the same range of $\mu$, $\mathcal{M}_{\mu}$ does not intersect $\Gamma_{0}$ and $\Gamma_{1}$ in $(0,1]\times\mathbb{R}$ as well.

We consider such a configuration, and define, for all $x\in(0,1)$, the time $T_{-}(x,\mu)$ to connect the point $(x,-k_{-}\min\{x^{2},1-x\})\in\mathfrak{B}_{-}$ to the symmetric point $(x,k_{-}\min\{x^{2},1-x\})$ along the level lines of \eqref{eq-energy-mu}. Thanks to Propositions~\ref{pr-3.1} and~\ref{pr-3.2}, we have that $T_{-}(\cdot,\mu)$ is strictly increasing in $(0,x^*]$ and strictly decreasing in $[x^*,1)$. Moreover, we also extend $T_-$ by continuity at $x=0$ and $x=1$. Thus, for all $x\in[0,1]$,
\begin{equation*}
T_{-}(x,\mu) \leq T_{-}(x^{*},\mu) = 2\int_{m(x^{*})}^{x^{*}}\frac{\mathrm{d}u}{\sqrt{(y^{*})^2+2\mu(G(u)-G(x^{*}))}}.
\end{equation*}
We observe that this integral converges to $0$ as $\mu\to+\infty$. Indeed, the integrand converges to $0$ and, using the conservation of the energy $H_{\mu}$ (cf., \eqref{eq-energy-mu}), it is easy to see that the level line passing through $(x^{*},y^{*})$ approaches a vertical line as $\mu\to +\infty$, thus $m(x^{*})\to x^{*}$ as $\mu\to+\infty$. We conclude that
\begin{equation}\label{eq:limit_T_mu-B2}
\lim_{\mu\to+\infty} T_{-}(\cdot,\mu)=0 \quad \text{uniformly in $[0,1]$.}
\end{equation}

Then, we define the maximal connection time from $\Gamma_{0}$ to $\Gamma_{1}$ (we refer to the constructions and notations of Section~\ref{section-4.2}). If $T_{1}$ is the unique connection time between $\Gamma_{0}$ and $\Gamma_{1}$ (cf., Figure~\ref{fig-noisola}), we set
\begin{equation*}
T_{\max}(s,\mu):=T_1(s,\mu), \quad \text{for all $s\in[0,1]$}.
\end{equation*}
Otherwise, if more connection times are present, as in Figure~\ref{fig-08}, we set
\begin{equation*}
T_{\max}(s,\mu):=
\begin{cases}
\, T_{1}(s,\mu), & \text{if $s\in[0,s_{0}^{\omega})\cup(s_{1}^{\tau},1]$,} \\
\, T_{3}(s,\mu), & \text{if $s\in[s_{0}^{\omega},s_{1}^{\tau}]$.}
\end{cases}
\end{equation*}
We observe that, in both cases, $T_{\max}(\cdot,\mu)$ is upper semi-continuous in $[0,1]$ and, moreover,
\begin{equation}\label{eq:order-Tmax}
T_{i}(s,\mu)\leq T_{\max}(s,\mu), \quad \text{for all $s\in[0,1]\cap D_{i}(\mu)$}.
\end{equation}
for all $i=1,2,3$.

Finally, we compare $T_{\max}$ and $T_{-}$: for all $s\in(0,1)$, we denote by $\zeta_{s}$ the largest abscissa of the intersections of $\mathfrak{B}_{-}$ with the level line of \eqref{eq-energy-mu} passing through $(u_{s}(\sigma),v_{s}(\sigma))$. Since $\zeta_{s}>u_{s}(\sigma)$, we have
\begin{equation}\label{eq:limit_T_mu-B3}
T_{\max}(s,\mu)<T_{-}(\zeta_{s},\mu), \quad \text{for all $s\in(0,1)$}.
\end{equation}
By combining~\eqref{eq:limit_T_mu-B2}, \eqref{eq:order-Tmax}, and~\eqref{eq:limit_T_mu-B3}, we conclude that, for all $\lambda\in(0,\lambda^{*})$, there exists $\mu_{-}=\mu_{-}(\lambda)>0$ such that
\begin{equation}\label{eq:limit_T_mu-B}
T_{i}(s,\mu)<1-2\sigma, \quad \text{for all $\mu>\mu_{-}$ and $s\in[0,1]\cap D_{i}(\mu)$,}
\end{equation}
for all $i=1,2,3$.

Moreover, we can adapt the previous arguments to the case $\lambda\in[\lambda^{*},+\infty)$ as follows. First of all, by recalling the quantities $s_{0}$ and $s_{1}$ introduced in Proposition~\ref{pr-2.1} (where for $\lambda=\lambda^{*}$, $s_{0}=s_{1}=s^{*}$), we have 
\begin{equation}\label{eq-lim-s}
\begin{aligned}
&\lim_{\mu\to+\infty}s_0^{\mathcal{M}}(\mu)= s_0, &&\lim_{\mu\to+\infty}s_1^{\mathcal{M}}(\mu)= s_1,\\
&\lim_{\mu\to+\infty}s_0^\tau(\mu)=: \bar{s}_0, &&\lim_{\mu\to+\infty}s_1^\tau(\mu)=: \bar{s}_1,
\end{aligned}
\end{equation}
with $\bar{s}_0\in(0,s_0)$ and $\bar{s}_1\in(s_1,1)$. Then, we consider the barrier $\mathfrak{B}_{-}$ as above, and, by taking $k_{-}$ arbitrarily large, its intersection points with $\Gamma_0$ can be made arbitrarily close to $(-\xi_{0},0)$ and $(-\xi_{1},0)$ (cf., the notation introduced in Corollary~\ref{cor-2.1}). With these constructions, it is possible to show that, for $\lambda\in[\lambda^{*},+\infty)$,
\begin{equation}\label{eq:limit_T_mu-C}
\lim_{\mu\to+\infty}T_{i}(\cdot,\mu)=0 \quad \text{locally uniformly in $D_{i,\infty}$,}
\end{equation}
for all $i=1,2,3$, where $D_{1,\infty}:=[0,s_{0})\cup(s_{1},1]$ and $D_{2,\infty}=D_{3,\infty}:=(0,s_{0})\cup(s_{1},\bar{s}_1]$. Notice that $D_{i,\infty}$ are the limiting domains of $T_i(\cdot,\mu)$ as $\mu\to+\infty$.

\smallskip

We are now ready to give the proofs of Theorems~\ref{th-intro1} and~\ref{th-intro2}.

\begin{proof}[Proof of Theorem~\ref{th-intro1}] $(i)$ For all $\lambda>0$, we claim that
\begin{equation*}
\mu_0^{*}(\lambda):=\min \biggl{\{} \mu >0 \colon \min_{s}T_1(s,\mu)= 1-2\sigma\biggr{\}}
\end{equation*}
is well-defined and positive. Indeed, the set $\{ \mu >0 \colon \min_{s}T_1(s,\mu)= 1-2\sigma\}$ is nonempty, bounded, bounded away from $0$ and closed thanks to~\eqref{eq:limit_T_mu-A}, \eqref{eq:limit_T_mu-B}, \eqref{eq:limit_T_mu-C} and the continuity of  $\mu\mapsto T_1(s,\mu)$ (notice that all the quantities in \eqref{def-T123} depend continuously on $\mu$). Then, we show that
\begin{equation}\label{eq:minT1-musmall}
\min_{s} T_1(s,\mu)>1-2\sigma, \quad \text{for all $\mu\in(0,\mu_0^{*}(\lambda))$},
\end{equation}
which, together with \eqref{eq-order-Ti-general} implies that problem \eqref{eq-main} has no solution, since no connection from $\Gamma_{0}$ to $\Gamma_{1}$ is possible in time $1-2\sigma$. To show this, let $\bar{\mu}\in(0,\mu_0^{*}(\lambda))$ be fixed and assume by contradiction that $\min_{s} T_1(s,\bar{\mu})\leq 1-2\sigma$. The equality cannot hold, because it would contradict the minimality of $\mu_0^{*}(\lambda)$. If $\min_{s} T_1(s,\bar{\mu})< 1-2\sigma$, instead, thanks to \eqref{eq:limit_T_mu-A} and the continuity in $\mu$, there would be a value of $\mu$ smaller than $\bar{\mu}$ (and thus smaller than $\mu_0^{*}(\lambda)$) for which $\min_{s} T_1(s,\mu)= 1-2\sigma$, against the minimality of $\mu_0^{*}(\lambda)$ again.
	
\smallskip
\noindent
$(ii)$ For all $\lambda\in(0,\lambda^{*})$, we claim that
\begin{equation*}
\mu_{0}^{**}(\lambda):=\max \biggl{\{} \mu >0 \colon \max_{s} T_{\max}(s,\mu) = 1-2\sigma\biggr{\}}
\end{equation*}
is well-defined, finite and satisfies $\mu_{0}^{**}(\lambda)>\mu_{0}^{*}(\lambda)$. The existence follows as above, while~\eqref{eq:limit_T_mu-B} gives that $\mu_0^{**}(\lambda)\leq\mu_{-}(\lambda)<+\infty$. If we assume by contradiction that  $\mu_{0}^{**}(\lambda)=\mu_{0}^{*}(\lambda)$, then $T_{\max}\left(s,\mu_{0}^{*}(\lambda)\right)=T_{1}\left(s,\mu_{0}^{*}(\lambda)\right)=1-2\sigma$ for all $s\in[0,1]$, which contradicts \eqref{eq:order_L0_L1}. If, instead, it was  $\mu_{0}^{**}(\lambda)<\mu_{0}^{*}(\lambda)$, by \eqref{eq:minT1-musmall} we would have
\begin{equation*}
1-2\sigma<\min_{s}T_{1}(s,\mu_{0}^{**}(\lambda))\leq\max_{s}T_{\max}(s,\mu_{0}^{**}(\lambda))=1-2\sigma,
\end{equation*}
again a contradiction. Now, we claim that
\begin{equation*}
\max_{s} T_{\max}(s,\mu)<1-2\sigma, \quad \text{for all $\mu>\mu_0^{**}(\lambda)$},
\end{equation*}
which, together with \eqref{eq:order-Tmax}, implies that problem \eqref{eq-main} has no solution, since no connection from $\Gamma_{0}$ to $\Gamma_{1}$ is possible in time $1-2\sigma$. To show this, let $\bar{\mu}>\mu_0^{**}(\lambda)$ and assume that $\max_{s} T_{\max}(s,\bar{\mu})\geq 1-2\sigma$. Equality cannot hold, because it would contradict the maximality of $\mu_0^{**}(\lambda)$. If, instead, $\max_{s} T_{\max}(s,\bar{\mu})> 1-2\sigma$, thanks to \eqref{eq:limit_T_mu-B} and the continuity in $\mu$, there would be a value of $\mu$ greater than $\bar{\mu}$ (and thus greater than $\mu_0^{**}(\lambda)$) for which $\max_{s} T_{\max}(s,\mu)= 1-2\sigma$, against the maximality of $\mu_0^{**}(\lambda)$ again.
\end{proof}

\begin{proof}[Proof of Theorem~\ref{th-intro2}] We start by setting, for all $\lambda>0$,
\begin{equation}
\mu_1^{*}(\lambda):=\max \biggl{\{} \mu >0 \colon \min_{s}T_1(s,\mu)= 1-2\sigma\biggr{\}}, \label{eq-mu1}
\end{equation}
which immediately gives $\mu_{0}^{*}(\lambda)\leq\mu_{1}^{*}(\lambda)$, and
\begin{equation}
\mu_2^{*}(\lambda):=\frac{2\lambda\sigma}{1-2\sigma}. \label{eq-mu2}
\end{equation}
We observe that, with such a definition, \eqref{eq-aux-time1} guarantees that
\begin{equation}\label{eq:order-L0-interval}
\mathcal{L}_{0}=\mathcal{L}_{0}(\mu)
\begin{cases}
\, <1-2\sigma, & \text{if $\mu>\mu_{2}^{*}(\lambda)$,} \\
\, =1-2\sigma, & \text{if $\mu=\mu_{2}^{*}(\lambda)$,} \\
\, >1-2\sigma, & \text{if $\mu<\mu_{2}^{*}(\lambda)$.}
\end{cases}
\end{equation}
In addition, we show that $\mu_1^{*}(\lambda)<\mu_2^{*}(\lambda)$ for all $\lambda>0$. Indeed, if by contradiction $\mu_1^{*}(\lambda)\geq\mu_2^{*}(\lambda)$ for some $\lambda>0$, then, \eqref{eq:order_L0_L1} and~\eqref{eq:order-L0-interval} would give $\mathcal{L}_{1}(\mu_1^{*}(\lambda))<\mathcal{L}_{0}(\mu_1^{*}(\lambda))\leq1-2\sigma$. Thus, $T_{1}(s,\mu_1^{*}(\lambda))<1-2\sigma$ for $s\sim 1$, against the fact that $\min_{s}T_{1}(s,\mu_1^{*}(\lambda))=1-2\sigma$.	
	
Moreover, for all $\lambda\in(0,\lambda^{*})$, we set
\begin{equation}\label{eq-mu2**}
\mu_{2}^{**}(\lambda):=\min \biggl{\{} \mu >0 \colon \max_{s} T_{\max}(s,\mu) = 1-2\sigma\biggr{\}},
\end{equation}
which immediately gives $\mu_{2}^{**}(\lambda)\leq\mu_{0}^{**}(\lambda)$,
and we show that $\mu_{2}^{*}(\lambda)<\mu_{2}^{**}(\lambda)$. Indeed, if by contradiction $\mu_2^{*}(\lambda)\geq\mu_2^{**}(\lambda)$ for some $\lambda\in(0,\lambda^{*})$, then \eqref{eq:order-L0-interval} would give $\mathcal{L}_{0}(\mu_2^{**}(\lambda))\geq1-2\sigma$. Thus, \eqref{eq:5_T1'(0)} implies $T_{\max}(s,\mu_2^{**}(\lambda))\geq T_{1}(s,\mu_2^{**}(\lambda))>1-2\sigma$ for $s\sim 0$, against the fact that $\max_{s}T_{\max}(s,\mu_2^{**}(\lambda))=1-2\sigma$.
	
Once we have introduced the quantities above, we prove the statements of this theorem.
	
\smallskip

First, dealing with case $(i)$, we consider a fixed $\lambda\in(0,\lambda^*)$.

\smallskip
\noindent
$(i.a)$ We prove that
\begin{equation}\label{eq:minT1-mularge}
\min_{s} T_{1}(s,\mu)<1-2\sigma, \quad \text{for all $\mu>\mu_{1}^{*}(\lambda)$}.
\end{equation}
Fix $\bar{\mu}>\mu_{1}^{*}(\lambda)$ and assume by contradiction that $\min_{s} T_{1}(s,\bar{\mu})>1-2\sigma$ (equality is excluded by the maximality of $\mu_{1}^{*}(\lambda)$). Thanks to~\eqref{eq:limit_T_mu-B} and the continuity in $\mu$, there exists $\mu>\bar{\mu}$ (thus $\mu>\mu_{1}^{*}(\lambda)$) such that $\min_{s} T_{1}(s,\mu)=1-2\sigma$, against the maximality of $\mu_{1}^{*}(\lambda)$. Thus~\eqref{eq:minT1-mularge} is proved. Now, using~\eqref{eq:limit_T_mu-A}, a similar contradiction argument allows us to show that
\begin{equation}\label{eq:maxTmax-musmall}
\max_{s\in[0,1]} T_{\max}(s,\mu)>1-2\sigma, \quad \text{for all $\mu<\mu_{2}^{**}(\lambda)$}.
\end{equation}
By combining~\eqref{eq:minT1-mularge}, \eqref{eq:maxTmax-musmall}, and since, in accord with Section~\ref{section-4.2},
\begin{equation}\label{eq:image_interval}
\bigcup_{i=1,2,3}T_{i}\left(D_{i}(\mu),\mu\right)=\left[\min_{s}T_1(s,\mu),\max_{s}T_{\max}(s,\mu)\right].
\end{equation}
we infer, for all $\mu\in(\mu^{*}_{1}(\lambda),\mu^{**}_{2}(\lambda))$, the existence of at least one value $\xi_1^{\mu}\in(0,1)$ such that $T_{i}(\xi_1^{\mu},\mu)=1-2\sigma$ for some $i\in\{1,2,3\}$. Thus, problem~\eqref{eq-main} has a solution $u$ such that $u(0)=s\in\{\xi_{1}^{\mu}\}$.
		
\smallskip
\noindent
$(i.b)$ Fix $\mu\in(\mu_{2}^{*},\mu_{2}^{**})$. Since $\mu>\mu_{2}^{*}$, by~\eqref{eq:order_L0_L1} and~\eqref{eq:order-L0-interval}, there exists $\bar{s}=\bar{s}(\mu)>0$, with $\bar{s}\sim 0$, such that
\begin{equation}\label{eq:T1-mu-greater-mu2*}
T_{1}(s,\mu)<1-2\sigma, \quad \text{for all $s\in(0,\bar{s})\cup(1-\bar{s},1)$}.
\end{equation}
This, together with~\eqref{eq:maxTmax-musmall} and~\eqref{eq:image_interval}, guarantees the existence of at least two values $\xi_{2,1}^{\mu},\xi_{2,2}^{\mu}$ such that
\begin{equation}\label{2sol-T1}
\begin{aligned}
&0<\xi_{2,1}^{\mu} < \xi_{2,2}^{\mu} < 1, \\
&T_i(\xi_{2,1}^{\mu},\mu)=T_j(\xi_{2,2}^{\mu},\mu)=1-2\sigma, \quad\text{for some $i,j\in\{1,2,3\}$}.
\end{aligned}
\end{equation}
As a consequence, for $\mu\in(\mu^{*}_{2}(\lambda),\mu^{**}_{2}(\lambda))$, problem~\eqref{eq-main} has two solutions $u$ such that $u(0)=s\in\{\xi_{2,1}^{\mu},\xi_{2,2}^{\mu}\}$. The proof of~$(i.b)$ is complete.
	
\smallskip

Next, dealing with case $(ii)$, we consider a fixed $\lambda\in[\lambda^*,+\infty)$.

\smallskip
\noindent
$(ii.a)$ We observe that \eqref{eq:minT1-mularge} can be obtained in this case as above using~\eqref{eq:limit_T_mu-C}. Moreover, since
\begin{equation}\label{eq:aux-asymptotes}
\lim_{s\to(s_{0}^{\mathcal{M}})^-}T_{1}(s,\mu)=\lim_{s\to(s_{1}^{\mathcal{M}})^+}T_{1}(s,\mu)=+\infty \quad \text{and} \quad \ell_{1,2}^{1}(\mu)<\ell_{3}^{1}(\mu)
\end{equation}
(cf., Figure~\ref{fig-06}), by continuity there exists, for all $\mu>\mu^{*}_{1}(\lambda)$, at least one value $\xi_1^{\mu}\in(0,s_{0}^{\mathcal{M}}(\mu))\cup(s_{1}^{\mathcal{M}}(\mu),1)$ such that $T_{1}(\xi_1^{\mu},\mu)=1-2\sigma$, and so problem~\eqref{eq-main} has a solution $u$ such that $u(0)=s\in\{\xi_{1}^{\mu}\}$.
	
\smallskip
\noindent
$(ii.b)$ Take any $\mu>\mu_{2}^{*}$. Then, \eqref{eq:T1-mu-greater-mu2*} holds as above and, thanks to~\eqref{eq:aux-asymptotes}, we obtain the existence of at least two values $\xi_{2,1}^{\mu},\xi_{2,2}^{\mu}$ such that
\begin{equation*}
0<\xi_{2,1}^{\mu} <s_{0}^{\mathcal{M}}(\mu)\leq s_{1}^{\mathcal{M}}(\mu)< \xi_{2,2}^{\mu} < 1 \quad \text{and} \quad T_{1}(\xi_{2,j}^{\mu},\mu)=1-2\sigma,\,\, j=1,2,
\end{equation*}
which completes the proof in this case.

\smallskip
\noindent	
$(ii.c)$ Set
\begin{equation} \label{eq-mu4}
\mu_{4}^{*}(\lambda):=\max \Bigl{\{}\mu_{2}^{*}(\lambda),\, \max \bigl{\{} \mu>0 \colon \ell^{0}_{1,2}(\mu) = 1-2\sigma \bigr{\}} \Bigr{\}},
\end{equation}
which is well-defined, finite, and obviously satisfies $\mu_{2}^{*}(\lambda)\leq\mu_{4}^{*}(\lambda)$. Observe that the finiteness is a consequence of
\begin{equation}\label{eq:limit_l120}
\lim_{\mu\to+\infty}\ell_{1,2}^{0}(\mu)=0.
\end{equation}
This follows from \eqref{eq:limit_T_mu-C} and the fact that $s_{0}^{\tau}(\mu)$ is bounded away from $0$ and $s_{0}$ as $\mu\to+\infty$, since the level lines of \eqref{eq-energy-mu} approach vertical lines as $\mu\to+\infty$, thus $s_{0}^{\tau}(\mu)$ converges, as $\mu\to+\infty$, to the value of $s$ for which $\Gamma_{0}$ has a vertical tangent in the $(u,v)$-plane.  Moreover, we have
\begin{equation*}
\min_{s\in(0,s_{0}^{\mathcal{M}})}T_{2}(s,\mu)\leq\ell_{1,2}^{0}(\mu)<1-2\sigma \; \text{ and } \;  \mathcal{L}_{1}(\mu)<\mathcal{L}_{0}(\mu)<1-2\sigma, \quad \text{for all $\mu>\mu_{4}^{*}$.}
\end{equation*}
Indeed, the second relation follows since $\mu_{4}^{*}\geq\mu_{2}^{*}$, while the first one from~\eqref{eq:limit_T_mu-A} and~\eqref{eq:limit_l120}. Then, recalling that
\begin{equation*}
\lim_{s\to 0^{+}}T_2(s,\mu)=\lim_{s\to (s_0^\mathcal{M})^{-}}T_2(s,\mu)=+\infty,
\end{equation*}
from the continuity of $T_2(\cdot,\mu)$ in $(0,s_0^\mathcal{M})$, we obtain, for all $\mu>\mu_{4}^{*}$, the existence of two values $\xi_{4,1}^{\mu},\xi_{4,2}^{\mu}$ such that
\begin{equation*}
0<\xi_{4,1}^{\mu}<s_{0}^{\tau}<\xi_{4,2}^{\mu}<s_{0}^{\mathcal{M}} \quad \text{and} \quad T_2(\xi_{4,1}^{\mu},\mu)=T_2(\xi_{4,2}^{\mu},\mu)=1-2\sigma.
\end{equation*}
Moreover, since $\mu^{*}_4(\lambda)\geq\mu^{*}_2(\lambda)$, arguing as in the proof of~$(ii.b)$ and using the facts that $T_{1}(\xi_{4,2}^{\mu},\mu)<1-2\sigma$ and $\lim_{s\to(s_{0}^{\mathcal{M}})^{-}} T_{1}(s,\mu)=+\infty$, we have the existence of one value $\xi_{4,3}^{\mu}$ such that
\begin{equation}\label{eq-xi4-1}
0 < \xi_{4,1}^{\mu}<s_{0}^{\tau}<\xi_{4,2}^{\mu}<\xi_{4,3}^{\mu}<s_{0}^{\mathcal{M}}\quad \text{and} \quad T_1(\xi_{4,3}^{\mu},\mu)=1-2\sigma.
\end{equation}
Next, as in the proof of $(ii.a)$, i.e., using the facts that $T_{1}\left((s_{1}^{\mathcal{M}},1],\mu\right) \supseteq [\mathcal{L}_{1},+\infty)$ and $\mathcal{L}_{1}<1-2\sigma$ we have the existence of one value $\xi_{4,4}^{\mu}$ such that
\begin{equation}\label{eq-xi4-2}
s_{1}^{\mathcal{M}} < \xi_{4,4}^{\mu} < 1 \quad \text{and} \quad T_1(\xi_{4,4}^{\mu},\mu)=1-2\sigma.
\end{equation}
As a consequence, for $\mu>\mu^{*}_{4}(\lambda)$, problem~\eqref{eq-main} has four solutions $u$ such that $u(0)=s\in\{\xi_{4,1}^{\mu},\xi_{4,2}^{\mu},\xi_{4,3}^{\mu},\xi_{4,4}^{\mu}\}$. We notice that the values $\xi_{4,3}^{\mu},\xi_{4,4}^{\mu}$ can be taken to coincide with $\xi_{2,1}^{\mu},\xi_{2,2}^{\mu}$ defined in the proof of $(ii.b)$, respectively. The proof of~$(ii.c)$ is thus complete.
	
\smallskip
\noindent
$(ii.d)$ Set
\begin{equation}\label{eq-mu8}
\mu_{8}^{*}(\lambda):=\max \Bigl{\{}\mu_{2}^{*}(\lambda),\, \max \bigl{\{} \mu>0 \colon \ell^{0}_{3}(\mu) = 1-2\sigma \bigr{\}} \Bigr{\}},
\end{equation}
which is well-defined, finite, and satisfies  $\mu_{4}^{*}(\lambda)\leq\mu_{8}^{*}(\lambda)$ since $\ell^{0}_{1,2}(\mu)<\ell^{0}_{3}(\mu)$.  Observe that the finiteness is a consequence of
\begin{equation}
\lim_{\mu\to+\infty}\ell_{3}^{0}(\mu)=0,
\end{equation}
which can be proved by reasoning as for~\eqref{eq:limit_l120}.

For each $\mu>\mu^{*}_8(\lambda)$, we can infer that 
\begin{align*}
&\min_{s\in(0,s_0^\mathcal{M})}T_3(s,\mu) \leq \ell_{3}^{0}(\mu) <1-2\sigma, 
\\
&\min_{s\in (s_{1}^{\mathcal{M}},s_{1}^{\tau}]} T_{1}(s,\mu) \leq \min_{s\in (s_{1}^{\mathcal{M}},s_{1}^{\tau}]} T_{2}(s,\mu) \leq \ell_{1,2}^{1}(\mu)=\ell_{3}^{0}(\mu) < 1-2\sigma.
\end{align*}
Recalling that
\begin{equation*}
\lim_{s\to 0^{+}}T_3(s,\mu)=\lim_{s\to (s_0^\mathcal{M})^{-}}T_3(s,\mu)=
\lim_{s\to (s_{1}^{\mathcal{M}})^{+}}T_1(s,\mu)=\lim_{s\to (s_{1}^{\mathcal{M}})^{+}} T_2(s,\mu)=+\infty,
\end{equation*}
from the continuity of $T_3(\cdot,\mu)$ in $(0,s_0^\mathcal{M})$, of $T_{1}(\cdot,\mu)$ in $(s_{1}^{\mathcal{M}},s_{1}^{\tau}]$, and of $T_{2}(\cdot,\mu)$ in $(s_{1}^{\mathcal{M}},s_{1}^{\tau}]$, we immediately obtain that there exist four values $\xi_{8,2}^{\mu},\xi_{8,3}^{\mu},\xi_{8,6}^{\mu},\xi_{8,7}^{\mu}$ such that
\begin{align*}
&0<\xi_{8,2}^{\mu}<s_{0}^{\tau}<\xi_{8,3}^{\mu}<s_{0}^{\mathcal{M}} \quad \text{and} 
\quad T_3(\xi_{8,2}^{\mu},\mu)=T_3(\xi_{8,3}^{\mu},\mu)=1-2\sigma,
\\
&s_{1}^{\mathcal{M}}<\xi_{8,6}^{\mu}<\xi_{8,7}^{\mu}<s_{1}^{\tau} \quad \text{and} 
\quad T_1(\xi_{8,6}^{\mu},\mu)=T_2(\xi_{8,7}^{\mu},\mu)=1-2\sigma.
\end{align*}
Moreover, since
\begin{equation*}
\begin{aligned}
&\quad T_{2}(\xi_{8,2}^{\mu},\mu)<1-2\sigma, \qquad T_{1}(\xi_{8,3}^{\mu},\mu)<T_{2}(\xi_{8,3}^{\mu},\mu)<1-2\sigma<T_{3}(\xi_{8,7}^{\mu},\mu) \\
&\lim_{s\to 0^{+}}T_2(s,\mu)=\lim_{s\to (s_{0}^{\tau})^{-}}T_1(s,\mu)=+\infty, \quad \mathcal{L}_{1}(\mu)<\mathcal{L}_{0}(\mu)<1-2\sigma,
\end{aligned}
\end{equation*}
arguing as in the proof of~$(ii.c)$, we have, for all $\mu>\mu^*_8$, the existence of four values $\xi_{8,1}^{\mu},\xi_{8,4}^{\mu},\xi_{8,5}^{\mu},\xi_{8,8}^{\mu}$ satisfying
\begin{equation*}
\begin{aligned}
&0<\xi_{8,1}^{\mu}<\xi_{8,2}^{\mu}<s_{0}^{\tau}<\xi_{8,3}^{\mu}<\xi_{8,4}^{\mu}<\xi_{8,5}^{\mu}<s_{0}^{\mathcal{M}}\leq s_{1}^{\mathcal{M}}<\xi_{8,6}^{\mu}<\xi_{8,7}^{\mu}<\xi_{8,8}^{\mu}<1,
\\
&T_2(\xi_{8,1},\mu)=T_2(\xi_{8,4}^{\mu},\mu)=T_{1}(\xi_{8,5}^{\mu},\mu)=T_{i}(\xi_{8,8}^{\mu},\mu)=1-2\sigma,
\end{aligned}
\end{equation*}
where $i=3$, provided $\xi_{8,8}^{\mu}\in(s_{1}^{\mathcal{M}},s_{1}^{\tau}]$, and $i=1$, provided $\xi_{8,8}^{\mu}\in(s_{1}^{\tau},1]$. In particular, they can be taken to coincide with $\xi_{4,1}^{\mu},\xi_{4,2}^{\mu},\xi_{4,3}^{\mu},\xi_{4,4}^{\mu}$ introduced in the proof of $(ii.c)$, respectively. As a consequence, for $\mu>\mu^{*}_{8}(\lambda)$, problem~\eqref{eq-main} has eight solutions $u$ such that $u(0)=s\in\{\xi_{8,1}^{\mu},\xi_{8,2}^{\mu},\xi_{8,3}^{\mu},\xi_{8,4}^{\mu},\xi_{8,5}^{\mu},\xi_{8,6}^{\mu},\xi_{8,7}^{\mu},\xi_{8,8}^{\mu}\}$. The proof of~$(ii.d)$ is thus complete.
\end{proof}

\begin{remark}[Conjecture]\label{rem-conjecture}
We conjecture that the values $\mu^{*}_1(\lambda)$ and $\mu^{**}_2(\lambda)$ in Theorems~\ref{th-intro1} and \ref{th-intro2} coincide with $\mu^{*}_0(\lambda)$ and $\mu^{**}_0(\lambda)$, respectively. To obtain such a result, it would be sufficient, for example, to prove some properties of monotonicity with respect to $\mu$ for the connection times $T_i$, with $i=1,2,3$, introduced and analyzed in Section~\ref{section-4}.
\hfill$\lhd$
\end{remark}

\section{Bifurcation diagrams in the $\left(\mu,u(0)\right)$-plane}\label{section-6}

In this section, we provide a description of all possible bifurcation diagrams concerning the number of solutions of problem~\eqref{eq-main}. We thus plot, for every fixed $\lambda\in(0,+\infty)$, the initial data $u(0)=s$ in~\eqref{eq-initial1}, that identifies a solution of~\eqref{eq-main}, against the main bifurcation parameter $\mu\in(0,+\infty)$. 
To this purpose, we exploit the behavior of the connection times $T_i$ defined in Section~\ref{section-4}, and so we divide the discussion into two parts corresponding to the cases $\lambda\in[\lambda^*,+\infty)$ and $\lambda\in(0,\lambda^*)$. Overall, we will show that three different topological diagrams arise.

\subsection{The case $\lambda\in[\lambda^*,+\infty)$}\label{section-6.1}

Let $\lambda\in[\lambda^*,+\infty)$ be fixed. We aim to show that the minimal bifurcation diagram to problem~\eqref{eq-main} with respect to $\mu\in(0,+\infty)$ behaves as in Figure~\ref{fig-bifurcation-1}. We therefore introduce the continuous functions
\begin{equation}\label{fun-f}
\begin{aligned}
&f_{1,2}^l(s,\mu) :=
\begin{cases}
\, T_{1}(s,\mu), &\text{if $s\in[0,s_{0}^{\tau}(\mu)]$,} \\
\, T_{2}(s,\mu), &\text{if $s\in[s_{0}^{\tau}(\mu),s_{0}^{\mathcal{M}}(\mu))$,} 
\end{cases}
\\
&f_{2,1}^l(s,\mu) :=
\begin{cases}
\, T_{2}(s,\mu), &\text{if $s\in(0,s_{0}^{\tau}(\mu)]$,} \\
\, T_{1}(s,\mu), &\text{if $s\in[s_{0}^{\tau}(\mu),s_{0}^{\mathcal{M}}(\mu))$,} 
\end{cases}
\\
&f_{3,3}^l(s,\mu) := T_{3}(s,\mu), \quad \text{if $s\in(0,s_{0}^{\mathcal{M}}(\mu))$,} 
\\
&f_{1}^r(s,\mu) :=T_{1}(s,\mu), \quad\text{if $s\in(s_{1}^{\mathcal{M}}(\mu),s_{1}^{\tau}(\mu))$,} 
\\
&f_{2}^r(s,\mu) :=T_{2}(s,\mu), \quad\text{if $s\in(s_{1}^{\mathcal{M}}(\mu),s_{1}^{\tau}(\mu)]$,} 
\\
&f_{3,1}^r(s,\mu) :=
\begin{cases}
\, T_{3}(s,\mu), &\text{if $s\in(s_{1}^{\mathcal{M}}(\mu),s_{1}^{\tau}(\mu)]$,} \\
\, T_{1}(s,\mu), &\text{if $s\in(s_{1}^{\tau}(\mu),1]$,}
\end{cases}
\end{aligned}
\end{equation}
(see Figure~\ref{fig-06}). Throughout this section, we also assume that:
\begin{quote}
\textit{the functions defined in~\eqref{fun-f} have at most one local extremum point  in the $s$-variable.}
\end{quote}
This assumption is suggested by numerical computations of the graphs of the connection times $T_i,$ $i=1,2,3$. In particular, it implies that the functions $f^l_{1,2}$, $f^r_{2}$, and $f^r_{3,1}$ are strictly monotone in $s$. As we shall see, they also lead to a bifurcation diagram made of eight unbounded continuous branches $b_i(\mu)$, with $i=1,\dots,8$. If these additional assumptions for the functions in~\eqref{fun-f} are dropped, there may be more turning points in the bifurcation diagrams compared to the minimal configuration we are going to discuss. 

By exploiting the symmetry properties in~\eqref{eq-symmetry}, under this additional hypothesis we also deduce that $s_{0}^{\tau}(\mu)$ is the critical point of $f^l_{2,1}(\cdot,\mu)$ for all $\mu>0$. By denoting the critical points of $f^l_{3,3}(\cdot,\mu)$ and $f_{1}^r(\cdot,\mu)$ through $s_{0}^{*}(\mu)$ and $s_1^*(\mu)$, respectively, thanks to~\eqref{eq-symmetry}, we obtain that $0<s_{0}^{*}(\mu)\leq s_{0}^{\tau}(\mu)$ (by the strict monotonicity of $f^r_{2}(\cdot,\mu)$), $s_{1}^{\mathcal{M}}(\mu)< s_{1}^{*}(\mu) \leq s_{1}^{\tau}(\mu)$, and $f^l_{3,3}(s_0^*(\mu),\mu)=f_{1}^r(s_1^*(\mu),\mu)$. Next, let us take $\mu^{*}_1=\mu^{*}_1(\lambda)$, $\mu^*_2=\mu^*_2(\lambda)$, and $\mu^*_4=\mu^*_4(\lambda)$ as the values defined in~\eqref{eq-mu1},~\eqref{eq-mu2}, and~\eqref{eq-mu4}, respectively.
We also define $\mu^*_8=\mu^*_8(\lambda)$ similarly as in~\eqref{eq-mu8} where $\ell_{3}^{0}(\mu)$ is replaced by $f^l_{3,3}(s_{0}^{*}(\mu),\mu)$. 

The additional assumptions considered here allow us to improve the results in Theorem~\ref{th-intro2} by proving in a similar way that $0<\mu^*_1<\mu^*_2<\mu^*_4<\mu^*_8$ and 
\begin{itemize}
\item for every $\mu>\mu^{*}_1$, there is exactly one value $\xi_{1}^\mu\in(0,1)$ such that $f^r_{3,1}(\xi_{1}^\mu,\mu)=1-2\sigma$;
\item for every $\mu>\mu^{*}_2$, there are exactly two values $\xi_{2,i}^\mu\in(0,1)$, $i=1,2$, ordered as in~\eqref{2sol-T1}, and such that $f_{1,2}^l(\xi_{2,1}^\mu,\mu)=f^r_{3,1}(\xi_{2,2}^\mu,\mu)=1-2\sigma$ with $\xi_{2,2}^\mu=\xi_{1}^\mu$;
\item for every $\mu>\mu^{*}_4$, there are exactly four values $\xi_{4,i}^\mu\in(0,1)$, $i=1,\dots,4$, ordered as in~\eqref{eq-xi4-1}-\eqref{eq-xi4-2}, and such that $f^l_{2,1}(\xi_{4,1}^\mu,\mu)=f_{1,2}^l(\xi_{4,2}^\mu,\mu)=f^l_{2,1}(\xi_{4,3}^\mu,\mu)=f^r_{3,1}(\xi_{4,4}^\mu,\mu)=1-2\sigma$ with $\xi_{4,2}^\mu=\xi_{2,1}^\mu$, and $\xi_{4,4}^\mu=\xi_{2,2}^\mu$;
\item for every $\mu>\mu^{*}_8$, there are exactly eight values $\xi_{8,i}^\mu\in(0,1)$, $i=1,\dots,8$, with
$0<\xi_{8,1}^{\mu}<\xi_{8,2}^{\mu}<s_{0}^{*}(\mu)<\xi_{8,3}^{\mu}<\xi_{8,4}^{\mu}<\xi_{8,5}^{\mu}<s_{0}^{\mathcal{M}}(\mu)\leq s_{1}^{\mathcal{M}}(\mu)<\xi_{8,6}^{\mu}<\xi_{8,7}^{\mu}<\xi_{8,8}^{\mu}<1,$ such that
$f^l_{2,1}(\xi_{8,1}^\mu,\mu)=f^l_{3,3}(\xi_{8,2}^\mu,\mu)=f^l_{3,3}(\xi_{8,3}^\mu,\mu)=f_{1,2}^l(\xi_{8,4}^\mu,\mu)=f^l_{2,1}(\xi_{8,5}^\mu,\mu)=f_{1}^{r}(\xi_{8,6}^\mu,\mu)=f_{2}^{r}(\xi_{8,7}^\mu,\mu)=f^r_{3,1}(\xi_{8,8}^\mu,\mu)=1-2\sigma$, and $\xi_{8,1}^\mu=\xi_{4,1}^\mu$, $\xi_{8,4}^\mu=\xi_{4,2}^\mu$, $\xi_{8,5}^\mu=\xi_{4,3}^\mu$, $\xi_{8,8}^\mu=\xi_{4,4}^\mu$.
\end{itemize}

Taking into account the solutions $\xi^\mu_{i,j}\in(0,1)$, with $i\in\{1,2,4,8\}$ and $j\in\{1,\dots,i\}$, we can univocally determine the solutions of~\eqref{eq-main} as $\mu$ varies. Hence, we define the branches $b_i(\mu)$, for $i=1,\dots,8$, of the bifurcation diagram through the following functions:
\begin{align*}
&b_{1}(\mu) :=
\begin{cases}
\, s_0^\tau(\mu^*_4), &\text{if $\mu=\mu^*_4$,} \\
\, \xi_{4,1}^\mu, &\text{if $\mu\in(\mu^*_4,+\infty)$,} \\
\end{cases}
&&b_{2}(\mu) :=
\begin{cases}
\, s_0^*(\mu^*_8), &\text{if $\mu=\mu^*_8$,} \\
\, \xi_{8,2}^\mu, &\text{if $\mu\in(\mu^*_8,+\infty)$,} \\
\end{cases}
\\
&b_{3}(\mu) :=
\begin{cases}
\, s_0^*(\mu^*_8), &\text{if $\mu=\mu^*_8$,} \\
\, \xi_{8,3}^\mu, &\text{if $\mu\in(\mu^*_8,+\infty)$,} \\
\end{cases}
&&b_{4}(\mu) :=
\begin{cases}
\, 0, &\text{if $\mu=\mu^*_2$,} \\
\, \xi_{2,1}^\mu, &\text{if $\mu\in(\mu^{*}_2,+\infty)$,} 
\end{cases}
\\
&b_{5}(\mu) :=
\begin{cases}
\, s_0^\tau(\mu^*_4), &\text{if $\mu=\mu^*_4$,} \\
\, \xi_{4,3}^\mu, &\text{if $\mu\in(\mu^*_4,+\infty)$,} \\
\end{cases}
&&b_{6}(\mu) :=
\begin{cases}
\, s_1^*(\mu^*_8), &\text{if $\mu=\mu^*_8$,} \\
\, \xi_{8,6}^\mu, &\text{if $\mu\in(\mu^*_8,+\infty)$,} \\
\end{cases}
\\
&b_{7}(\mu) :=
\begin{cases}
\, s_1^*(\mu^*_8), &\text{if $\mu=\mu^*_8$,} \\
\, \xi_{8,7}^\mu, &\text{if $\mu\in(\mu^*_8,+\infty)$,} \\
\end{cases}
&&b_{8}(\mu) :=
\begin{cases}
\, 1, &\text{if $\mu=\mu^{*}_1$,} \\
\, \xi_{1}^\mu, &\text{if $\mu\in(\mu^{*}_1,+\infty)$.}
\end{cases}
\end{align*}

Thanks to the continuity of the functions in~\eqref{fun-f}, we have that
\begin{align*}
&\lim_{\mu\to(\mu^*_1)^{+}}\xi_{1}^{\mu}=1, \qquad\lim_{\mu\to(\mu^*_2)^{+}}\xi_{2,1}^{\mu}=0,\\
&\lim_{\mu\to\mu^*_4}\xi_{2,1}^{\mu}=\lim_{\mu\to(\mu^*_4)^{+}}\xi_{4,1}^{\mu}=\lim_{\mu\to(\mu^*_4)^{+}}\xi_{4,3}^{\mu}=s_0^\tau(\mu^*_4),\\
&\lim_{\mu\to(\mu^*_8)^{+}}\xi_{8,2}^{\mu}=\lim_{\mu\to(\mu^*_8)^{+}}\xi_{8,3}^{\mu}=s_0^*(\mu^*_8),\\
&\lim_{\mu\to(\mu^*_8)^{+}}\xi_{8,6}^{\mu}=\lim_{\mu\to(\mu^*_8)^{+}}\xi_{8,7}^{\mu}=s_1^*(\mu^*_8),
\end{align*}
and so the branches $b_i(\mu)$, for $i=1,\dots,8$, are continuous functions in their domains. 

We now investigate the behavior of $b_i(\mu)$ as $\mu\to+\infty.$ To this purpose, by recalling~\eqref{eq-lim-s} and~\eqref{eq:limit_T_mu-C}, we deduce that $\lim_{\mu\to+\infty}\xi_{8,i}^\mu\in\{0, s_0, s_1\}$ for $i=1,\dots,8$. 
Moreover, from the order of the points $\xi_{8,i}^{\mu}$, we obtain
\begin{align*}
&\lim_{\mu\to+\infty}\xi_{8,1}^{\mu}=\lim_{\mu\to+\infty}\xi_{8,2}^{\mu}=0,\\
&\lim_{\mu\to+\infty}\xi_{8,3}^{\mu}=\lim_{\mu\to+\infty}\xi_{8,4}^{\mu}=\lim_{\mu\to+\infty}\xi_{8,5}^{\mu}=s_0, \\
&\lim_{\mu\to+\infty}\xi_{8,6}^{\mu}=\lim_{\mu\to+\infty}\xi_{8,7}^{\mu}=\lim_{\mu\to+\infty}\xi_{8,8}^{\mu}=s_1.
\end{align*}
Concerning the above limits, we clarify that for $\xi_{8,3}^{\mu}\in[s_{0}^{*}(\mu),s_0^{\mathcal{M}}(\mu))$ since the other ones follow straightforwardly. If by contradiction we suppose that $\liminf_{\mu\to+\infty}\xi_{8,3}^\mu=:\hat{s}<s_0$, then, thanks to the monotonicity assumptions on $f_{3,3}$, it would follow that $T_3(s,\mu)\geq1-2\sigma$ for all sufficiently large $\mu$ and $s\in[\hat{s},s_0^{\mathcal{M}}(\mu))$; indeed, otherwise, $f_{3,3}(\cdot,\mu)$ would have another local minimum point for a certain $s>\xi_{8,3}^\mu>s_0^*(\mu)$. But this is impossible since $T_3(\cdot,\mu)\to 0$ as $\mu\to+\infty$, locally uniformly in $(0,s_0)$.
As a consequence, it follows that: $b_1(\mu),b_2(\mu)\to 0$; $b_3(\mu),b_4(\mu),b_5(\mu)\to s_0$; and $b_6(\mu),b_7(\mu),b_8(\mu)\to s_1$, as $\mu\to+\infty$.

\begin{figure}[htb]
\centering
\begin{tikzpicture}
\begin{axis}[
  tick label style={font=\scriptsize},
  axis y line=middle, 
  axis x line=middle,
  xtick={0,30.727,172.780,400,1165},
  ytick={0,0.22,0.79,1},
  xticklabels={0,\scriptsize{$\mu_{1}^{*}$},\scriptsize{$\mu_{2}^{*}$},\scriptsize{$\mu_{4}^{*}$},\scriptsize{$\mu_{8}^{*}$}},
  yticklabels={0,\scriptsize{$s_{0}$},\scriptsize{$s_{1}$},1},
  xlabel={\small $\mu$},
  ylabel={\small $u(0)$},
every axis x label/.style={
    at={(ticklabel* cs:1.0)},
    anchor=west,
},
every axis y label/.style={
    at={(ticklabel* cs:1.0)},
    anchor=south,
},
  width=10.5cm,
  height=5cm,
  xmin=-60,
  xmax=2000,
  ymin=-0.05,
  ymax=1.1]
\addplot [color=black,line width=0.8pt,smooth] coordinates {(30.7300,1.0000) (66.5800,0.9901) (142.0000,0.9777) (255.6000,0.9634) (405.9000,0.9477) (591.5000,0.9313) (811.1000,0.9146) (1063.0000,0.8984) (1346.0000,0.8831) (1659.0000,0.8695) (2000.0000,0.8580)};
\addplot [color=black,line width=0.8pt,smooth] coordinates {(172.8000,0.0000) (217.3000,0.0136) (266.4000,0.0266) (320.1000,0.0392) (378.6000,0.0513) (441.9000,0.0628) (510.0000,0.0739) (583.0000,0.0844) (661.0000,0.0944) (743.9000,0.1038) (831.8000,0.1128) (924.9000,0.1211) (1023.0000,0.1289) (1126.0000,0.1361) (1235.0000,0.1427) (1349.0000,0.1488) (1468.0000,0.1542) (1593.0000,0.1591) (1723.0000,0.1634) (1859.0000,0.1670) (2000,0.17)};
\addplot [color=darktangerine,line width=0.8pt,smooth] coordinates {(2000.0000,0.0260) (1428.0000,0.0273) (992.5000,0.0280) (683.4000,0.0306) (491.1000,0.0378) (406.0000,0.0520) (424.3000,0.0748) (566.4000,0.1037) (858.4000,0.1352) (1327.0000,0.1658) (1998.0000,0.1920)};
\addplot [color=darkpastelgreen,line width=0.8pt,smooth] coordinates {(2000.0000,0.8080) (1698.0000,0.8234) (1466.0000,0.8396) (1304.0000,0.8547) (1210.0000,0.8668) (1185.0000,0.8740) (1226.0000,0.8751) (1331.0000,0.8709) (1497.0000,0.8631) (1721.0000,0.8529) (2000.0000,0.8420)};
\addplot [color=darkpastelgreen,line width=0.8pt,smooth] coordinates {(2000.0000,0.0500) (1654.0000,0.0535) (1414.0000,0.0577) (1274.0000,0.0630) (1193.0000,0.0722) (1166.0000,0.0860) (1179.0000,0.1001) (1219.0000,0.1103) (1306.0000,0.1178) (1559.0000,0.1286) (2000.0000,0.1440)};
\addplot [color=black,dashed] coordinates {(0,1) (10000,1)};
\addplot [color=black,dashed] coordinates {(0,0.79) (10000,0.79)};
\addplot [color=black,dashed] coordinates {(0,0.22) (10000,0.22)};
\end{axis}
\end{tikzpicture} 
\caption{For $\lambda\in[\lambda^*,+\infty)$ fixed, a minimal qualitative bifurcation diagram for~\eqref{eq-main} with $\mu$ as bifurcation parameter. The topological configuration involves: two unbounded branches bifurcating from $0$ and $1$ (black) and three unbounded branches (yellow, green) originating from a supercritical pitchfork bifurcation and two supercritical turning points, respectively.
} 
\label{fig-bifurcation-1}
\end{figure}
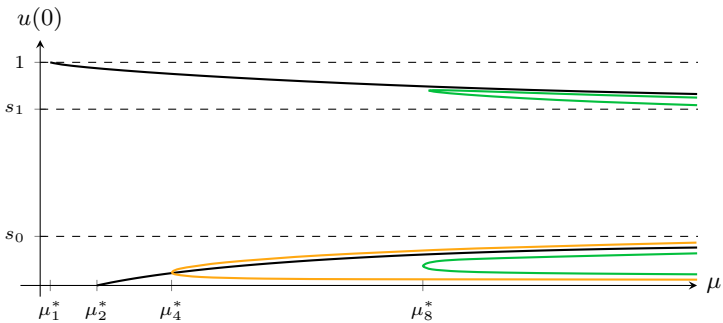

At last, we notice that the branches $b_4$ and $b_8$ bifurcate supercritically from~$0$ at $\mu=\mu^*_2$ and from $1$ at $\mu=\mu^{*}_1$, respectively. Instead, the branches $b_2$ and $b_3$ merge together at $\mu^*_8$, so that $\left(\mu^*_8,s_0^*(\mu^*_8)\right)$ is a supercritical turning point by the definition of $\mu^*_8$. Analogously, the branches $b_6$ and $b_7$ merge together at $\mu^*_8$, and so $\left(\mu^*_8,s_1^*(\mu^*_8)\right)$ is a supercritical turning point too. Concerning $b_1$, $b_4$, and $b_5$, we observe that these branches merge together at $\mu^*_4$ and one can easily show that $\left(\mu_4^*, s_0^\tau(\mu^*_4)\right)$ is a supercritical pitchfork bifurcation point since $s_0^\tau(\mu^*_4)$ is the global minimum of the function $f_{2,1}^l(\cdot,\mu^*_4)$.

\subsection{The case $\lambda \in (0,\lambda^{*})$}

Let $\lambda \in (0,\lambda^{*})$ be fixed. Here, we aim to show that any minimal bifurcation diagram of problem~\eqref{eq-main} with respect to $\mu\in(0,+\infty)$ has at least a bounded branch that joins $0$ to $1$ as in Figure~\ref{fig-bifurcation-2}. Let us define the time needed to move, along the level lines associated with~\eqref{eq-energy-mu}, from a point $(u_s(\sigma),v_s(\sigma))$ of the curve $\Gamma_0$ to its symmetric with respect to the axis $\{v=0\}$:
\begin{equation*}
T_{sym}(s,\mu)= 2 \int_{m(u_{s}(\sigma))}^{u_{s}(\sigma)} \dfrac{\mathrm{d}u}{\sqrt{(v_{s}(\sigma))^2+2\mu(G(u)-G(u_{s}(\sigma)))}}.
\end{equation*}
We notice that this map is defined, respectively, for $s\in(0,s_0^{\mathcal{M}}(\mu))\cup(s_1^{\mathcal{M}}(\mu),1)$, if $\mu\in(0,\tilde{\mu}(\lambda))$; for $s\in(0,1)\setminus\{s^{\mathcal{M}}(\mu)\}$, if $\mu=\tilde{\mu}(\lambda)$; and for $s\in(0,1)$, if $\mu\in(\tilde{\mu}(\lambda),+\infty)$. As above, we extend by continuity $T_{sym}$ in $s=0$ and $s=1$. Furthermore, the map $T_{sym}(s,\mu)$ is continuous in $\mu$ for $s$ belonging to a compact set of its domain. 

\begin{figure}[htb]
\centering
\begin{tikzpicture}
\begin{axis}[
  tick label style={font=\scriptsize},
  axis y line=middle, 
  axis x line=middle,
  xtick={0, 15.4, 30.0200},
  ytick={0,1},
  xticklabels={0,\scriptsize{$\mu_{1}^{*}$},\scriptsize{$\mu_{2}^{*}$}},
  yticklabels={0,1},
  xlabel={\small $\mu$},
  ylabel={\small $u(0)$},
every axis x label/.style={
    at={(ticklabel* cs:1.0)},
    anchor=west,
},
every axis y label/.style={
    at={(ticklabel* cs:1.0)},
    anchor=south,
},
  width=3.5cm,
  height=5cm,
  xmin=-9,
  xmax=70,
  ymin=-0.05,
  ymax=1.1]
\addplot [color=black,line width=0.8pt, smooth] coordinates {(15.4000,0.9993) (15.9600,0.9907) (16.5700,0.9815) (17.2200,0.9716) (17.9200,0.9611) (18.6500,0.9500) (19.4300,0.9382) (20.2500,0.9258) (21.1100,0.9128) (22.0100,0.8991) (22.9500,0.8847) (23.9400,0.8698) (24.9700,0.8542) (26.0400,0.8379) (27.1100,0.8215) (28.1600,0.8054) (29.1600,0.7896) (30.1400,0.7741) (31.0800,0.7588) (31.9800,0.7438) (32.8600,0.7292) (33.7000,0.7148) (34.5000,0.7007) (35.2800,0.6869) (36.0100,0.6734) (36.7200,0.6602) (37.3900,0.6472) (38.0300,0.6346) (38.6300,0.6222) (39.2000,0.6102) (39.7400,0.5984) (40.2400,0.5869) (40.7100,0.5757) (41.1500,0.5648) (41.5500,0.5542) (41.9200,0.5439) (42.2600,0.5338) (42.5600,0.5241) (42.8300,0.5146) (43.0600,0.5055) (43.2700,0.4965) (43.4700,0.4876) (43.6500,0.4789) (43.8200,0.4702) (43.9700,0.4616) (44.1200,0.4532) (44.2400,0.4448) (44.3600,0.4365) (44.4600,0.4284) (44.5400,0.4203) (44.6100,0.4124) (44.6700,0.4045) (44.7200,0.3968) (44.7500,0.3891) (44.7600,0.3816) (44.7600,0.3741) (44.7500,0.3668) (44.7300,0.3595) (44.6900,0.3524) (44.6300,0.3453) (44.5700,0.3384) (44.4900,0.3316) (44.3900,0.3248) (44.2800,0.3182) (44.1600,0.3117) (44.0200,0.3052) (43.8700,0.2988) (43.7100,0.2922) (43.5300,0.2855) (43.3300,0.2787) (43.1200,0.2718) (42.8900,0.2647) (42.6500,0.2575) (42.4000,0.2501) (42.1200,0.2427) (41.8400,0.2351) (41.5400,0.2273) (41.2200,0.2195) (40.8900,0.2115) (40.5400,0.2034) (40.1800,0.1952) (39.8100,0.1868) (39.4200,0.1783) (39.0100,0.1697) (38.5900,0.1610) (38.1500,0.1521) (37.7000,0.1431) (37.2300,0.1340) (36.7500,0.1247) (36.2600,0.1153) (35.7500,0.1058) (35.2200,0.0962) (34.7000,0.0867) (34.1900,0.0775) (33.7100,0.0688) (33.2500,0.0604) (32.8100,0.0523) (32.3900,0.0447) (31.9900,0.0374) (31.6100,0.0305) (31.2500,0.0240) (30.9100,0.0178) (30.5900,0.0120) (30.2900,0.0066) (30.0200,0.0016)};
\addplot [color=black,dashed] coordinates {(0,1) (200,1)};
\end{axis}
\end{tikzpicture} 
\;
\begin{tikzpicture}
\begin{axis}[
  tick label style={font=\scriptsize},
  axis y line=middle, 
  axis x line=middle,
  xtick={0,22.2600,67.3600,170},
  ytick={0,1},
  xticklabels={0,\scriptsize{$\mu_{1}^{*}$},\scriptsize{$\mu_{2}^{*}$},\scriptsize{$\mu_{4}^{*}$}},
  yticklabels={0,1},
  xlabel={\small $\mu$},
  ylabel={\small $u(0)$},
every axis x label/.style={
    at={(ticklabel* cs:1.0)},
    anchor=west,
},
every axis y label/.style={
    at={(ticklabel* cs:1.0)},
    anchor=south,
},
  width=5.5cm,
  height=5cm,
  xmin=-20,
  xmax=280,
  ymin=-0.05,
  ymax=1.1]
\addplot [color=black,line width=0.8pt,smooth] coordinates {(22.2600,0.9999) (22.3600,0.9994) (22.4600,0.9989) (22.5600,0.9985) (22.6600,0.9980) (22.7600,0.9975) (22.8600,0.9970) (22.9600,0.9966) (23.0600,0.9961) (23.1600,0.9956) (23.2600,0.9952) (23.3600,0.9947) (23.4600,0.9942) (23.6600,0.9933) (23.8600,0.9924) (24.0600,0.9914) (24.2600,0.9905) (24.4600,0.9896) (24.6600,0.9887) (24.8600,0.9878) (25.0600,0.9869) (25.2600,0.9861) (25.4600,0.9852) (25.9600,0.9830) (26.4600,0.9808) (26.9600,0.9787) (27.4600,0.9766) (27.9600,0.9745) (28.4600,0.9725) (28.9600,0.9705) (29.4600,0.9685) (29.9600,0.9665) (30.4600,0.9646) (31.4600,0.9607) (32.4600,0.9570) (33.4600,0.9533) (34.4600,0.9497) (35.4600,0.9462) (36.4500,0.9427) (37.4500,0.9393) (38.4500,0.9360) (39.4500,0.9327) (40.4500,0.9295) (42.4500,0.9232) (44.4500,0.9171) (46.4500,0.9111) (48.4500,0.9053) (50.4500,0.8997) (52.4500,0.8942) (54.4500,0.8889) (56.4500,0.8836) (58.4500,0.8785) (60.4500,0.8735) (64.4500,0.8637) (68.4500,0.8542) (72.4500,0.8451) (76.4500,0.8362) (80.4500,0.8276) (84.4500,0.8191) (88.4500,0.8109) (92.4500,0.8028) (96.4500,0.7948) (100.5000,0.7870) (105.5000,0.7774) (110.5000,0.7679) (115.5000,0.7585) (120.5000,0.7492) (125.5000,0.7400) (130.5000,0.7308) (135.5000,0.7217) (140.5000,0.7125) (145.5000,0.7033) (150.5000,0.6940) (155.5000,0.6847) (160.5000,0.6752) (165.5000,0.6656) (170.5000,0.6557) (175.5000,0.6456) (180.5000,0.6351) (185.5000,0.6242) (190.5000,0.6128) (195.5000,0.6007) (200.5000,0.5876) (205.5000,0.5732) (210.5000,0.5568) (215.5000,0.5369) (218.0000,0.5244) (219.2000,0.5170) (220.5000,0.5083) (221.7000,0.4969) (222.3000,0.4889) (222.6000,0.4833) (222.8000,0.4794) (222.9000,0.4768) (222.9000,0.4751) (223.0000,0.4727) (223.0000,0.4705) (223.0000,0.4699) (223.0000,0.4693) (223.0000,0.4692) (223.0000,0.4691) (223.0000,0.4689) (223.0000,0.4687) (223.0000,0.4686) (223.0000,0.4685) (223.0000,0.4685) (223.0000,0.4684) (223.0000,0.4684) (223.0000,0.4683) (223.0000,0.4683) (223.0000,0.4682) (223.0000,0.4682) (223.0000,0.4681) (223.0000,0.4678) (223.0000,0.4676) (223.0000,0.4673) (223.0000,0.4672) (223.0000,0.4670) (223.0000,0.4668) (223.0000,0.4667) (223.0000,0.4666) (223.0000,0.4664) (223.0000,0.4663) (223.0000,0.4653) (223.0000,0.4649) (222.9000,0.4645) (222.9000,0.4642) (222.9000,0.4639) (222.9000,0.4636) (222.9000,0.4633) (222.9000,0.4630) (222.9000,0.4628) (222.9000,0.4625) (222.8000,0.4587) (222.8000,0.4573) (222.7000,0.4560) (222.7000,0.4549) (222.6000,0.4538) (222.6000,0.4528) (222.5000,0.4518) (222.5000,0.4509) (222.4000,0.4501) (222.4000,0.4493) (221.9000,0.4424) (221.6000,0.4395) (221.4000,0.4369) (221.1000,0.4345) (220.9000,0.4323) (220.6000,0.4301) (220.4000,0.4281) (220.1000,0.4262) (219.9000,0.4243) (219.6000,0.4225) (219.1000,0.4191) (218.6000,0.4159) (218.1000,0.4129) (217.6000,0.4100) (217.1000,0.4072) (216.6000,0.4046) (216.1000,0.4020) (215.6000,0.3995) (215.1000,0.3971) (214.6000,0.3948) (214.1000,0.3925) (213.6000,0.3903) (213.1000,0.3881) (212.6000,0.3860) (212.1000,0.3840) (211.6000,0.3819) (211.1000,0.3799) (210.6000,0.3780) (210.1000,0.3760) (209.6000,0.3741) (209.1000,0.3723) (208.6000,0.3704) (208.1000,0.3686) (207.6000,0.3669) (207.1000,0.3651) (206.6000,0.3634) (206.1000,0.3616) (205.6000,0.3599) (205.1000,0.3583) (204.6000,0.3566) (204.1000,0.3550) (203.6000,0.3533) (203.1000,0.3517) (202.6000,0.3501) (202.1000,0.3486) (201.6000,0.3470) (201.1000,0.3454) (200.6000,0.3439) (200.1000,0.3424) (199.6000,0.3409) (198.6000,0.3379) (197.6000,0.3349) (196.6000,0.3320) (195.6000,0.3292) (194.6000,0.3263) (193.6000,0.3235) (192.6000,0.3208) (191.6000,0.3180) (190.6000,0.3153) (189.6000,0.3127) (187.6000,0.3074) (185.6000,0.3021) (183.6000,0.2970) (181.6000,0.2920) (179.6000,0.2870) (177.6000,0.2820) (175.6000,0.2771) (173.6000,0.2723) (171.6000,0.2675) (169.6000,0.2627) (167.6000,0.2579) (165.6000,0.2532) (163.6000,0.2485) (161.6000,0.2438) (159.6000,0.2392) (157.6000,0.2345) (155.6000,0.2299) (153.6000,0.2252) (151.6000,0.2206) (149.6000,0.2160) (145.6000,0.2067) (141.6000,0.1975) (137.6000,0.1882) (133.6000,0.1788) (129.6000,0.1694) (125.6000,0.1599) (121.6000,0.1503) (117.6000,0.1406) (113.6000,0.1307) (109.6000,0.1207) (105.6000,0.1105) (101.6000,0.1002) (97.6100,0.0896) (93.6100,0.0788) (89.6100,0.0677) (85.6100,0.0564) (83.6100,0.0506) (81.6100,0.0447) (79.6100,0.0388) (77.6100,0.0327) (75.6100,0.0266) (74.6100,0.0235) (73.6100,0.0203) (72.6100,0.0172) (71.6100,0.0140) (70.6100,0.0108) (70.1100,0.0092) (69.6100,0.0075) (69.1100,0.0059) (68.6100,0.0043) (68.1100,0.0026) (67.8600,0.0018) (67.6100,0.0009) (67.4900,0.0005) (67.4200,0.0003) (67.3600,0.0001)};
\addplot [color=magenta,line width=0.8pt, smooth] coordinates {(200.9000,0.4413) (201.1000,0.4420) (201.3000,0.4426) (201.5000,0.4432) (201.7000,0.4438) (201.9000,0.4443) (202.1000,0.4448) (202.3000,0.4453) (202.5000,0.4457) (202.7000,0.4460) (202.9000,0.4464) (203.1000,0.4466) (203.3000,0.4469) (203.5000,0.4470) (203.7000,0.4472) (203.9000,0.4473) (204.1000,0.4473) (204.3000,0.4474) (204.5000,0.4473) (204.7000,0.4473) (204.9000,0.4472) (205.1000,0.4470) (205.3000,0.4468) (205.5000,0.4466) (205.7000,0.4463) (205.9000,0.4460) (206.1000,0.4457) (206.3000,0.4453) (206.5000,0.4449) (206.7000,0.4444) (206.9000,0.4440) (207.1000,0.4434) (207.3000,0.4429) (207.5000,0.4423) (207.7000,0.4416) (207.9000,0.4409) (208.1000,0.4402) (208.3000,0.4395) (208.5000,0.4387) (208.7000,0.4378) (208.9000,0.4369) (209.1000,0.4360) (209.3000,0.4350) (209.5000,0.4339) (209.7000,0.4328) (209.9000,0.4317) (210.1000,0.4304) (210.3000,0.4291) (210.5000,0.4278) (210.7000,0.4263) (210.9000,0.4247) (211.1000,0.4231) (211.3000,0.4213) (211.5000,0.4193) (211.7000,0.4172) (211.9000,0.4148) (212.1000,0.4122) (212.3000,0.4091) (212.5000,0.4053) (212.6000,0.4030) (212.7000,0.4003) (212.8000,0.3987) (212.8000,0.3967) (212.9000,0.3956) (212.9000,0.3942) (212.9000,0.3934) (212.9000,0.3924) (212.9000,0.3912) (212.9000,0.3904) (212.9000,0.3899) (212.9000,0.3893) (212.9000,0.3889) (212.9000,0.3886) (213.0000,0.3883) (213.0000,0.3881) (213.0000,0.3879) (213.0000,0.3877) (213.0000,0.3876) (213.0000,0.3875) (213.0000,0.3874) (213.0000,0.3874) (213.0000,0.3874) (213.0000,0.3874) (213.0000,0.3874) (213.0000,0.3873) (213.0000,0.3873) (213.0000,0.3873) (213.0000,0.3873) (213.0000,0.3873) (213.0000,0.3873) (213.0000,0.3873) (213.0000,0.3872) (213.0000,0.3871) (213.0000,0.3870) (213.0000,0.3870) (213.0000,0.3869) (213.0000,0.3868) (213.0000,0.3868) (213.0000,0.3867) (213.0000,0.3867) (213.0000,0.3866) (212.9000,0.3862) (212.9000,0.3861) (212.9000,0.3859) (212.9000,0.3858) (212.9000,0.3857) (212.9000,0.3855) (212.9000,0.3854) (212.9000,0.3853) (212.9000,0.3852) (212.9000,0.3851) (212.9000,0.3850) (212.9000,0.3848) (212.9000,0.3846) (212.9000,0.3845) (212.9000,0.3843) (212.9000,0.3842) (212.9000,0.3841) (212.9000,0.3839) (212.9000,0.3838) (212.9000,0.3837) (212.9000,0.3826) (212.9000,0.3822) (212.9000,0.3818) (212.9000,0.3814) (212.9000,0.3810) (212.9000,0.3807) (212.9000,0.3804) (212.9000,0.3801) (212.9000,0.3797) (212.9000,0.3795) (212.8000,0.3749) (212.7000,0.3731) (212.7000,0.3715) (212.6000,0.3700) (212.6000,0.3687) (212.5000,0.3674) (212.5000,0.3662) (212.4000,0.3650) (212.4000,0.3639) (212.3000,0.3628) (211.8000,0.3537) (211.6000,0.3497) (211.3000,0.3461) (211.1000,0.3427) (210.8000,0.3394) (210.6000,0.3363) (210.3000,0.3333) (210.1000,0.3303) (209.8000,0.3275) (209.6000,0.3247) (208.6000,0.3143) (208.1000,0.3093) (207.6000,0.3044) (207.1000,0.2996) (206.6000,0.2949) (206.1000,0.2903) (205.6000,0.2858) (205.1000,0.2813) (204.6000,0.2768) (204.1000,0.2725) (203.1000,0.2640) (202.1000,0.2560) (201.1000,0.2486) (200.1000,0.2420) (199.1000,0.2364) (198.1000,0.2316) (197.1000,0.2278) (196.1000,0.2246) (195.1000,0.2222) (194.1000,0.2202) (193.1000,0.2186) (192.1000,0.2175) (191.1000,0.2166) (190.1000,0.2160) (189.1000,0.2157) (188.1000,0.2155) (187.1000,0.2155) (186.1000,0.2157) (185.1000,0.2161) (184.1000,0.2167) (183.1000,0.2174) (182.1000,0.2182) (181.1000,0.2193) (180.1000,0.2205) (179.1000,0.2218) (178.1000,0.2234) (177.1000,0.2252) (176.1000,0.2273) (175.1000,0.2297) (174.1000,0.2325) (173.1000,0.2359) (172.1000,0.2402) (171.6000,0.2428) (171.1000,0.2460) (170.8000,0.2479) (170.6000,0.2502) (170.3000,0.2531) (170.2000,0.2550) (170.1000,0.2561) (170.1000,0.2575) (170.0000,0.2583) (170.0000,0.2593) (170.0000,0.2599) (170.0000,0.2607) (170.0000,0.2612) (170.0000,0.2617) (170.0000,0.2621) (170.0000,0.2624) (170.0000,0.2627) (170.0000,0.2629)};
\addplot [color=magenta,line width=0.8pt, smooth] coordinates {(200.9000,0.4413) (200.7000,0.4405) (200.5000,0.4398) (200.3000,0.4390) (200.1000,0.4381) (199.9000,0.4373) (199.7000,0.4364) (199.5000,0.4355) (199.3000,0.4345) (199.1000,0.4336) (198.9000,0.4326) (198.7000,0.4316) (198.5000,0.4306) (198.3000,0.4296) (198.1000,0.4285) (197.9000,0.4275) (197.7000,0.4264) (197.5000,0.4254) (197.3000,0.4243) (197.1000,0.4232) (196.9000,0.4222) (196.7000,0.4211) (196.5000,0.4200) (196.3000,0.4189) (196.1000,0.4178) (195.9000,0.4167) (195.7000,0.4156) (195.5000,0.4146) (195.3000,0.4135) (195.1000,0.4124) (194.9000,0.4113) (194.7000,0.4102) (194.5000,0.4092) (194.3000,0.4081) (194.1000,0.4070) (193.9000,0.4060) (193.7000,0.4049) (193.5000,0.4039) (193.3000,0.4028) (193.1000,0.4018) (192.9000,0.4007) (192.4000,0.3981) (191.9000,0.3955) (191.4000,0.3930) (190.9000,0.3905) (190.4000,0.3879) (189.9000,0.3854) (189.4000,0.3830) (188.9000,0.3805) (188.4000,0.3781) (187.9000,0.3757) (186.9000,0.3708) (185.9000,0.3661) (184.9000,0.3613) (183.9000,0.3566) (182.9000,0.3519) (181.9000,0.3471) (180.9000,0.3423) (179.9000,0.3375) (178.9000,0.3326) (177.9000,0.3276) (176.9000,0.3225) (175.9000,0.3171) (174.9000,0.3115) (173.9000,0.3056) (172.9000,0.2991) (171.9000,0.2918) (171.4000,0.2876) (170.9000,0.2828) (170.7000,0.2800) (170.4000,0.2768) (170.3000,0.2749) (170.2000,0.2726) (170.1000,0.2713) (170.1000,0.2697) (170.0000,0.2687) (170.0000,0.2676) (170.0000,0.2668) (170.0000,0.2659) (170.0000,0.2653) (170.0000,0.2649) (170.0000,0.2647) (170.0000,0.2643) (170.0000,0.2641) (170.0000,0.2640) (170.0000,0.2638) (170.0000,0.2636) (170.0000,0.2635) (170.0000,0.2635) (170.0000,0.2635) (170.0000,0.2635)};
\addplot [color=black,dashed] coordinates {(0,1) (500,1)};
\end{axis}
\end{tikzpicture} 
\;
\begin{tikzpicture}
\begin{axis}[
  tick label style={font=\scriptsize},
  axis y line=middle, 
  axis x line=middle,
  xtick={0,25.4100,97.37,224.7,425.7},
  ytick={0,1},
  xticklabels={0,\scriptsize{$\mu_{1}^{*}$},\scriptsize{$\mu_{2}^{*}$},\scriptsize{$\mu_{4}^{*}$},\scriptsize{$\mu_{8}^{*}$}},
  yticklabels={0,1},
  xlabel={\small $\mu$},
  ylabel={\small $u(0)$},
every axis x label/.style={
    at={(ticklabel* cs:1.0)},
    anchor=west,
},
every axis y label/.style={
    at={(ticklabel* cs:1.0)},
    anchor=south,
},
  width=5.5cm,
  height=5cm,
  xmin=-35,
  xmax=950,
  ymin=-0.05,
  ymax=1.1]
\addplot [color=black,line width=0.8pt,smooth] coordinates {(25.4100,0.9999) (25.5100,0.9997) (25.6100,0.9994) (25.7100,0.9992) (25.8100,0.9989) (25.9100,0.9987) (26.1100,0.9982) (26.3100,0.9976) (26.5100,0.9972) (26.7100,0.9967) (26.9100,0.9962) (27.1100,0.9957) (27.3100,0.9952) (27.5100,0.9947) (27.7100,0.9942) (27.9100,0.9937) (28.4100,0.9926) (28.9100,0.9914) (29.4100,0.9902) (29.9100,0.9891) (30.4100,0.9880) (30.9100,0.9869) (31.4100,0.9858) (31.9100,0.9847) (32.4100,0.9836) (32.9100,0.9826) (33.9100,0.9805) (34.9100,0.9785) (35.9100,0.9765) (36.9100,0.9746) (37.9100,0.9727) (38.9100,0.9709) (39.9100,0.9690) (40.9100,0.9673) (41.9100,0.9655) (42.9100,0.9638) (44.9100,0.9605) (46.9100,0.9572) (48.9100,0.9541) (50.9100,0.9511) (52.9100,0.9482) (54.9100,0.9453) (56.9100,0.9425) (58.9100,0.9398) (60.9100,0.9372) (62.9100,0.9347) (67.9100,0.9285) (72.9100,0.9227) (77.9100,0.9172) (82.9100,0.9119) (87.9100,0.9069) (92.9100,0.9021) (97.9100,0.8974) (102.9000,0.8930) (107.9000,0.8887) (112.9000,0.8845) (122.9000,0.8766) (132.9000,0.8691) (142.9000,0.8620) (152.9000,0.8553) (162.9000,0.8488) (172.9000,0.8426) (182.9000,0.8367) (192.9000,0.8310) (202.9000,0.8255) (212.9000,0.8201) (232.9000,0.8099) (252.9000,0.8002) (272.9000,0.7910) (292.9000,0.7821) (312.9000,0.7736) (332.9000,0.7654) (352.9000,0.7574) (372.9000,0.7496) (392.9000,0.7420) (412.9000,0.7345) (442.9000,0.7236) (472.9000,0.7128) (502.9000,0.7022) (532.9000,0.6917) (562.9000,0.6811) (592.9000,0.6705) (622.9000,0.6597) (652.9000,0.6486) (682.9000,0.6372) (712.9000,0.6252) (742.9000,0.6124) (772.9000,0.5983) (802.9000,0.5823) (832.9000,0.5624) (847.9000,0.5493) (855.4000,0.5410) (859.2000,0.5359) (862.9000,0.5296) (864.8000,0.5257) (866.7000,0.5207) (867.6000,0.5172) (868.1000,0.5148) (868.3000,0.5132) (868.5000,0.5108) (868.6000,0.5104) (868.6000,0.5098) (868.6000,0.5091) (868.6000,0.5084) (868.6000,0.5082) (868.6000,0.5082) (868.6000,0.5081) (868.6000,0.5080) (868.6000,0.5080) (868.6000,0.5080) (868.6000,0.5079) (868.6000,0.5079) (868.6000,0.5079) (868.6000,0.5079) (868.6000,0.5078) (868.6000,0.5078) (868.6000,0.5077) (868.6000,0.5077) (868.6000,0.5077) (868.6000,0.5077) (868.6000,0.5077) (868.6000,0.5077) (868.6000,0.5077) (868.6000,0.5076) (868.6000,0.5076) (868.6000,0.5076) (868.6000,0.5075) (868.6000,0.5074) (868.6000,0.5074) (868.6000,0.5073) (868.6000,0.5073) (868.6000,0.5073) (868.6000,0.5072) (868.6000,0.5072) (868.6000,0.5072) (868.6000,0.5071) (868.6000,0.5067) (868.6000,0.5065) (868.6000,0.5064) (868.6000,0.5062) (868.6000,0.5061) (868.6000,0.5060) (868.6000,0.5059) (868.6000,0.5058) (868.6000,0.5057) (868.6000,0.5056) (868.5000,0.5042) (868.4000,0.5036) (868.4000,0.5032) (868.3000,0.5027) (868.3000,0.5023) (868.2000,0.5020) (868.2000,0.5016) (868.1000,0.5013) (868.1000,0.5010) (868.0000,0.5007) (867.0000,0.4962) (866.5000,0.4945) (866.0000,0.4930) (865.5000,0.4916) (865.0000,0.4904) (864.5000,0.4892) (864.0000,0.4881) (863.5000,0.4870) (863.0000,0.4860) (862.5000,0.4851) (860.5000,0.4816) (858.5000,0.4785) (856.5000,0.4757) (854.5000,0.4731) (852.5000,0.4706) (850.5000,0.4683) (848.5000,0.4662) (846.5000,0.4641) (844.5000,0.4621) (842.5000,0.4602) (837.5000,0.4557) (832.5000,0.4515) (827.5000,0.4475) (822.5000,0.4438) (817.5000,0.4402) (812.5000,0.4368) (807.5000,0.4335) (802.5000,0.4303) (797.5000,0.4272) (792.5000,0.4242) (782.5000,0.4185) (772.5000,0.4129) (762.5000,0.4076) (752.5000,0.4024) (742.5000,0.3974) (732.5000,0.3925) (722.5000,0.3877) (712.5000,0.3830) (702.5000,0.3784) (692.5000,0.3739) (672.5000,0.3649) (652.5000,0.3561) (632.5000,0.3475) (612.5000,0.3389) (592.5000,0.3304) (572.5000,0.3218) (552.5000,0.3133) (532.5000,0.3047) (512.5000,0.2959) (492.5000,0.2871) (462.5000,0.2736) (432.5000,0.2596) (402.5000,0.2451) (372.5000,0.2299) (342.5000,0.2139) (312.5000,0.1968) (282.5000,0.1784) (252.5000,0.1583) (222.5000,0.1361) (192.5000,0.1112) (162.5000,0.0827) (147.5000,0.0667) (140.0000,0.0581) (132.5000,0.0492) (128.8000,0.0445) (125.0000,0.0397) (121.3000,0.0348) (117.5000,0.0298) (113.8000,0.0247) (111.9000,0.0220) (110.0000,0.0193) (108.2000,0.0166) (106.3000,0.0139) (104.4000,0.0111) (102.5000,0.0082) (101.6000,0.0068) (100.7000,0.0054) (99.7200,0.0039) (99.2500,0.0032) (98.7800,0.0024) (98.3100,0.0017) (97.8400,0.0009) (97.6100,0.0006) (97.3700,0.0002)};
\addplot [color=magenta,line width=0.8pt, smooth] coordinates {(330.9000,0.2377) (332.9000,0.2389) (334.9000,0.2400) (336.9000,0.2411) (338.9000,0.2421) (340.9000,0.2432) (342.9000,0.2443) (344.9000,0.2454) (346.9000,0.2464) (348.9000,0.2475) (350.9000,0.2486) (355.9000,0.2512) (360.9000,0.2537) (365.9000,0.2563) (370.9000,0.2588) (375.9000,0.2613) (380.9000,0.2638) (385.9000,0.2662) (390.9000,0.2686) (395.9000,0.2710) (400.9000,0.2734) (410.9000,0.2781) (420.9000,0.2828) (430.9000,0.2873) (440.9000,0.2919) (450.9000,0.2963) (460.9000,0.3008) (470.9000,0.3052) (480.9000,0.3095) (490.9000,0.3139) (500.9000,0.3182) (520.9000,0.3268) (540.9000,0.3353) (560.9000,0.3439) (580.9000,0.3525) (600.9000,0.3612) (620.9000,0.3701) (640.9000,0.3791) (660.9000,0.3885) (680.9000,0.3982) (700.9000,0.4084) (720.9000,0.4193) (740.9000,0.4312) (760.9000,0.4447) (780.9000,0.4612) (790.9000,0.4718) (800.9000,0.4863) (805.9000,0.4995) (806.6000,0.5027) (806.9000,0.5052) (807.0000,0.5073) (807.1000,0.5077) (807.1000,0.5083) (807.1000,0.5087) (807.1000,0.5091) (807.1000,0.5092) (807.1000,0.5093) (807.1000,0.5093) (807.1000,0.5094) (807.1000,0.5094) (807.1000,0.5094) (807.1000,0.5095) (807.1000,0.5095) (807.1000,0.5095) (807.1000,0.5095) (807.1000,0.5095) (807.1000,0.5095) (807.1000,0.5095) (807.1000,0.5095) (807.1000,0.5095) (807.1000,0.5095) (807.1000,0.5095) (807.1000,0.5095) (807.1000,0.5095) (807.1000,0.5095) (807.1000,0.5095) (807.1000,0.5095) (807.1000,0.5095) (807.1000,0.5095) (807.1000,0.5095) (807.1000,0.5095) (807.1000,0.5095) (807.1000,0.5096) (807.1000,0.5096) (807.1000,0.5096) (807.1000,0.5096) (807.1000,0.5096) (807.1000,0.5096) (807.1000,0.5096) (807.1000,0.5096) (807.1000,0.5096) (807.1000,0.5096) (807.1000,0.5096) (807.1000,0.5097) (807.1000,0.5097) (807.1000,0.5097) (807.1000,0.5097) (807.1000,0.5097) (807.1000,0.5097) (807.1000,0.5098) (807.1000,0.5098) (807.1000,0.5099) (807.1000,0.5100) (807.1000,0.5101) (807.1000,0.5101) (807.1000,0.5102) (807.1000,0.5102) (807.1000,0.5103) (807.1000,0.5103) (807.1000,0.5104) (807.1000,0.5104) (807.1000,0.5108) (807.1000,0.5111) (807.1000,0.5113) (807.0000,0.5116) (807.0000,0.5118) (807.0000,0.5119) (807.0000,0.5121) (807.0000,0.5123) (807.0000,0.5124) (807.0000,0.5125) (806.9000,0.5137) (806.8000,0.5142) (806.8000,0.5146) (806.7000,0.5150) (806.7000,0.5154) (806.6000,0.5157) (806.6000,0.5161) (806.5000,0.5164) (806.5000,0.5167) (806.4000,0.5170) (805.9000,0.5194) (805.7000,0.5204) (805.4000,0.5214) (805.2000,0.5222) (804.9000,0.5231) (804.7000,0.5238) (804.4000,0.5245) (804.2000,0.5252) (803.9000,0.5259) (803.7000,0.5265) (801.7000,0.5310) (799.7000,0.5346) (797.7000,0.5378) (795.7000,0.5407) (793.7000,0.5433) (791.7000,0.5458) (789.7000,0.5481) (787.7000,0.5503) (785.7000,0.5523) (783.7000,0.5543) (778.7000,0.5589) (773.7000,0.5631) (768.7000,0.5671) (763.7000,0.5708) (758.7000,0.5743) (753.7000,0.5776) (748.7000,0.5808) (743.7000,0.5839) (738.7000,0.5869) (733.7000,0.5898) (723.7000,0.5953) (713.7000,0.6006) (703.7000,0.6056) (693.7000,0.6105) (683.7000,0.6152) (673.7000,0.6197) (663.7000,0.6242) (653.7000,0.6285) (643.7000,0.6327) (633.7000,0.6369) (613.7000,0.6450) (593.7000,0.6529) (573.7000,0.6606) (553.7000,0.6682) (533.7000,0.6757) (513.7000,0.6832) (493.7000,0.6907) (473.7000,0.6982) (453.7000,0.7058) (443.7000,0.7096) (438.7000,0.7116) (433.7000,0.7137) (431.2000,0.7148) (429.9000,0.7153) (429.8000,0.7154) (429.6000,0.7155) (429.5000,0.7156) (429.4000,0.7157) (429.3000,0.7157) (429.3000,0.7157) (429.3000,0.7157) (429.3000,0.7157) (429.3000,0.7158) (429.3000,0.7158) (429.3000,0.7158) (429.3000,0.7158) (429.3000,0.7158) (429.3000,0.7158) (429.3000,0.7158) (429.3000,0.7158) (429.3000,0.7158) (429.3000,0.7158) (429.3000,0.7158) (429.3000,0.7158) (429.3000,0.7158) (429.3000,0.7158) (429.3000,0.7158) (429.3000,0.7158) (429.3000,0.7158) (429.3000,0.7158) (429.3000,0.7158) (429.3000,0.7158) (429.3000,0.7158) (429.3000,0.7158) (429.3000,0.7158) (429.3000,0.7158) (429.3000,0.7158) (429.3000,0.7158) (429.3000,0.7158) (429.3000,0.7158) (429.3000,0.7158) (429.3000,0.7158) (429.3000,0.7158) (429.3000,0.7158) (429.3000,0.7158) (429.3000,0.7158) (429.3000,0.7158) (429.3000,0.7158) (429.3000,0.7158) (429.3000,0.7158) (429.3000,0.7158) (429.3000,0.7158) (429.3000,0.7158) (429.3000,0.7158) (429.3000,0.7158) (429.3000,0.7158) (429.3000,0.7158) (429.3000,0.7158) (429.3000,0.7158) (429.3000,0.7158) (429.3000,0.7158) (429.4000,0.7158) (429.4000,0.7158) (429.4000,0.7158) (429.4000,0.7158) (429.4000,0.7158) (429.4000,0.7158) (429.4000,0.7158) (429.4000,0.7158) (429.4000,0.7158) (429.5000,0.7158) (429.6000,0.7158) (429.6000,0.7158) (429.7000,0.7158) (429.7000,0.7158) (429.8000,0.7158) (429.8000,0.7158) (429.9000,0.7157) (429.9000,0.7157) (430.0000,0.7157) (430.5000,0.7156) (430.7000,0.7155) (431.0000,0.7155) (431.2000,0.7154) (431.5000,0.7153) (431.7000,0.7153) (432.0000,0.7152) (432.2000,0.7151) (432.5000,0.7151) (432.7000,0.7150) (434.7000,0.7144) (436.7000,0.7138) (438.7000,0.7131) (440.7000,0.7125) (442.7000,0.7118) (444.7000,0.7112) (446.7000,0.7105) (448.7000,0.7099) (450.7000,0.7092) (452.7000,0.7085) (457.7000,0.7069) (462.7000,0.7052) (467.7000,0.7035) (472.7000,0.7018) (477.7000,0.7001) (482.7000,0.6984) (487.7000,0.6968) (492.7000,0.6951) (497.7000,0.6934) (502.7000,0.6917) (512.7000,0.6883) (522.7000,0.6849) (532.7000,0.6815) (542.7000,0.6780) (552.7000,0.6746) (562.7000,0.6711) (572.7000,0.6677) (582.7000,0.6642) (592.7000,0.6607) (602.7000,0.6572) (622.7000,0.6500) (642.7000,0.6427) (662.7000,0.6353) (682.7000,0.6276) (702.7000,0.6196) (722.7000,0.6113) (742.7000,0.6026) (762.7000,0.5933) (782.7000,0.5832) (802.7000,0.5721) (822.7000,0.5591) (842.7000,0.5430) (852.7000,0.5323) (857.7000,0.5256) (862.7000,0.5166) (865.2000,0.5100) (866.5000,0.5048) (867.1000,0.4996) (867.4000,0.4977) (867.6000,0.4984) (867.7000,0.4991) (867.9000,0.4999) (868.0000,0.5008) (868.2000,0.5018) (868.4000,0.5030) (868.5000,0.5046) (868.6000,0.5059) (868.6000,0.5071) (868.6000,0.5074) (868.6000,0.5075) (868.6000,0.5075) (868.6000,0.5077) (868.6000,0.5077) (868.6000,0.5078) (868.6000,0.5078) (868.6000,0.5078) (868.6000,0.5078) (868.6000,0.5078) (868.6000,0.5078) (868.6000,0.5079)};
\addplot [color=magenta,line width=0.8pt, smooth] coordinates {(330.9000,0.2377) (328.9000,0.2366) (326.9000,0.2355) (324.9000,0.2344) (322.9000,0.2332) (320.9000,0.2321) (318.9000,0.2309) (316.9000,0.2297) (314.9000,0.2286) (312.9000,0.2274) (310.9000,0.2262) (305.9000,0.2232) (300.9000,0.2200) (295.9000,0.2169) (290.9000,0.2136) (285.9000,0.2102) (280.9000,0.2068) (275.9000,0.2032) (270.9000,0.1995) (265.9000,0.1956) (260.9000,0.1915) (250.9000,0.1826) (240.9000,0.1720) (235.9000,0.1657) (230.9000,0.1577) (228.4000,0.1525) (227.2000,0.1490) (226.6000,0.1469) (225.9000,0.1440) (225.6000,0.1418) (225.5000,0.1397) (225.4000,0.1383) (225.4000,0.1383) (225.3000,0.1383) (225.3000,0.1383) (225.2000,0.1382) (225.2000,0.1382) (225.2000,0.1382) (225.1000,0.1381) (225.1000,0.1381) (225.0000,0.1381) (225.0000,0.1380) (225.0000,0.1380) (224.9000,0.1380) (224.9000,0.1380) (224.9000,0.1379) (224.8000,0.1379) (224.8000,0.1379) (224.7000,0.1378) (224.7000,0.1378) (224.7000,0.1378)};
\addplot [color=magenta,line width=0.8pt, smooth] coordinates {(330.9000,0.0775) (332.9000,0.0771) (334.9000,0.0767) (336.9000,0.0763) (338.9000,0.0759) (340.9000,0.0756) (342.9000,0.0752) (344.9000,0.0748) (346.9000,0.0745) (348.9000,0.0741) (350.9000,0.0737) (355.9000,0.0729) (360.9000,0.0721) (365.9000,0.0713) (370.9000,0.0705) (375.9000,0.0697) (380.9000,0.0690) (385.9000,0.0683) (390.9000,0.0676) (395.9000,0.0669) (400.9000,0.0663) (410.9000,0.0651) (420.9000,0.0639) (430.9000,0.0628) (440.9000,0.0617) (450.9000,0.0607) (460.9000,0.0598) (470.9000,0.0589) (480.9000,0.0580) (490.9000,0.0572) (500.9000,0.0564) (520.9000,0.0549) (540.9000,0.0536) (560.9000,0.0523) (580.9000,0.0511) (600.9000,0.0501) (620.9000,0.0491) (640.9000,0.0482) (660.9000,0.0474) (680.9000,0.0466) (700.9000,0.0459) (720.9000,0.0453) (740.9000,0.0448) (760.9000,0.0444) (780.9000,0.0442) (790.9000,0.0442) (800.9000,0.0444) (805.9000,0.0447) (806.6000,0.0448) (806.9000,0.0449) (807.0000,0.0450) (807.1000,0.0450) (807.1000,0.0451) (807.1000,0.0451) (807.1000,0.0451) (807.1000,0.0451) (807.1000,0.0451) (807.1000,0.0451) (807.1000,0.0451) (807.1000,0.0451) (807.1000,0.0451) (807.1000,0.0451) (807.1000,0.0451) (807.1000,0.0451) (807.1000,0.0451) (807.1000,0.0451) (807.1000,0.0451) (807.1000,0.0451) (807.1000,0.0451) (807.1000,0.0451) (807.1000,0.0451) (807.1000,0.0451) (807.1000,0.0451) (807.1000,0.0451) (807.1000,0.0451) (807.1000,0.0451) (807.1000,0.0451) (807.1000,0.0451) (807.1000,0.0451) (807.1000,0.0451) (807.1000,0.0451) (807.1000,0.0451) (807.1000,0.0451) (807.1000,0.0451) (807.1000,0.0451) (807.1000,0.0451) (807.1000,0.0451) (807.1000,0.0451) (807.1000,0.0451) (807.1000,0.0451) (807.1000,0.0451) (807.1000,0.0451) (807.1000,0.0451) (807.1000,0.0451) (807.1000,0.0451) (807.1000,0.0451) (807.1000,0.0451) (807.1000,0.0451) (807.1000,0.0451) (807.1000,0.0451) (807.1000,0.0451) (807.1000,0.0451) (807.1000,0.0451) (807.1000,0.0451) (807.1000,0.0451) (807.1000,0.0452) (807.1000,0.0452) (807.1000,0.0452) (807.1000,0.0452) (807.1000,0.0452) (807.1000,0.0452) (807.1000,0.0452) (807.1000,0.0452) (807.1000,0.0452) (807.0000,0.0452) (807.0000,0.0452) (807.0000,0.0452) (807.0000,0.0452) (807.0000,0.0453) (807.0000,0.0453) (807.0000,0.0453) (806.9000,0.0453) (806.8000,0.0453) (806.8000,0.0454) (806.7000,0.0454) (806.7000,0.0454) (806.6000,0.0454) (806.6000,0.0454) (806.5000,0.0455) (806.5000,0.0455) (806.4000,0.0455) (805.9000,0.0456) (805.7000,0.0457) (805.4000,0.0457) (805.2000,0.0458) (804.9000,0.0459) (804.7000,0.0459) (804.4000,0.0459) (804.2000,0.0460) (803.9000,0.0460) (803.7000,0.0461) (801.7000,0.0464) (799.7000,0.0467) (797.7000,0.0469) (795.7000,0.0472) (793.7000,0.0474) (791.7000,0.0476) (789.7000,0.0478) (787.7000,0.0480) (785.7000,0.0482) (783.7000,0.0484) (778.7000,0.0490) (773.7000,0.0495) (768.7000,0.0499) (763.7000,0.0504) (758.7000,0.0509) (753.7000,0.0514) (748.7000,0.0519) (743.7000,0.0524) (738.7000,0.0529) (733.7000,0.0534) (723.7000,0.0544) (713.7000,0.0554) (703.7000,0.0564) (693.7000,0.0575) (683.7000,0.0586) (673.7000,0.0598) (663.7000,0.0610) (653.7000,0.0622) (643.7000,0.0635) (633.7000,0.0649) (613.7000,0.0678) (593.7000,0.0710) (573.7000,0.0746) (553.7000,0.0787) (533.7000,0.0834) (513.7000,0.0889) (493.7000,0.0956) (473.7000,0.1041) (453.7000,0.1159) (443.7000,0.1244) (438.7000,0.1301) (433.7000,0.1381) (431.2000,0.1443) (429.9000,0.1494) (429.8000,0.1504) (429.6000,0.1515) (429.5000,0.1530) (429.4000,0.1542) (429.3000,0.1550) (429.3000,0.1557) (429.3000,0.1558) (429.3000,0.1560) (429.3000,0.1562) (429.3000,0.1565) (429.3000,0.1565) (429.3000,0.1565) (429.3000,0.1566) (429.3000,0.1566) (429.3000,0.1566) (429.3000,0.1566) (429.3000,0.1566) (429.3000,0.1566) (429.3000,0.1566) (429.3000,0.1566) (429.3000,0.1566) (429.3000,0.1566) (429.3000,0.1566) (429.3000,0.1566) (429.3000,0.1566) (429.3000,0.1566) (429.3000,0.1566) (429.3000,0.1566) (429.3000,0.1566) (429.3000,0.1566) (429.3000,0.1566) (429.3000,0.1566) (429.3000,0.1566) (429.3000,0.1566) (429.3000,0.1566) (429.3000,0.1566) (429.3000,0.1567) (429.3000,0.1567) (429.3000,0.1567) (429.3000,0.1567) (429.3000,0.1568) (429.3000,0.1568) (429.3000,0.1568) (429.3000,0.1568) (429.3000,0.1568) (429.3000,0.1568) (429.3000,0.1570) (429.3000,0.1570) (429.3000,0.1571) (429.3000,0.1572) (429.3000,0.1573) (429.3000,0.1573) (429.3000,0.1574) (429.3000,0.1574) (429.3000,0.1575) (429.3000,0.1575) (429.3000,0.1579) (429.4000,0.1582) (429.4000,0.1584) (429.4000,0.1586) (429.4000,0.1588) (429.4000,0.1590) (429.4000,0.1592) (429.4000,0.1593) (429.4000,0.1595) (429.4000,0.1596) (429.5000,0.1608) (429.6000,0.1613) (429.6000,0.1617) (429.7000,0.1622) (429.7000,0.1625) (429.8000,0.1629) (429.8000,0.1632) (429.9000,0.1636) (429.9000,0.1639) (430.0000,0.1642) (430.5000,0.1667) (430.7000,0.1678) (431.0000,0.1688) (431.2000,0.1697) (431.5000,0.1705) (431.7000,0.1713) (432.0000,0.1721) (432.2000,0.1729) (432.5000,0.1736) (432.7000,0.1742) (434.7000,0.1791) (436.7000,0.1831) (438.7000,0.1866) (440.7000,0.1899) (442.7000,0.1929) (444.7000,0.1957) (446.7000,0.1983) (448.7000,0.2008) (450.7000,0.2032) (452.7000,0.2055) (457.7000,0.2110) (462.7000,0.2160) (467.7000,0.2207) (472.7000,0.2252) (477.7000,0.2294) (482.7000,0.2335) (487.7000,0.2375) (492.7000,0.2413) (497.7000,0.2450) (502.7000,0.2486) (512.7000,0.2555) (522.7000,0.2621) (532.7000,0.2684) (542.7000,0.2745) (552.7000,0.2805) (562.7000,0.2863) (572.7000,0.2919) (582.7000,0.2975) (592.7000,0.3029) (602.7000,0.3083) (622.7000,0.3188) (642.7000,0.3291) (662.7000,0.3394) (682.7000,0.3496) (702.7000,0.3599) (722.7000,0.3704) (742.7000,0.3811) (762.7000,0.3923) (782.7000,0.4041) (802.7000,0.4169) (822.7000,0.4314) (842.7000,0.4490) (852.7000,0.4604) (857.7000,0.4675) (862.7000,0.4768) (865.2000,0.4836) (866.5000,0.4889) (867.1000,0.4940) (867.4000,0.4977) (867.6000,0.4984) (867.7000,0.4991) (867.9000,0.4999) (868.0000,0.5008) (868.2000,0.5018) (868.4000,0.5030) (868.5000,0.5046) (868.6000,0.5059) (868.6000,0.5071) (868.6000,0.5074) (868.6000,0.5075)};
\addplot [color=magenta,line width=0.8pt, smooth] coordinates {(330.9000,0.0775) (328.9000,0.0779) (326.9000,0.0783) (324.9000,0.0788) (322.9000,0.0792) (320.9000,0.0796) (318.9000,0.0801) (316.9000,0.0805) (314.9000,0.0810) (312.9000,0.0815) (310.9000,0.0820) (305.9000,0.0832) (300.9000,0.0845) (295.9000,0.0859) (290.9000,0.0874) (285.9000,0.0889) (280.9000,0.0906) (275.9000,0.0923) (270.9000,0.0942) (265.9000,0.0963) (260.9000,0.0985) (250.9000,0.1038) (240.9000,0.1106) (235.9000,0.1151) (230.9000,0.1211) (228.4000,0.1254) (227.2000,0.1284) (226.6000,0.1303) (225.9000,0.1330) (225.6000,0.1351) (225.5000,0.1371) (225.4000,0.1383) (225.4000,0.1383) (225.3000,0.1383) (225.3000,0.1383) (225.2000,0.1382) (225.2000,0.1382) (225.2000,0.1382) (225.1000,0.1381) (225.1000,0.1381) (225.0000,0.1381)};
\addplot [color=black,dashed] coordinates {(0,1) (1250,1)};
\end{axis}
\end{tikzpicture} 
\caption{For $\lambda\in(0,\lambda^*)$ fixed, minimal qualitative bifurcation diagrams of~\eqref{eq-main} with $\mu$ as bifurcation parameter. Depending on $\lambda$, different topological configurations may appear: only a bounded connected branch (black) connecting $0$ to $1$ producing one to two solutions (left) or a connected branch crossed by a loop (magenta) which increases the number of solutions up to either four (center) or eight (right).} 
\label{fig-bifurcation-2}
\end{figure}
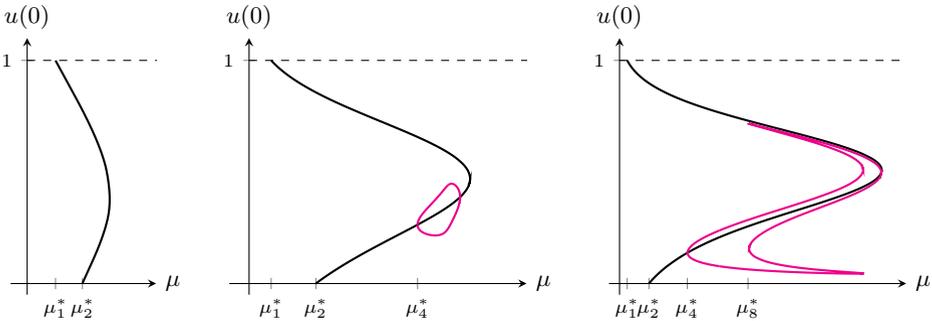

Next, when $\mu\in(0,\tilde{\mu}(\lambda)]$, $T_{sym}$ is an unbounded function that coincides with $f^{l}_{1,2}$ in $[0,s_0^{\mathcal{M}}(\mu))\times(0,\tilde{\mu}(\lambda)]$ and with $f^r_{3,1}$ in $(s_1^{\mathcal{M}}(\mu),1]\times(0,\tilde{\mu}(\lambda)]$ (cf.,~\eqref{fun-f} and recall that, for $\mu=\tilde{\mu}(\lambda)$, $s_0^{\mathcal{M}}(\mu)=s_1^{\mathcal{M}}(\mu)=s^{\mathcal{M}}(\mu)$). On the other hand, when $\mu\in(\tilde{\mu}(\lambda),+\infty)$, $T_{sym}(s,\mu)$ coincides either with $f_1(s,\mu):=T_1(s,\mu)$
when the configuration of the connection times is as in Figure~\ref{fig-noisola}, or with 
\begin{equation*}
f_{1,2,3}(s,\mu) :=
\begin{cases}
\, T_{1}(s,\mu), &\text{if $s\in[0,s_{0}^{\tau}(\mu)]\cup(s_1^{\tau}(\mu),1]$,} \\
\, T_{2}(s,\mu), &\text{if $s\in[s_{0}^{\tau}(\mu),s_{1}^{\omega}(\mu)]$,}\\
\, T_{3}(s,\mu), &\text{if $s\in[s_{1}^{\omega}(\mu),s_{1}^{\tau}(\mu)]$,}\\
\end{cases}
\end{equation*}
when the configuration is as in Figure~\ref{fig-08}. Accordingly, the map $T_{sym}$ is a bounded function provided $\mu>\tilde{\mu}(\lambda)$. In order to explain the transition between an unbounded and bounded configuration, we recall, from Section~\ref{section-4}, that $s_1^{\omega}(\mu)\to s^\mathcal{M}(\tilde{\mu}(\lambda))$ as $\mu\to(\tilde{\mu}(\lambda))^{+}$. We thus have $\kappa_4(\mu)\to+\infty$ as $\mu\to(\tilde{\mu}(\lambda))^{+}$ which leads to a breakdown in continuity for the function $T_{sym}(\cdot, \tilde{\mu}(\lambda))$ at $s=s^\mathcal{M}(\tilde{\mu}(\lambda))$.

Throughout this section, suggested as above by numerical computations, we also assume that:
\begin{quote}
\textit{the functions $f_{1}$ and $f_{1,2,3}$ have exactly one local maximum point in the $s$-variable.}
\end{quote}
Next, let us take $\mu^{*}_1=\mu^{*}_1(\lambda)$ and $\mu^*_2=\mu^*_2(\lambda)$ as the values defined in~\eqref{eq-mu1} and~\eqref{eq-mu2}, respectively.
We define $\mu^{**}_2=\mu^{**}_2(\lambda)$ similarly as in~\eqref{eq-mu2**}:
\begin{equation*}
\mu_{2}^{**}(\lambda)= \min \biggl{\{} \mu>\tilde{\mu}(\lambda) \colon \max_{s} T_{sym}(s,\mu) = 1-2\sigma\biggr{\}}.
\end{equation*}
We notice that $\mu_1^*<\mu_2^*<\mu_2^{**}$. Arguing as in Section~\ref{section-6.1}, we can infer that:
\begin{itemize}
\item for every $\mu\in(\mu^{*}_1,\mu^{*}_2]\cup\{\mu^{**}_2\}$, there is exactly one value $\xi_{1}^\mu\in(0,1)$ such that $T_{sym}(\xi_{1}^\mu,\mu)=1-2\sigma$;
\item for every $\mu\in(\mu^{*}_2,\mu^{**}_2)$, there are exactly two values $\xi_{2,i}^\mu\in(0,1)$, $i=1,2,$ such that $\xi_{2,1}^\mu<\xi_{2,2}^\mu:=\xi_{1}^\mu$ and $T_{sym}(\xi_{2,i}^\mu,\mu)=1-2\sigma$.
\end{itemize}
We stress that these values coincide with those in Section~\ref{section-6.1}, provided that $\mu\in(0,\tilde{\mu}(\lambda)]$. We can thus ensure the existence of two branches:
\begin{align*}
&\tilde{b}_{1}(\mu) :=
\begin{cases}
\, 0, &\text{if $\mu=\mu^*_2$,} \\
\, \xi_{2,1}^\mu, &\text{if $\mu\in(\mu^*_2,\mu^{**}_2)$,} \\
\end{cases}
&\tilde{b}_{2}(\mu) :=
\begin{cases}
\, 1, &\text{if $\mu=\mu^*_1$,} \\
\, \xi_{1}^\mu, &\text{if $\mu\in(\mu^*_1,\mu^{**}_2]$.} \\
\end{cases}
\end{align*}
These branches bifurcate supercritically from $0$ at $\mu=\mu^*_2$ and from $1$ at $\mu=\mu^*_1$, respectively. Since for every $\mu>\tilde{\mu}(\lambda)$ the function $T_{sym}$ is a bounded continuous function in $[0,1]$, then $\tilde{b}_{1}$ and $\tilde{b}_{2}$ merge together at $\mu^{**}_2$, where a subcritical turning point arises. Thus, we have shown that, for $\lambda\in(0,\lambda^*)$, the branches bifurcating from $0$ and $1$ belong to the same connected component.

\begin{remark}\label{rem-6.1}
As observed in Section~\ref{section-4}, when $\mu\in(\tilde{\mu}(\lambda),+\infty)$ and it is sufficiently close to $\tilde{\mu}(\lambda)$, a loop in the connection times appears around $T_{sym}$. If the level $1-2\sigma$ intersects such a loop, we obtain a topologically different configuration for the bifurcation diagrams to problem~\eqref{eq-main} which involves a corresponding loop surrounding the branches $\tilde{b}_1$ and $\tilde{b}_2$, see Figure~\ref{fig-bifurcation-2}~(center and right). Moreover, depending on the number of maxima and minima of the connection times forming the loop (see Figure~\ref{fig-08}), there might be some additional turning points on the loop in the bifurcation diagram, as shown in Figure~\ref{fig-bifurcation-2}~(right). Actually, by continuity, such a situation should occur for $\lambda\in(0,\lambda^*)$ sufficiently close to $\lambda^*$.
\hfill$\lhd$
\end{remark}

\begin{remark}\label{rem-6.2}
In this work, we have considered only a symmetric step-wise weight function. For completeness, we point out that our analysis can be adapted, in the spirit of the one done in \cite{LGTe-14}, also for non-symmetric step-wise weights. The main differences in the results regard the structure of the bifurcation diagrams constructed in this section. Indeed, the secondary bifurcation points on the branches break and give rise to separate connected components (see also~\cite{Te-15} for a similar behavior).
\hfill$\lhd$
\end{remark}

\section{High multiplicity for $\mu=K\lambda$, with large $\lambda$: Proof of Theorem~\ref{th-intro3}}\label{section-7}

In the previous sections, the positive parameters $\lambda$ and $\mu$, defining $a_{\lambda,\mu}$ in \eqref{eq-weight}, have been taken independent of each other. In this section, instead, we take $\mu=K\lambda$, with $K>0$, and deal with problem~\eqref{eq-aux-intro} where the weight $\tilde{a}$ is defined as in~\eqref{eq-weight-intro} in order to prove Theorem~\ref{th-intro3}.
For ease of notation, we do not explicitly write the dependence on $\lambda$ when it is clear from the context.

\begin{proof}[Proof of Theorem~\ref{th-intro3}]
Without loss of generality, we take $\lambda \in[\lambda^{*},+\infty)$, since we aim to determine the asymptotic behavior, as $\lambda\to+\infty$, of the quantities introduced in the previous sections. 
		
\smallskip
\noindent
\textit{Proof of $(i)$.} Let $K>0$. We recall that the continuous map $f^r_{3,1}$ in~\eqref{fun-f} satisfies 
\begin{equation*}
\lim_{s\to (s_1^{\mathcal{M}})^{+}}f^r_{3,1}(s,K\lambda)=+\infty, \qquad \lim_{s\to 1^{-}}f^r_{3,1}(s,K\lambda)=\mathcal{L}_1(\lambda),
\end{equation*}
(see also Figure \ref{fig-06}). Thus, we conclude by noticing that, for $\mu=K\lambda$ and $\lambda\to+\infty$, \eqref{eq-aux-time2} ensures that $\mathcal{L}_1(\lambda)\to 0$.

\smallskip
\noindent
\textit{Proof of $(ii)$.} Let $K>\frac{2\sigma}{1-2\sigma}$. From~\eqref{eq-aux-time1}, with $\mu=K\lambda$, we have
\begin{equation*}
\lim_{s\to 0^{+}} T_{1}(s)=\mathcal{L}_0 =  \dfrac{2 \sigma}{K}<1-2\sigma.
\end{equation*}
The properties of the connection times in Section~\ref{section-4.1} guarantee the existence of at least one solution of problem~\eqref{eq-aux-intro}, corresponding to some $s\in(0,s_0^{\mathcal{M}})$, for all $\lambda\in[\lambda^{*},+\infty)$. Moreover, thanks to \eqref{eq:order_L0_L1}, we have, again for all $\lambda\in[\lambda^{*},+\infty)$, another solution corresponding to some $s\in(s_1^{\mathcal{M}},1)$.

\smallskip
\noindent

For the proof of $(iii)$ and $(iv)$, we need some intermediate technical results on the asymptotic behavior of $\Gamma_0(\lambda)$ and the connection times that we present through the following steps.

\smallskip
\noindent
\textit{Step~1. Asymptotic behavior of $s_{0}(\lambda)$ and $s_1(\lambda)$.}
Let $s_{0}=s_{0}(\lambda)$, $s_1=s_1(\lambda)$, and $s^{*}$ be the points in the interval $(0,1)$ introduced in Proposition~\ref{pr-2.1}. We stress that $s^{*}$ does not depend on $\lambda$ and $0<s_{0}<s^{*}<s_{1}<1$.
	
Given $\lambda\in[\lambda^{*},+\infty)$, due to the monotonicity of the function $s\mapsto T_{0}(s,\lambda)$ defined in \eqref{eq-formulaT_0}, we have that $s_{0}$ and $s_1$ are the unique values $s$ such that $T_{0}(s)=\sigma$ in $(0,s^{*})$ and in $(s^{*},1)$, respectively. Since, for fixed $s\in(0,1)$, the function $\lambda\mapsto T_{0}(s,\lambda)$ is strictly decreasing, we have that $s_{0}=s_{0}(\lambda)$ is strictly decreasing and $s_{1}=s_{1}(\lambda)$ is strictly increasing. Moreover, the function $\lambda\mapsto T_{0}(s,\lambda)$ converges to $0$ locally uniformly for $s\in(0,1)$ as $\lambda\to+\infty$; therefore, we have
\begin{equation}\label{eq-7.s01}
\lim_{\lambda\to+\infty}s_{0}(\lambda)=0 \qquad \text{and} \qquad 	\lim_{\lambda\to+\infty} s_1(\lambda)=1.
\end{equation}
	
\smallskip
\noindent
\textit{Step~2. Asymptotic behavior of $s_{0}^{\mathcal{M}}(\lambda)$ and of the corresponding solution $u_{s_{0}^{\mathcal{M}}(\lambda)}$.} We consider the point $s_{0}^{\mathcal{M}}=s_{0}^{\mathcal{M}}(\lambda)$, defined as the first value of $s\in(0,s_{0})$ such that $(u_{s}(\sigma),v_{s}(\sigma)) \in \Gamma_{0}\cap\mathcal{M}_{\mu}$ (see~Section~\ref{section-4.1}).
From the definition of $\mathcal{M}_{\mu}$ given in \eqref{def-stable-manifold}, we have
\begin{equation*}
(v_{s_{0}^{\mathcal{M}}}(\sigma))^2-2K\lambda G(u_{s_{0}^{\mathcal{M}}}(\sigma))=0,
\end{equation*}
while the energy conservation in $[0,\sigma]$ (cf.,~\eqref{eq-energy-lambda}) implies
\begin{equation*}
(v_{s_{0}^{\mathcal{M}}}(\sigma))^2+2\lambda G(u_{s_{0}^{\mathcal{M}}}(\sigma))=2\lambda G(s_{0}^{\mathcal{M}}).
\end{equation*}
From the above equalities, we deduce that
\begin{equation}\label{eq-7.sh-impl}
G(s_{0}^{\mathcal{M}})-(1+K)G(u_{s_{0}^{\mathcal{M}}}(\sigma))=0.
\end{equation}
We stress the dependence of $u_{s_{0}^{\mathcal{M}}}(\sigma)$ on $\lambda$ arising from~\eqref{eq-initial0} and $s_{0}^{\mathcal{M}}=s_{0}^{\mathcal{M}}(\lambda)$ by writing $u_{s_{0}^{\mathcal{M}}(\lambda)}(\sigma,\lambda)$.
From the fact that $0 < u_{s_{0}^{\mathcal{M}}}(\sigma)<s_{0}^{\mathcal{M}}<s_{0}$ and thanks to~\eqref{eq-7.s01}, if we divide \eqref{eq-7.sh-impl} by $(s_{0}^{\mathcal{M}})^3$, we get
\begin{equation}\label{eq-7.xsh}
L_0^{\mathcal{M}} :=\lim_{\lambda\to+\infty} \dfrac{u_{s_{0}^{\mathcal{M}}(\lambda)}(\sigma,\lambda)}{s_{0}^{\mathcal{M}}(\lambda)}=\left(\frac{1}{1+K}\right)^{\!\frac{1}{3}}.
\end{equation}
Finally, from \eqref{eq-2.t} with $t=\sigma$ and $s=s_{0}^{\mathcal{M}}$, it follows
\begin{equation*}
l_0^{\mathcal{M}}:=\lim_{\lambda\to+\infty}\lambda s_{0}^{\mathcal{M}}(\lambda)=\left(\frac{I(L_0^{\mathcal{M}})}{\sqrt{2}\sigma}\right)^{\!2},
\end{equation*}
where we have set
\begin{equation}
\label{eq:defI}
I(r):= \sqrt{3} \int_{r}^{1} \dfrac{\mathrm{d}\xi}{\sqrt{1-\xi^{3}}}.
\end{equation}
	
\smallskip
\noindent
\textit{Step~3. Asymptotic behavior of the points $\theta s_{0}^{\mathcal{M}}(\lambda)$, with $\theta\in(0,1)$, and of the corresponding solutions $u_{\theta s_{0}^{\mathcal{M}}(\lambda)}$.}
Let us now consider a generic point $\hat{s}(\theta,\lambda):=\theta s_{0}^{\mathcal{M}}(\lambda)$, with $\theta\in(0,1)$ independent of $\lambda$. From Step~2, we have
\begin{equation}\label{eq-7.s}
\hat{l}(\theta):=\lim_{\lambda\to+\infty}\lambda \hat{s}(\theta,\lambda)=\theta l_0^{\mathcal{M}} \in(0,l_0^{\mathcal{M}} ),
\end{equation}
and, from \eqref{eq-2.t} with $t=\sigma$ and $s=\hat{s}(\theta,\lambda)$, it follows that
\begin{equation*}
\hat{L}(\theta):=\lim_{\lambda\to+\infty}\frac{u_{\hat{s}(\theta,\lambda)}(\sigma,\lambda)}{\hat{s}(\theta,\lambda)}
\end{equation*}
satisfies
\begin{equation}\label{eq-7.l}
\hat{l}(\theta)=\biggl{(}\frac{I(\hat{L}(\theta))}{\sqrt{2}\sigma}\biggr{)}^{\!2}.
\end{equation}
The previous relation, together with the implicit function theorem and the properties of $I(r)$, gives that the function $\theta\mapsto\hat{L}(\theta)$ is $\mathcal{C}^1(0,1)$, strictly decreasing, and satisfies
\begin{equation} \label{eq:lim_hat_L_theta}
\lim_{\theta \to 0^{+}} \hat{L}(\theta)=1, \qquad	\lim_{\theta \to1^{-}} \hat{L}(\theta)=L_0^{\mathcal{M}}=\left(\frac{1}{1+K}\right)^{\!\frac{1}{3}}.
\end{equation}
	
Moreover, from \eqref{eq-energy-lambda} and \eqref{eq-energy-mu}, in $(\sigma,1-\sigma)$ the energy of the level line of~\eqref{eq-energy-lambda} passing through $(u_{\hat{s}(\theta,\lambda)}(\sigma,\lambda),v_{\hat{s}(\theta,\lambda)}(\sigma),\lambda)$ is
\begin{equation*}
h_{K\lambda}(\hat{s}(\theta,\lambda))=2\lambda G(\hat{s}(\theta,\lambda))-2\lambda(1+K)G(u_{\hat{s}(\theta,\lambda)}(\sigma,\lambda))
\end{equation*} 
where $h_{K\lambda}$ is defined as in~\eqref{def-hs}, and it satisfies
\begin{equation}\label{eq-7.sign}
\lim_{\lambda\to+\infty}\frac{3h_{K\lambda}(\hat{s}(\theta,\lambda))}{2\lambda (\hat{s}(\theta,\lambda))^3}=1-(1+K)(\hat{L}(\theta))^3<1-(1+K)(L_0^{\mathcal{M}})^3=0.
\end{equation}
	
\smallskip
\noindent
\textit{Step~4. Characterization of the number of intersections between $\Gamma_0$ and the level lines of~\eqref{eq-energy-mu} for large $\lambda$.}
Let $\hat{s}(\theta,\lambda):=\theta s_{0}^{\mathcal{M}}(\lambda)$, with $\theta\in(0,1)$, as in Step~3, and assume now that $\hat{s}(\theta,\lambda) < s_{0}^{\tau}(\lambda)$. Then, there exists a point $\tilde{s}(\theta,\lambda)\in(s_{0}^{\tau}(\lambda),s_{0}^{\mathcal{M}}(\lambda))$ such that $h_{K\lambda}(\hat{s}(\theta,\lambda))=h_{K\lambda}(\tilde{s}(\theta,\lambda))$, that is
\begin{equation}\label{eq-7.energy}
G(\hat{s}(\theta,\lambda))-(1+K)G(u_{\hat{s}(\theta,\lambda)}(\sigma,\lambda)) = G(\tilde{s}(\theta,\lambda))-(1+K)G(u_{\tilde{s}(\theta,\lambda)}(\sigma,\lambda)).
\end{equation}
The same discussion can be performed for $\hat{s}(\theta,\lambda)\in\left( s_{0}^{\tau}(\lambda),s_{0}^{\mathcal{M}}(\lambda)\right)$, with the difference that, in such a case, $\tilde s(\theta,\lambda)$ lies in $(0,s_0^\tau(\lambda))$.
	
Again, from~\eqref{eq-2.t} with $t=\sigma$ and $s=\tilde{s}(\theta,\lambda)$, we have
\begin{equation}\label{eq-7.s2impl}
\sigma=\frac{1}{\sqrt{2\lambda \hat{s}(\theta,\lambda) \frac{\tilde{s}(\theta,\lambda)}{ \hat{s}(\theta,\lambda)}}} \int_{\frac{u_{\tilde{s}(\theta,\lambda)}(\sigma,\lambda)}{\tilde{s}(\theta,\lambda)}}^{1} \frac{\mathrm{d}\xi}{\sqrt{\frac{1-\xi^{3}}{3}-\tilde{s}(\theta,\lambda)\frac{1-\xi^{4}}{4}}}.
\end{equation}
Since the integral in \eqref{eq-7.s2impl} is bounded and \eqref{eq-7.s} holds, it entails that $\liminf$ and $\limsup$ of $\tilde{s}(\theta,\lambda)/\hat{s}(\theta,\lambda)$, as $\lambda\to+\infty$, are positive real numbers. As a consequence, by dividing \eqref{eq-7.energy} by $(\hat{s}(\theta,\lambda))^3$, we also obtain that $\liminf$ and $\limsup$ of $u_{\tilde{s}(\theta,\lambda)}(\sigma,\lambda)/\tilde{s}(\theta,\lambda)$, as $\lambda\to+\infty$, are positive real numbers. Let us denote
\begin{align*}
\ell^-(\theta)&:=\liminf_{\lambda\to+\infty}\frac{\tilde{s}(\theta,\lambda)}{\hat{s}(\theta,\lambda)}, & 
\ell^{+}(\theta)&:=\limsup_{\lambda\to+\infty}\frac{\tilde{s}(\theta,\lambda)}{\hat{s}(\theta,\lambda)}, \\
\tilde L^-(\theta)&:=\liminf_{\lambda\to+\infty}\frac{u_{\tilde{s}(\theta,\lambda)}(\sigma,\lambda)}{\tilde{s}(\theta,\lambda)}, &
\tilde L^{+}(\theta)&:=\limsup_{\lambda\to+\infty}\frac{u_{\tilde{s}(\theta,\lambda)}(\sigma,\lambda)}{\tilde{s}(\theta,\lambda)}.
\end{align*}
By taking the $\liminf$ and $\limsup$ in \eqref{eq-7.s2impl}, we obtain
\begin{equation*}
\frac{I(\tilde L^-(\theta))}{\sqrt{2 \hat{l}(\theta)\, \ell^-(\theta)}}=\sigma=\frac{I(\tilde L^{+}(\theta))}{\sqrt{2 \hat{l}(\theta)\, \ell^{+}(\theta)}}.
\end{equation*}
Suppose by contradiction that $\ell^-(\theta)<\ell^{+}(\theta)$: this would imply that $I(\tilde L^{+}(\theta))>I(\tilde L^-(\theta))$ and, since $I$ is strictly decreasing, that $\tilde L^{+}(\theta)<\tilde L^-(\theta)$, which is impossible. Then, $\ell^-(\theta)=\ell^{+}(\theta)$ and $\tilde L^{+}(\theta)=\tilde L^-(\theta)$, that is the quantities
\begin{equation*}
\ell(\theta):=\lim_{\lambda\to+\infty}\frac{\tilde{s}(\theta,\lambda)}{\hat{s}(\theta,\lambda)}, \qquad
\tilde L(\theta):=\lim_{\lambda\to+\infty}\frac{u_{\tilde{s}(\theta,\lambda)}(\sigma,\lambda)}{\tilde{s}(\theta,\lambda)}
\end{equation*}
exist and, by taking the limits in \eqref{eq-7.energy} divided by $(\hat{s}(\theta,\lambda))^3$ and in \eqref{eq-7.s2impl}, they satisfy
\begin{equation*}
1-(1+K)(\hat{L}(\theta))^3=(\ell(\theta))^3 \left(1-(1+K)(\tilde L(\theta))^3\right), 
\qquad
\sigma=\dfrac{I(\tilde L(\theta))}{\sqrt{2\hat{l}(\theta)\, \ell(\theta)}}.
\end{equation*}
From the second equality and \eqref{eq-7.l}, we get
\begin{equation*}
\ell(\theta)=\left(\frac{I(\tilde L(\theta))}{\sqrt{2\hat{l}(\theta)}\,\sigma}\right)^{\!2}=\biggl{(}\frac{I(\tilde L(\theta))}{I(\hat{L}(\theta))}\biggr{)}^{\!2}
\end{equation*}
and, by substituting in the first one, we have
\begin{equation*}
\bigl{(}I(\tilde L(\theta))\bigr{)}^6 \bigl{(}1-(1+K)(\tilde{L}(\theta))^3 \bigr{)}=\bigl{(}I(\hat{L}(\theta))\bigr{)}^6 \bigl{(}1-(1+K) (\hat{L}(\theta))^3\bigr{)}<0,
\end{equation*}
where the inequality comes from \eqref{eq-7.sign}.
	
We study now the differentiable function $f_{K}\colon (L_{0}^{\mathcal{M}},1) \to \mathbb{R}$, defined as
\begin{equation*}
f_K(r):=\bigl{(}I(r)\bigr{)}^6\bigl{(}1-(1+K)r^3\bigr{)}.
\end{equation*}
First, recalling \eqref{eq-7.xsh}, we notice that $f_{K}(r)<0$ for all $r\in(L_{0}^{\mathcal{M}},1)$, and $f_{K}(r)\to0$ as $r\to (L_{0}^{\mathcal{M}})^{+}$ and as $r\to1^{-}$. Therefore, $f_K$ has a minimum point. We claim that it is unique. Indeed, since $I(r)>0$ for all $r\in(0,1)$, the zeros of
\begin{equation*}
f'_K(r)=3(I(r))^5 \bigl{(} 2I'(r)(1-(1+K)r^3)-I(r)(1+K)r^2 \bigr{)}
\end{equation*}
are the solutions of $w_1(r)=w_2(r)$, where
\begin{equation*}
w_1(r):=2\,\frac{I'(r)}{I(r)}, \qquad 
w_2(r):=\frac{(1+K)r^2}{1-(1+K)r^3}.
\end{equation*}
For all $r\in(L_{0}^{\mathcal{M}},1)$, since $I''(r)=-\frac{r^2}{2}\bigl{(}\frac{1-r^3}{3}\bigr{)}^{-\frac{3}{2}}<0$, we have
\begin{equation}\label{der-w}
\begin{aligned}
&w_1'(r)=2\,\frac{I''(r)I(r)-(I'(r))^2}{(I(r))^2}<0,
\\
&w_2'(r)=\frac{(1+K)(2+(1+K)r^3)r}{(1-(1+K)r^3)^2}>0.
\end{aligned}
\end{equation}
Therefore, the negative differentiable function $f_K$ has a unique critical point.
Consequently, we deduce that the equation
\begin{equation*}
f_K(r)=\rho, \quad \text{with $M:=\min\{f_K(r)\colon r\in(L_{0}^{\mathcal{M}},1)\}<\rho<0$,}
\end{equation*}
has exactly two solutions. Finally, by exploiting the notation used in Step~3, the continuity of $\theta\mapsto\hat{L}(\theta)$, the definition of $f_K$, and \eqref{eq:lim_hat_L_theta}, we have that
\begin{equation*}
\bigl{\{}f_K(\hat{L}(\theta))\colon\theta\in(0,1)\bigr{\}}=[M,0).
\end{equation*}
This analysis shows that, for $\rho\in(M,0)$, there exist, for sufficiently large $\lambda$, two families of distinct points $s^-(\lambda)$, $s^{+}(\lambda)$ such that the corresponding (distinct) points on $\Gamma_0(\lambda)$ lie on the same level line of~\eqref{eq-energy-mu} and the limits for $\lambda\to+\infty$ of $u_{s^-(\lambda)}(\sigma,\lambda)/s^-(\lambda)$ and $u_{s^{+}(\lambda)}(\sigma,\lambda)/s^{+}(\lambda)$ are the two solutions of $f_K(r)=\rho$ lying in $(L_{0}^{\mathcal{M}},1)$.

\smallskip
\noindent
\textit{Step~5. Asymptotic behavior of $s_{0}^{\tau}(\lambda)$ and of the corresponding solution.}
The analysis done in Step~4 guarantees that the point in $(L_{0}^{\mathcal{M}},1)$ where $f_K$ achieves its minimum, which will be denoted by $L_{0}^{\tau}$, corresponds to the tangent level line of~\eqref{eq-energy-mu}, in the sense that
\begin{equation}\label{eq-7.Ltau}
L_{0}^{\tau} =\lim_{\lambda\to+\infty}\frac{u_{s_{0}^{\tau}(\lambda)}(\sigma,\lambda)}{s_{0}^{\tau}(\lambda)}.
\end{equation}
With the notation introduced in Step~4, we have that $w_{1}(L_{0}^{\tau})=w_{2}(L_{0}^{\tau})$, that is 
\begin{equation}\label{eq-7.Ltau-impl}
2\frac{I'(L_{0}^{\tau})}{I(L_{0}^{\tau})}=\frac{(1+K)(L_{0}^{\tau})^2}{1-(1+K)(L_{0}^{\tau})^3}.
\end{equation}
As above, we remark that $0 < u_{s_{0}^{\tau}(\lambda)}(\sigma,\lambda)<s_{0}^{\tau}(\lambda)<s_{0}(\lambda)$, thus all these quantities tend to $0$ as $\lambda\to+\infty$, and we have
\begin{equation}\label{eq-7.stau}
l_{0}^{\tau}:=\lim_{\lambda\to+\infty}\lambda s_{0}^{\tau}(\lambda)=\biggl{(}\frac{I(L_{0}^{\tau})}{\sqrt{2}\sigma}\biggr{)}^{\!2}.
\end{equation}
According to the notation introduced in Section~\ref{section-3}, we now denote by $m(u_{s_{0}^{\tau}}(\sigma))$ the abscissa of the intersection point between the level line of~\eqref{eq-energy-mu} through $(u_{s_{0}^{\tau}}(\sigma),v_{s_{0}^{\tau}}(\sigma))$ and the $u$-axis. We exploit the conservation of the energies (cf.,~\eqref{eq-energy-lambda} and \eqref{eq-energy-mu}) to deduce
\begin{equation*}
G(s_{0}^{\tau})-(1+K)G(u_{s_{0}^{\tau}}(\sigma))=-KG(m(u_{s_{0}^{\tau}}(\sigma))).
\end{equation*}
By dividing this relation by $(s_{0}^{\tau})^3$ and passing to the limit as $\lambda\to+\infty$, we have
\begin{equation}\label{eq-7.Mtau}
M_{0}^{\tau}:=\lim_{\lambda\to+\infty}\frac{m(u_{s_{0}^{\tau}(\lambda)}(\sigma,\lambda))}{s_{0}^{\tau}(\lambda)}=\biggl{(}\frac{(1+K)(L_{0}^{\tau})^3-1}{K}\biggr{)}^{\!\frac{1}{3}}.
\end{equation}
	
\smallskip
\noindent
\textit{Step~6. Asymptotic behavior of $\ell_{1,2}^{0}(\lambda)$.}
Recalling the definition \eqref{def-l120} of $\ell_{1,2}^{0}$, that is the common limit of the connection times $T_{1}$ and $T_{2}$ as $s\to s_{0}^{\tau}$, we investigate how $\ell_{1,2}^{0}$ depends on $\lambda$.
Observing that, from the conservation of the energy \eqref{eq-energy-mu}, we have
\begin{equation*}
(v_{s_{0}^{\tau}}(\sigma))^2-2K\lambda G(u_{s_{0}^{\tau}}(\sigma)) = -2K\lambda G(m(u_{s_{0}^{\tau}}(\sigma))),
\end{equation*}
from \eqref{def-l120} we deduce
\begin{align}
\lim_{\lambda\to+\infty} \ell_{1,2}^{0} &=
\lim_{\lambda\to+\infty}\frac{2}{\sqrt{2K\lambda}}
\int_{m(u_{s_{0}^{\tau}}(\sigma))}^{u_{s_{0}^{\tau}}(\sigma)}\frac{\mathrm{d}u}{\sqrt{G(u)-G(m(u_{s_{0}^{\tau}}(\sigma)))}} \\
&=\lim_{\lambda\to+\infty}\frac{2}{\sqrt{2K\lambda s_{0}^{\tau}}}
\int_{\frac{m(u_{s_{0}^{\tau}}(\sigma))}{s_{0}^{\tau}}}^{\frac{u_{s_{0}^{\tau}}(\sigma)}{s_{0}^{\tau}}}\frac{\mathrm{d}\tilde{u}}{\sqrt{\frac{G(s_{0}^{\tau}\tilde{u})-G(m(u_{s_{0}^{\tau}}(\sigma)))}{(s_{0}^{\tau})^3}}} \\
&=\frac{2}{\sqrt{2Kl_{0}^{\tau}}}
\int_{M_{0}^{\tau}}^{L_{0}^{\tau}}\frac{\mathrm{d}\tilde{u}}{\sqrt{\frac{\tilde{u}^3-(M_{0}^{\tau})^3}{3}}} \\
&=\frac{2\sigma}{I(L_{0}^{\tau})\sqrt{K}}
\int_{M_{0}^{\tau}}^{L_{0}^{\tau}}\frac{\mathrm{d}\tilde{u}}{\sqrt{\frac{\tilde{u}^3-(M_{0}^{\tau})^3}{3}}} =: \Theta_1(K), \label{eq-7.Theta1}
\end{align}
where, to get the last equality, we have used \eqref{eq-7.stau}. 
	
\smallskip
\noindent
\textit{Step~7. Asymptotic behavior of $\ell_{3}^{0}(\lambda)$.}
Recalling the definition \eqref{def-l30} of $\ell_{3}^{0}$, that is the limit of the connection time $T_{3}$ as $s\to s_{0}^{\tau}$, we now investigate how $\ell_{3}^{0}$ depends on $\lambda$. Arguing as above, we have
\begin{equation}\label{eq:l30}
\begin{aligned}
\ell_{3}^{0} 
&=\frac{2}{\sqrt{2K\lambda}}\int_{m(u_{s_{0}^{\tau}}(\sigma))}^{u_{s_{0}^{\tau}}(\sigma)}\dfrac{\mathrm{d}u}{\sqrt{G(u)-G(m(u_{s_{0}^{\tau}}(\sigma)))}}
\\
&\quad +\frac{1}{\sqrt{2K\lambda}}\int_{u_{s_{0}^{\tau}}(\sigma)}^{u_{s_{1}^{\tau}}(\sigma)}\dfrac{\mathrm{d}u}{\sqrt{G(u)-G(m(u_{s_{0}^{\tau}}(\sigma)))}},
\end{aligned}
\end{equation}
where $s_{1}^{\tau}\in(s_1,1)$ is the value of $s\in(0,1)$ giving the other intersection between the level line $H_{\mu}(u,v)=H_{\mu}(u_{s_0^{\tau}}(\sigma),v_{s_0^{\tau}}(\sigma))$ and $\Gamma_0$, apart from $(u_{s_{0}^{\tau}}(\sigma),v_{s_{0}^{\tau}}(\sigma))$. First of all, thanks to \eqref{eq-7.s01}, we observe that
\begin{equation}\label{eq-7.stau+}
\lim_{\lambda\to+\infty}s_{1}^{\tau}(\lambda)=1.
\end{equation}
Moreover, as in \eqref{eq-7.energy}, the conservation of the energies implies that
\begin{equation*}
G(s_{0}^{\tau})-(1+K)G(u_{s_{0}^{\tau}}(\sigma)) = G(s_{1}^{\tau})-(1+K)G(u_{s_{1}^{\tau}}(\sigma)),
\end{equation*}
thus \eqref{eq-7.Ltau}, \eqref{eq-7.stau}, and \eqref{eq-7.stau+} give that $L_{1}^{\tau}:=\lim_{\lambda\to+\infty}u_{s_{1}^{\tau}(\lambda)}(\sigma,\lambda)$ satisfies
\begin{equation*}
G(L_{1}^{\tau})=\frac{G(1)}{1+K} \in (0,G(1)),
\end{equation*}
and, so, $L_{1}^{\tau}\in(0,1)$. As a consequence, from \eqref{eq:l30} we have
\begin{align}
\lim_{\lambda\to+\infty}\ell_{3}^{0}
&=\lim_{\lambda\to+\infty}\frac{2}{\sqrt{2K\lambda s_{0}^{\tau}}}\int_{\frac{m(u_{s_{0}^{\tau}}(\sigma))}{s_{0}^{\tau}}}^{\frac{u_{s_{0}^{\tau}}(\sigma)}{s_{0}^{\tau}}}\frac{\mathrm{d}\tilde{u}}{\sqrt{\frac{G(s_{0}^{\tau}\tilde{u})-G(m(u_{s_{0}^{\tau}}(\sigma)))}{(s_{0}^{\tau})^3}}} \\
&\quad+\lim_{\lambda\to+\infty}\frac{1}{\sqrt{2K\lambda s_{0}^{\tau}}}\int_{\frac{u_{s_{0}^{\tau}}(\sigma)}{s_{0}^{\tau}}}^{\frac{u_{s_{1}^{\tau}}(\sigma)}{s_{0}^{\tau}}}\frac{\mathrm{d}\tilde{u}}{\sqrt{\frac{G(s_{0}^{\tau}\tilde{u})-G(m(u_{s_{0}^{\tau}}(\sigma)))}{(s_{0}^{\tau})^3}}} \\
&=\Theta_1(K)+\frac{\sigma}{I(L_{0}^{\tau})\sqrt{K}}
\int_{L_{0}^{\tau}}^{+\infty}\frac{\mathrm{d}\tilde{u}}{\sqrt{\frac{\tilde{u}^3-(M_{0}^{\tau})^3}{3}}} =: \Theta_2(K). \label{eq-7.Theta2}
\end{align}
	
\smallskip
\noindent
\textit{Proof of $(iii)$.}
We analyze the behavior of $\Theta_1(K)$ defined in~\eqref{eq-7.Theta1} as $K\to+\infty$.

First of all, we rewrite~\eqref{eq-7.Ltau-impl} as
\begin{equation}\label{eq-L0tau}
	\frac{1}{1+K}=\left(L_{0}^{\tau}(K)\right)^2\left(L_{0}^{\tau}(K)+\frac{I(L_{0}^{\tau}(K))}{2I'(L_{0}^{\tau}(K))}\right)
	\end{equation} 	
and assume by contradiction that $\liminf_{K\to+\infty}L_{0}^{\tau}(K)=0$. Then, the $\liminf$ for $K\to+\infty$ of the last factor in~\eqref{eq-L0tau} would be equal to $I(0)/(2I'(0))<0$. This would imply that, for some sequence of $K$ converging to $+\infty$, the right-hand side in~\eqref{eq-L0tau} is negative, contradicting the positivity of the left-hand side. The same reasoning guarantees that $\limsup_{K\to+\infty}L_{0}^{\tau}(K)>0$. Moreover, by letting $K\to+\infty$ in~\eqref{eq-L0tau}, we have that both $\liminf$ and $\limsup$ satisfy 
\begin{equation*}
2\frac{I'(r)}{I(r)}=-\frac{1}{r}.
\end{equation*}
This equation has a unique solution $\bar{L}$ thanks to~\eqref{der-w}. Thus $\lim_{K\to+\infty}L_{0}^{\tau}(K)=\bar{L}$. Furthermore, thanks to \eqref{eq-7.Mtau}, also $\lim_{K\to+\infty}M_{0}^{\tau}(K)=\bar L$. From \eqref{eq-7.Theta1} we conclude that 
\begin{equation*}
\lim_{K\to+\infty}\Theta_1(K)=0.
\end{equation*}
Next, we claim that
\begin{equation} \label{eq:claim_Theta1}
\lim_{K\to 0^{+}}\Theta_1(K)=+\infty.
\end{equation}
Once this is proved, by the continuity of $\Theta_1(K)$, there exists $K_4>0$ such that
\begin{equation}\label{eq-7.K4}
\Theta_1(K_4)=1-2\sigma \quad \text{ and } \quad \Theta_1(K)<1-2\sigma, \quad \text{for all $K>K_4$.}
\end{equation}
With this definition, we have that, for all $K>K_4$, $\lim_{\lambda\to+\infty} \ell_{1,2}^{0}<1-2\sigma$, and the properties of the connection times shown in Section~\ref{section-4.1} guarantee the existence of at least four solutions of problem~\eqref{eq-aux-intro}, once $K>K_4$ has been fixed and $\lambda$ is sufficiently large (depending on $K$).
	
To conclude this step, it only remains to prove \eqref{eq:claim_Theta1}. First of all, since, from Step~5, $L_0^{\tau}(K)\in(L_0^{\mathcal{M}},1)$, \eqref{eq-7.xsh} implies that $\lim_{K\to 0^+}L_0^{\tau}(K)=1$. Then, thanks to the definition \eqref{eq-7.Theta1}, since $I$ is continuous and $I(1)=0$, it will be enough to show that
\begin{equation*}
\int_{M_0^{\tau}(K)}^{L_0^{\tau}(K)}\frac{\mathrm{d}\tilde u}{\sqrt{\tilde u^3-(M_0^\tau(K))^3}}
\end{equation*}
is bounded away from zero, as $K\to 0^+$. This will follow if we prove that
\begin{equation} \label{eq:lim_M0tau}
\lim_{K\to 0^+}M_0^{\tau}(K)\in[0,1).
\end{equation}
From \eqref{eq-7.Mtau} and L'H\^{o}pital's rule, we have
\begin{align}
\lim_{K\to 0^{+}}(M_0^{\tau}(K))^3&=\lim_{K\to 0^+}\frac{(1+K)(L_0^\tau(K))^3-1}{K}\\
&=\lim_{K\to 0^{+}}(L_0^\tau(K))^3 + 3(1+K) (L_0^\tau(K))^2 (L_0^\tau(K))',
\end{align}
provided that the latter limit exists. Thus, to get \eqref{eq:lim_M0tau}, it will suffice to prove that $\lim_{K\to 0^+}(L_0^\tau(K))'$ is a negative real number. Using \eqref{eq:defI}, we rewrite~\eqref{eq-L0tau} as
\begin{equation*}
\frac{1}{1+K}=(L_0^{\tau}(K))^3-\frac{1}{2\sqrt{3}}I(L_0^{\tau}(K))\sqrt{1-(L_0^{\tau}(K))^3}(L_0^{\tau}(K))^2,
\end{equation*}
we differentiate with respect to $K$ 
\begin{align}
-\frac{1}{(1+K)^2}&=(L_0^\tau(K))'\left(3(L_0^{\tau}(K))^2+\frac{1}{2}(L_0^\tau(K))^2+\frac{\sqrt{3}I(L_0^{\tau}(K))(L_0^{\tau}(K))^4}{4\sqrt{1-(L_0^{\tau}(K))^3}}\right) \\
&\quad-(L_0^\tau(K))'\frac{L_0^{\tau}(K)}{\sqrt{3}}I(L_0^{\tau}(K))\sqrt{1-(L_0^{\tau}(K))^3},
\end{align}
and, by taking the limit as $K\to 0^+$ and using that $\lim_{r\to 1^{-}}\frac{I(r)}{\sqrt{1-r}}=2$, we obtain that $\lim_{K\to 0^+}(L_0^\tau(K))'=-\frac{1}{4}$, which concludes the proof of our claim.

\smallskip
\noindent
\textit{Proof of $(iv)$.}
We analyze the behavior of $\Theta_2(K)$ defined in \eqref{eq-7.Theta2} as $K\to+\infty$. By reasoning as in the proof of~$(iii)$, we have that
\begin{equation*}
\lim_{K\to+\infty}\Theta_2(K)=0, \qquad \lim_{K\to0^{+}}\Theta_2(K)=+\infty,
\end{equation*}
thus, by continuity, there exists $K_8$ such that
\begin{equation}\label{eq-7.K8}
\Theta_2(K_8)=1-2\sigma \quad \text{ and } \quad \Theta_2(K)<1-2\sigma, \quad \text{for all $K>K_8$.}
\end{equation}
With this definition, we have that, for all $K>K_8$, $\lim_{\lambda\to+\infty} \ell_{3}^{0}<1-2\sigma$, and the properties of the connection times shown in Section~\ref{section-4.1} guarantee the existence of at least eight solutions of  problem~\eqref{eq-aux-intro}, once $K>K_8$ has been fixed and $\lambda$ is sufficiently large (depending on $K$).
	
In addition, from \eqref{eq-7.Theta2} and \eqref{eq-7.K4}, we have that
\begin{equation*}
\Theta_2(K_4)>\Theta_1(K_4)=1-2\sigma;
\end{equation*}
hence, by \eqref{eq-7.K8}, we conclude that $K_8>K_4$.
\end{proof}

We point out that $K_4$ and $K_8$ only depend on $\sigma$, as it is clear from their definition (cf.,~\eqref{eq-7.Theta1}, \eqref{eq-7.Theta2}, \eqref{eq-7.K4}, and \eqref{eq-7.K8}). 

We conclude this section by stating Theorem~\ref{th-intro3} under a dual viewpoint, which consists in fixing $K>0$ and letting $\sigma\in\left(0,\frac{1}{2}\right)$ vary to obtain our multiplicity results. 

\begin{corollary}\label{th-intro3bis}
Let $\tilde{a}$ be as in~\eqref{eq-weight-intro} with fixed $K>0$. Then, there exist two values $\sigma_4(K)$, $\sigma_8(K)$, with $0<\sigma_8(K)<\sigma_4(K)$, such that the following assertions hold:
\begin{enumerate}
\item[$(i)$] for all $\sigma\in\left(0,\frac{1}{2}\right)$, there exists $\hat\lambda_1^{*}(\sigma)>0$ such that problem~\eqref{eq-aux-intro} admits at least one solution for all $\lambda>\hat\lambda_1^{*}(\sigma)$;
\item[$(ii)$] for all $\sigma\in\left(0,\frac{K}{2(K+1)}\right)$,  problem~\eqref{eq-aux-intro} admits at least two solutions for all $\lambda>\lambda^*$, where $\lambda^*$ is as in Theorem~\ref{th-intro1};
\item[$(iii)$] for all $\sigma\in(0,\sigma_4(K))$, there exists $\hat\lambda^{*}_4(\sigma)>0$ such that problem~\eqref{eq-aux-intro} admits at least four solutions for all $\lambda>\hat\lambda^{*}_4(\sigma)$;
\item[$(iv)$] for all $\sigma\in(0,\sigma_8(K))$, there exists $\hat\lambda^{*}_8(\sigma)>0$ such that problem \eqref{eq-aux-intro} admits at least eight solutions for all $\lambda>\hat\lambda^{*}_8(\sigma)$.
\end{enumerate}
\end{corollary}

\begin{proof}
Let $K>0$ be fixed. It is enough to apply the procedure exploited for the proof of Theorem~\ref{th-intro3}. Indeed, we consider the quantities $\mathcal{L}_0$ and $\mathcal{L}_1$ introduced in \eqref{eq-aux-time1} and \eqref{eq-aux-time2} respectively, which in this setting read as
\begin{equation}
\mathcal{L}_0=\frac{2\sigma}{K}, \qquad \mathcal{L}_1=\dfrac{2}{\sqrt{K\lambda}}\arctan\left(\dfrac{ \tanh(\sqrt{\lambda}\sigma)}{\sqrt{K}}\right).
\end{equation}
Our aim is to study how the previous quantities, as well as those in~\eqref{eq-7.Theta1} and~\eqref{eq-7.Theta2}, vary with respect to $\sigma$ and if they are above or below the level $1-2\sigma$. For this reason, we explicitly denote here the dependence of all these quantities on $\sigma$.
	
For point $(i)$, we observe that $\mathcal{L}_1$ is arbitrarily small for sufficiently large $\lambda$, while, for point $(ii)$, we recall that $\mathcal{L}_0>\mathcal{L}_1$, and observe that
\begin{equation*}
\mathcal{L}_0(\sigma)<1-2\sigma \quad \text{if and only if} \quad \sigma<\frac{K}{2(1+K)}.
\end{equation*}
For points $(iii)$ and $(iv)$, we have that the functions $\Theta_1(\sigma)$ and $\Theta_2(\sigma)$ satisfy
\begin{equation*}
\Theta_1(0)=0=\Theta_2(0), \qquad 0<\Theta_1(\sigma)<\Theta_2(\sigma), \quad \text{for all $\sigma\in(0,1/2)$.}
\end{equation*}
On the other hand, $1-2\sigma$ decreases from $1$ to $0$ as $\sigma$ increases from $0$ to $1/2$. Hence, there exists a unique $\sigma_4$ and a unique $\sigma_{8}$ such that
\begin{align*}
&\Theta_1(\sigma_4)=1-2\sigma_4 \quad \text{ and } \quad \Theta_1(\sigma)<1-2\sigma,\quad \text{for all $\sigma<\sigma_4$,} \\
&\Theta_2(\sigma_8)=1-2\sigma_8 \quad \text{ and } \quad \Theta_2(\sigma)<1-2\sigma,\quad \text{for all $\sigma<\sigma_8$.}
\end{align*}
At last, by reasoning as in the proof of Theorem~\ref{th-intro3}, we notice that $0<\sigma_{8}<\sigma_{4}$ and complete the proof of these statements using the behavior of the connection times shown in Section~\ref{section-4.1}.
\end{proof}

\begin{remark}\label{rem-7.1}
Both Theorem~\ref{th-intro3} and Corollary~\ref{th-intro3bis} assert that problem~\eqref{eq-aux-intro} admits at least two solutions, for all $\lambda\in[\lambda^{*},+\infty)$ provided that the mean value of the weight $\int_0^1\tilde{a}(t)\,\mathrm{d}t$ is negative. Moreover, there is a higher multiplicity, with at least four and eight solutions, provided that the mean is sufficiently negative and $\lambda$ is sufficiently large. These multiplicity results are compatible with those in \cite{BoFeSo-20,FeSo-18non}, where the existence of $3^{m}-1$ solutions of \eqref{eq-main} is obtained for a weight term $a_{\lambda,\mu}$ with $m$ intervals of positivity separated by intervals of negativity, a negative mean value, $\lambda>\lambda^{*}$ and $\mu$ sufficiently large.
\hfill$\lhd$
\end{remark}

\begin{remark}\label{rem-7.2}
We conjecture that, in Theorem~\ref{th-intro3}, $K_4$ is larger than $\frac{2\sigma}{1-2\sigma}$, which is the threshold appearing in point $(ii)$. To obtain such a result (by using the notation of the proof of Theorem~\ref{th-intro3}), it would be sufficient, for example, to prove that $\Theta_1(K)>\frac{2\sigma}{K}$, for all $K>0$, as numerical simulations suggest. In the light of Corollary~\ref{th-intro3bis}, this conjecture reads as $\sigma_4(K)<\frac{K}{2(K+1)}$, for all $K>0$.
\hfill$\lhd$
\end{remark}

\appendix
\section{Monotonicity of the time maps $T_{p}$ and $T_{l}$}\label{appendix-A}

This appendix contains the omitted details in the proofs of Propositions~\ref{pr-3.1} and~\ref{pr-3.2} concerning the properties of the time maps $T_{p}$ and $T_{l}$. To this purpose we maintain the notation introduced in Section~\ref{section-3}.

\begin{lemma}\label{le-A.1}
For every $x\in(0,x_{p})$ and $\xi\in(0,1)$, if $y(x)=-kx^{2}$, with $k>0$, and $N(x,\xi)$ is as in \eqref{eq-3.10}, then $\partial_\xi N(x,\xi)<0$.
\end{lemma}

\begin{proof}
From \eqref{eq-3.10} we obtain that
\begin{equation}\label{eq-A.1}
\begin{aligned}
\frac{\partial_\xi N(x,\xi)}{\mu(x-m(x))}
&=-(1-m'(x))g(x-(x-m(x))\xi)\\
&\quad +(x-m(x))g'(x-(x-m(x))\xi)(1-(1-m'(x))\xi).
\end{aligned}
\end{equation}
By using \eqref{eq-parabola} and \eqref{eq-3.5}, relation \eqref{eq-3.11} becomes
\begin{equation*}
m'(x)=\frac{\mu g(x)-y(x)y'(x)}{\mu g(m(x))}=\frac{\mu g(x)-2k^2x^3}{\mu g(m(x))}
=\frac{xg(x)-4\left(G(x)-G(m(x))\right)}{x g(m(x))}
\end{equation*}
and, by substituting in \eqref{eq-A.1}, we get
\begin{equation*}
\begin{aligned}
&xg(m(x))\frac{\partial_\xi N(x,\xi)}{\mu(x-m(x))}=
\\
&=\Bigl{(}x(g(x)-g(m(x)))-4(G(x)-G(m(x)))\Bigr{)} g(x-(x-m(x))\xi)
\\
&\quad +(x-m(x))g'(x-(x-m(x))\xi)\Bigl{(}x g(m(x))+(x(g(x)-g(m(x)))
\\
&\hspace{200pt} -4(G(x)-G(m(x))))\xi \Bigr{)}.
\end{aligned}
\end{equation*}
By recalling that $0<m(x)<x<1$ for all $x\in(0,1)$, the previous relation shows that, in order to prove that $\partial_\xi N(x,\xi)<0$ for all $x\in(0,x_{p})$ and $\xi\in(0,1)$, it is sufficient to show that
\begin{equation*}
\tilde N_p(x,m,\xi)<0, \quad \text{for all $\xi\in(0,1)$, $x\in(0,1)$ and $m\in(0,x)$}, 
\end{equation*}
where
\begin{align*}
&\tilde N_p(x,m,\xi):=\Bigl{(}x(g(x)-g(m))-4(G(x)-G(m))\Bigr{)}g(x-(x-m)\xi)
\\
&+(x-m)g'(x-(x-m)\xi)\Bigl{(}xg(m)+\left(x(g(x)-g(m))-4(G(x)-G(m))\right)\xi\Bigr{)}.
\end{align*}
By using the expressions of $g$ and $G$, we have
\begin{equation}\label{eq-A.3}
3\tilde N_p(x,m,\xi)=(x-m)(x-(x-m)\xi)\hat N_p(x,m,\xi),
\end{equation}
where
\begin{align}
&\hat N_p(x,m,\xi):=a(x,m)\xi^2+b(x,m)\xi+c(x,m),&
\label{eq-A.4}
\\
&a(x,m):=2 \left(x- m \right)^2 \left(3 m^3 -4 m^2 - m x - x^2\right),&
\label{eq-A.5}
\\
&b(x,m):=(x -m) (-4 m^2 + 3 m^3 - m x + 13 m^2 x - 12 m^3 x - x^2 + m x^2 + x^3), \quad&
\label{eq-A.6}
\\
&c(x,m):=x (2 m^2 - 3 m^3 - m x - 5 m^2 x + 6 m^3 x - x^2 + m x^2 + x^3).&
\label{eq-A.7}
\end{align}
The first two factors in the right-hand side of \eqref{eq-A.3} are positive in the range of the variables that we are considering. We claim that the third one is negative, and, hence, conclude the proof.

As it is apparent from \eqref{eq-A.4}, $\hat N_p(x,m,\xi)$ is a parabola in the $\xi$-variable. We aim to show that:
\begin{itemize}
\item[$(i)$] $a(x,m)<0$,
\item[$(ii)$] $c(x,m)<0$,
\item[$(iii)$] if $b(x,m)>0$, then $\Delta(x,m)<0$, where $\Delta(x,m):=b(x,m)^2-4a(x,m)c(x,m)$ is the discriminant of the parabola.
\end{itemize}
The first point says that the parabola is concave. The second one says that the parabola has a negative value for $\xi=0$; thus, if the vertex has negative abscissa, the parabola is strictly decreasing and, hence, negative, for all $\xi\in(0,1)$. On the contrary, when the vertex has positive abscissa, the first and last points guarantee that the parabola lies always below the level $0$. As a consequence, if we show $(i)$, $(ii)$ and $(iii)$, our claim is proved.

To prove $(i)$, we can equivalently show that $f_1(x,m):=3 m^3-4 m^2  - m x - x^2<0$ in the region 
\begin{equation*}
\mathcal{R}:=\bigl{\{}(x,m) \colon 0<m<x<1 \bigr{\}}.
\end{equation*}
Since $\frac{\partial f_1}{\partial x}=-m-2x<0$ in $\mathcal{R}$, the function $f_1$ does not have critical points in $\mathcal{R}$. On the boundary of $\mathcal{R}$, we have
\begin{itemize}
\item $f_1(x,0)=-x^2<0$ for all $x\in(0,1)$;
\item $f_1(1,m)=3m^3-4m^2-m-1$, which is negative for all $m\in(0,1)$; indeed, it is a cubic with positive leading coefficient, it is negative for $m=0$, and its derivative vanishes for $m=-1/9$ and $m=1$, thus it is strictly decreasing for all $m\in(0,1)$;
\item $f_1(x,x)=3x^2(x-2)<0$ for all $x\in(0,1)$;
\end{itemize}
hence $f_1$ is negative in $\mathcal{R}$.

For point $(ii)$, we consider
\begin{equation*}
f_2(x,m):=2 m^2 - 3 m^3 - m x - 5 m^2 x + 6 m^3 x - x^2 + m x^2 + x^3
\end{equation*}
and, as above, we prove, first of all, that it does not have critical points in $\mathcal{R}$. Indeed, the relation $\frac{\partial f_2}{\partial x}=\frac{\partial f_2}{\partial m}$ has two roots
\begin{equation*}
x^{\pm}(m):=\frac{1}{4}\Bigl{(}1-12m+18m^2\pm\sqrt{\Delta_2(m)}\Bigr{)},
\end{equation*}
where $\Delta_2(m):=1 + 16 m + 148 m^2 - 480 m^3 + 324 m^4$. Moreover, $\Delta_2(m)$ has two zeros $m^\pm$ in $(0,1)$ with $\frac{1}{2} < m^- < m^{+} < 1$, and is positive in $(0,1)\setminus[m^-,m^{+}]$. Observe that $x^-(m)<x^{+}(m)$, and, in addition, it is easy to show that:
\begin{itemize}
	\item $x^-(m)<0$ for every $0<m<\frac{1}{2}$;
	\item $x^{+}(m)\leq\frac{1}{2}\leq m$ for every $\frac{1}{2}\leq m<m^-$;
	\item $x^-(m)>1$ for every $m^{+}\leq m<1$.
\end{itemize}
Therefore, for $\frac{\partial f_2}{\partial x}(x,m)=\frac{\partial f_2}{\partial m}(x,m)$ to hold in $\mathcal{R}$, $m$ has to lie in $\left(0,\frac{1}{2}\right)$, and $x$ has to be equal to $x^{+}(m)$. Moreover, the relation $\frac{\partial f_2}{\partial x}=0$ has only one positive root for $0<m<1$, which is
\begin{equation*}
x_0(m):=\frac{1}{3}\Bigl{(}1-m+\sqrt{1+m+16m^2-18m^3}\Bigr{)},
\end{equation*}
and it holds that $x_0(m)>x^{+}(m)$ for all $0<m<\frac{1}{2}$. This shows that $f_2$ cannot admit critical points in $\mathcal{R}$. As a second step, we study $f_2$ on the boundary of $\mathcal{R}$:
\begin{itemize}
\item $f_2(x,0)=x^2(x-1)<0$ for all $x\in(0,1)$;
\item $f_2(1,m)=3m^2(m-1)<0$ for all $m\in(0,1)$;
\item $f_2(x,x)=6x^3(x-1)<0$ for all $x\in(0,1)$.
\end{itemize}
This concludes the proof of point $(ii)$.

We pass to $(iii)$, thus we assume that $b(x,m)>0$, and, from \eqref{eq-A.6}, we get
\begin{equation*}
f_3(x,m):=-4 m^2 + 3 m^3 - m x + 13 m^2 x - 12 m^3 x - x^2 + m x^2 + x^3>0.
\end{equation*}
Our goal is to prove that $\Delta(x,m)<0$. Direct computations give
\begin{equation*}
\Delta(x,m)=(x-m)^2\delta_1(x,m)\delta_2(x,m),
\end{equation*}
where
\begin{align*}
\delta_1(x,m)&:=-4 m^2 + 3 m^3 - m x + m^2 x - x^2 + m x^2 + x^3, \\
\delta_2(x,m)&:=-4 m^2 + 3 m^3 - m x + 9 m^2 x - x^2 + 9 m x^2 + 9 x^3,
\end{align*}
thus we want to show that $\delta_1$ and $\delta_2$ have opposite sign. This follows since
\begin{equation}\label{eq-A.10}
\delta_1(x,m)<0 \quad \text{and} \quad 0< f_3(x,m)<\delta_2(x,m) \quad \text{in $\mathcal{R}$}.
\end{equation}
Indeed, for the first inequality, we recall that the discriminant of a cubic equation $ax^3+bx^2+cx+d$ is given by $18 abcd-4b^3d+b^2c^2-4ac^3-27a^2d^2$. Therefore, the discriminant of the cubic equation $x\mapsto\delta_1(x,m)$ is
\begin{equation*}
-3m^2 (5 + 4 m + 108 m^2 - 176 m^3 + 68 m^4),
\end{equation*}
which is negative since
\begin{equation*}
108 m^2 - 176 m^3 + 68 m^4=m^{2}(108 - 176 m + 68 m^2) > 0,
\end{equation*}
for all $m\in(0,1)$. This shows that the cubic $x\mapsto\delta_1(x,m)$ has a unique real root, which is greater than $1$, since $\delta_1(1,m)=3m^2(m-1)<0$ and $\delta_1(x,m)\to+\infty$ as $x\to+\infty$. This shows the first relation of \eqref{eq-A.10}. 
The second relation is equivalent to $4x (-m^2 + 3 m^3 + 2 m x + 2 x^2 )>0$. If we denote
\begin{equation*}
f_{4}(x,m):=-m^2 + 3 m^3 + 2 m x + 2 x^2,
\end{equation*}
we have that $\frac{\partial f_{4}}{\partial x}=4x+2m>0$, thus $f_{4}(x,m)$ has no critical points in $\mathcal{R}$. Moreover, on the boundary of such a region we have:
\begin{itemize}
\item $f_{4}(x,0)=2x^2>0$ for all $x\in(0,1)$;
\item $f_{4}(1,m)=2+2m-m^2+3m^3>2-m^{2}>0$ for all $m\in(0,1)$;
\item $f_{4}(x,x)=3x^2(x+1)>0$ for all $x\in(0,1)$.
\end{itemize}
Therefore, the second relation in \eqref{eq-A.10} follows, and the proof is complete.
\end{proof}

\begin{lemma}\label{le-A.2}
For every $x\in(x_{l},1)$ and $\xi\in(0,1)$, if $y(x)=-k(1-x)$, with $k>0$, and $N(x,\xi)$ is as in \eqref{eq-3.10}, then $\partial_\xi N(x,\xi)>0$.
\end{lemma}

\begin{proof}
By using \eqref{eq-line} and \eqref{eq-3.5}, relation \eqref{eq-3.11} now gives
\begin{align*}
m'(x)&=\frac{\mu g(x)-y(x)y'(x)}{\mu g(m(x))}=\frac{\mu g(x)+k^2(1-x)}{\mu g(m(x))}
\\ &=\frac{(1-x)g(x)+2\left(G(x)-G(m(x))\right)}{(1-x) g(m(x))}.
\end{align*}
and, by substituting in \eqref{eq-A.1}, we get
\begin{align*}
&(1-x)g(m(x))\frac{\partial_\xi N(x,\xi)}{\mu(x-m(x))}=
\\&= g(x-(x-m(x))\xi)\Bigl{(}(1-x)(g(x)-g(m(x)))+2(G(x)-G(m(x)))\Bigr{)}
\\&\quad +(x-m(x))g'(x-(x-m(x))\xi)\cdot
\\&\qquad \cdot\Bigl{(}(1-x) g(m(x))+((1-x)(g(x)-g(m(x)))+2(G(x)-G(m(x))))\xi\Bigr{)}.
\end{align*}
Thus, as in Lemma~\ref{le-A.1}, if we set
\begin{align*}
\tilde N_l(x,m,\xi)
&:=g(x-(x-m)\xi) \Bigl{(}(1-x)(g(x)-g(m))+2(G(x)-G(m))\Bigr{)}
\\
&\quad +(x-m)g'(x-(x-m)\xi)\cdot
\\
&\qquad \cdot\Bigl{(}(1-x)g(m)+((1-x)(g(x)-g(m))+2(G(x)-G(m)))\xi\Bigr{)},
\end{align*}
it is sufficient to show that $\tilde N_l(x,m,\xi)>0$ for all $\xi\in(0,1)$, $x\in(0,1)$ and $m\in(0,x)$. By using the expressions of $g$ and $G$, we have
\begin{equation}\label{eq-A.11}
-6\tilde N_l(x,m,\xi)=(x-m)(x-(x-m)\xi)\hat N_l(x,m,\xi),
\end{equation}
where
\begin{align*}
&\hat N_l(x,m,\xi):=a(x,m)\xi^2+b(x,m)\xi+c(x,m),
\\
&a(x,m):=2 (x- m)^2 \bigl{(}-6 m + 2 m^2 + 3 m^3 - 6 x + 8 m x - 3 m^2 x + 8 x^2 
\\ &\hspace{263pt} - 3 m x^2 - 3 x^3\bigr{)},
\\
&b(x,m):=(x-1) (x-m) \bigl{(}6 m + 16 m^2 - 21 m^3 + 6 x - 8 m x + 3 m^2 x - 8 x^2 
\\ &\hspace{263pt} + 3 m x^2 + 3 x^3\bigr{)},
\\
&c(x,m):=(x-1) \bigl{(}12 m^2 - 12 m^3 + 6 m x - 20 m^2 x + 15 m^3 x + 6 x^2 -8 m x^2
\\ &\hspace{195pt} + 3 m^2 x^2 - 8 x^3 + 3 m x^3 + 3 x^4\bigr{)}.
\end{align*}
The first two factors in \eqref{eq-A.11} are positive; the third one can be shown to be negative for the values of the parameters that we are considering with similar arguments as in the proof of Lemma~\ref{le-A.1}, i.e., by proving that $a(x,m)$ and $c(x,m)$ are negative and, here, that $b(x,m)$ is negative too.
\end{proof}

\section{Behavior of $T_1(s)$ near $s=0$}\label{appendix-B}

The purpose of this appendix is proving that the connection time $T_1(s)$ defined in \eqref{def-T1_4.2} is increasing in a right neighborhood of $s=0$. To do so, we will show that
\begin{equation}\label{eq:5_T1'(0)}
\lim_{s\to 0^{+}}T_1'(s)>0.
\end{equation}
As a consequence, all the subsequent expansions are meant to be valid for $s\to 0^{+}$. Starting from \eqref{def-T1_4.2}, direct transformations give
\begin{align}
T_1(s)
&=2\int_{0}^{u_s(\sigma)-m(u_s(\sigma))}\frac{\mathrm{d}\tilde{u}}{\sqrt{\left(v_s(\sigma)\right)^2+2\mu\left(G(\tilde{u}+m(u_s(\sigma)))-G(u_s(\sigma))\right)}} \\
&=2\int_{0}^{1}\frac{\mathrm{d}\xi}{\sqrt{\frac{\left(v_s(\sigma)\right)^2+2\mu\left(G(\left(u_s(\sigma)-m(u_s(\sigma))\right)\xi+m(u_s(\sigma)))-G(u_s(\sigma))\right)}{\left(u_s(\sigma)-m(u_s(\sigma))\right)^2}}}. \label{def-T1_App}
\end{align}
To achieve our goal, we will establish some asymptotic expansions of the quantities appearing in expression \eqref{def-T1_App}. We start the investigation with the terms depending only on the problem in the interval $[0,\sigma]$.
We start by proving that
\begin{equation}\label{eq:u_s-2nd_order-unif}
u_s(t)=s-\frac{\lambda t^2}{2}s^2+o(s^2), \qquad \text{for all $t\in [0,\sigma]$.}
\end{equation}
The first term in the development has already been proved in \eqref{eq-u_s-unif}. Moreover, from \eqref{eq-2.10}, for all $t\in [0,\sigma]$, we deduce that
\begin{equation}\label{eq:u_s-2nd_order_int}
\sqrt{2\lambda}t
=\int_{u_s(t)-s}^0\frac{\mathrm{d}\tilde{u}}{\sqrt{G(s)-G(\tilde{u}+s)}}
=\int_{\frac{u_s(t)-s}{s^2}}^0\frac{\mathrm{d}\xi}{\sqrt{\frac{G(s)-G(s^2\xi+s)}{s^4}}}.
\end{equation}
Since the radicand converges, for $\xi<0$, to $-\xi$ as $s\to 0^{+}$,	if we set
\begin{equation*}
l_2^-:=\liminf_{s\to 0^{+}} \frac{u_s(t)-s}{s^2}
\end{equation*}
and take the $\liminf$ in \eqref{eq:u_s-2nd_order_int}, the dominated convergence theorem gives that $l_2^-$ satisfies
\begin{equation} 
\sqrt{2\lambda} \, t
=\int_{l_2^-}^0\frac{\mathrm{d}\xi}{\sqrt{-\xi}}
=2\sqrt{-l_2^-}, \qquad \text{for all $t\in [0,\sigma]$.}
\end{equation}
Reasoning similarly with the $\limsup$, we obtain that
\begin{equation*} 
\lim_{s\to 0^{+}} \frac{u_s(t)-s}{s^2}=-\frac{\lambda t^2}{2}, \qquad \text{for all $t\in [0,\sigma]$.}
\end{equation*}
Therefore, \eqref{eq:u_s-2nd_order-unif} holds true.

The next step consists in proving that
\begin{equation} \label{eq:v_s-3_ord}
v_s(\sigma)=-\lambda\sigma s^2+\frac{\lambda\sigma}{3}\left(3+\lambda\sigma^2\right)s^3+o(s^3).
\end{equation}
We notice that the first term in the development has already been proved in~\eqref{limit-v_s-to-0}.
Moreover, we observe that~\eqref{eq:v_s-3_ord} is equivalent to
\begin{equation*} 
\lim_{s\to 0}\frac{u_s'(\sigma)+\lambda\sigma s^2}{s^3}=\frac{\lambda\sigma}{3}\left(3+\lambda\sigma^2\right).
\end{equation*}
By integrating~\eqref{eq-initial0} in $[0,\sigma]$, we obtain
\begin{equation*} 
\frac{u_s'(\sigma)+\lambda\sigma s^2}{s^3}=-\lambda \int_0^\sigma \frac{g\left(u_s(t)\right)-s^2}{s^3}\,\mathrm{d}t,
\end{equation*}
and, thanks to \eqref{eq:u_s-2nd_order-unif}, the integrand converges to $-1-\frac{\lambda t^2}{2}$, as $s\to 0^{+}$. Thus, the dominated convergence theorem ensures that
\begin{equation*} 
\lim_{s\to 0^{+}}\frac{u_s'(\sigma)+\lambda\sigma s^2}{s^3}=-\lambda \int_0^\sigma \left(-1-\frac{\lambda t^2}{2}\right)\,\mathrm{d}t=\lambda\left(\sigma+\frac{\lambda\sigma^3}{3}\right),
\end{equation*}
as desired.
Finally, to conclude this analysis in $[0,\sigma]$, we will improve the expansion \eqref{eq:u_s-2nd_order-unif} of $u_s(\sigma)$ by computing the value of $\alpha_3$ such that
\begin{equation} \label{eq:u_s-3_ord}
u_s(\sigma)=s-\frac{\lambda\sigma^2}{2}s^2+\alpha_3s^3+o(s^3).
\end{equation}
First of all, we observe that such value $\alpha_3$ exists since the differentiable dependence theorem with respect to parameters ensures that the function $s\mapsto u_s(\sigma)$ is $\mathcal{C}^{\infty}$. Then, we recall the energy conservation in $[0,\sigma]$, which reads
\begin{equation*}
(v_s(\sigma))^2+2\lambda G(u_s(\sigma))=2\lambda G(s)=2\lambda\biggl{(}\frac{s^3}{3}-\frac{s^4}{4}\biggr{)},
\end{equation*}
and, by considering \eqref{eq:u_s-3_ord} and \eqref{eq:v_s-3_ord}, we compute the corresponding expansion up to the fifth order, and so
\begin{equation*}
\alpha_3=\frac{\lambda\sigma^2}{12}\left(6+\lambda\sigma^2\right).
\end{equation*}
For notational convenience, we denote the coefficients in \eqref{eq:u_s-3_ord} and \eqref{eq:v_s-3_ord} as follows:
\begin{equation}\label{eq:coefficients}
\alpha_2:=-\frac{\lambda\sigma^2}{2}, \qquad \beta_2:=-\lambda\sigma, \qquad \beta_3:=\frac{\lambda\sigma}{3}\left(3+\lambda\sigma^2\right).
\end{equation}

We pass to study $m(u_s(\sigma))$, which, due to the energy conservation in $[\sigma,1-\sigma]$, satisfies 
\begin{equation}\label{eq:en2}
\left(v_s(\sigma)\right)^2-2\mu G(u_s(\sigma))=-2\mu G(m(u_s(\sigma))),
\end{equation}
and, again thanks to the differentiable dependence theorem with respect to parameters, it is a regular function of $s$; thus it can be written as
\begin{equation*}
m(u_s(\sigma))=m_1 s + m_2 s^2 + m_3 s^3 + o(s^3),
\end{equation*}
for some coefficients $m_1,m_2,m_3\in\mathbb{R}$. By plugging this expansion in \eqref{eq:en2}, using \eqref{eq:u_s-3_ord}, \eqref{eq:v_s-3_ord}, and \eqref{eq:coefficients} and equaling the coefficients of the same powers of $s$, we obtain
\begin{align}
m_1&=1, \\
m_2&=\alpha_2-\frac{\beta_2^2}{2\mu}=-\frac{\lambda(\lambda+\mu)\sigma^2}{2\mu}, \label{eq:m2} \\
m_3&=m_2-m_2^2-\alpha_2+\alpha_2^2+\alpha_3-\frac{\beta_2\beta_3}{\mu}
\\
&=\frac{\lambda^2\sigma^2}{12\mu^2}\left(6\mu-3\lambda^2\sigma^2-2\lambda\mu\sigma^2\right)+\frac{\lambda\sigma^2}{12}\left(6+\lambda\sigma^2\right).
\end{align}
We now set $N(s,\xi)$ to be the numerator in the square root in~\eqref{def-T1_App}, and $D(s)$ the denominator. Then, by differentiating \eqref{def-T1_App} with respect to $s$, we obtain
\begin{equation}\label{eq:T1'}
T_1'(s)=-\int_0^1 \left(\frac{N(s,\xi)}{D(s)}\right)^{\!-\frac{3}{2}}\frac{\partial_s N(s,\xi) D(s)-N(s,\xi)D'(s)}{\left(D(s)\right)^2} \,\mathrm{d}\xi.
\end{equation}
The expansions determined above give
\begin{equation}\label{eq:ND}
N(s,\xi)=n_4(\xi)s^4+n_5(\xi)s^5+o(s^5), \qquad D(s)=d_4s^4+d_5s^5+o(s^5),
\end{equation}
where
\begin{align}
n_4(\xi)&=\beta_2^2+2\mu\bigl{(}(\alpha_2-m_2)\xi+m_2-\alpha_2\bigr{)}=2\mu(\alpha_2-m_2)\xi=\lambda^2\sigma^2\xi, \\
n_5(\xi)&=2\beta_2\beta_3+2\mu\Bigl{(}\bigl{(}(\alpha_2-m_2)\xi+m_2\bigr{)}^2-\bigl{(}(\alpha_2-m_2)\xi+m_2\bigr{)}+(\alpha_3-m_3)\xi+m_3\Bigr{)} \\[-1em]
	&\quad -2\mu(\alpha_2^2-\alpha_2+\alpha_3), \\
d_4&=(\alpha_2-m_2)^2, \\
d_5&=2(\alpha_2-m_2)(\alpha_3-m_3),
\end{align}
(in the computation for $n_4(\xi)$ we have used \eqref{eq:m2}). Thus, the first factor in the integral of \eqref{eq:T1'} satisfies
\begin{equation*}
\lim_{s\to 0^{+}} \left(\frac{N(s,\xi)}{D(s)}\right)^{\!-\frac{3}{2}}=\left(\frac{\lambda\sigma}{2\mu}\right)^{\!3}\xi^{-\frac{3}{2}}.
\end{equation*}
Regarding the second factor, by differentiating the expressions \eqref{eq:ND} with respect to $s$, we see that the coefficient of $s^7$ in the numerator vanishes (it is equal to $4n_4(\xi)d_4-4n_4(\xi)d_4$), and the coefficient of $s^8$ is
\begin{equation*}
n_5(\xi)d_4-n_4(\xi)d_5=\frac{\lambda ^7 \xi  \sigma ^8
(3 \lambda  (\xi -3)-8 \mu)}{24 \mu ^3}.
\end{equation*}
As a consequence, by passing to the limit in \eqref{eq:T1'}, we obtain
\begin{align}
\lim_{s\to 0^{+}}T_1'(s)&=-\left(\frac{\lambda\sigma}{2\mu}\right)^3
\frac{\lambda ^7   \sigma ^8}{24 \mu ^3}\frac{16
\mu ^4}{\lambda ^8 \sigma ^8}
\int_0^1\xi^{-\frac{1}{2}}(3 \lambda  (\xi -3)-8 \mu)\,\mathrm{d}\xi \\
&=-\frac{\lambda^2\sigma^3}{12\mu^2}\left(3\lambda\int_0^1\sqrt{\xi}\,\mathrm{d}\xi-(9\lambda+8\mu)\int_0^1\frac{\mathrm{d}\xi}{\sqrt{\xi}}\right)\\
&=\frac{4\lambda^2(\lambda+\mu)\sigma^3}{3\mu^2}>0,
\end{align}
as desired.

\begin{remark}\label{rem-B.1}
We observe that, with the same asymptotic expansions, by passing to the limit in \eqref{def-T1_App}, we can recover the value of $\lim_{s\to 0^+} T_1(s)$ obtained in \eqref{eq-aux-time1}.
\hfill$\lhd$
\end{remark}

\section*{Acknowledgments}

This work has been designed during the visit of A.T. and E.S. at the Department of Mathematics, Computer Science and Physics of the University of Udine, whose members they thank for the very warm hospitality.

\bibliographystyle{elsart-num-sort}
\bibliography{FeSoTe-biblio}

\end{document}